\documentclass[11pt]{amsart}
\usepackage[utf8]{inputenc}

\usepackage[margin=1.2in]{geometry}
\usepackage{comment}
\usepackage{amsmath, bm}
\usepackage{amsfonts}
\usepackage{amssymb}
\usepackage{amsthm}
\usepackage{mathtools}
\usepackage{caption}
\usepackage{subcaption}
\usepackage{bbm}
\usepackage[export]{adjustbox}
\usepackage{mathrsfs}
\usepackage[normalem]{ulem}
\usepackage{accents}
\usepackage{subfiles}
\usepackage[shortlabels]{enumitem}
\usepackage{pifont}
\usepackage{float}
\usepackage[x11names]{xcolor}
\usepackage{tensor}
\usepackage{leftindex}

\usepackage{todonotes}

\usepackage[all]{xy}

\usepackage{stmaryrd}

\usepackage{bbold}

\usepackage{tikz-cd}
\usetikzlibrary{matrix}
\usepackage{graphicx} 
\usepackage{epstopdf}
\usetikzlibrary{positioning}

\usepackage[linktocpage]{hyperref}
\hypersetup{
    colorlinks=true,
    linkcolor=blue,
    citecolor=blue,      
    urlcolor=blue,
}

\usepackage{color}
\usepackage{romannum}
\usepackage[new]{old-arrows}

\definecolor{note}{rgb}{0,0,1} 

\theoremstyle{plain}
\newtheorem{theorem}{Theorem}
\newtheorem{proposition}[theorem]{Proposition}
\newtheorem{lemma}[theorem]{Lemma}
\newtheorem{claim}[theorem]{Claim}
\newtheorem{corollary}[theorem]{Corollary}

\newtheorem{conjecture}[theorem]{Conjecture}
\newtheorem{question}[theorem]{Question}

\theoremstyle{definition}
\newtheorem{definition}[theorem]{Definition}
\newtheorem{constr}[theorem]{Construction}
\newtheorem{remark}[theorem]{Remark}

\numberwithin{equation}{section}
\numberwithin{theorem}{section}
\usepackage[english]{babel}

\newcommand{\op}{\operatorname}

\usepackage{hyphenat}

\usepackage[backend=bibtex,style=alphabetic,maxalphanames=5,maxnames=5]{biblatex}

\def\C{\mathbb C}
\def\R{\mathbb R}
\def\Q{\mathbb Q}
\def\Z{\mathbb Z}

\def\cyl{\mathcal{C}_b}
\def\cylcomp{\overline{\mathcal{C}}_b}
\def\cylsm{\widehat{\mathcal{C}}_b}

\def\hurwitzaction{\widetilde{\mathcal{H}}^{\bm \mu, \bm \mu'}_{\kappa, \chi}}
\def\hurwitz{\mathcal{H}^{\bm \mu, \bm \mu'}_{\kappa, \chi}}
\def\hurwitzstandard{\accentset{\circ}{\mathcal{H}}^{\bm \mu, \bm \mu'}_{\kappa, \chi}}
\def\hurwitzstandardcomp{\accentset{\circ}{\overline{\mathcal{H}}}^{\bm \mu, \bm \mu'}_{\kappa, \chi}}
\def\hurwitzcomp{\overline{\mathcal{H}}^{\bm \mu, \bm \mu'}_{\kappa, \chi}}

\def\hurwitzsm{\widehat{\mathcal{H}}^{\bm \mu, \bm \mu'}_{\kappa, \chi}}
\def\hurwitzstandard{\accentset{\circ}{\mathcal{H}}^{\bm \mu, \bm \mu'}_{\kappa, \chi}}
\def\hurwitzstandardcomp{\accentset{\circ}{\overline{\mathcal{H}}}^{\bm \mu, \bm \mu'}_{\kappa, \chi}}

\def\hurwitzbr{\mathscr{H}^{\bm \mu, \bm \mu'}_{\kappa, \chi}}
\def\hurwitzbraction{\widetilde{\mathscr{H}}^{\bm \mu, \bm \mu'}_{\kappa, \chi}}
\def\hurwitzcompbr{\overline{\mathscr{H}}^{\bm \mu, \bm \mu'}_{\kappa, \chi}}
\def\hurwitzsmbr{\widehat{\mathscr{H}}^{\bm \mu, \bm \mu'}_{\kappa, \chi}}

\def\hurwitzactionbr{\widetilde{\mathscr{H}}^{\bm \mu, \bm \mu'}_{\kappa, \chi}}

\def\hurwitzactionunsym{\widetilde{\mathcal{H}}^{\sigma, \sigma'}_{\kappa, \chi}}

\def\hurwitzcompunsym{\overline{\mathcal{H}}^{\sigma, \sigma'}_{\kappa, \chi}}
\def\hurwitzsmunsym{\widehat{\mathcal{H}}^{\sigma, \sigma'}_{\kappa, \chi}}
\def\hurwitzsmbrunsym{\widehat{\mathscr{H}}^{\sigma, \sigma'}_{\kappa, \chi}}
\def\hurwitzbrunsym{\mathscr{H}^{\sigma, \sigma'}_{\kappa, \chi}}

\def\modulifloerbr{\mathscr{M}^\chi(\bm x, \bm x'; H^{\mathrm{tot}}, J)}
\def\modulifloerbraction{\widetilde{\mathscr{M}}^\chi(\bm x, \bm x'; H^{\mathrm{tot}}, J)}
\def\modulifloercompbr{\overline{\mathscr{M}}^\chi(\bm x, \bm x'; H^{\mathrm{tot}}, J)}
\def\modulifloerbrshort{\mathscr{M}^\chi(\bm x, \bm x')}
\def\modulifloercompbrshort{\overline{\mathscr{M}}^\chi(\bm x, \bm x')}
\def\modulifloerbrKR{\mathscr{M}^{\chi, \sharp}(\bm x, \bm x'; H^{\mathrm{tot}}, J)}
\def\modulifloerbrclosedKR{\mathscr{M}^{\chi', \sharp}}

\def\modulifloerbrunsym{\mathscr{M}_{\mathrm{unsym}}^\chi(\bm x, \bm x'; H^{\mathrm{tot}}, J)}
\def\modulifloerbractionunsym{\widetilde{\mathscr{M}}_{\mathrm{unsym}}^\chi(\bm x, \bm x'; H^{\mathrm{tot}}, J)}

\def\hurwitzhalf{\mathcal{H}^{\sigma, \sigma_0, \ge 0}_{\kappa, \chi}}
\def\hurwitzhalfsm{\widehat{\mathcal{H}}^{\sigma, \sigma_0, \ge 0}_{\kappa, \chi}}
\def\hurwitzhalfcomp{\overline{\mathcal{H}}^{\sigma, \sigma_0, \ge 0}_{\kappa, \chi}}
\def\hurwitzhalfbrsm{\widehat{\mathscr{H}}^{\sigma, \sigma_0, \ge 0}_{\kappa, \chi}}

\renewbibmacro{in:}{}

\DeclareDelimFormat[bib,biblist]{nametitledelim}{\addcomma\space}

\DeclareFieldFormat*{title}{\mkbibitalic{#1}\addcomma}
\DeclareFieldFormat*{journaltitle}{#1}
\DeclareFieldFormat*{volume}{\mkbibbold{#1}}
\DeclareFieldFormat{pages}{#1}
\DeclareFieldFormat[misc]{date}{preprint {#1}}
\DeclareFieldFormat{mr}{%
  MR\addcolon\space
  \ifhyperref
    {\href{http://www.ams.org/mathscinet-getitem?mr=MR#1}{\nolinkurl{#1}}}
    {\nolinkurl{#1}}}
    
\AtEveryBibitem{
  \clearfield{url}
  \clearfield{number}
  \clearfield{doi}
  \clearfield{issn}
  \clearfield{isbn}
  \clearfield{eprintclass}
}
\AtEveryBibitem{\ifentrytype{book}{\clearfield{pages}}{}}

\bibliography{hecke}
\usepackage{fancyhdr}
\usepackage{nomencl} 
\makenomenclature

\nomenclature[11]{$\widetilde{\mathcal{C}}_{b}$}{set of stable cylinders, p.~\pageref{definition-stable-cylinder}}
\nomenclature[12]{$\mathcal{C}_{b}$}{space of stable cylinders up to equivalence, p.~\pageref{stable-cylinders-equivalence}}
\nomenclature[15]{$\widetilde{\mathcal{C}}^{\operatorname{sm}}_{b,n}$}{set of smooth cylindrical buildings of height $n$, p.~\pageref{vertical-automorphism-quotient}}
\nomenclature[14]{$\widehat{\mathcal{C}}_{b}$}{smooth cylindrical buildings up to vertical translations, p.~\pageref{eq: smooth-buildings-moduli}}
\nomenclature[13]{$\overline{\mathcal{C}}_{b}$}{moduli space of cylindrical buildings, p.~\pageref{eq: moduli-cylindrical-buildings}}
\nomenclature[01]{$\hurwitzaction$}{Hurwitz space, p.~\pageref{def-hurwitz-space}}
\nomenclature[02]{$\hurwitz$}{the quotient of Hurwitz space by the vertical translations, p.~\pageref{def-hurwitz-space}}
\nomenclature[03]{$\hurwitzcomp$}{moduli space of branched cover buildings, p.~\pageref{hurwitz-compactification}}
\nomenclature[04]{$\hurwitzsm$}{moduli space of smooth branched cover buildings, p.~\pageref{defn-hurwitz-smooth}}
\nomenclature[05]{$\hurwitzsmbr$}{wnbc groupoid modification of $\hurwitzsm$, p.~\pageref{hurwitz-smooth-branched}}
\nomenclature[06]{$\hurwitzbr$}{the open stratum of $\hurwitzsmbr$, p.~\pageref{hurwitz-smooth-branched-open}}
\nomenclature[07]{$\hurwitzactionunsym$}{Hurwitz space for the unsymmetrized version, p.~\pageref{def-hurwitz-space-unsym}}
\nomenclature[21]{$\mathscr{M}^\chi(\bm x, \bm x'; H^{\operatorname{tot}}, J)$}{weighted branched groupoid of pseudo-holomorphic curves ``connecting" $\bm x$ and $\bm x'$, p.~\pageref{moduli-floer-branched}}
\nomenclature[31]{$\mathfrak{X}_{\operatorname{sw}}(\ell, \bm h, \bm \tau, \bm \rho)$}{space of perturbation data for MFLS with $\ell$ switchings of type $(\bm h, \bm \tau, \bm \rho)$, p.~\pageref{eqn: sw1}}
\nomenclature[32]{$\mathfrak{X}^{\op{back}}_{\operatorname{sw}}(\ell, \bm h, \bm \tau, \bm \rho)$}{space of background perturbation data for MFLS with $\ell$ switchings of type $(\bm h, \bm \tau, \bm \rho)$, p.~\pageref{eqn: sw2}}
\nomenclature[39]{$\widetilde{\mathcal{M}}_\ell(\bm \gamma,\bm \gamma'; \bm Y, \bm h, \bm \tau, \bm \rho)$}{moduli space of Morse flow lines with $\ell$ switchings of type $(\bm h, \bm \tau , \bm \rho)$ ``connecting" $\bm \gamma$ and $\bm \gamma'$, p.~\pageref{def-morse-diff}}
\nomenclature[41]{$\mathcal{T}^{\sigma_0}_{\chi, \ell}(\bm x,\bm \gamma; H^{\mathrm{tot}}, J, \bm Y, \bm \tau, \bm h, \bm \rho)$}{mixed moduli space ``connecting" $\bm x$ and $\bm \gamma$, p.~\pageref{mixed-moduli-tuples}}

\title{Heegaard Floer Symplectic homology and Viterbo's isomorphism theorem in the context of multiple particles}
\author{Roman Krutowski}
\address{University of California, Los Angeles, Los Angeles, CA 90095}
\email{romakrut@math.ucla.edu}\urladdr{}
\author{Tianyu Yuan}
\address{School of Mathematical Sciences, Eastern Institute of Technology, Ningbo, Zhejiang, 315200, China}
\email{tyyuan@eitech.edu.cn} \urladdr{}

\date{\today}
\subjclass[2010]{Primary 53D40; Secondary 57R17.}
\begin{document}
\pagenumbering{arabic}
\begin{abstract}
Given a Liouville manifold $M$, we introduce an invariant of $M$ that we call the \emph{Heegaard Floer symplectic cohomology} $SH^*_\kappa(M)$ for any $\kappa \ge 1$ that coincides with the symplectic cohomology for $\kappa=1$. Writing $\hat{M}$ for the completion of $M$, the differential counts pseudoholomorphic curves of arbitrary genus in $\R \times S^1 \times \hat{M}$ that are required to be branched $\kappa$-sheeted covers when projected to the  $\R \times S^1$-direction; this resembles the cylindrical reformulation of Heegaard Floer homology by Lipshitz. These cohomology groups provide a closed-string analogue of higher-dimensional Heegaard Floer homology introduced by Colin, Honda, and Tian.
  
When $\hat{M}=T^*Q$ with $Q$ an orientable manifold, we introduce a Morse-theoretic analogue of Heegaard Floer symplectic cohomology, which we call the \emph{free multiloop complex} of $Q$. When $Q$ has vanishing relative second Stiefel-Whitney class, we prove a generalized version of Viterbo's isomorphism theorem by showing that the cohomology groups $SH^*_\kappa(T^*Q)$ are isomorphic to the cohomology groups of the free multiloop complex of $Q$.
\end{abstract}
\maketitle
\tableofcontents

\section{Introduction}
Since its introduction in Viterbo's foundational paper \cite{viterbo1999functors}, symplectic cohomology has emerged as a powerful and useful invariant of Liouville domains. Let $(M,\lambda)$ be a Liouville domain, and $\hat{M}=M \cup_{\partial M} ([1, \infty) \times \partial M)$ be its completion. One way to define the symplectic cohomology of $M$ (or of $\hat{M}$) is to count pseudoholomorphic cylinders connecting closed orbits of the usual one-particle motion Hamiltonian system in the completion $\hat{M}$ according to \cite{seidel2006biased, abouzaidseidel2010}. In this paper, we construct an invariant of $M$ where, instead of one-particle motion, we consider the dynamics of multiple identical particles in $\hat{M}$. The closed orbits then can be regarded as loops in the unordered configuration space $\operatorname{UConf}_\kappa(\hat{M})$, where $\kappa \ge 1$ is the number of particles under consideration. 

Naturally, one would want to count pseudoholomorphic cylinders in the symmetric product $\operatorname{Sym}^\kappa(\hat{M})$ connecting orbits as above. The main obstacle is that, in general,
the symmetric product of a manifold is not a smooth manifold. One strategy is to study Floer theory of the symplectic orbifold $\operatorname{Sym}^\kappa(\hat{M})$ and it has been undertaken in \cite{maksmith2021}, based on the approach from \cite{chopoddar2014},  for $\kappa=2$ to obtain non-displaceability results for Lagrangian links in $S^2 \times S^2$. Instead, in the present work we study moduli spaces of curves in $\R \times S^1 \times \hat{M}$ mimicking Lipshitz's cylindrical reformulation \cite{lipshitz2006cylindrical} of the Heegaard Floer homology. This provides a closed string analogue of higher-dimensional Heegaard Floer homology (HDHF) introduced by Colin, Honda, and Tian in \cite{colin2020applications}. We recall that HDHF provides a model for the Lagrangian Floer homology of the symmetric product $\operatorname{Sym}^{\kappa}(M)$.

In this paper, we present two closely related versions of such a sought-after invariant which we refer to as the \emph{unsymmetrized version} $SH^*_{\kappa, \mathrm{unsym}}(M)$ of Heegaard Floer symplectic cohomology (HFSH) and the \emph{partially symmetrized version} $SH^*_{\kappa, \mathrm{psym}}(M)$ of HFSH. It is expected that these two versions have a clear algebraic relation between their homology, but we do not present it here. Ideally, one would want to ``fully symmetrize" and work with symplectic cohomology of the orbifold symplectic product as in \cite{MSS2025orbifold}, but we do not pursue this approach here.
 
\begin{theorem}[Theorems~\ref{theorem-d2},~\ref{theorem-linear-to-quadratic} and~\ref{theorem-d2-unsym}]\label{thm-main}
    The Heegaard Floer symplectic cohomology groups $SH^*_{\kappa, \mathrm{psym}}(M)$ and $SH^*_{\kappa, \mathrm{unsym}}(M)$ are well defined and are invariants of the Liouville domain $M$, independent of all the intrinsic choices of the Floer data required for its setup.
\end{theorem}

The main obstacle for translating the HDHF setup to the closed string context is that unlike in \cite{colin2020applications} the tuples of orbits of cumulative time $\kappa$ may contain multiply covered orbits. This may lead to the multiply covered curves in the boundary of moduli spaces of solutions to an appropriate $\bar{\partial}_{J, H}$-equation. The tool we develop here to tackle this issue is a certain perturbation scheme on a branched manifold that we associate with Hurwitz spaces $\hurwitz$ serving as moduli spaces of branched $\kappa$-sheeted covers of the cylinder $\R \times S^1$. We point out that this setting simplifies numerous problems that arise in the context of higher-genus SFT \cite{EGH2000}. 

Let us now consider the case where $\hat{M}=T^*Q$. We recall that Abbondandolo-Schwarz \cite{abbondandolo2006floer} and Abouzaid \cite{abouzaid2011cotangent} show that for an orientable manifold $Q$ there is an isomorphism between the symplectic cohomology $SH^*_b(T^*Q)$ twisted by a certain background class and the homology of the free loop space of $Q$. Abbondandolo and Schwarz construct a chain map from the Morse complex associated with action functional $\mathcal{A}_L$ on the space $\Lambda^1(Q)$ of free loops in the class $W^{1,2}$ to the symplectic cochain complex. They prove that this chain map is a quasi-isomorphism, whereas Abouzaid constructs a chain map in the opposite direction and proves that it is the inverse of the former. 

In this paper, in an attempt to compute the unsymmetrized version of Heegaard Floer symplectic cohomology, we introduce a Morse-theoretic model which is intended to generalize the Morse complex $CM_*(\Lambda^1(Q))$. We consider the space $\Lambda^1_{[\kappa]}(Q)$ of \emph{free multiloops} of class $W^{1,2}$ which can be understood as free loops in $\operatorname{UConf}_{\kappa}(Q)$ of class $W^{1,2}$. The differential of this Morse complex $CM_*(\Lambda^1_{[\kappa]}(Q))$ counts not only gradient trajectories but also \emph{piecewise gradient trajectories} where the multiloops are allowed to undergo \emph{switchings} at finitely many moments. These switchings are meant to correspond to ramification points on the Heegaard Floer side. This model is parallel to the Morse model for HDHF in \cite{honda2025morse}.

For any orientable manifold $Q$ of dimension $n$ with vanishing second Stiefel-Whitney class $w_2(TQ)\in H^2(T^*Q; \mathbb{Z}/2)$ we then construct a map 
$$\mathcal{F} \colon SC_{\kappa, \mathrm{unsym}}^*(T^*Q) \to CM_{\kappa n-*}(\Lambda^1_{[\kappa]}(Q)),$$
which for $\kappa=1$ coincides with the chain map constructed by Abouzaid. It also extends the map constructed in \cite{honda2022higher} in the context of HDHF for $Q$ of dimension $2$ to higher dimensions. We then prove the following:
\begin{theorem}[Theorems~\ref{thm-f-isomorphism} and~\ref{thm-F-chain-map}]\label{main-theorem2}
    For $Q$ an orientable manifold with vanishing second Stiefel-Whitney class $w_2(TQ)~\in~H^2(T^*Q; \mathbb{Z}/2)$ the map $\mathcal{F}$ is a chain map and induces an isomorphism on the homology
    $$SH^*_{\kappa, \mathrm{unsym}}(T^*Q) \cong HM_{\kappa n-*}(\Lambda^1_{[\kappa]}(Q)).$$
\end{theorem}
In the situation of ordinary symplectic cohomology Abouzaid \cite{abouzaid2011cotangent, abouzaid2015symplectic} constructs an isomorphism of this kind when $w_2(TQ)\neq 0$ between the symplectic cohomology with twisted coefficients and the Morse homology of the free loop space. We hope that $\mathcal{F}$ as above also can be extended to such a more general setting. The first work which highlighted the necessity of twisted coefficients is that of Kragh \cite{kragh2018}.

There are numerous questions one can ask regarding HFSH and the Morse model above, and here we list just some of them. First, the original symplectic cohomology possesses various algebraic structures and we refer the reader to \cite{seidel2006biased, ritter2013, abouzaid2015symplectic} for a broad review. Therefore, it is natural to ask:
\begin{question}
    How much of various algebra structures including the BV-algebra structure, Viterbo restriction, TQFT-structure, etc., can be extended from symplectic cohomology to HFSH?
\end{question}

We expect most of the algebraic operations related to symplectic cohomology to have their analogues in the context of HFSH. We further hope that the corresponding structures will come in handy to show the generation criterion for HDHF wrapped Fukaya category $\mathcal{W}_\kappa(T^*Q)$, extending the result of Abouzaid \cite{abouzaid2010geometric}. More precisely, we pose the following: 

\begin{conjecture}
    Given a closed oriented manifold $Q$ with vanishing relative second Stiefel-Whitney class, the HDHF wrapped Fukaya category $\mathcal{W}_\kappa(T^*Q)$ is split-generated by any $\kappa$-tuple of disjoint cotangent fibers.
\end{conjecture}

Furthermore, based on the work of Ganatra \cite{ganatra2012symplectic} we expect a close relationship between Hochschild homology $HH_*(\mathcal{W}_\kappa(\hat{M}))$ of the wrapped HDHF Fukaya category and HFSH and we state it as the following: 

\begin{question}
    What is the relationship between the Hochschild homology $HH_*(\mathcal{W}_\kappa(\hat{M}))$ and the Heegaard Floer symplectic cohomology $SH^*_\kappa(M)$ for $M$ a Liouville domain?
\end{question}
{\bf Structure of the paper.}

Section 2 is devoted to the setup of Heegaard Floer symplectic cohomology. We first introduce basic definitions and notation in Sections~\ref{section-liouville-manifold}-\ref{section-almost-complex} in the context of the partially symmetrized Heegaard Floer symplectic cohomology. In Sections~\ref{section-moduli-cylinders}-\ref{hurwitz-spaces} we introduce Hurwitz spaces, which serve as model spaces for the Floer data required to define HFSH. Our treatment is based on the long lineage of work devoted to the study of Hurwitz schemes, for instance, \cite{ACV2003, deopurkar2014, harris-morrison2006, HM1982Hurwitz, ionel2002recursive, prufer2017}. 
We adapt this circle of ideas to our particular problem in the SFT setting. Furthermore, we introduce a branched manifold structure on the Hurwitz spaces in the sense of \cite{mcduff2006} in Sections~\ref{branched-manifolds}-\ref{branched-structure} and assign consistent Floer data to Hurwitz spaces.

In Sections~\ref{section-pseudo-holomophic-curves}-\ref{section-differential} we show that the above data indeed provide a cochain complex $SC^*_{\kappa, \mathrm{psym}}(\hat{M}; H, J)$, where $(H,J)$ denotes some Floer data adapted to a quadratic at infinity Hamiltonian function $H_0$ and a time-dependent almost complex structure $J_t$. Following the original approach in \cite{seidel2006biased} we give an alternative definition of the partially symmetrized HFSH via tuples of linear Hamiltonian functions in Section~\ref{section-linear-hamiltonians}. We also show that this definition via the direct limit over all tuples of linear Hamiltonians coincides with the definition given by means of a quadratic at infinity $H_0$ as above. The approach in terms of linear Hamiltonians is known to provide a simple way of showing the invariance of symplectic cohomology, and we adopt it to show that the partially symmetrized HFSH also gives an invariant of a Liouville domain $M$.

In Section 3 we introduce the free multiloop complex associated with an oriented manifold $Q$. This section is mostly devoted to showing that multiloops with a differential given by counting Morse flow lines with switchings form a cochain complex. The methods we use rely greatly on the approach of Abbondandolo-Schwarz in \cite{abbondandolo2010floer} and Abbondandolo-Majer in \cite{abbondandolomajer2006lectures}.

In Section 4 we prove Theorem~\ref{main-theorem2}. The chain map is constructed by means of moduli spaces of pseudoholomorphic curves associated with branched $\kappa$-sheeted covers of the half-cylinder $\R_{\ge 0} \times S^1$ extending the ideas in \cite{abbondandolo2006floer, abouzaid2011cotangent, abouzaid2015symplectic} to the Heegaard Floer context.

\vskip0.1in
 \noindent \textit{Acknowledgements}.
RK and TY thank Ko Honda for his guidance and emotional support. RK also thanks Sheel Ganatra, Robert Lipshitz, Cheuk Yu Mak, Ivan Smith, Mohan Swaminathan, Gleb Terentiuk and Bogdan Zavyalov for stimulating discussions. RK would also like to thank the organizers of the Symplectic Paradise summer school for creating an inspiring atmosphere that influenced the work on this project.
TY thanks Yin Tian for helpful discussions. We thank the anonymous referee for numerous valuable comments and corrections.
TY is supported by China Postdoctoral Science Foundation 2023T160002 and 2023M740106.

\section{The Floer complex}
\label{section-floer}

\subsection{The Liouville manifold}\label{section-liouville-manifold}

Let $(M,\lambda)$ be a $2n$-dimensional Liouville domain, i.e., $M$ is a compact exact symplectic manifold with boundary, $\lambda$ is a primitive of a symplectic form, and the Liouville vector field $V$ (satisfying $L_V d\lambda = d\lambda$) points outwards along the boundary $\partial M$.  

The Liouville form $\lambda$ restricts to a contact form $\alpha\coloneqq \lambda|_{\partial M}$ on $\partial M$. 
We then complete $M$ to $\hat{M}$ by attaching the symplectization end $([1,\infty)_r\times \partial M, r\alpha)$ along $\partial M$.
Denote the Reeb vector field of $(\partial M,\alpha)$ by $R$.
We require that
\begin{itemize}
    \item all Reeb orbits of $\alpha$ are non-degenerate.
\end{itemize}

To define an integral grading for the Floer complex, we further require that 
\begin{itemize}
    \item $2c_1(M)\in H^2(M,\mathbb{Z})$ vanishes.
\end{itemize}

\subsection{Tuples of Hamiltonian orbits}
\label{section-hamiltonian-partsym}

Here, we introduce the cochain complex for the partially symmetrized version of Heegaard Floer symplectic cohomology. The cochains to be introduced correspond to tuples of Hamiltonian orbits of some function of cumulative time $\kappa$. Describing them in this context requires working separately with each partition of $\kappa$. That is why we would have to reserve significantly more Hamiltonians than $\kappa$, which is enough for the unsymmetrized version; see Section~\ref{section-hamiltonian-unsym}.

Let $\bm \mu =(\mu_1, \ldots, \mu_{l(\bm \mu)})$ satisfying $\mu_1 \ge \mu_2 \ge \dots \ge \mu_{l(\bm \mu)}$ be a partition of $\kappa$,  where $l(\bm{\mu})$ denotes the \emph{length} of $\bm{\mu}$, i.e., the number of elements in the partition. Following standard notation, we set $\bm \mu! \coloneqq \mu_1 \cdots \mu_{l(\bm \mu)}$. We put 
\begin{gather}
     m_i(\bm \mu) =\#\{j\colon \mu_j=i\}, \label{partition-constants1} \\
    R_{\bm \mu}=\prod_im_i(\bm \mu)!, \label{partition-constants2} \\
    N_{\bm \mu}=\bm \mu! \cdot R_{\bm \mu} \label{partition-constants3}.   
\end{gather}
For the $i$-th element $\mu_i$ of the partition $\bm \mu$ we assign an additional ordering given by
$\op{ord}_i(\bm \mu)= i-\sum_{\mu_i < k\le \mu_1} m_k(\bm \mu)$. This ordering is meant to record the encountering of each particular value in the partition $\bm \mu$, e.g., for $\bm \mu=(2,2,1,1,1)$ we have $\op{ord}_4(\bm \mu)=2$ and therefore $\mu_4$ is the second term in this partition of value $\mu_4=1$. 

Let $H_0\colon S^1_t\times \hat{M}\to\mathbb{R}$ be a time-independent Hamiltonian function on $\hat{M}$ which satisfies $H_0=r^2/2$ on $[R,\infty)_r\times \partial M$ for some $R \ge 1$. We denote the class of such functions by $\mathcal{H}(\hat{M}) \subset C^{\infty}(S^1 \times \hat{M})$.
We choose $F_0 \colon \mathbb{R}/\mathbb{Z} \times \hat{M} \to \R$ to be a smooth nonnegative function satisfying:
\begin{enumerate}[(F1)]
    \item \label{function-perturbation-properties1} the functions $F_0$ and $\lambda(X_{F_0})$ are uniformly bounded in absolute value;
    \item \label{function-perturbation-properties2} all time-$1$ periodic orbits of $X_{H_0^{\mathrm{tot}}}$, the Hamiltonian vector field of the function $H_0^{\mathrm{tot}}=H_0+F_0$, are non-degenerate.
\end{enumerate}
The choice of such $F_0$ for a given $H_0 \in \mathcal{H}(M)$ exists by \cite[Lemma 5.1]{abouzaid2010geometric}.  We also recall that for such $F_0$ all time $1$ orbits of $H_0^{\operatorname{tot}}$ are isolated and, therefore, there are only finitely many of them in any compact subset of $\hat{M}$. As we are interested in tuples of Hamiltonian orbits of $H^{\mathrm{tot}}_0$ of cumulative time $\kappa$ we now introduce further perturbations to ensure that such tuples are non-degenerate.
That is, we choose a collection of functions $F_{m,j}$ for $1 \le m \le \kappa$, $1 \le j \le \lfloor\frac{\kappa}{m}\rfloor$ satisfying
\begin{enumerate}[(H1)]
    \item \label{hamiltonian-properties1} $F_{m,j}\colon\R/m\Z \times \hat{M} \longrightarrow \R$ is non-negative such that $|F_{m,j}-F_0|$ is absolutely bounded by a small constant and vanishes away from the locus of time-$m$ orbits of $H^{\mathrm{tot}}_0$, note that here we assume $F_0$ to be $1$-periodic and $F_{m,j}$ only $m$-periodic;
    \item \label{hamiltonian-properties2} all time-$m$ orbits of $H^{\mathrm{tot}}_{m,j}=H_0+F_{m,j}$ are non-degenerate, and each such orbit lies in a small neighborhood of the corresponding time-$m$ orbits of $H^{\mathrm{tot}}_0$;
    \item \label{hamiltonian-properties3} Hamiltonian time-$m$ orbits for $H^{\mathrm{tot}}_{m,j}$'s for all pairs $(m,j)$ as above are pairwise disjoint.
\end{enumerate}
Notice that existence of $F_{m,j}$'s satisfying~\ref{hamiltonian-properties1} and~\ref{hamiltonian-properties-unsym-2} can be shown as in the proof of \cite[Lemma 5.1]{abouzaid2010geometric}, with correspondence in~\ref{hamiltonian-properties-unsym-2} shown as follows: even if the initial time-$m$ orbit is degenerate, a local perturbation of the function produces new non-degenerate Hamiltonian orbits within a small neighborhood (see, e.g., \cite[Chapter 5]{audin-damian2014} for a thorough discussion). The last property is again easy to achieve using the finiteness of the number of orbits in any compact region and their non-degeneracy.  We also point out that possible degeneracy of time $k$ orbits for $H_0^{\operatorname{tot}}$ (even given non-degeneracy of time-$1$ orbits) implies that correspondence in~\ref{hamiltonian-properties2} does not have to be bijective.
 
 We denote the collections of Hamiltonians 
 $ H^{\mathrm{tot}}=\{H^{\mathrm{tot}}_{m,j}\}$ satisfying all of the above properties by ${\mathcal{ \bm H}}_{S^1}^\kappa(\hat{M})$. 

Let us note that, due to rotational symmetry, a non-multiply covered time-$m$ orbit of $H_0^{\mathrm{tot}}$ corresponds to $m$ orbits of $H_0^{\mathrm{tot}}$ with the same image. Therefore, these $m$ orbits produce $m$ distinct orbits of $H^{\mathrm{tot}}_{m,j}$ for any $j$, $1 \le j \le \lfloor \frac{\kappa}{m} \rfloor$.  See Figure~\ref{figure-H^2} for an illustration of this phenomenon with $m=2$.
\begin{figure}
    \centering
    \includegraphics[width=7cm]{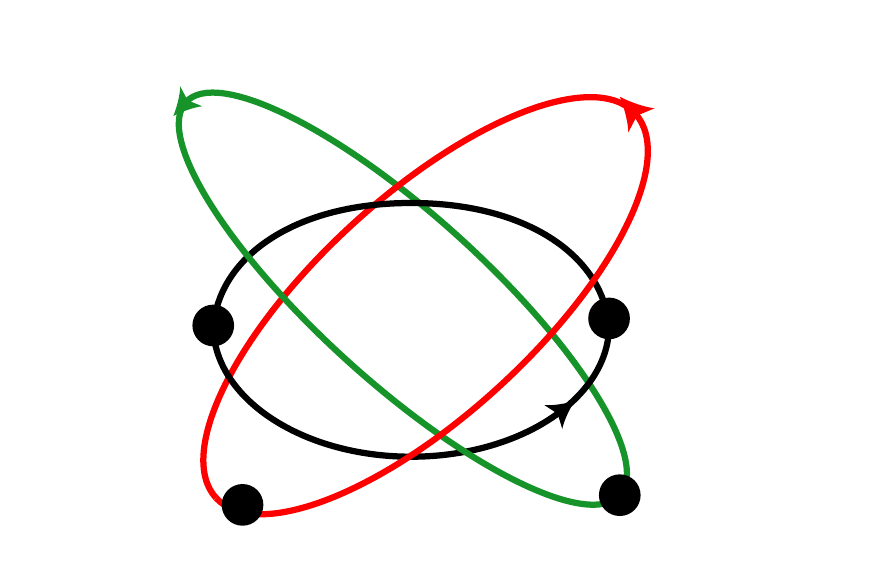}
    \caption{The black curve denotes a time-$2$ orbit of $H_0^{\mathrm{tot}}$. The green and red curves denote the two nearby orbits of $H^{\mathrm{tot}}_{2,1}$. The dots correspond to the starting points of these orbits.}
    \label{figure-H^2}
\end{figure}
More generally, assume we are given a tuple of Hamiltonian orbits of $H^{\mathrm{tot}}_0$ of cumulative time $\kappa$ with no multiply covered orbits among them, and let the periods of these orbits form a partition $\bm \mu=(\mu_1, \ldots, \mu_{l(\bm \mu)})$ of $\kappa$ with $\mu_1 \ge \mu_2 \ge \dots \ge \mu_{l(\bm \mu)}$; we use the notation $l(\bm \mu)$ to denote the length of the partition. There are $R_{\bm \mu}$ ways to order these orbits preserving the time-ordering. Given such an order on the orbits $x_1, \ldots, x_{l(\bm \mu)}$, there are $\bm \mu!$ distinct tuples of orbits of the form $x_1', \ldots, x_{l(\bm \mu)}'$, where each $x_i'$ is a Hamiltonian orbit of $H^{\mathrm{tot}}_{\mu_i, \op{ord}_i(\bm \mu)}$. All these $\bm \mu!$ tuples are in the small neighborhood of the initial tuple $x_1, \ldots, x_{l(\bm \mu)}$.
Hence, in total near the initial tuple, there are $N_{\bm \mu}$ new ``desymmetrized" tuples with disjoint images, since there are $R_{\bm \mu}$ orderings and a total of $\bm \mu!$ rotations. 

We set $\mathcal{P}_{\mathrm{psym}}$ to be
\begin{align}
    \label{eq-paths}
    \mathcal{P}_{\mathrm{psym}}\coloneqq\left\{\bm x=(x_1,\ldots,x_{l(\bm \mu)})\left|
        \begin{array}{ll}
            \text{$x_i\colon[0,\mu_i]\to\hat{M}$ is the integral flow of $\phi^t_{H^{\mathrm{tot}}_{\mu_i, \op{ord}_i(\bm{\mu})}}$,}\\
            \text{$x_i(0)=x_i( \mu_i)$, $\bm \mu$ is a partition of $\kappa$.}
        \end{array}
    \right.
    \right\}
\end{align}
\begin{figure}
    \centering
    \includegraphics[width=12cm]{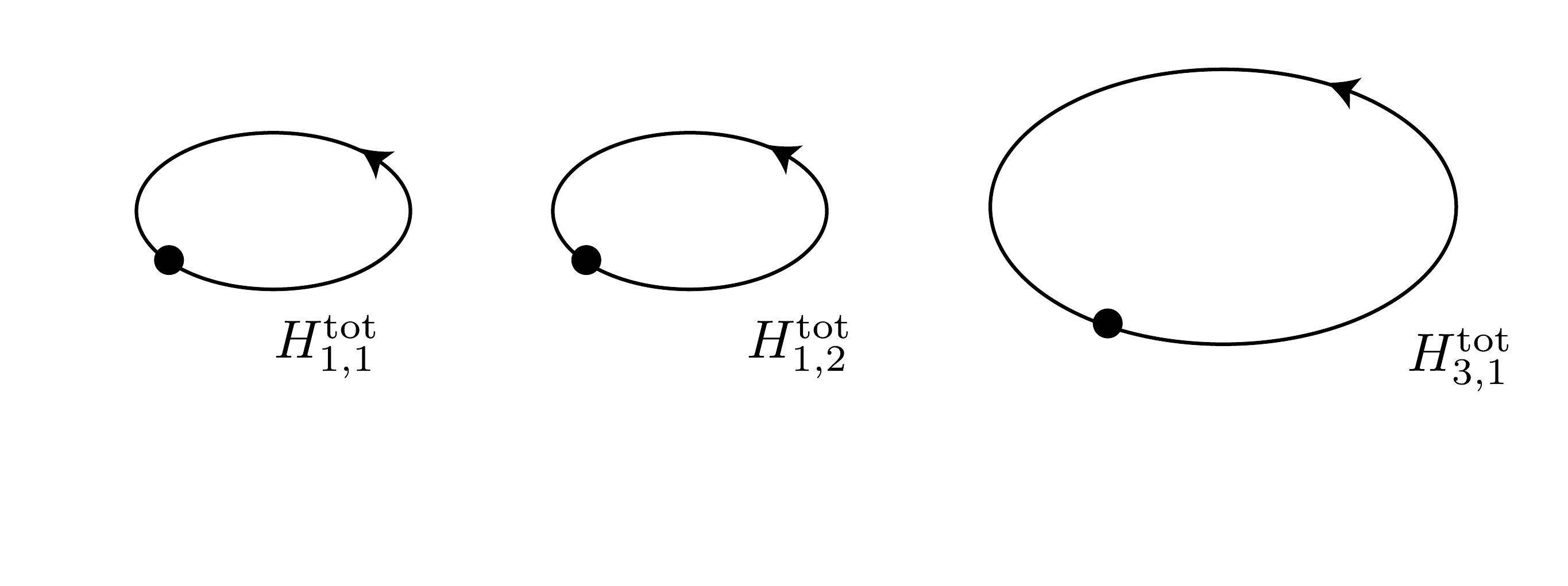}
    \caption{A sample tuple of Hamiltonian orbits in $\mathcal{P}_{\mathrm{psym}}$ for $\kappa=5$.}
    \label{fig-orbit-tuple-psym}
\end{figure}
To an orbit $x_{\mu_i}$ we associate a differential operator
\begin{equation}\label{loop-operator}
    D_{x_{\mu_i}} \colon W^{1,2}(\C, \C^{n}) \to L^2 (\C, \C^n),
\end{equation}
 following \cite[Section C.6]{abouzaid2010geometric}. Using this operator, one defines the \emph{orientation line} of $x_{\mu_i}$
 \begin{equation}
     o_{x_{\mu_i}} = |\operatorname{det}(D_{x_{\mu_i}})|,
 \end{equation}
i.e., $o_{x_{\mu_i}}$ is a rank-$1$ graded free abelian group with two generators representing two possible orientations of $\operatorname{det}(D_{x_{\mu_i}})$ and the relation that their sum is zero; see Section~\ref{section-orientation-lines} for the notations that we use here. The grading of $o_{x_{\mu_i}}$ is equal to the degree $|x_{\mu_i}|$ of the orbit $x_{\mu_i}$, which is equal to $n-CZ(x_{\mu_i})$, where the latter term is the Conley-Zehnder index.
 \begin{definition}\label{orientation-line}
     The \emph{orientation line} $o_{\bm x}$ of $\bm x \in \mathcal{P}_\mathrm{psym}$ is given as follows:
     \begin{equation}\label{orientation-eq}
         o_{\bm x} = o_{x_{\mu_1}} \otimes \dots \otimes o_{x_{\mu_{l(\bm \mu)}}}.
     \end{equation}
 \end{definition}
  We observe that the grading of $o_{\bm x}$ is equal to $|\bm x| \coloneqq |x_{\mu_1}|+\dots+|x_{\mu_{l(\mu)}}|$, i.e.
 \begin{equation}\label{grading-partsym}
    |\bm x|=\operatorname{deg}(o_{\bm x})=\kappa n-CZ(x_1)-\dots-CZ(x_{l(\bm \mu( \bm x))}).
\end{equation}

 \begin{remark}
    Note that for $\kappa=1$ the grading above agrees with the convention in \cite{abouzaid2010geometric}.
\end{remark}

Regarding $o_{\bm x}$ as a $\Z$-module we set
    \begin{equation}\label{orientation-rational-eq}
        o_{\bm x}^{\Q} = \Q \otimes_{\Z}o_{\bm x}.
   \end{equation}
We then set 
\begin{equation}  
V_{\kappa, \mathrm{psym}} = \bigoplus_{\bm x \in \mathcal{P}_{\mathrm{psym}}}o_{\bm x}^{\Q}.
\end{equation}
Each $\bm x \in \mathcal{P}_{\mathrm{psym}}$ comes with a grading that we discuss in Section~\ref{section-almost-complex}. Hence the above is a graded vector space with 
\begin{equation}
    V^i_{\kappa, \mathrm{psym}}=\bigoplus_{|\bm x|=i}o_{\bm x}^{\Q}.
\end{equation}
Consider a formal variable $\hbar$ with grading $|\hbar|=2-n$. We then set $SC^*_{\kappa, \mathrm{psym}}(\hat{M})$ to be a graded vector subspace in $V_{\kappa, \mathrm{psym}}\llbracket\hbar\rrbracket$ with $m$-th graded summand given by:
\begin{equation}
SC_{\kappa, \mathrm{psym}}^m(\hat{M}) \coloneqq \sum_{i+(2-n)j=m}V^i_{\kappa, \mathrm{psym}}\hbar^j = \sum_{i+(2-n)j=m}V^i_{\kappa, \mathrm{psym}}\bigl[(n-2)j\bigr],
\end{equation}
where in the last term we used the shift convention from Section~\ref{section-orientation-lines}, and the summation sign means direct product. We notice that $SC^*_{\kappa, \operatorname{psym}}(\hat{M})$ is a graded $\mathbb{Q}[h]$-module (but it is not a $\mathbb{Q}\llbracket\hbar\rrbracket$-module). In particular, multiplication by $\hbar$ induces a $(2-n)$-degree shift on $SC^*_{\kappa, \mathrm{psym}}(\hat{M})$.

For each tuple of Hamiltonian orbits $\bm x=(x_1, \ldots, x_{l(\bm \mu)}) \in \mathcal{P}_{\mathrm{psym}}$, we define the action of $\bm x$ by
\begin{equation}\label{action-psym}
    \mathcal{A}(\bm x)=\sum^{l(\bm \mu)}_{i=1}\mathcal{A}(x_i),
\end{equation}
where the action of the orbit $x_i$ is given by the expression
\begin{equation}
    \mathcal{A}(x_i)=\int^{\mu_i}_0 x_i^*\lambda-\int^{\mu_i}_0 H^{\mathrm{tot}}_{\mu_i, \op{ord}_i{\bm \mu}}(t,x_i(t))dt.
\end{equation}

We will further denote by $\bm \mu(\bm x) \vdash \kappa$ the partition associated with $\bm x$.
A tuple of Hamiltonian orbits $\bm x$ can be also viewed as a loop in $\mathrm{Conf}_\kappa(\hat{M})$ by considering a collection of $\kappa$ paths, where each path takes the form $\gamma \colon [0,1] \rightarrow \hat{M}$ with $$\gamma (t) \equiv x_i(k+t),\, t \in[0,1]$$ for some $x_i \in \bm x$ and $k \in \Z$ such that $0 \le k < \mu_i$. 

\begin{remark}
   The main difference between $\mathcal{P}_{\mathrm{unsym}}$ and $\mathcal{P}_{\mathrm{psym}}$ which we introduce in Section~\ref{section-hamiltonian-unsym}, is that an unperturbed loop in $\mathrm{Conf}_\kappa(\hat{M})$ (given by $\kappa$ time-$1$ flows of $H_0$) provides $\kappa !$ elements of $\mathcal{P}_{\mathrm{unsym}}$ and only $N_{\bm \mu}$ (see~\eqref{partition-constants2} and~\eqref{partition-constants3}) elements of $\mathcal{P}_{\mathrm{psym}}$, where $\bm \mu !$ stands for a number of ways to choose a starting point on each of the orbits and $R_{\bm \mu}$ is responsible for ordering the orbits. 
   
   In the toy example with a tuple consisting of Hamiltonian orbits of $H_0$ of respective times $\bm \mu=(2,1,1)$ we have $\bm \mu !=2$, as there are two ways to pick the starting point on the time-$2$ orbit, and $N_{\bm \mu}=2$ as there are two ways to give an order on the $2$ time-$1$ orbits. Therefore, in this example $\#\mathcal{P}_{\mathrm{psym}}=4$ and $\#\mathcal{P}_{\mathrm{unsym}}=24$.
\end{remark}

\subsection{Almost complex structures and gradings}\label{section-almost-complex}
Denote the contact hyperplane field $\operatorname{Ker}\lambda|_{\partial M}\subset T(\partial M)$ by $\xi$. 
Then we can write
\begin{equation}
    T([1,\infty)\times\partial M)=\mathbb{R}^2\oplus\xi,
\end{equation}
where $\mathbb{R}^2$ is spanned by $\partial_r$ and the Reeb vector field $R$.

Let $\mathcal{J}(\hat{M})$ be the set of almost complex structures on $\hat{M}$ tamed by $d\lambda$ which are of \emph{rescaled contact type} on the symplectization end:
\begin{equation}
    -\frac{1}{r}\lambda\circ J=dr,
\end{equation}
\noindent that is,
\begin{align}
    \left\{
        \begin{array}{ll}
            J\colon \xi\to\xi,\\
            J(R)=\partial_r,
        \end{array}
    \right.
\end{align}
on $[1,\infty)\times\partial M$ for $J\in\mathcal{J}(\hat{M})$. From now on, we fix a generic family of almost complex structures of rescaled contact type 
\begin{equation}\label{fixed-almost-complex}
    J_t \colon [0,1] \to \mathcal{J}(\hat{M}).
\end{equation}
 In addition to this, we fix a generic almost complex structure of rescaled contact type $J^0 \in \mathcal{J}(\hat{M})$.
Given any almost complex structure in $\mathcal{J}(\hat{M})$, we can define a grading on $SC_\kappa(\hat{M})$. 
Since $2c_1(M)=0$, the bicanonical bundle $(\Lambda_{\mathbb{C}}^2 T^*\hat{M})^{\otimes 2}$ is trivial. One can then define the usual Conley-Zehnder indices of Hamiltonian orbits.

\subsection{Moduli space of marked cylinders}\label{section-moduli-cylinders}
\begin{definition}\label{definition-stable-cylinder}
    A \emph{stable cylinder} is a tuple $\bm C =(C, j, B, D)$, where
    \begin{itemize}
        \item $(C, j)$ is Riemann surface consisting of the \emph{main component} which is a cylinder $C_0= \R_s \times S^1_t$ with $j|_{C_0}$ a standard complex structure and a collection of projective planes (i.e., $\C \mathbb{P}^1$'s up to biholomorphism) with $j$ restricting to the standard complex structure;
        \item $B$ is the finite ordered tuple comprised of \emph{marked points} $q_1, \ldots, q_{|B|}$ on $C$;
        \item $D$ is the set consisting of unordered nonintersecting pairs of points $\{d, d'\} \subset C$ which are disjoint from $B$ and belong to different components, where any such pair represents a node;
    \item each component is \emph{stable};
    \item the nodal Riemann surface $\bm C^D \coloneqq C/\sim_D$, where $d\sim_D d'$ if $\{d,d'\} \in D$, is connected and the genus of $\bm C^D$ equals $0$.
    \end{itemize}
\end{definition}

We denote by $\widetilde{\mathcal{C}}_b$ the set of all stable cylinders with $b$ marked points. The collection $\bm \beta \colon \bm C \to \bm C'$ of biholomorphisms $\{\beta_i\}$ between components of these two stable cylinders is said to be an \emph{equivalence}, if it is an isomorphism between $\bm C$ and $\bm C'$ as nodal marked curves which sends the main component $C_0$ to the main component $C_0'$ under some $\beta_0$, and $\beta_0$ is a vertical parallel translation, i.e., $\beta_0(s,t)=(s+s_0,t)$ for some $s_0 \in \R$.
We denote by $\mathcal{C}_{b}$~\label{stable-cylinders-equivalence} the quotient of $\widetilde{\mathcal{C}}_b$ under such equivalences.

One may regard the main component $\R_s \times S^1_t$ as a sphere $\mathbb{P}^1$ punctured at $2$ marked points $q_{+\infty}= \infty \in \mathbb{P}^1$ and $q_{-\infty}=0\in\mathbb{P}^1$. Compactifying the main component yields a genus $0$ curve with two marked points $q_{-\infty}$ and $q_{+\infty}$.
From this perspective the quotient of $\widetilde{\mathcal{C}}_b$ by all automorphisms may be regarded as a part of Deligne-Mumford compactification $\overline{\mathcal{M}}_{0,b+2}$ of the moduli space of genus zero curves with $k+2$ marked points  $(q_{+\infty}, q_{-\infty}, q_1, \ldots, q_{b})$. Namely, classes of stable cylinders up to automorphisms are in bijection with classes of genus zero nodal curves with $q_{+ \infty}$ and $q_{-\infty}$ staying in the same rational component. Notice that we have a natural map $\mathcal{C}_b \to \overline{\mathcal{M}}_{0,b+2}$ which surjects onto the portion described above.
 
For the main component of a given stable cylinder $\bm C$ as above to be stable, it has to contain at least one special point (a marked point or a node). A cylinder $\R \times S^1$ with no special points would be further called a \emph{trivial cylinder}. 

We also note that one may identify $S^1 \times \{\pm \infty\}$ with the corresponding space of tangent directions
\begin{equation}\label{tangent-directions-at-infinity}
    S_{q_{\pm \infty}}\mathbb{P}^1 = T_{q_{\pm \infty}}\mathbb{P}^1/ \R.
\end{equation}

\begin{definition}
    A \emph{cylindrical building} of height $n$ is a Riemann surface $$\bm C=(\overrightarrow{C}=(C_1, \ldots, C_n), j, B, D=D_r \cup D_s),$$ 
    where each $C_i$ is a stable cylinder and $B$ is an ordered set of marked points obtained as a union of marked points of each $C_i$. The set of \emph{regular nodes} $D_r$ is an unordered collection of nodes of all the $C_i$ and the set of \emph{special nodes} $D_s$ is a set of all the pairs $\{0_i, \infty_{i+1}\}$, where $0_i$ is a point $0$ of the main component of $C_i$ and $\infty_{i+1}$ is a point $\infty$ of the main component of $C_{i+1}$.
\end{definition}
Notice that by smoothing all the nodes (together with special ones) of the height $n$ cylindrical building, one obtains a genus $0$ connected Riemann surface. 

We denote the set of all cylindrical buildings of height $n$ with $b$ marked points by $\widetilde{\mathcal{C}}_{b,n}$. In the special case of $n=1$ we have $\widetilde{\mathcal{C}}_{b,1}=\widetilde{\mathcal{C}}_b$. 
\begin{definition}
   We say that a cylindrical building $\bm C$ of height $n$ is \emph{smooth} if the set $D_r$ is empty; equivalently, each $C_i$ consists only of the main component. 
\end{definition}
We denote the set of smooth cylindrical buildings of height $n$ by $\widetilde{\mathcal{C}}^{\text{sm}}_{b, n}$. 
On $\widetilde{\mathcal{C}}^{\text{sm}}_{b, n}$ there is a natural $\R^n$-action given by vertical translations on the main components of the $C_i$'s. Let us denote the quotient by
   \begin{equation}\label{vertical-automorphism-quotient}
    \widehat{\mathcal{C}}_{b, n} = \widetilde{\mathcal{C}}^{\text{sm}}_{b, n}/\R^n.
\end{equation}
\noindent Similarly, we denote by $\mathcal{C}_{b,n}$ the quotient of $\widetilde{\mathcal{C}}_{b,n}$ by the above-mentioned $\R^n$-action on the main components and further action by equivalences on each stable cylinder in the building.

For any cylindrical building $\bm C$ of height $n$ we denote by $\text{Aut}_v(\bm C)$ the group of \emph{vertical automorphisms} that acts by vertical translations on the main components and is the $\R^n$ by which we have been taking quotients above, e.g. in~\eqref{vertical-automorphism-quotient}.
\begin{remark}
    We stress that we are not going to quotient out the $S^1$-action on the main components in our considerations. 
\end{remark}

By taking the disjoint union of all $\mathcal{C}_{b,n}$ one obtains a moduli space 

\begin{equation}\label{eq: moduli-cylindrical-buildings}
    \cylcomp = \bigsqcup\limits_{n \in \mathbb{Z}_{>0}} \mathcal{C}_{b, n}.
\end{equation}
It has a subspace $\cylsm \subset \cylcomp$ serving as a moduli space of smooth cylindrical buildings
    \begin{equation}\label{eq: smooth-buildings-moduli}
       \widehat{\mathcal{C}}_b=\bigsqcup_{n \in \mathbb{Z}_{>0}} \, \widehat{\mathcal{C}}_{b, n}.
   \end{equation}
   
   \begin{remark}\label{rmk-smooth-cylindrical}
      The space $\widehat{\mathcal{C}}_b$ is the portion of the full compactification $\cylcomp$ of our main interest, as for the actual moduli space of pseuholomorphic curves all other forms of degeneration do not appear generically; see Section~\ref{section-compactness}. We use hat-notation to describe the portion of a ``presumable" compactification consisting of curves with multiple \emph{smooth} (i.e., without nodes) levels.
   \end{remark}
\begin{remark}
Notice that for a given convergent sequence $\bm C_k \in \cylsm$ of smooth cylindrical buildings such that no two marked points ``collide" it holds that $\lim_{k \rightarrow \infty} \bm C_k \in \cylsm$.
\end{remark}

\begin{remark}\label{moduli-of-rational-curves-remark}
    The set of \emph{cylindrical buildings} up to all the automorphisms forms the full Deligne-Mumford compactification $\overline{\mathcal{M}}_{0, k+2}$. We notice that the $t$-rotations in $S^1$-directions of the main components and possible permutations of $C_i$'s are among the automorphisms that are not quotiented out in the definition of $\mathcal{C}_{b,n}$.
    The height of a building is just the distance between the leaves corresponding to $q_{+\infty}$ and $q_{-\infty}$ in the tree associated with the building regarded as a nodal rational curve.
\end{remark}
We set the topology on these spaces following \cite{BEHWZ2003sft}. One can show that

\begin{lemma}
    The space $\cylcomp$ is a compact metric space and serves as the compactification of the space $\cyl$, i.e., it coincides with the closure of $\cyl$ viewed as a subspace of $\cylcomp$. In particular, every sequence of smooth cylinders with $b$ marked points has a subsequence that converges to a cylindrical building.
\end{lemma}
\begin{proof}[Sketch of the proof]
    For simplicity, assume we are given a sequence of smooth cylinders $\bm C_n \in \cyl$. Following \cite[Section 2]{HT2007gluing} one can pick a subsequence $\bm C_{n_k}$ with marked points denoted by $q_1^k, \ldots, q_b^k$  and a partition $\bm \lambda \vdash [b]$ of the set $[b]=\{1, \ldots, b\}$ (in particular each $\lambda_i$ is a subset of $[b]$), such that
    \begin{itemize}
        \item there is a constant $D$ satisfying $\forall j, j' \in \lambda_i \colon \text{dist}(q^k_j, q_{j'}^k)<D$ for any term $\lambda_i$ of $\bm \lambda$;
        \item if $j$ and $j'$ do not belong to the same element of $\bm \lambda$, then $\text{dist}(q_j^k, q_{j'}^k)>k$.
    \end{itemize}

    In other words, one may divide marked points into ``clusters" such that the $s$-coordinate difference between points of different clusters goes to $\infty$. Moreover, the elements of $\bm \lambda$ can be ordered by their height (i.e., the $s$-coordinates of elements of one cluster are bigger than the $s$-coordinates of the points in another cluster). If no marked points ``collide" then the limit of such a sequence is a smooth height $\ell(\bm \lambda)$ cylindrical building. Otherwise, bubbles appear and the limit is a stable height $\ell(\bm \lambda)$ cylindrical building.
\end{proof}

For further usage, we note that the boundary $\partial \cylcomp$ can be covered by images of the natural maps
\begin{equation}\label{bubble}
    \bigsqcup_{b_1+b_2=b} \overline{\mathcal{C}}_{b_1+1} \times \overline{\mathcal{M}}_{0,b_2+1} \to \partial \cylcomp,
\end{equation}
\begin{equation}\label{newlevel}
    \bigsqcup_{b_1+b_2=b} \overline{\mathcal{C}}_{b_1} \times \overline{\mathcal{C}}_{b_2} \to \partial \cylcomp,
\end{equation}
where the space $\overline{\mathcal{M}}_{b_2+1}$ in \eqref{bubble} refers to bubbling on the main component of a cylinder and \eqref{newlevel} describes the appearance of a new level in the building. Recall that in the map~\eqref{bubble} one identifies the last marked point on the first component and the first marked point on the bubble component, and one additionally needs a choice of length $2$ partition of $[b]$.
\begin{remark}
   Despite forcing out levels that could be trivial cylinders for elements of $\cylcomp$ later on, we would consider cylindrical buildings with a finite number of trivial cylinders inserted between some of the levels. A sequence of cylindrical buildings that converges to an element of $\cylcomp$ also converges to the one obtained by adding some trivial cylinder levels in between the levels of this element.
\end{remark}

To shed more light on the relation between the moduli space of cylindrical buildings $\cylcomp$ and $\overline{\mathcal{M}}_{0, b+2}$ (see Remark~\ref{moduli-of-rational-curves-remark}), we note that there is an $U(1)$-principal bundle
\begin{equation}
    U(1)\to \widehat{\mathcal{C}}_{b,1} \to \mathcal{M}_{0, b+2},
\end{equation}
where the $U(1)$ acts by counterclockwise rotations on a given cylinder. Denote by $\mathcal{M}_{0, (b_1+2, b_2+2)}$ the subset of $\partial \overline{\mathcal{M}}_{0, b+2}$ consisting of nodal curves with $2$ components $S_1$, $S_2$, such that $q_{+\infty} \in S_1, q_{-\infty} \in S_2$ and $b_1+b_2=b$. Then there is a principal bundle 
\begin{equation}\label{s1-bundle-boundary}
    U(1)^2 \to \widehat{\mathcal{C}}_{b,2} \to \mathcal{M}_{0, (b_1+2, b_2+2)},
\end{equation}
and more generically there is $U(1)^n$-principal bundle 
\begin{equation}
    U(1)^n \to \widehat{\mathcal{C}}_{b,n} \to \mathcal{M}_{0, (b_1+2, b_2+2, \ldots, b_n+2)}
\end{equation}
where the last space is defined similarly. 

We note that $\mathcal{M}_{0, (b_1+2, b_2+2)}$ is the image of the natural map
\begin{equation}
    \mathcal{M}_{0, b_1+2} \times \mathcal{M}_{0, b_2+2} \to \overline{\mathcal{M}}_{0, b+2}
\end{equation}
given by gluing two curves $(C_1; q_{+\infty}, q_1, \ldots, q_{b_1}, q_{-\infty})$ and $(C_2; q'_{+\infty}, q'_1, \ldots, q'_{b_2}, q'_{-\infty})$  at the node $\{q_{-\infty}, q'_{+\infty}\}$.

We denote by $L_{q_i}$ the line bundle over $\overline{\mathcal{M}}_{0, b+2}$ with fiber $T_{q_i}^*C$ over $$(C; q_{+\infty}, q_1, \ldots, q_{b_1}, q_{-\infty}) \in \overline{\mathcal{M}}_{0, b+2}.$$

We now recall facts from deformation theory following \cite[Section 1]{ionel2002recursive}; we refer the reader to the related descriptions of a neighborhood of a boundary stratum given in \cite[Section 8]{HK2014}, \cite[Section 4]{IP2004}, \cite[Section 2]{Siebert1997}. 
Namely,  the local coordinates in the normal direction near $\mathcal{M}_{0, (b_1+2, b_2+2)}$ 
are given by 
\begin{equation}\label{deformation-line}
    \zeta \in L^*_{q_{-\infty}} \otimes L^*_{q'_{+\infty}},
\end{equation}
which is a line bundle over $\mathcal{M}_{0, (b_1+2, b_2+2)}$ with fiber $T_{q_{-\infty}}C_1 \otimes T_{q'_{+\infty}}C_2$.
It turns out that the line bundle given in~\eqref{deformation-line} is essentially the normal complex line bundle to $\mathcal{M}_{0, (b_1+2, b_2+2)} \subset \overline{\mathcal{M}}_{0, b+2}$. According to the discussion which follows after \cite[Equation (1.18)]{ionel2002recursive}, in local coordinates $(w_1, w_2)$ on $C_1$ and $C_2$ near the node 
$$w_1w_2=0$$ 
the curve associated with $\zeta$ is locally a solution of
\begin{equation}\label{node-resolution}
    w_1w_2=\zeta.
\end{equation}

Analogously, we may define a complex line bundle
\begin{equation}\label{deformation-line-cyl}
     \hat{L}^*_{q_{-\infty}} \otimes \hat{L}^*_{q'_{+\infty}}
\end{equation}
on $\widehat{\mathcal{C}}_{b,2}$ with fiber $T_{q_{-\infty}}C_1 \otimes T_{q'_{+\infty}}C_2$ at a point $((C_1; q_{+\infty}, \ldots, q_{-\infty}), (C_2; q'_{+\infty}, \ldots, q'_{-\infty}))$. More explicitly, over $\widetilde{\mathcal{C}}_{b,2}$ both bundles with fibers $ T_{q_{-\infty}}C_1$ and $T_{q'_{+\infty}}C_2$ are trivializable (by the means of global coordinates on cylinders) and $\R^2$-equivariant. Thus, these vector bundles induce quotient vector bundles on $\widehat{\mathcal{C}}_{b,2}$ which we denote by $\hat{L}^*_{q_{-\infty}}$ and $\hat{L}^*_{q'_{+\infty}}$. We additionally note that positive projectivizations of real vector bundles $\hat{L}^*_{q_{-\infty}}$ and $\hat{L}^*_{q'_{+\infty}}$ are trivial. Hence, for~\eqref{deformation-line-cyl} viewed as a real $2$-dimensional bundle, there is a global choice of angle coordinate $\varphi$.

We now describe a normal neighborhood of $\widehat{\mathcal{C}}_{b,2} \subset \widehat{\mathcal{C}}_b$ similar to~\eqref{deformation-line}. First, we pick a global section of $\widetilde{\mathcal{C}}_{b,2} \to \widehat{\mathcal{C}}_{b,2}$ as follows. Given a representative 
$$((C_1; q_{+\infty}, \ldots, q_{-\infty}), (C_2; q'_{+\infty}, \ldots, q'_{-\infty}))$$ of a point in $\widehat{\mathcal{C}}_{b,2}$ we assign to it another representative of the same point $(\widetilde{C}_1, \widetilde{C}_2)$ obtained by translating $C_1$ vertically by a negative of the minimum of the $s$-coordinates of its marked points and translating $C_2$ by a negative of the maximum of the $s$-coordinates of its marked points (to make such section smooth one should pick smooth approximations of $min$ and $max$ differing by not more than a constant). The point $(\widetilde{C}_1, \widetilde{C}_2)$ will then always have all marked points of $\widetilde{C}_1$ above $s=0$ and all marked points of $\widetilde{C}_2$ below $s=0$.

Now pick an $R \in (0, +\infty]$ and glue $(\widetilde{C}_1, \widetilde{C}_2)$ via the standard procedure. Namely, set 
$$\widetilde{C}_1^R=\widetilde{C}_1 \setminus (-\infty, -R] \times S^1,$$
$$\widetilde{C}_2^R=\widetilde{C}_2 \setminus [R, +\infty) \times S^1,$$
$$\widetilde{C}^R = \widetilde{C}_1^R \sqcup \widetilde{C}_2^R/\sim,$$
where the identification $\sim$ is given by solving
\begin{equation}\label{nodal-resolution-cyl}
e^{s_1+2\pi \sqrt{-1}t_1} \cdot e^{-s_2-2\pi \sqrt{-1}t_2}=e^{-R}    
\end{equation}
for $-R\le s_1 \le 0$ and for $R=+\infty$ we have $\widetilde{C}^R = \widetilde{C}_1^R \sqcup \widetilde{C}_2^R$.

Hence, we just constructed a diffeomorphism of a neighborhood of $\widehat{C}_{b,2} \subset \widehat{C}_b$ and 
\begin{equation}\label{collar-nbhd}
    \widehat{C}_{b,2} \times [0, +\infty].
\end{equation}

Let us notice that given a parameter $R$ as above, we can assign a section 
\begin{equation}\label{gluing-section-hurwitz}
\tilde{\zeta} \in \hat{L}^*_{q_{-\infty}} \otimes \hat{L}^*_{q'_{+\infty}}
\end{equation}
with $\varphi( \tilde{\zeta})=0$ (recall $\varphi$ is the angular coordinate) and, moreover, $\tilde{\zeta}=e^{-R}$ in the local trivialization associated with $(\widetilde{C}_1, \widetilde{C}_2)$ as in~\eqref{nodal-resolution-cyl}. Hence, in local coordinates on cylinders, the equation~\eqref{nodal-resolution-cyl} is given via
\begin{equation}
    w_1w_2=\tilde{\zeta},
\end{equation}
which is the analogue of~\eqref{node-resolution}; here we use identifications $w_1=e^{s_1+2\pi \sqrt{-1}t_1}, w_2=e^{-s_2-2\pi \sqrt{-1}t_2}$. 
We point out that unlike in the case of moduli space of curves $\overline{\mathcal{M}}_{0, b+2}$ the only sections of $\hat{L}^*_{q_{-\infty}} \otimes \hat{L}^*_{q'_{+\infty}}$ that give a family of curves close to breaking are those with $\varphi=0$.

More generally, there is a similar description of a neighborhood of $\widehat{C}_{b,n}$ diffeomorphic to
\begin{equation}\label{codim-n-stratum}
    \widehat{C}_{b,n} \times [0, +\infty]^{n-1}.
\end{equation}

Given any curve $C$ and a point $q \in C$ we introduce the notation
\begin{equation}\label{eq-angular-circle}
    \mathcal{S}^1_{q} \bm C=(T_qC\setminus \{0\}) / \R_{>0}.
\end{equation}
For a nodal curve $\bm C$ given by two components $C_1$ and $C_2$ at the node $\{q_{-\infty}, q_{+\infty}'\}$ as above we additionally introduce
\begin{equation}\label{eq-angular-circle-node}
    \mathcal{S}^1_{q_{-\infty}, q_{+\infty}'}C=\bigl(T_{q_{-\infty}}C_1 \otimes T_{q_{+\infty}'}C_2\setminus \{0\} \bigr)/ \R_{>0}.
\end{equation}
The circle $\mathcal{S}^1_{q}C$ can then be regarded as the set of angular coordinates at point $q \in C$. There is a natural $S^1$-bundle $\mathcal{S}^1_q$ over $\mathcal{C}_b$ with such a fiber on $C \in \mathcal{C}_b$. Similarly, there is a line bundle $\mathcal{S}^1_{q_{-\infty}, q_{+\infty}'}$ over $\widehat{C}_{b,2}$ with a fiber $\mathcal{S}^1_{q_{-\infty}, q_{+\infty}'} \bm C$. 

We now note that $\mathcal{S}^1_{q_{-\infty}, q_{+\infty}'}$ is a trivial $S^1$ bundle, as it is essentially the projectivization of the trivial bundle~\eqref{deformation-line-cyl}. Using unit-circle coordinates, we will denote the $0$-angular coordinate as in the discussion above by 
\begin{equation}
  1 \in \mathcal{S}^1_{q_{-\infty}, q_{+\infty}'} \bm C.  
\end{equation}
For future use, we also introduce the following notation for the projectivization of the tensor power of the above line bundle by setting:

\begin{equation}\label{eq-angular-circle-node-power}
    (\mathcal{S}^1_{q_{-\infty}, q_{+\infty}'})^{\otimes k}C=\bigl((T_{q_{-\infty}}C_1 \otimes T_{q_{+\infty}'}C_2)^{\otimes k}\setminus \{0\} \bigr)/ \R_{>0}.
\end{equation}

\subsection{Hurwitz type spaces}\label{hurwitz-spaces}
In this subsection, we describe the moduli spaces of certain branched covers of a cylinder and their compactifications, closely following \cite{HM1982Hurwitz}. In particular, we describe the smooth structure near the boundary of the compactification of the Hurwitz space in the case of multiple-level smooth degenerations; see Theorem~\ref{hurwitz-level-boundary-neighborhood} for the precise statement.

Pick two partitions $\bm \mu, \bm \mu' \vdash \kappa$. We recall that  that $\mu_{i} \ge \mu_{i+1}$ for $i=1, \dots, l(\bm \mu)-1$.
\vspace{-0.1in}
\begin{definition}\label{def-hurwitz-space}
    The \emph{Hurwitz space} $\hurwitzaction$ is the space of isomorphism classes of \emph{branched covers}, where a branched cover
    \begin{equation*}
        (S, \pi \colon S~\longrightarrow~\R ~\times~S^1, \bm p^+, \bm a^+, \bm p^-, \bm a^-, B)
    \end{equation*}
    is the following collection of data: 
    \begin{enumerate}
        \item A Riemann surface $S$ without boundary of Euler characteristic $\chi$ punctured at positive ordered punctures $\bm p^+=\{p^+_1, \ldots, p^+_{l(\bm \mu)}\}$ and negative ordered punctures $\bm p^-=\{p^-_1, \ldots, p^-_{l(\bm \mu ')}\}$. Denote by $\overline{S}$ the compactification of $S$ such that $S= \overline{S} \setminus (\bm p^+ \cup \bm p^-)$. 
        \item For each positive puncture $p_i^+$ (resp. negative puncture $p_i^-$) an \emph{asymptotic marker} $a_i^+ \in \mathcal{S}^1_{p_i^+}S \coloneqq (T_{p_i^+}\overline{S} \setminus \{0\})/\R_{>0}$ (resp. an asymptotic marker $a_i^- \in \mathcal{S}^1_{p_i^-}S$).
        \item A holomorphic degree $\kappa$ branched cover $\pi \colon S \longrightarrow \R \times S^1$ with simple branched points away from $\pm \infty$. At $+\infty$ (resp. at $- \infty$) the branching profile is given by $\bm \mu$ (resp. by $\bm \mu '$), i.e., a point $p_i^+$ (resp. a point $p_i^-$) is a ramification point of multiplicity $\mu_i$ (resp. multiplicity $\mu_i'$) of an extended map $\overline{\pi} \colon \overline{S} \longrightarrow S^2$. Here $\R \times S^1= S^2\setminus \{+\infty, -\infty\}$. We additionally require that the map induced by $\pi$ from $\mathcal{S}^1_{p_i^+}\overline{S}$ to $\mathcal{S}^1_{+ \infty} S^2$ sends $a_i^+$ to the marker corresponding to $\R \times \{0\}$ which is regarded as $1 \in \mathcal{S}^1_{+ \infty} S^2$  together with similar requirements for negative punctures of $S$.
        \item An ordered tuple $B$ of points $q_1, \ldots, q_b$ in $\R \times S^1$ of simple branching, where $b= - \chi$ according to the Riemann-Hurwitz theorem.
    \end{enumerate}

    Two branched covers $(S, \pi, \bm p^\pm, \bm a^\pm, B)$ and $(S', \pi', \bm p'^\pm, \bm a'^\pm, B')$ are said to be \emph{isomorphic} if there exists a biholomorphism $\alpha \colon \overline{S} \rightarrow \overline{S'}$ such that $\alpha(p^\pm_i)=p_i'^\pm$, $\alpha(a^\pm_i)=a_i'^\pm$, and $\pi' \circ \alpha = \pi$, and it is required that $B$ and $B'$are equal as ordered subsets of $\R \times S^1$. Isomorphic branched covers represent the same point of $\hurwitzaction$.
\end{definition}
\begin{figure}[H]
    \centering
    \includegraphics[width=9cm]{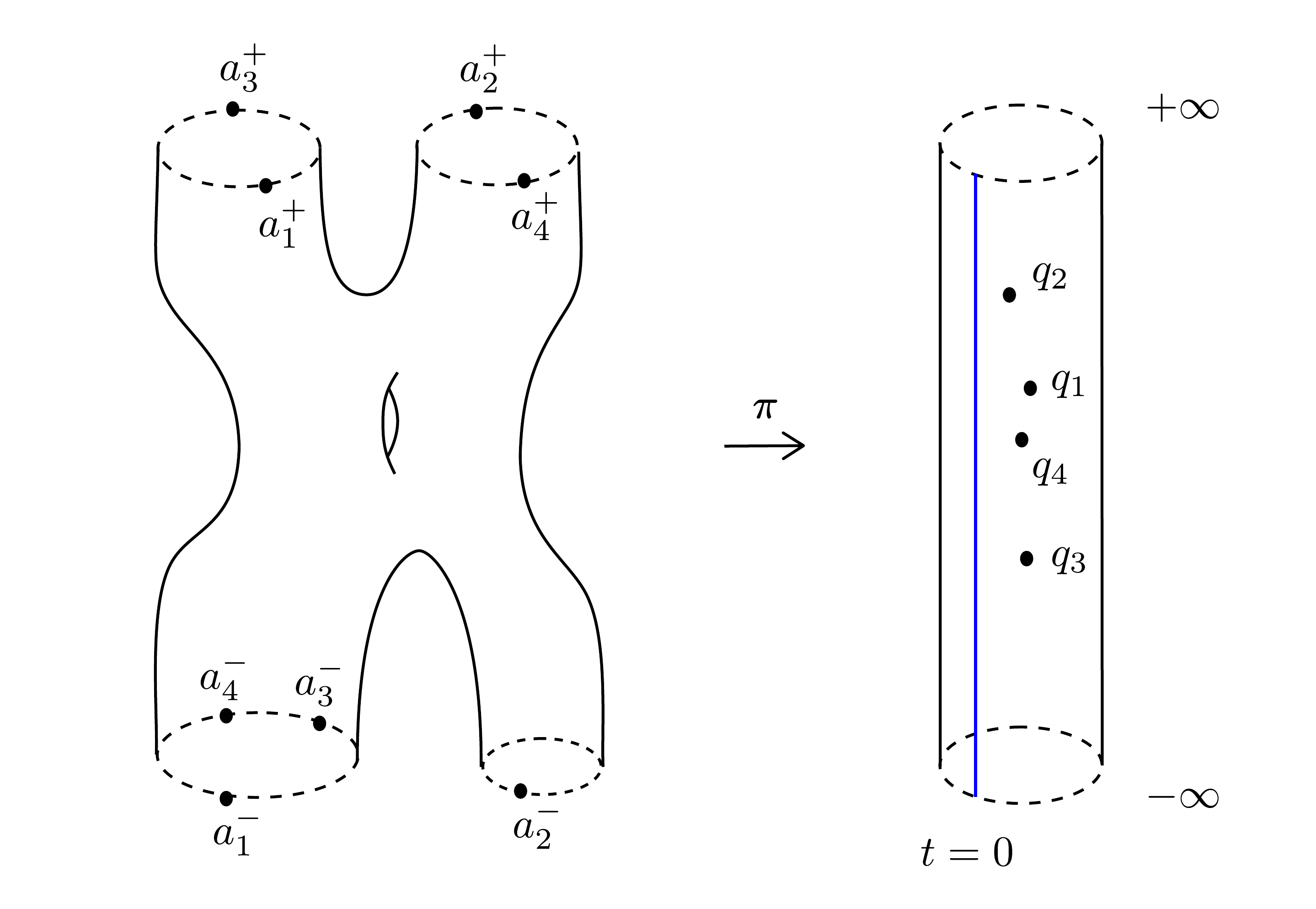}
    \caption{A sample element of $\mathcal{H}^{(2,2), (3,1)}_{4, -4}$.}
    \label{fig-hurwitz-elem}
\end{figure}
\begin{remark}
    Notice that given a geometric branched cover $\pi \colon S \rightarrow \R \times S^1$ for each puncture $p_i^+$ (resp. $p_i^-$), there are only $\mu_i$ (resp. $\mu_i'$) possible choices for the marker $a_i^+$ (resp. $a_i^-$).
\end{remark}

Recall that $\text{Aut}_v(\R \times S^1) \cong \R$ is a group of vertical translations on the standard cylinder.
This group induces an $\R$-action on $\hurwitzaction$ by translating the target of a given branched cover. We denote by $\hurwitz$ the quotient space by this action.

\begin{remark}\label{rmk: chi=0 case}
    We also consider the space unbranched $\kappa$-fold covers, i.e. the case when $\chi=0$ and $\bm \mu=\bm \mu'$. In this case, to obtain $\hurwitzaction$ we additionally take a quotient by $S_1$-action on the target, and $\hurwitz$ is a further quotient by $\R$-action. The space $\hurwitz$ consists of finitely many points.
\end{remark}

\begin{remark}\label{hurwitz-scheme-remark}
    Note that to get the standard Hurwitz scheme $\hurwitzstandard$  as defined in \cite{fulton1969} or in the analytic setting in \cite{ionel2002recursive} one has to drop the asymptotic marker data and take a quotient by the action of the whole group $\text{Aut}(\R \times S^1)$ containing  $\text{Aut}_v(\R \times S^1)$ as its subgroup.
\end{remark}

\begin{definition}\label{cover-building}
    A \emph{branched cover building}
     $$\left(\bm C, S=S_1\cup \ldots \cup S_n, \pi, (\bm v^1, \ldots, \bm v^{n-1}), (\bm m^1, \ldots, \bm m^{n-1})\right),$$
     with Euler characteristic $\chi$ and profile $(\bm \mu, \bm \mu')$ consists of the following data:
    \begin{enumerate}
        \item A cylindrical building $\bm C=[(\overrightarrow{C}=(C_1, \ldots, C_n), j, (q_1, \ldots, q_b), D) )] \in \cylcomp$.
        \item A punctured nodal Riemann surface $S$ decomposed as $S=S_1 \cup \ldots \cup S_n$ with punctures at two distinct \emph{ordered tuples} of points $\bm p^+, \bm p^-$, a marked collection of \emph{special nodes} 
        $$v^1_1, \ldots, v^1_{\ell_1}, v^2_1, \ldots, v^2_{\ell_2}, v^3_{1}, \ldots, v^{n-1}_1, \ldots, v^{n-1}_{\ell_{n-1}}$$
        and maybe some other nodes. We require that $S_i \cap S_{i+1}=\{v^i_1, \ldots, v^i_{l_i}\}= \bm v^i$;  writing each node as $v^i_j=(v^i_{j,+}, v^i_{j,-})$, we use the convention that $v^i_{j,+} \in S_{i+1}$ and $v^i_{j,-} \in S_i$.
        We regard each tuple $\bm v^i$ as an \emph{unordered collection} of nodes, and the lower indices above are used here to simplify the notation.
          \item An \emph{admissible cover} $\pi \colon S \longrightarrow \bm C$ , i.e., a branched cover that is simply branched at $q_1, \ldots, q_b$, has a branching profile $\bm \mu$ at $+\infty_1$ (by which we mean $+\infty$ of the $C_1$ level of $\bm C$) and a branching profile $\bm \mu'$ at $-\infty_n$ and is possibly branched at nodes.
        It is required that $\pi^{-1}(C_i)=S_i$ and that the preimage of a node $\{-\infty_i, +\infty_{i+1}\}$ under $\pi$ is given by $S_i \cap S_{i+1}$. The covering $\pi$ has the same degree when restricted to each of the $S_i$. 
        \item For each positive puncture $p_i^+$ (resp. negative puncture $p_i^-$) an \emph{asymptotic marker} $a_i^+ \in \mathcal{S}^1_{p_i^+}S = (T_{p_i^+}\overline{S} \setminus \{0\})/\R_{>0}$ (resp. an asymptotic marker $a_i^- \in \mathcal{S}^1_{p_i^-}S$), satisfying that $\pi_*(a_i^+)=1 \in \mathcal{S}^1_{+\infty}\overline{C_1}$ (resp. $\pi_*(a_i^-)=1 \in \mathcal{S}^1_{-\infty}\overline{C_n}$). 
        \item For each of the nodes $v^i_j$ as above, there is an assigned \emph{matching marker} given by the choice of  $m^i_j \in \mathcal{S}^1_{v^i_j}S$   such that it is mapped to 
        $$1 \in \mathcal{S}^1_{q_{-\infty}^i, q_{+\infty}^{i+1}} \bm C = (T_{q_{-\infty}^i}C_i \otimes T_{q_{+\infty}^{i+1}}C_{i+1} \setminus 0)/ \R_{>0}$$
        under the following composition (see~\eqref{eq-angular-circle-node} and~\eqref{eq-angular-circle-node-power} for the notation that we use here)
        \[\begin{tikzcd}
	\mathcal{S}^1_{v^i_j}S && (\mathcal{S}^1_{v^i_j})^{\otimes e_{\pi}(v^i_j)}S \\
	&& \mathcal{S}^1_{q_{-\infty}^i, q_{+\infty}^{i+1}} \bm C.
	\arrow[from=1-1, to=2-3]
	\arrow["{\xi \mapsto \xi^{e_{\pi}(v^i_j)}}", from=1-1, to=1-3]
	\arrow["{\pi_*}", from=1-3, to=2-3]
\end{tikzcd}\]
Here $e_\pi(v^i_j)$ is the ramification index of $\pi$ at the node $v^i_j$ and the map $\pi_*$ is induced by natural linear maps associated with $\pi$:
\begin{gather}\label{eq: sheaves-map}
(T_{v_{j,-}^{i}}S_{i})^{\otimes e_{\pi}(v^i_j)} \to T_{q_{-\infty}^i}C_i,\\
(T_{v_{j,+}^i}S_i)^{\otimes e_{\pi}(v^i_j)} \to  T_{q_{+\infty}^{i+1}}C_{i+1}. \nonumber
\end{gather}
In local coordinates, where $\pi(z)=\lambda_mz^m+o(z^m)$, the map~\eqref{eq: sheaves-map} is given by multiplication by $\lambda_m \in \C$ (clearly, this is independent of the choice of coordinates). 

We refer the reader to~\eqref{eq-angular-circle},~\eqref{eq-angular-circle-node} and~\eqref{eq-angular-circle-node-power} for the notation we use here. We point out that the horizontal arrow in the diagram above is nothing but an $e_{\pi}(v^i_j)$-sheeted cover of a circle by another circle, and that $\pi_*$ is a diffeomorphism. We will denote the tuple of matching markers corresponding to the tuple of nodes $\bm v^i$ ``connecting" $S_i$ and $S_{i+1}$ by $\bm m^i$.    
    \end{enumerate}

    Two branched cover buildings $$\left(\bm C, S=S_1\cup \ldots \cup S_n, \pi, (\bm v^1, \ldots, \bm v^{n-1}), (\bm m^1, \ldots, \bm m^{n-1})\right),$$
    $$\left(\bm C', S'=S_1'\cup \ldots \cup S_n', \pi', (\bm v'^1, \ldots, \bm v'^{n-1}), (\bm m'^1, \ldots, \bm m'^{n-1})\right)$$
    are \emph{isomorphic} if there exists a pair $(\bm \alpha \colon S \rightarrow S', \bm \beta)$ with $\bm \alpha$ consisting of biholomorphisms $\alpha_i \colon S_i \rightarrow S_i'$ and $\bm \beta \colon  \bm C \to \bm C'$ (see Section~\ref{section-moduli-cylinders}) is an isomorphism of cylindrical buildings:
\begin{enumerate}
    \item $\alpha_1(p^+_i)=p_i'^+$,  $\alpha_n(p^-_i)=p_i'^-$;
    \item $\alpha_1(a^+_i)=a_i'^+$,  $\alpha_n(a^-_i)=a_n'^-$;
    \item $\bm \alpha (\bm v^i)= \bm v'^i$ as unordered tuples; given some orderings on $\bm v^i$ and $\bm v'^i$ we denote by $j(\bm \alpha)$ the index satisfying $\bm \alpha(v^i_j)=v'^i_{j(\bm \alpha)}$;
    \item for the orderings as above $\bm \alpha^i_*(m^i_j)=m'^i_{j(\bm \alpha)}$, where $\alpha^i_*$ is the induced diffeomorphism
$$\alpha^i_* \colon \mathcal{S}^1_{v^i_j}S \to \mathcal{S}^1_{v'^i_{j(\bm \alpha)}}S';$$
    \item $\bm \beta(q_j)=q_j'$ and $\bm \beta \circ \pi' \circ \bm \alpha = \pi$.
\end{enumerate}
\end{definition}

We stress that in the definition above, the tuples $\bm v^i=\{v^i_1, \ldots, v^i_{l_i} \}$ are unordered (i.e., reordering these collections in the definition above does not change the class of a branched cover building).
 \begin{remark}
     A matching marker as above can be interpreted as a choice of gluing between two cylindrical ends corresponding to a given node.
 \end{remark}
 
 We denote by $\hurwitzcomp$ the set of all branched cover buildings up to isomorphism. \label{hurwitz-compactification} 
 \begin{definition}\label{defn-hurwitz-smooth}
      Let $\hurwitzsm$ be the subset of $\hurwitzcomp$ consisting of \emph{branched cover buildings with smooth levels}, by which we mean those branched cover buildings 
       $$\left(\bm C, S=S_1\cup \ldots \cup S_n, \pi, (\bm v^1, \ldots, \bm v^{n-1}), (\bm m^1, \ldots, \bm m^{n-1})\right)$$
       for which the only nodes of $S$ are the special nodes connecting different levels, i.e., those nodes which are contained in the tuples $(\bm v^1, \ldots, \bm v^{n-1})$.
 \end{definition}

\begin{remark}
    As per Remark~\ref{rmk-smooth-cylindrical}, the space of branched cover buildings with smooth levels $\hurwitzsm$ is of main interest to us as we will be able to show that it is enough to consider only such types of domains for moduli spaces of pseudo-holomorphic curves involved in the definition of Heegaard Floer symplectic cohomology; see Theorem~\ref{theorem-cpt-moduli} for details.
\end{remark}
 We now discuss the relationship between $\hurwitzcomp$ and Harris-Mumford compactification via admissible covers $\hurwitzstandardcomp$, continuing the discussion in Remarks~\ref{moduli-of-rational-curves-remark} and~\ref{hurwitz-scheme-remark}. In particular,  we give a more precise description of the role of the matching data. 

As mentioned in Remark~\ref{hurwitz-scheme-remark} there is a fiber bundle
\begin{equation}
\hurwitz \to \hurwitzstandard.    
\end{equation}
We claim that this map extends to
\begin{equation}
    \hurwitzcomp \to \hurwitzstandardcomp,
\end{equation}
and fits into the diagram
\[\begin{tikzcd}
	\hurwitzcomp & \hurwitzstandardcomp \\
	\overline{\mathcal{C}}_{-\chi} & \overline{\mathcal{M}}_{0,-\chi+2}.
	\arrow[from=1-1, to=1-2]
	\arrow[from=1-2, to=2-2]
	\arrow[from=1-1, to=2-1]
	\arrow[from=2-1, to=2-2]
\end{tikzcd}\]
As in the case of the Hurwitz scheme, there is a map
\begin{equation}\label{to-domain}
    \hurwitzcomp \to \overline{\mathcal{M}}_{\chi, l(\bm \mu)+l(\bm \mu')},
\end{equation}
defined by taking the domain of a branched cover. Here $\overline{\mathcal{M}}_{\chi, l(\bm \mu)+l(\bm \mu')}$ is the moduli space of nodal curves of Euler characteristic $\chi$ with $l(\bm \mu)$ positive and $l(\bm \mu')$ negative  marked points (recall that $l(\bm \mu)$ refers to the length of the partition $\bm \mu$). We may then pull-back the universal family of curves over $\overline{\mathcal{M}}_{\chi, l(\bm \mu)+l(\bm \mu')}$ to obtain a universal family of curves 
\begin{equation}\label{eq: universal-curve-hurwitz}
    \mathscr{C}^{\bm\mu, \bm\mu'}_{\kappa, \chi} \to \hurwitzcomp.
\end{equation}

Given two nodal curves $C_1$ and $C_2$ with $n_1+k$ and $n_2+k$ marked points, one may produce a new curve $C$ by taking the disjoint union of $C_1$ and $C_2$ and identifying the last $k$ marked points of $C_1$ and the first $k$ marked points of $C_2$. This map also requires a datum of partition of $[n_1+n_2]=\{1, \ldots, n_1+n_2\}$ into two subsets of cardinalities $n_1$ and $n_2$. This datum allows one to order marked points on the glued curve. This construction induces the \emph{attaching map}
\begin{equation}\label{attaching-map}
    \overline{\mathcal{M}}_{\chi_1, n_1+k} \times \overline{\mathcal{M}}_{\chi_2, n_2+k} \longrightarrow \overline{\mathcal{M}}_{\chi_1+\chi_2, n_1+n_2}.
\end{equation}
The attaching map gives rise to the description of boundary strata for both $\hurwitzcomp$ and $\hurwitzstandardcomp$, which arises from the following diagram (we draw the diagram only for $\hurwitzcomp$)
\begin{equation}\label{boundary-map-diagram}
\begin{tikzcd}
	{\overline{\mathcal{H}}^{\bm \mu, \bm \mu''}_{\kappa, \chi_1} \times \overline{\mathcal{H}}^{\bm \mu'', \bm \mu'}_{\kappa, \chi_2} } & \hurwitzcomp \\
	 \overline{\mathcal{M}}_{\chi_1, l(\bm \mu)+l(\bm \mu'')} \times \overline{\mathcal{M}}_{\chi_2, l(\bm \mu')+l(\bm \mu'')} & \overline{\mathcal{M}}_{\chi,l(\bm \mu)+l(\bm \mu')}.
	\arrow[from=1-1, to=1-2]
	\arrow[from=1-2, to=2-2]
	\arrow[from=1-1, to=2-1]
	\arrow[from=2-1, to=2-2]
\end{tikzcd}
\end{equation}
The maps induced by the attaching map fit into the following commutative diagram
\[\begin{tikzcd}
	{\overline{\mathcal{H}}^{\bm \mu, \bm \mu''}_{\kappa, \chi_1} \times \overline{\mathcal{H}}^{\bm \mu'', \bm \mu'}_{\kappa, \chi_2} } & \hurwitzcomp \\
	\accentset{\circ}{\overline{\mathcal{H}}}^{\bm \mu, \bm \mu''}_{\kappa, \chi_1} \times \accentset{\circ}{\overline{\mathcal{H}}}^{\bm \mu'', \bm \mu'}_{\kappa, \chi_2} 
 & \hurwitzstandardcomp.
	\arrow[from=1-1, to=1-2]
	\arrow[from=1-2, to=2-2]
	\arrow[from=1-1, to=2-1]
	\arrow[from=2-1, to=2-2]
\end{tikzcd}\]
Note that to define a map given by the arrow in the top row of the diagram above, one has to assign matching data based on asymptotic markers at punctures corresponding to $\bm \mu''$-profiles, unlike the arrow in the bottom row, where all this data is not present.

\begin{definition}\label{def-hurwitz-2-level}
    We define 
\begin{equation}\label{eq-hurwitz-2-level}
    {\mathcal{H}^{\bm \mu, \bm \mu''}_{\kappa, \chi_1} \times_{\bm m} \mathcal{H}^{\bm \mu'', \bm \mu'}_{\kappa, \chi_2} } \subset \partial \hurwitzcomp
\end{equation}
to be the part of $\hurwitzsm$ consisting of branched cover buildings with smooth levels of height $2$ $$\left(\bm C=(C_1, C_2), S=S_1 \cup S_2, \pi, \bm v=\{v_1^1, \ldots, v^1_{l(\bm \mu'')}\}, \bm m=\{m_1^1, \ldots, m_{l(\bm \mu'')}^1\}\right)$$ 
with:
\begin{itemize}
    \item $\chi(S_1)=\chi_1$, $\chi(S_2)=\chi_2$ and $\chi=\chi_1+\chi_2$;
    \item ramification indices of $\pi$ at nodes $\bm v$ constituting a partition $\bm \mu''$ which appears in the notation~\eqref{eq-hurwitz-2-level}.
\end{itemize}

\end{definition}

We point out that an element of $\mathcal{H}^{\bm \mu, \bm \mu''}_{\kappa, \chi_1} \times_{\bm m} \mathcal{H}^{\bm \mu'', \bm \mu'}_{\kappa, \chi_2}$, unlike an element of the direct product $\mathcal{H}^{\bm \mu, \bm \mu''}_{\kappa, \chi_1} \times \mathcal{H}^{\bm \mu'', \bm \mu'}_{\kappa, \chi_2}$, possesses a choice of asymptotic markers at nodes corresponding to cylindrical ends with assigned partition $\bm \mu''$, i.e. to the nodes $\bm v$. On the contrary, for this boundary stratum, there is a natural ``boundary forgetful" covering
\begin{equation}\label{hurwitz-boundary-cover}
    bf^1 \colon \mathcal{H}^{\bm \mu, \bm \mu''}_{\kappa, \chi_1} \times \mathcal{H}^{\bm \mu'', \bm \mu'}_{\kappa, \chi_2} \to \mathcal{H}^{\bm \mu, \bm \mu''}_{\kappa, \chi_1} \times_{\bm m} \mathcal{H}^{\bm \mu'', \bm \mu'}_{\kappa, \chi_2}
\end{equation}
given by matching the negative asymptotic markers of a given element of $\mathcal{H}^{\bm \mu, \bm \mu''}_{\kappa, \chi_1}$ with corresponding (via the order on punctures) positive asymptotic markers of $\mathcal{H}^{\bm \mu, \bm \mu''}_{\kappa, \chi_2}$ and forgetting these two sets of asymptotic markers.
Explicitly, given a pair 
\begin{equation}\label{eq-pair-hurwitz}
    \left((S_1, \pi_1, \bm p^{\pm}_1, \bm a^{\pm}_1, B_1), (S_2, \pi_2, \bm p^{\pm}_2, \bm a^{\pm}_2, B_2) \right) \in \mathcal{H}^{\bm \mu, \bm \mu''}_{\kappa, \chi_1} \times \mathcal{H}^{\bm \mu'', \bm \mu'}_{\kappa, \chi_2},
\end{equation}
we set $\bm C$ to be simply a union of two cylinders $C_1$ and $C_2$ with one special node. Further, we set $S$ to be the union of $S_1$ and $S_2$ glued along pairs of punctures $v^1_i=\{p^-_{1i}, p^+_{2i}\}$. To each such pair, we associate a matching marker
\begin{equation}\label{eq-asymptotic-to-matching}
    m^1_i=a^-_{1i} \otimes a^+_{2i} \in \mathcal{S}^1_{\{p^-_{1i}, p^+_{2i}\}}S=(T_{p_{1i}^-}\overline{S_1} \otimes T_{p_{2i}^+}\overline{S_2} \setminus \{0\})/ \R.
\end{equation}
This marker evidently satisfies (5) in Definition~\ref{cover-building} as both $a_{1i}^-$ and $a_{2i}^+$ satisfy the property (3) in the same definition.
Then the map~\eqref{hurwitz-boundary-cover} sends a pair as in~\eqref{eq-pair-hurwitz} to $$(\bm C=(C_1, C_2), S=S_1 \cup S_2, \pi= \pi_1 \cup_{\bm v} \pi_2, \bm v =\{v^1_1, \ldots, v^1_{l(\bm \mu'')}\}, \bm m=\{m^1_1, \ldots, m^1_{l(\bm \mu'')}\}).$$

Essentially, the subspace $\mathcal{H}^{\bm \mu, \bm \mu''}_{\kappa, \chi_1} \times_{\bm m} \mathcal{H}^{\bm \mu'', \bm \mu'}_{\kappa, \chi_2} \subset \hurwitzcomp$ is the image of $\mathcal{H}^{\bm \mu, \bm \mu''}_{\kappa, \chi_1} \times \mathcal{H}^{\bm \mu'', \bm \mu'}_{\kappa, \chi_2}$ under the top row map in the diagram~\eqref{boundary-map-diagram}. It is evident that this forgetful map is an $N_{\bm \mu''}$-sheeted covering (see~\eqref{partition-constants3}). The reason for this is that for any matching $m_i^1$ there are $\mu''_i$ pairs of asymptotic markers that are sent to $m^1_i$ under~\eqref{eq-asymptotic-to-matching} and there are $R_{\bm \mu''}$ ways to give an order on the nodes of an element of $\mathcal{H}^{\bm \mu, \bm \mu''}_{\kappa, \chi_1} \times_{\bm m} \mathcal{H}^{\bm \mu'', \bm \mu'}_{\kappa, \chi_2}$.

Similarly, for the space of branched cover buildings with smooth levels of height $n$
$$\mathcal{H}^{\bm \mu, \bm \mu_1''}_{\kappa, \chi_1} \times_{\bm m^1} \dots \times_{\bm m^{n-1}} \mathcal{H}^{\bm \mu_n'', \bm \mu'}_{\kappa, \chi_n}, $$
defined in a similar fashion to Definition~\ref{def-hurwitz-2-level}, there is forgetful covering
\begin{equation}
    bf^{n-1} \colon \mathcal{H}^{\bm \mu, \bm \mu_1''}_{\kappa, \chi_1} \times \dots \times \mathcal{H}^{\bm \mu_n'', \bm \mu'}_{\kappa, \chi_n} \to \mathcal{H}^{\bm \mu, \bm \mu_1''}_{\kappa, \chi_1} \times_{\bm m^1} \dots \times_{\bm m^{n-1}} \mathcal{H}^{\bm \mu_n'', \bm \mu'}_{\kappa, \chi_n} 
\end{equation}
of degree $N_{\bm \mu_1''} \cdots N_{\bm \mu_n''}$.

We now give an interpretation of the matchings in terms of local deformations following \cite{ionel2002recursive}. In particular, we claim that the neighborhood of the image of $\mathcal{H}^{\bm \mu, \bm \mu''}_{\kappa, \chi_1} \times \mathcal{H}^{\bm \mu'', \bm \mu'}_{\kappa, \chi_2} $ is smooth.

We pick some element 
\begin{equation}\label{point-in-boundary-hurwitz}
    \left(\left(\pi_1 \colon S_1 \to C_1 \right),\left(\pi_2 \colon S_2 \to C_2\right)\right) \in \mathcal{H}_{\kappa, \chi_1}^{\bm \mu, \bm \mu^{\prime \prime}} \times_{\bm m} \mathcal{H}_{\kappa, \chi_2}^{\bm \mu^{\prime \prime},\bm \mu^{\prime}}.
\end{equation}

Recall that the normal bundle to the boundary stratum containing $(S_1, S_2) \in \overline{\mathcal{M}}_{\chi,  l(\bm\mu)+  l(\bm\mu')}$ given by
\begin{equation}\label{normal-bundle-higher-genus}
    \bigoplus_{i=1}^{  l(\bm\mu'')} L^*_{p_{1,i}^-} \otimes L^*_{p_{2,i}^+}
\end{equation}
at $(S_1, S_2)$ has fiber 
\begin{equation}
    \bigoplus_{i=1}^{  l(\bm\mu'')} T_{p_{1,i}^-}S_1 \otimes T_{p_{2,i}^+}S_2.
\end{equation}
A section of $ L^*_{q_{-\infty}^1} \otimes L^*_{q_{+\infty}^2}$ given by $\zeta \in  T_{q_{-\infty}^1}C_1 \otimes T_{q_{+\infty}^2}C_2$  sets a restriction on sections of~\eqref{normal-bundle-higher-genus} which give rise to the deformation of the point~\eqref{point-in-boundary-hurwitz} regarded as a point of $\hurwitzstandardcomp$ as shown in \cite{ionel2002recursive}. More specifically, such a section given by $$(\xi_1, \ldots, \xi_{  l(\bm\mu'')}) \in \bigoplus_{i=1}^{  l(\bm\mu'')} T_{p_{1,i}^-}S_1 \otimes T_{p_{2,i}^+}S_2$$ 
has to satisfy (see \cite[Equation (1.21)]{ionel2002recursive}
\begin{equation}\label{deformation-hurwitz-standard}
    \zeta=\lambda_1\xi_1^{\mu''_1}=\dots=\lambda_{  l(\bm\mu'')}\xi^{\mu''_{l(\bm\mu'')}}_{  l(\bm\mu'')},  
\end{equation}
where $\lambda_1, \ldots, \lambda_{l(\bm \mu'')}\in \C$ are the values determined by $\pi_1$ and $\pi_2$ as in Definition~\ref{cover-building} (5).
We now recall that in the context of $\hurwitzcomp$ the non-zero section $\tilde{\zeta}$ as in~\eqref{gluing-section-hurwitz} giving rise to local perturbation of $(C_1,C_2) \in \mathcal{C}_{b,2}$ has to satisfy $\varphi(\tilde{\zeta})=0$. We then claim that a section of~\eqref{normal-bundle-higher-genus} produces a point of $\hurwitz$ near~\eqref{point-in-boundary-hurwitz} whenever 
\begin{equation}\label{deformation-hurwitz}
    \tilde{\zeta}=a_1\xi_1^{\mu''_1}=\dots=a_{  l(\bm\mu'')}\xi^{\mu''_{l(\bm\mu'')}}_{  l(\bm\mu'')}.
\end{equation}

At last, we claim that given a matching marker $m^1_j$ as in Definition~\ref{def-hurwitz-2-level} there is a unique choice of $\xi_j$ satisfying~\eqref{deformation-hurwitz} with $[\xi_j]=m^1_j$ in $\mathcal{S}^1_{v^i_j}S$.

Hence, the coordinate chart~\eqref{collar-nbhd} (and, moreover, the chart~\eqref{codim-n-stratum}) near the boundary of the space of cylindrical buildings gives lifts to coordinate charts near the corresponding portions of the boundary of $\hurwitzcomp$.

Notice that $\mathcal{H}^{\bm \mu, \bm \mu''}_{\kappa, \chi_1} \times_{\bm m} \mathcal{H}^{\bm \mu'', \bm \mu'}_{\kappa, \chi_2}$ is a smooth manifold, as there are no non-trivial automorphisms of a curve fixing a marked point and an asymptotic marker at this point.

We conclude this discussion with the following theorem and refer the reader to \cite[Sections 3 and 4]{mochizuki1995}, \cite[Theorem 4]{HM1982Hurwitz}, \cite[Section 1]{ionel2002recursive}, \cite[Theorem 6.1.1 and Section 8]{HV2010augmented}  for further details.

\begin{theorem}\label{hurwitz-level-boundary-neighborhood}
    The space $\hurwitzcomp$ is a compact metric space and serves as a compactification of the space $\hurwitz$, i.e., it coincides with the closure of $\hurwitz$ viewed as a subspace of $\hurwitzcomp$. Moreover, there is a smooth normal neighborhood near the space of height $2$ buildings with smooth levels
    $$\mathcal{H}^{\bm \mu, \bm \mu''}_{\kappa, \chi_1} \times_{\bm m} \mathcal{H}^{\bm \mu'', \bm \mu'}_{\kappa, \chi_2} \subset \hurwitzcomp$$
    diffeomorphic to 
    \begin{equation}\label{hurwitz-collar-nbhd-codim1}
        \mathcal{H}^{\bm \mu, \bm \mu''}_{\kappa, \chi_1} \times_{\bm m} \mathcal{H}^{\bm \mu'', \bm \mu'}_{\kappa, \chi_2} \times [0, +\infty].
    \end{equation}
    More generally, there is a normal neighborhood of 
    $$\mathcal{H}^{\bm \mu, \bm \mu_1''}_{\kappa, \chi_1} \times_{\bm m^1} \dots \times_{\bm m^{n-1}} \mathcal{H}^{\bm \mu_n'', \bm \mu'}_{\kappa, \chi_n} \subset \hurwitzcomp$$ diffeomorphic to
    \begin{equation}\label{hurwitz-collar-nbhd}
        \mathcal{H}^{\bm \mu, \bm \mu_1''}_{\kappa, \chi_1} \times_{\bm m^1} \dots \times_{\bm m^{n-1}} \mathcal{H}^{\bm \mu_n'', \bm \mu'}_{\kappa, \chi_n}  \times [0,+\infty]^{n-1}.
    \end{equation}
\end{theorem}

\begin{remark}
    We note that the whole $\hurwitzcomp$ is not a manifold with corners since there may be non-special nodes corresponding to bubbles on the cylindrical side.
    Nevertheless, the choice of matching data allows us to get a smooth structure on the preimage of smooth cylinders $\widetilde{C}_{-\chi}$. Recall that in the algebraic setting to get a smooth stack serving as compactification of $\hurwitzstandard$, one should add \emph{twisted covers} as introduced in \cite[Definition 2.1.4]{ACV2003} and more generally in \cite{deopurkar2014}. Working with $\overline{\mathcal{C}}_{-\chi}$ instead of $\overline{\mathcal{M}}_{0, \chi+2}$ allows us to avoid introducing twisted covers.
\end{remark}

\begin{remark}\label{remark-bubble-components}
     We also point out that each stratum of the boundary of $\hurwitzcomp$ which has bubbles on the cylindrical side is of codimension at least $2$, whereas strata that we describe above are of codimension $1$. This is the main reason we do not introduce matchings or twisted covers to the compactification (although we could) to make the whole space into an orbifold (famously, the stack of admissible covers is not smooth, see e.g. \cite[Section 4.2]{ACV2003}).
\end{remark}
\smallskip
We will refer to $[0, +\infty]^{n-1}$ part in~\eqref{hurwitz-collar-nbhd} as \emph{neck coordinates}.
For further use, we also discuss orientations on Hurwitz spaces. Any space $\hurwitzaction$ is naturally oriented as it is a cover space over $\text{Conf}_{-\chi}(\C \setminus \{0, \infty\})$ which is a complex manifold. The choice of vertical direction $\partial_s$  leads to a further natural choice of orientation on $\hurwitz$. We fix such choices on all spaces $\hurwitz$.

We then claim that boundary orientation on the image of ${\mathcal{H}}^{\bm \mu, \bm \mu''}_{\kappa, \chi_1} \times \mathcal{H}^{\bm \mu'', \bm \mu'}_{\kappa, \chi_2}$ under the map described in~(\ref{boundary-map-diagram}) coincides with the product orientation. In other words, there is a natural isomorphism
\begin{equation}\label{orientation-boundary}
    |v| \otimes |{\mathcal{H}}^{\bm \mu, \bm \mu''}_{\kappa, \chi_1}| \otimes |\mathcal{H}^{\bm \mu'', \bm \mu'}_{\kappa, \chi_2}| \cong |\hurwitz|
\end{equation}
where $v$ is an outward pointing vector (corresponding to the normal direction described in~\eqref{hurwitz-collar-nbhd-codim1}).  This follows immediately since the reordering of marked points on the target cylinder does not affect the orientation, as they can be regarded locally as providing complex coordinates.  Here we use conventions discussed in Section~\ref{section-orientation-lines}.

\begin{remark}
We point out that assuming an order on the set of branched points as in the Definition~\ref{def-hurwitz-space} allows us to get a description of the portion of the boundary of $\hurwitz$ which is of interest to us. Nevertheless, in practice, it leads to ``overcounting" geometric curves connecting tuples of Hamiltonian orbits. We ignore this disadvantage since Heegaard Floer symplectic cohomology is eventually defined over $\Q$. To eliminate this nuance, one could use the approach of \cite{mochizuki1995} towards compactification of the Hurwitz scheme with unordered branched points.
\end{remark}

\subsection{Review of branched manifolds}\label{branched-manifolds}
We start by reviewing the general framework of \emph{weighted nonsingular branched groupoids} from \cite{mcduff2006}.

Denote the spaces of objects and morphisms of a 
small topological category $\mathscr{X}$ by the letters $X_0$ and $X_1$, respectively.  The source and 
target maps are denoted via  $s,t\colon X_1 \to X_0$, and the identity map is given by 
$\operatorname{id}\colon X_0\to X_1$, $x \mapsto \text{id}_x$.
The composition map is set as
$$
m\colon {X_1}\;_s\!\times_t X_1\;\to \;X_1,\quad (f,g)\mapsto f \circ g.
$$
The space of morphisms from $x$ to $y$ is denoted by $\text{Hom}(x,y)$.
 \begin{definition}     
A \emph{nonsingular (oriented) sse (smooth, stable, \'etale) groupoid with corners} $\mathscr{X}$ is a small topological category such that the following conditions hold:
 \begin{itemize}
\item {\bf (Groupoid)}   All morphisms are invertible, i.e., there is a structure map $\iota \colon X_1\to X_1$ that takes each $f\in X_1$ to its inverse $f^{-1}$.
\item
{\bf (Smooth)}  The spaces $X_0$  of objects and $X_1$ of morphisms are (oriented) manifolds with corners, and all structure maps ($s,t,m,\iota, \text{id}$) are smooth (and orientation-preserving). 
\item
{\bf (Proper)} The map $s \times t \colon X_1 \to X_0 \times X_0$ is proper.

\item {\bf (\'Etale)} All structure maps are local diffeomorphisms.
 \item {\bf (Nonsingular)}  For each $x\in X_0$ the set of self-morphisms $\text{Hom}(x,x) \eqqcolon G_x$ is trivial  (\emph{stability} is a weaker condition which only requires finiteness of these groups).
 \end{itemize}
 \end{definition} 

We denote by $\mathscr{X}^{\vee}$ the quotient space of a small topological category groupoid $\mathscr{X}$ given by $X_0/\sim $, where $x \sim x'$ if $\text{Hom}(x, x')$ is non-empty. 
Denote by $\pi$ the projection $X_0 \to \mathscr{X}^{\vee}$.

 Further, we will simply call $\mathscr{X}$ a \emph{nonsingular sse groupoid} if $X_0$ and $X_1$ are manifolds (without boundary or corners).

Let us recall \cite[Lemma 3.1]{mcduff2006} that any topological space $Y$ admits a maximal Hausdorff quotient $(f, Y_H)$, i.e., a projection $f \colon Y \to Y_H$ to a Hausdorff space $Y_H$, such that any other continuous surjection $f' \colon Y \to Y'$ with $Y'$ Hausdorff factors through $f$.

For a given sse groupoid with boundary $\mathscr{B}$ denote by $\pi_H^{\vee}$ the projection $\mathscr{B}^{\vee} \to \mathscr{B}^{\vee}_H$
and by $\pi_H$ the projection $B_0\to \mathscr{B}^{\vee}_H$.  Furthermore,  $U_H = \pi_H(U) $ denotes the image 
of $U\subset B_0$ in $\mathscr{B}^{\vee}_H$.

\begin{definition}\label{defn-wnb-groupoid}
A  \emph{weighted nonsingular branched groupoid with corners} (or \emph{wnbc groupoid} for short) is a pair $(\mathscr{B},\Lambda)$ consisting of an  oriented,  nonsingular sse groupoid with boundary together with a weighting function  
$\Lambda \colon \mathscr{B}^{\vee}_H\to (0,\infty)$, which satisfies the following compatibility conditions.
For each $p\in \mathscr{B}^{\vee}_H$ there is an open neighborhood $N$ of $p$ in $\mathscr{B}^{\vee}_H$, a collection  $U_1,\ldots,U_\ell$ of  
 disjoint open subsets of $\pi_H^{-1}(N)\subset B_0$ (called \emph{local branches}) and a set of positive weights $m_1,\ldots,m_\ell$ such that: 

\begin{itemize}
\item  $\pi^{\vee \, -1}_H(N) = U_1^{\vee}\cup\ldots \cup U_\ell^{\vee} \subset \mathscr{B}^{\vee}$, where $U_i^{\vee}$ are the spaces associated with the restricted groupoid structure to $U_i$;

\item for each $i=1,\ldots,\ell$ the projection $\pi_H\colon U_i\to 
(U_i^{\vee})_H$  is a homeomorphism onto a relatively closed subset of $N$;

\item for all $q\in N$, $\Lambda(q)$ is the sum of the weights of the local
branches whose image contains $q$:
\begin{equation}\label{eq: local-weights}
\Lambda(q) = \underset{i \colon q\in \,(U_i^{\vee})_H}\sum m_i.
\end{equation}
\end{itemize}
\end{definition}

We call $(\mathscr{B}, \Lambda)$ a \emph{wnb groupoid} if $\mathscr{B}$ is a nonsingular sse groupoid (i.e., with no boundary).

Later, we also make use of the notion of an \emph{orbibundle} over a wnb groupoid $\rho \colon \mathscr{E} \to \mathscr{B}$. This can be thought of as a vector bundle $\rho_0 \colon E_0 \to B_0$ with $E_1$ representing identifications between fibers $E_x$ and $E_y$ for any $f \in \text{Hom}(x, y)$; see \cite{mcduff2006} for details.  A set $\mathcal{S}(\mathscr{E})$ of sections of such a bundle consists of  smooth functors $\sigma \colon \mathscr{B} \to \mathscr{E}$. Notice that this space can be endowed with $C^k$- or Fr\'{e}chet topology assuming that $B_0$ and $B_1$ have finitely many components, see~\cite[Section 4.2]{mcduff2006}.
A section $\sigma$ is said to be transverse to the zero section if $\sigma(B_0)$ is \emph{transverse to the zero set} of $E_0$, and for each $x \in \sigma^{-1}(0)$ it holds that $d(ev_x) \colon T_\sigma \mathcal{S}(\mathscr{E}) \to E_x$ is surjective, where $ev_x$ is map which takes given section to its value at point $x$, see~\cite[Definition 4.9]{mcduff2006}. According to \cite[Lemma 4.10]{mcduff2006} this notion of transversality guarantees that generic section is transverse for ``sufficiently nice" wnb groupoid $\mathscr{B}$. 
Notice that a structure of a wnb groupoid on $\mathscr{B}$ induces a structure of a wnb groupoid in the zero set $\mathscr{Z}(\sigma)$ for transversal $\sigma$.

\subsection{Branched structure on Hurwitz spaces.}\label{branched-structure} 
In this section, we define a wnbc groupoid $\hurwitzsmbr$ by means of an inductive procedure. The motivation behind the construction to follow is discussed in Section~\ref{section-floer-data}.
\begin{figure}
    \centering
    \subfloat{{\includegraphics[width=15cm]{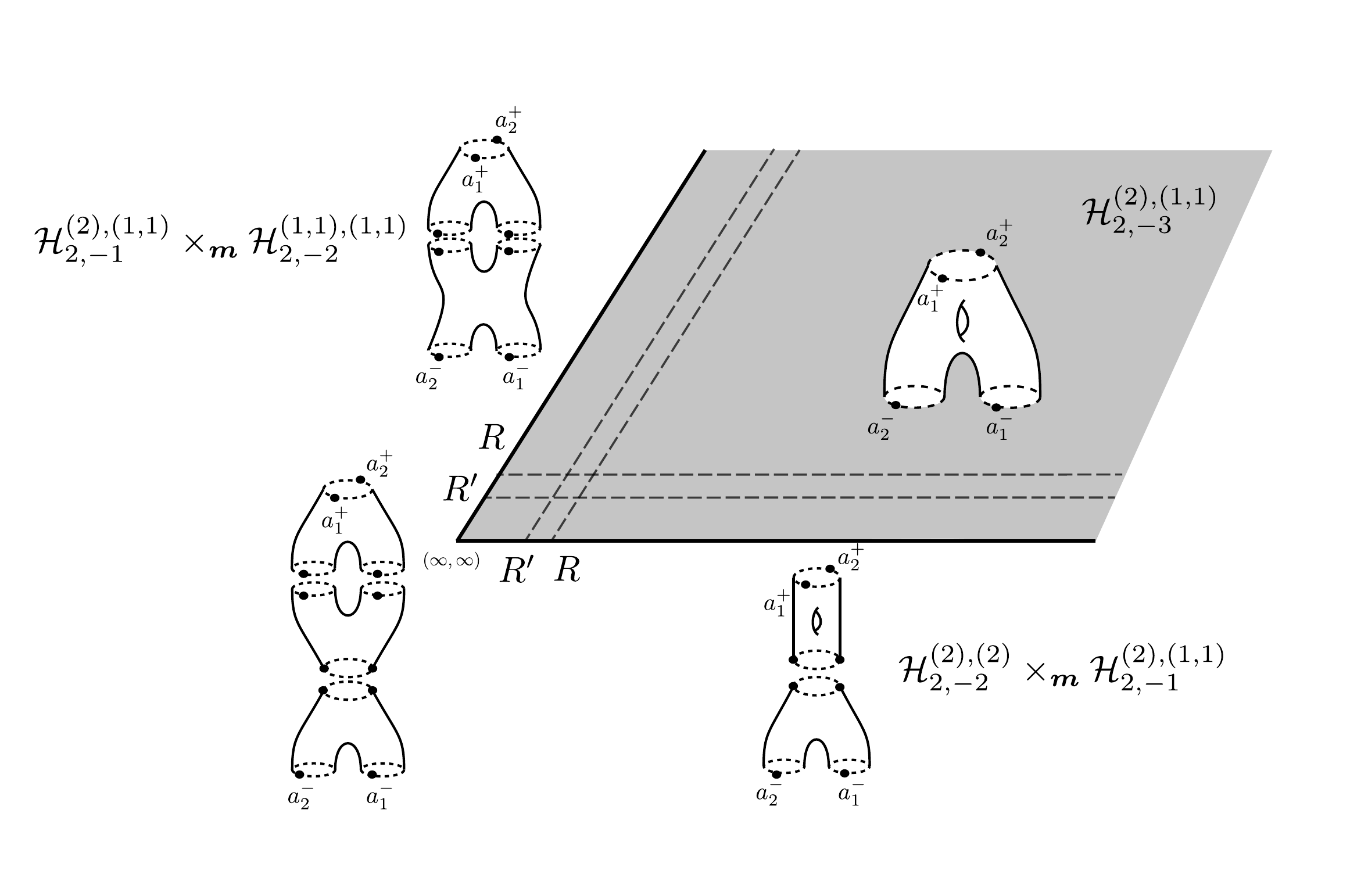} }}
    \caption{Here is a toy picture of $\overline{\mathcal{H}}^{(2), (1,1)}_{2, -3}$, i.e., we have $\kappa=2, \, \chi=-3, \,\bm \mu=(2), \, \bm \mu'=(1,1)$, and the interior  $\mathcal{H}^{(2), (1,1)}_{2, -3}$ is in gray. The black horizontal ray corresponds to a portion of the boundary associated with $\bm \mu''= (2)$; the vertical ray corresponds to $\bm \mu ''=(1,1)$. We also regard it as a cartoon for Step 0 in our construction.} 
    \label{branched-structure-example}
\end{figure}
\vskip 0.3in
{\bf Step $k=0$.} We start with $(\mathscr{B}, \Lambda)$, where $B_0=\hurwitz$, $B_1$ consists only of $\text{id}_x$ for $x \in B_0$ and $\Lambda \equiv 1$. Here, the convention is that $B_0$ denotes the objects of $\mathscr{B}$ and $B_1$ denotes its morphisms. We then modify $(\mathscr{B}, \Lambda)$ in the next steps.

\vskip 0.3in
{\bf Step $k=1$: adding buildings of height $2$.}

Recall that by Theorem~\ref{hurwitz-level-boundary-neighborhood} for any codimension-$1$ portion of the boundary $\partial \hurwitzsm$  consisting of smooth buildings of height $2$ of the form
$\mathcal{H}^{\bm \mu, \bm \mu''}_{\kappa, \chi_1} \times_{\bm m} \mathcal{H}^{\bm \mu'', \bm \mu'}_{\kappa, \chi_2}$ (for any pair $\chi_1, \chi_2$ such that $\chi_1+\chi_2=\chi$) there is a neighborhood of the form
\begin{equation}\label{boundary-neighborhood-2-level}
    U^{(\bm \mu'', \chi'')}_r\coloneqq\mathcal{H}^{\bm \mu, \bm \mu''}_{\kappa, \chi''} \times_{\bm m} \mathcal{H}^{\bm \mu'', \bm \mu'}_{\kappa, \chi-\chi''} \times (r; +\infty]
\end{equation}
for any $r \ge 0$. We also use the notation $\overline{U}_r^{\bm \mu''}$ to denote $\mathcal{H}^{\bm \mu, \bm \mu''}_{\kappa, \chi''} \times_{\bm m} \mathcal{H}^{\bm \mu'', \bm \mu'}_{\kappa, \chi-\chi''} \times [r; +\infty]$. Similarly, for boundary strata consisting of buildings of height $k$ there are neighborhoods
\begin{gather}\label{boundary-l-level-neighborhood}
    U^{(\bm \mu_1, \chi_1), \dots, (\bm\mu_{l-1}, \chi_{l-1})}_r\coloneqq\mathcal{H}^{\bm \mu, \bm \mu_1}_{\kappa, \chi_1} \times_{\bm m^1} \dots \times_{\bm m^{l-1}} \mathcal{H}^{\bm \mu_{l-1}, \bm \mu'}_{\kappa, \chi_l}  \times (r,+\infty]^{l-1},  \\ \, \text{with }\chi_1+\dots+\chi_l=\chi \nonumber
\end{gather}
for any $r \ge 0$, and we will also use notation $\overline{U}^{\bm \mu_1, \dots, \bm \mu_{l-1}}_r$, which amounts for allowing neck coordinates to be equal to $r$. We fix a large $R \gg 0$ and some $R' > R$.
We introduce a thickened forgetful covering
\begin{equation}
    \overline{bf}^1 \colon \mathcal{H}^{\bm \mu, \bm \mu''}_{\kappa, \chi''} \times \mathcal{H}^{\bm \mu'', \bm \mu'}_{\kappa, \chi-\chi''} \times (R; +\infty] \to \mathcal{H}^{\bm \mu, \bm \mu''}_{\kappa, \chi''} \times_{\bm m} \mathcal{H}^{\bm \mu'', \bm \mu'}_{\kappa, \chi-\chi''} \times (R; +\infty],
\end{equation}
which is equal to $(bf^1, \operatorname{Id})$. More generally, for $l \ge -\chi$ there are thickened forgetful covers
\begin{equation}
    \overline{bf}^{l-1}\colon \mathcal{H}^{\bm \mu, \bm \mu_1}_{\kappa, \chi_1} \times \dots \times \mathcal{H}^{\bm \mu_{l-1}, \bm \mu'}_{\kappa, \chi_{l}} \times (R; +\infty]^{l-1} \to \mathcal{H}^{\bm \mu, \bm \mu_1}_{\kappa, \chi_1} \times_{\bm m^1} \dots \times_{\bm m^{l-1}} \mathcal{H}^{\bm \mu_{l-1}, \bm \mu'}_{\kappa, \chi_{l}} \times (R; +\infty]^{l-1}.
\end{equation}
Set 
\begin{equation}\label{eq: big-interior}
    U_0= \hurwitz \setminus \Big(\bigcup_{0\ge\chi''\ge\chi, \, \bm \mu'' \vdash \kappa} \overline{U}^{(\bm \mu'', \chi'')}_{R'} \cup \bigcup_{-\chi \ge l \ge 2, \, 0\ge \chi_1+\dots+\chi_l\ge\chi, \, \bm \mu_1, \dots, \bm \mu_{l} \vdash \kappa} \overline{U}^{(\bm \mu_1, \chi_1), \ldots, (\bm\mu_{l}, \chi_{l})}_{R}\Big). 
\end{equation}
In addition, we define
\begin{equation}\label{eq: codim-1-nbhd-step-1}
\accentset{\circ}{U}^{(\bm \mu'', \chi'')}_R\coloneqq \overline{U}^{(\bm \mu'', \chi'')}_R \setminus\bigcup_{-\chi > l \ge 2, \, \chi_1+\dots+\chi_l\ge\chi, \, \bm \mu_1, \dots, \bm \mu_{l} \vdash \kappa} \overline{U}^{(\bm \mu_1, \chi_1), \ldots, (\bm\mu_{l}, \chi_{l})}_{R},
\end{equation}
and more generally 
\begin{equation}\label{eq: codim-2-nbhd-step-2}
\accentset{\circ}{U}^{(\bm \mu_1, \chi_1), \ldots, (\bm \mu_{l}, \chi_l)}_{R}\coloneqq \overline{U}^{(\bm \mu_1,\chi_1), \ldots, (\bm \mu_{l}, \chi_{l})}_R \setminus\bigcup_{-\chi > l' > l, \, \chi_1+\dots+\chi_{l'}\ge\chi, \, \bm \mu_1', \dots, \bm \mu_{l'}' \vdash \kappa} \overline{U}^{(\bm \mu_1,\chi_1), \ldots, (\bm\mu_{l'}, \chi_{l'})}_{R}.
\end{equation}
We update $(\mathscr{B}, \Lambda)$ by putting
\begin{equation}
    B_0^{\text{new}}= U_0 \sqcup \bigsqcup_{\chi_1+\chi_2=\chi, \, \bm \mu'' \vdash \kappa} (\overline{bf}^1)^{-1}(\accentset{\circ}{U}^{(\bm \mu'', \chi'')}_R).
\end{equation}

We then set the manifold $B_1^{\text{new}}$ to record the following identifications corresponding to the thickened forgetful maps. Namely given a point $x \in \mathcal{H}^{\bm \mu, \bm \mu''}_{\kappa, \chi''} \times \mathcal{H}^{\bm \mu'', \bm \mu'}_{\kappa, \chi-\chi''} \times (R; R') \subset B^{\operatorname{new}}_0$ we then set $\operatorname{Hom}(x, \overline{bf}^1(x))$ and $\operatorname{Hom}(\overline{bf}^1(x), x)$ each to consist of one arrow, and these arrows are meant to be inverses of one another, where we regard $\overline{bf}^1(x)$ as a point of $U_0$. We then complete $B_1^{\text{new}}$ by adding all the necessary arrows to make it into a groupoid. Note that in particular for any two points $x,y$ as above with $\overline{bf}^1(x)=\overline{bf}^1(y)$ the above implies that there is a unique arrow in $\operatorname{Hom}(x,y)$. We recall the notation introduced prior to Definition~\ref{defn-wnb-groupoid} for maps associated with groupoid $\mathscr{B}^{\operatorname{new}}$, and use it in the following discussion.

The weight function $\Lambda^{\text{new}}$ is set to be equal to $1$ on the closure of $\pi_H(U_0)$ and equal to $\frac{1}{N_{\bm \mu''}}$ on the image of $$(\overline{bf}^1)^{-1}(\accentset{\circ}{U}^{(\bm \mu'', \chi'')}_R) \cap \mathcal{H}^{\bm \mu, \bm \mu''}_{\kappa, \chi''} \times \mathcal{H}^{\bm \mu'', \bm \mu'}_{\kappa, \chi-\chi''} \times (R'; +\infty] \subset \mathcal{H}^{\bm \mu, \bm \mu''}_{\kappa, \chi''} \times\mathcal{H}^{\bm \mu'', \bm \mu'}_{\kappa, \chi-\chi''} \times (R; +\infty]$$ under $\pi_H$ for any $0\ge \chi''\ge\chi$ and $\bm \mu'' \vdash \kappa$. We note that these images are pairwise disjoint and all of them are disjoint from the closure of $\pi_H(U_0)$, since we have ``thrown away" the neighborhoods of buildings consisting of more than $2$ levels. It is straightforward to see that the updated pair $(\mathscr{B}^{\text{new}}, \Lambda^{\text{new}})$ is a wnbc groupoid, as one only needs to check the requirement of Definition~\ref{defn-wnb-groupoid} in the only nontrivial case of a point $q \in \mathscr{B}_H^{\operatorname{new}\vee}$ which has its neck coordinate equal to $R'$. According to our setup there is unique $\bm \mu''$ such that $q \in \pi_H((\overline{bf}^1)^{-1}(\accentset{\circ}{U}^{\bm \mu''}_R))$, and we also see that $\Lambda(q)=1$. We then notice that $q$ has exactly $N_{\bm \mu''}$ preimages under $\pi_H$, and for a small enough neighborhood $V$ of $q$ there are $N_{\bm \mu''}$ disjoint local branches $V_1, \ldots, V_{N_{\bm \mu''}} \subset (\overline{bf}^1)^{-1}(\accentset{\circ}{U}^{\bm \mu''}_R)$. Setting all the weights to be equal to $\frac{1}{N_{\bm \mu''}}$, we see that condition~\eqref{eq: local-weights} is satisfied.

\vskip 0.2in
{\bf Step $k=2$: adding buildings of height 3.}
To make our construction more illustrative, we describe in detail how we proceed to adding neighborhoods of codimension $2$ strata before describing the general inductive step.

First, we define (compare it with~\eqref{eq: big-interior})
\begin{gather}\label{eq: intermediary-subset-codim-1}
    \accentset{\circ}{U}^{(\bm \mu'', \chi'')}_{R, R'}\coloneqq \overline{U}^{(\bm \mu'', \chi'')}_{R} \setminus \\ \Big(\bigcup_{\chi_1+\chi_2 \ge \chi, \bm\mu_1, \bm\mu_2 \vdash \kappa} \overline{U}^{(\bm \mu_1,\chi_1),(\bm\mu_{2}, \chi_{2})}_{R'} \cup 
    \bigcup_{-\chi \ge l \ge 3,\, \chi_1+\dots+\chi_{l}\ge\chi, \, \bm \mu_1, \ldots, \bm \mu_{l} \vdash \kappa} \overline{U}^{(\bm \mu_1,\chi_1), \ldots,  (\bm\mu_{l}, \chi_{l})}_{R} \Big). \nonumber
\end{gather}

Denote by $\mathscr{B}^{\operatorname{old}}$ the groupoid that we obtained at the previous stage. Now we will obtain the new groupoid from it in two steps. First, we replace all sets $(\overline{bf}^1)^{-1}(\accentset{\circ}{U}_R^{\bm \mu''})$ with $(\overline{bf}^1)^{-1}(\accentset{\circ}{U}^{\bm \mu''}_{R, R'})$ and extending morphisms in a way similar to how they were defined at the Step $k=1$. Notice that we add morphisms between points with the same image in $\mathcal{H}^{\mu, \mu'}_{\kappa, \chi}$ under maps $\overline{bf}^1$. Denote the obtained topological space at the level of objects by $U_0$.
We observe that forgetful covering map $\overline{bf}^2$ factors through projections
$$\overline{bf}^2_1 \colon \mathcal{H}^{\bm \mu, \bm \mu_1}_{\kappa, \chi_1} \times \mathcal{H}^{\bm \mu_1, \bm \mu_2}_{\kappa, \chi_2} \times \mathcal{H}^{\bm \mu_2, \bm \mu'}_{\kappa, \chi_3}\times (R; +\infty]^2 \to \mathcal{H}^{\bm \mu, \bm \mu_1}_{\kappa, \chi_1} \times_{\bm m_1} \mathcal{H}^{\bm \mu_1, \bm \mu_2}_{\kappa, \chi_2} \times \mathcal{H}^{\bm \mu_2, \bm \mu'}_{\kappa, \chi_3}\times (R; +\infty]^2,$$
$$\overline{bf}^2_2 \colon \mathcal{H}^{\bm \mu, \bm \mu_1}_{\kappa, \chi_1} \times \mathcal{H}^{\bm \mu_1, \bm \mu_2}_{\kappa, \chi_2} \times \mathcal{H}^{\bm \mu_2, \bm \mu'}_{\kappa, \chi_3}\times (R; +\infty]^2 \to \mathcal{H}^{\bm \mu, \bm \mu_1}_{\kappa, \chi_1} \times \mathcal{H}^{\bm \mu_1, \bm \mu_2}_{\kappa, \chi_2} \times_{\bm m_2} \mathcal{H}^{\bm \mu_2, \bm \mu'}_{\kappa, \chi_3}\times (R; +\infty]^2.$$
Notice that $(\overline{bf}^1)^{-1}(\accentset{\circ}{U}_{R, R'}^{(\bm \mu_i, \chi_i)} \cap U_R^{{(\bm\mu_1, \chi_1), (\bm\mu_2, \chi_2)}})$ can be regarded as a subset of codomain of $\overline{bf}^2_{3-i}$ for $i=1,2$. By analogy with~\eqref{eq: codim-1-nbhd-step-1} we have

We set 
\begin{equation}
    B_0^{\operatorname{new}}=U_0 \sqcup \bigsqcup_{\chi_1+\chi_2 \ge \chi, \, \bm \mu_1, \bm \mu_2 \vdash \kappa} (\overline{bf}^2)^{-1}(\accentset{\circ}{U}^{(\bm \mu_1,\chi_1), (\bm \mu_2, \chi_2)}_R).
\end{equation}
We then add those morphisms that identify points $x, y \in (\overline{bf}^2)^{-1}(\accentset{\circ}{U}^{(\bm \mu_1,\chi_1), (\bm \mu_2, \chi_2)}_R)$ satisfying $\overline{bf}^2_i(x)=\overline{bf}_i^2(y)$ with their image in $(\overline{bf}^1)^{-1}(\accentset{\circ}{U}^{(\bm \mu_{3-i}, \chi_{3-i})}_{R, R'})$ for $i=1,2$, and add all the morphisms required to make it into a groupoid. To define the weight function we first extend $\Lambda^{\operatorname{old}}$ to $\overline{\pi_H(U_0)}$ in the natural way, and then set it to be equal to $\frac{1}{N_{\bm \mu_1}N_{\bm \mu_2}}$ at those points of $\pi_H((\overline{bf}^2)^{-1}(\accentset{\circ}{U}_R^{(\bm \mu_1,\chi_1), (\bm \mu_2, \chi_2)}))$ where it has not been determined above. 

Verification that the above construction yields a wnbc groupoid is straightforward, but tedious, and we only do it for a point $q$ in the image under $\pi_H$ of a point in $(\overline{bf}^2)^{-1}(\accentset{\circ}{U}^{(\bm \mu_1,\chi_1), (\bm \mu_2, \chi_2)}_R)$ with its neck coordinates both equal to $R'$. First, notice that $\Lambda(q)=1$, because it belongs to $\overline{\pi_H(U_0)}$. The local branches are all $N_{\bm \mu_1}\cdot N_{\bm \mu_2}$ preimages under $\pi_H$ of an open neighborhood of this point inside $(\overline{bf}^2)^{-1}(\accentset{\circ}{U}_R^{(\bm \mu_1,\chi_1), (\bm \mu_2, \chi_2)})$.

\vskip 0.2in
{\bf Inductive step $k$: adding buildings of height $k+1$.} 

First, we introduce the space
\begin{gather}\label{eq: intermediary-subset-codim-(k-1)}
    \accentset{\circ}{U}^{(\bm \mu_1, \chi_1), \ldots, (\bm \mu_{k-1}, \chi_{k-1})}_{R, R'}\coloneqq \overline{U}^{(\bm \mu_1, \chi_1), \ldots, (\bm \mu_{k-1}, \chi_{k-1})}_R \setminus \\ \Big( \bigcup_{ \, \chi_1'+\dots+\chi_{k}' \ge \chi, \, \bm \mu_1', \dots, \bm \mu_{k}' \vdash \kappa} \overline{U}^{(\bm \mu_1', \chi_1'), \ldots, (\bm\mu_{k}', \chi_{k}')}_{R'} \cup \bigcup_{-\chi \ge l >k ,\, \chi_1'+\dots+\chi_{l}' \ge \chi, \, \bm \mu_1', \dots, \bm \mu_{l}' \vdash \kappa} \overline{U}^{(\bm \mu_1', \chi_1'), \ldots, (\bm\mu_{l}', \chi_{l}')}_{R}\Big). \nonumber
\end{gather}
As in the case $k=2$ we first enlarge $B^{\operatorname{old}}_0$ that we obtained at the previous step by replacing sets $(\overline{bf}^{k-1})^{-1}(\accentset{\circ}{U}^{(\bm \mu_1, \chi_1), \ldots, (\bm \mu_{k-1}, \chi_{k-1})}_{R})$ with $(\overline{bf}^{k-1})^{-1}(\accentset{\circ}{U}^{(\bm \mu_1, \chi_1), \ldots, (\bm \mu_{k-1}, \chi_{k-1})}_{R, R'})$, and extending morphisms in the natural way. We denote by $U_0$ the space we obtain at the level of objects.

Fix $i$ such that $1 \le i \le k$, $\bm \mu'' \vdash \kappa$, and a pair $\chi_i', \chi_i''$ such that $\chi_i'+\chi_i''=\chi_i$. We notice that $\overline{bf}^k$ factors through cover 
\begin{gather}\label{eq-inductive-building-subset}
    \overline{bf}^{k-1}_i \colon  \mathcal{H}^{\bm \mu, \bm \mu_1}_{\kappa, \chi_1} \times \dots \times \mathcal{H}^{\bm \mu_{k-1}, \bm \mu'}_{\kappa, \chi_{k}} \times (R; +\infty]^{k-1} \to \\\mathcal{H}^{\bm \mu, \bm \mu_1}_{\kappa, \chi_1} \times \dots \times (\mathcal{H}^{\bm \mu_{i-1}, \bm \mu_i}_{\kappa, \chi_i'} \times_{\bm m} \mathcal{H}^{\bm \mu_i, \bm \mu_{i+1}}_{\kappa, \chi_i''}) \times \dots \times \mathcal{H}^{\bm \mu_{k-1}, \bm \mu'}_{\kappa, \chi_{k}} \times (R; +\infty)^{i-1} \times (R'; +\infty] \times (R; +\infty]^{k-i} .\nonumber
    \end{gather}

We put
\begin{equation}\label{eq-inductive-objects-update}
    B_0^{\text{new}}=U_0 \sqcup \bigsqcup_{\substack{\chi_1+\dots+\chi_{k}\ge\chi, \\ \bm \mu_1, \ldots, \bm \mu_k \vdash \kappa}} (\overline{bf}^k)^{-1}(\accentset{\circ}{U}^{(\bm \mu_1,\chi_1), \ldots, (\bm \mu_{k}, \chi_{k})}_R).
\end{equation}

To obtain $B_1^{\operatorname{new}}$ from $B_1^{\operatorname{old}}$ we add morphisms that identify points $x, y \in (\overline{bf}^k)^{-1}(\accentset{\circ}{U}^{(\bm \mu_1,\chi_1), \ldots, (\bm \mu_k, \chi_k)}_R)$ satisfying $\overline{bf}^k_i(x)=\overline{bf}_i^k(y)$ with their image in $(\overline{bf}^{k-1})^{-1}(\accentset{\circ}{U}^{(\bm \mu_{1}, \chi_{1}), \ldots, (\bm \mu_{i-1}, \chi_{i-1}), (\bm \mu_{i+1}, \chi_{i+1}), \ldots, (\bm \mu_k, \chi_k)}_{R, R'})$ for $i=1,\ldots, k$, and add all the morphisms required to make it into a groupoid.

At last, $\Lambda^{\text{new}}$ is set to coincide with $\Lambda^{\text{old}}$ at the points of $U_0$ and is set to be equal to $\frac{1}{N_{\bm \mu''_1} \cdots N_{\bm \mu''_k}}$ on the portion of the image under $\pi_H$ of $(\overline{bf}^k)^{-1}(\accentset{\circ}{U}^{(\bm \mu_1,\chi_1), \ldots, (\bm \mu_{k}, \chi_{k})}_R)$
disjoint from $\overline{\pi_H(U_0)} \subset \mathscr{B}_H^{\vee}$.

The verification that the resulting $(\mathscr{B}, \Lambda)$ is indeed a wnbc groupoid is omitted and left to the reader as an exercise.

\begin{figure}
    \centering
    \subfloat{{\includegraphics[width=13cm]{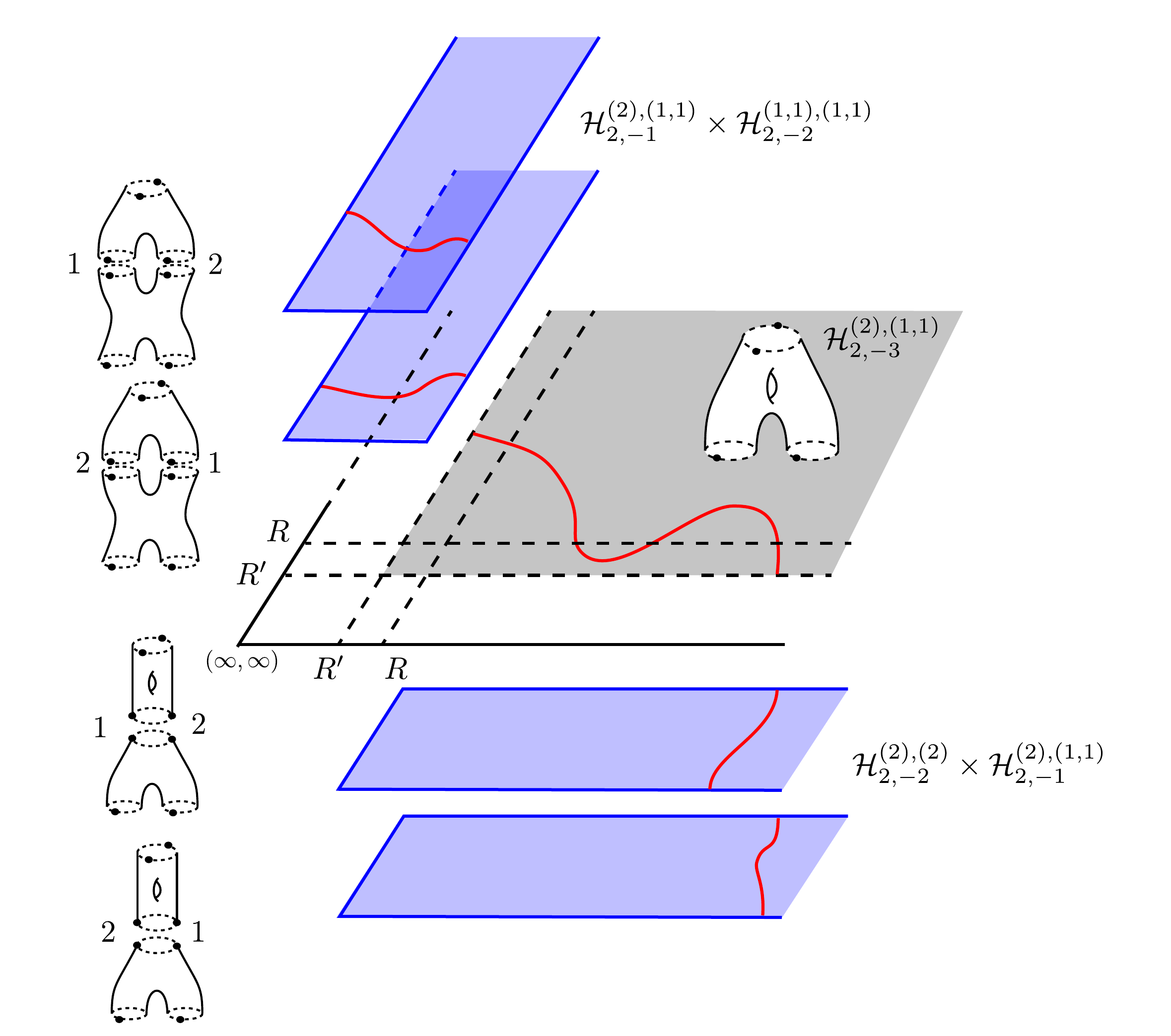} }}
    \caption{Step 1 of the construction of the wnbc groupoid $\widehat{\mathscr{H}}^{(2), (1,1)}_{2, -3}$ is shown here.  For $\bm \mu'' =(1,1)$ we have $N_{\bm \mu''}=2$ and it counts the number of ways to order a set of two orbits. For $\bm \mu'' =(2)$ we have $N_{\bm \mu''}=2$ and it counts the number of ways to choose an asymptotic marker on a time-$2$ orbit. Therefore, the construction provides four boundary neighborhoods given in blue. Notice that the projection of these blue regions is disjoint from a neighborhood strata of $2$-level buildings. The weight function takes the value $\Lambda= \frac{1}{2}$ in each of these regions. 
    The red curves depict $\mathscr{M}^{-3}(\bm x, \bm x'; H^{\mathrm{tot}}, J)$ for some $\bm x, \bm x' \in \mathcal{P}_{\mathrm{psym}}$.}
    \label{Step1}
\end{figure}
\vskip 0.3in
The wnbc groupoid obtained after the termination of the process above (after $-\chi-1$ steps) is further denoted by $(\hurwitzsmbr, \Lambda_{\hurwitzsmbr})$\label{hurwitz-smooth-branched}. We denote the interior wnb groupoid (obtained by throwing out the boundary at the level of $B_0$) by $\hurwitzbr$\label{hurwitz-smooth-branched-open}.

There is a natural projection
\begin{equation}\label{projection-from-branched}
    \mathscr{BF} \colon (\hurwitzsmbr)^{\vee} \to \hurwitzsm
\end{equation}
induced by forgetful maps. That is, for any $x \in (\hurwitzsmbr)^{\vee}$ we may pick a maximal height $k$ such that $x$ belongs to a portion of 
$$\mathcal{H}^{\bm \mu, \bm \mu_1}_{\kappa, \chi_1} \times \dots \times \mathcal{H}^{\bm \mu_k, \bm \mu'}_{\kappa, \chi_{k+1}} \times (R'; +\infty]^{k}$$
 in our inductive construction and we then send $x$ to $\overline{bf}^{k}(x)$. These locally defined maps clearly patch into a continuous map $\mathscr{BF}$. We point out that for a point $x$ as above, it holds that 
 \begin{equation}
     \Lambda_{\hurwitzsmbr}(\pi_H^{\vee}(x))=\frac{1}{N_{\bm \mu_1}\cdots N_{\bm \mu_k}}.
 \end{equation}

We will use the notation $\hurwitzbraction$ to denote the branched structure on $\hurwitzaction$ induced by $\hurwitzbr$.

\begin{remark}\label{branching-choices}
    The reasoning for this construction is provided in the next section. Roughly speaking, for breaking of profile $\bm \mu''$ one might want to have not only matching data and a non-decreasing order of entries of $\bm \mu''$ but a choice of actual markers (there are $\bm \mu''!$ choices) and a total order on partition $\bm \mu''$ (there are $R_{\bm \mu''}$ such orders). Note that our construction achieves that boundary points of $\hurwitzsm$ have the right number of preimages under $\mathscr{BF}$ to achieve this.
\end{remark}

\begin{remark}
    We point out that we may define $\hurwitzcompbr$ by extending this branched structure to the whole of $\hurwitzcomp$. Notice that one does not get a wnbc groupoid this way, since the compactification is not smooth, but the weighting conditions of Definition~\ref{defn-wnb-groupoid} are satisfied. 
\end{remark}

\subsection{Floer data}
\label{section-floer-data}

For each $\hurwitz$ and all elements $(S, \pi, \bm p^{\pm}, \bm a^{\pm}, B) \in \hurwitzaction$ we choose a \emph{cylindrical ends data} on $S$ such that $\bm p^+$ is the set of \emph{inputs} and $\bm p^-$ is the set of \emph{outputs}, i.e., a collection of holomorphic embeddings $\{\delta^i_+\}_{i=1}^{i=l(\bm \mu)}$ and  $\{\delta^i_-\}_{i=1}^{i=l(\bm \mu')}$ satisfying compatibility conditions that vary smoothly:

\begin{itemize}
    \item $\delta_+^i \colon \R_{\ge R_i^+} \times \R/\mu_i\Z \to S, \, \lim \limits_{s \to +\infty} \delta_+^i(s, \cdot)=p_i^+, \, \lim \limits_{s \to +\infty} -\frac{d}{ds}\delta_+^i(s, 0)=a_i^+;$
    \item \emph{compatibility with $\pi$: }$\pi (\delta_+^i(s,t))=(s,t)$, where~$$\pi \circ \delta_+^i \colon \R_{\ge R_i^+} \times \R/\mu_i\Z \to \R \times \R/\Z;$$
    \item $\delta_-^i \colon \R_{\le -R_i^-} \times \R/\mu_i'\Z \to S, \, \lim \limits_{s \to -\infty} \delta_-^i(s, \cdot)=p_i^-, \, \lim \limits_{s \to -\infty} -\frac{d}{ds}\delta_-^i(s, 0)=a_i^-$;
    \item \emph{compatibility with $\pi$: }$\pi (\delta_-^i(s,t))=(s,t)$, where~$$\pi \circ \delta_-^i \colon \R_{\le -R_i^-} \times \R/\mu_i'\Z \to \R \times \R/\Z.$$
\end{itemize}
In particular, the constants $R_i^+, R_i^-$ vary smoothly across $\hurwitzaction$.
It is required that the above choices are invariant under the $\R$-action on $\hurwitzaction$, which provides a choice of cylindrical ends data on $\hurwitz$.

A choice of cylindrical ends data on $(\hurwitzbr)^{\vee}_H$ is similar to the above; one has to choose such data on objects $\left(\hurwitzbr\right)_0$ making sure that it coincides for the pairs of points identified under morphisms $\left(\hurwitzbr\right)_1$. In particular, any cylindrical ends data on $\hurwitz$ provides cylindrical ends data on $\hurwitzbr$.

Now, assume there are choices of cylindrical ends data on $\hurwitzbr$ for all $\chi'$ with $\chi'> \chi$. Then we say that a choice of cylindrical end data is \emph{consistent} if for any (see~\eqref{boundary-l-level-neighborhood})
\begin{equation}\label{eq-point-near-bdry-branched}
([S=S_1 \cup \ldots \cup S_{k+1}, \pi], R_1, \ldots, R_k) \in U^{(\bm \mu_1, \chi_1), \dots, (\bm\mu_{k}, \chi_{k})}_R \subset \hurwitz 
\end{equation}
with $R_1, \ldots, R_k \gg 0$ the cylindrical ends coincide with the ones induced from positive (resp. negative) ends of $S_1$ (resp. $S_{k+1}$).

Furthermore, for a point of $[S, \pi] \in \hurwitzsm$ represented by~\eqref{eq-point-near-bdry-branched} with $R_1, \ldots, R_k > R'$ 
we may assign \emph{finite cylinders} data. That is, the domain $S'$ comes with $l(\bm \mu''_j)$ necks of length $2R_j$ induced by the gluing procedure described in~\eqref{nodal-resolution-cyl} and~\eqref{deformation-hurwitz} for each $j$ with $1 \le j \le k$. 

Moreover, since $[S, \pi]$ above is the element of a product of Hurwitz spaces, the cylindrical end data induces trivializations of portions of the necks as above. Namely, let $p^{i,-}_j$ be the $j$-th negative puncture of $S_i$ and $p^{i+1,+}_j$ be the $j$-th positive puncture of $S_{i+1}$. Both of the punctures come with associated cylindrical ends $\delta^{j}_{i, -} \colon \R_{\ge R^{i, -}_j} \times \R/ \mu''_{ij}\Z \to S'$ and $\delta^{j}_{i, +} \colon \R_{\le R^{i, +}_j} \times \R/ \mu''_{ij}\Z \to S'$. Then, under the gluing procedure dictated by the matching marker as in~\eqref{eq-asymptotic-to-matching}, a finite portion of $\delta^{j}_{i, -}$ is glued to a finite portion of $\delta^{j}_{i, +}$. This gives rise to a \emph{finite cylinder}
\begin{itemize}
    \item $\delta^j_{i,r} \colon [c_r, d_r] \times \R/ \mu''_{ij} \Z \to S$, which satisfies compatibility with $\pi'$ up to a vertical translation.
\end{itemize}

We then say that a choice of finite cylinders and cylindrical ends on $\hurwitzsmbr$ is a \emph{consistent cylindrical data} if for all points $[S, \pi] \in (\hurwitzsmbr)^{\vee}_H$ represented by points of the form~\eqref{eq-point-near-bdry-branched} the finite cylinders and cylindrical ends are induced from the boundary by the procedure above, and these choices vary smoothly over $\left(\hurwitzbr\right)_0$. Here we assume that finite cylinders $\delta^\cdot_{i,r}$ only show up for $R_i \ge R'$ (as in Section~\ref{branched-manifolds}) and for $R_i=R'$ they have length $0$.
We claim that such data exists and can be obtained inductively for any $\hurwitzsmbr$. Moreover, we claim that this data can be obtained from finite cylinder data. 
\begin{remark}
    We point out that only for in the branched manifold setting of $\hurwitzsmbr$ we can associate finite cylinders near the boundary consistently with with cylindrical ends of points in the boundary of $\hurwitzsmbr$. This would not be possible for $\hurwitzsm$ as boundary components do not come naturally with cylindrical ends for all levels. See also Remark~\ref{branching-choices}.
\end{remark}
\begin{definition}\label{floer-datum} A \emph{Floer datum}  on $[S, \pi] \in \hurwitzsmbr$ is a collection $(H_{(S,\pi)}, F_{(S, \pi)}, J_{(S, \pi)})$ where (note that we make here the same choices for all $(S, \pi)$ representing the same point in $\left(\hurwitzsmbr\right)^{\vee}_H$):
\begin{enumerate}[(FD1)]
    \item \label{def-floer-datum1} $H_{(S, \pi)} \colon S \to \mathcal{H}(\hat{M})$ is a domain-dependent Hamiltonian function, required to coincide with $H_0$ on the cylindrical ends and finite cylinders,  additionally it vanishes on the $\pi$-preimages of $\varepsilon$-neighbrohoods of branched points for a fixed small $\varepsilon>0$;
    \item \label{def-floer-datum2} $F_{(S, \pi)} \colon S \to C^\infty (\hat{M})$ is a domain-dependent function on $S$ such that
    \begin{itemize}
        \item  for each cylinder, i.e., cylindrical end $\delta^j_{i, \pm}$ or finite cylinder $\delta^j_{i,r}$, the function $F_{(S, \pi)}$ coincides with $F_{\mu''_{ij},\op{ord}_j(\bm \mu''_i)}$ away from a small collar region (for a more precise description of collar regions and behavior of $F_{(S, \pi)}$ see \cite[Section 4]{ganatra2012symplectic});
        \item $F_{(S, \pi)}$ is weakly monotonic, i.e. $\partial_s F_{(S, \pi)} \le 0$ on finite cylinders and cylindrical ends;
        \item $F_{(S, \pi)}$ is locally constant away from finite cylinders and cylindrical ends;
    \end{itemize}
    
    \item \label{def-floer-datum3} $J_{(S, \pi)}\colon S \to\mathcal{J}_{\hat{M}}$ is a domain-dependent almost complex structure coinciding with $J_t$ as in~(\ref{fixed-almost-complex}) on each cylindrical end and on each finite cylinder. It is required that $J \equiv J^0$ when restricted to the preimages under $\pi$ of $\varepsilon$-neighborhoods of branched points for some fixed small $\varepsilon >0$.
\end{enumerate}
\end{definition}

Given some Floer data as above we introduce the function $H^{\mathrm{tot}}_{(S, \pi)}=H_{(S, \pi)}+F_{(S, \pi)}$. We point out that $H^{\mathrm{tot}}_{(S, \pi)}$ coincides with $H^{\mathrm{tot}}_{\mu_i, \op{ord}_i(\bm \mu)}$ at the positive cylindrical end corresponding to $p_i^+$, with $H^{\mathrm{tot}}_{\mu'_i, \op{ord}_i(\bm \mu')}$ at the negative cylindrical end corresponding to $p_i^-$ and with $H^{\mathrm{tot}}_{\mu''_{ij},\op{ord}_j(\bm \mu''_i)}$ on each finite cylinder $\delta^j_{i,r}$.

We further associate a $1$-form $K_{(S, \pi)} \in \Omega^1(S, C^\infty(\hat{M}))$ of the form 
\begin{equation}\label{hamiltonian-form}
K_{(S,\pi)}=H^{\mathrm{tot}}_{(S, \pi)}\cdot \pi^*(dt).
\end{equation}
The $1$-form $K_{(S,\pi)}$ determines a vector-field-valued $1$-form $Y_{(S,\pi)} \in \Omega^1(S, C^\infty (T\hat{M}))$, such that $Y_{(S, \pi)}( \xi)$ for $\xi \in TS$ is the Hamiltonian vector field
of $K_{(S, \pi)}(\xi)$.

We denote by $X_{H^{\mathrm{tot}}_{(S, \pi)}}$ the domain-dependent Hamiltonian vector field associated with $H^{\mathrm{tot}}_{(S, \pi)}$. For $K_{(S, \pi)}$ as above the $1$-form $Y_{(S, \pi)}$ is given by 
\begin{equation}\label{vector-form}
Y_{(S, \pi)}=X_{H^{\mathrm{tot}}_{(S, \pi)}} \pi^*(dt).
\end{equation}

Following  \cite[Section 9i]{seidel2008fukaya} and \cite[Section 4]{ganatra2012symplectic} we introduce the notion of \emph{consistent Floer data} on $\hurwitzsmbr$. That is, a choice of Floer datum for all $\bm \mu, \bm \mu', \chi$ and for all $[S, \pi] \in \hurwitzsmbr$ is said to be consistent if it varies smoothly over each $\hurwitzbr$ and for $[S, \pi] \in \hurwitzbr$ it coincides with the one induced from the boundary for the points in the preimages of $\mathscr{BF}$ in the charts of the form~\eqref{boundary-l-level-neighborhood}.

\begin{remark}\label{data-on-compactification}
    One may extend Floer data to the full compactification $\hurwitzcompbr$ by requiring that it is equal to a pair $(0, J^0)$ at each bubble component. It is a continuous extension according to conditions~\eqref{def-floer-datum1} and~\eqref{def-floer-datum3}.
\end{remark}

Consider a universal family $\mathscr{S}^{\bm \mu, \bm \mu'}_{\kappa, \chi} \to \hurwitzbr$ (see \cite[Theorem 1.53]{harris-morrison2006} for the algebraic analog) whose fiber over a point $[S, \pi] \in \left(\hurwitzbr\right)_0$ is the domain $S$ of the cover $\pi$.
Then a choice of consistent Floer data on $\hurwitzsmbr$ provides a tuple
$$H \in C^\infty(\mathscr{S}^{\bm \mu, \bm \mu'}_{\kappa, \chi}, \mathcal{H}(\hat{M})), \, F \in C^\infty(\mathscr{S}^{\bm \mu, \bm \mu'}_{\kappa, \chi}, C^{\infty}_b(\hat{M})), \, J \in C^\infty(\mathscr{S}^{\bm \mu, \bm \mu'}_{\kappa, \chi}, \mathcal{J}(\hat{M})).$$
We also have a $1$-form $K$ associated with $H^{\mathrm{tot}}=H+F$:
$$ K \in \Omega^1_{\mathscr{S}^{\bm \mu, \bm \mu'}_{\kappa, \chi} / \hurwitzbr}(\mathscr{S}^{\bm \mu, \bm \mu'}_{\kappa, \chi}, C^\infty(\hat{M})).$$

A vector-field-valued $1$-form $Y$ is assigned to each such $K$ (see~(\ref{vector-form})).
Consistency allows one to extend this tuple to an analogous tuple $(\widehat{H}, \widehat{F}, \widehat{J})$ for$~\widehat{\mathscr{S}}^{\bm \mu, \bm \mu'}_{\kappa, \chi}~\to~\hurwitzsmbr$. 

Given consistent Floer data, one may assign a domain-dependent almost complex structure $\underline{J}^{H}$ on $\R \times S^1 \times \hat{M}$ as follows: 
\begin{align}
    \left\{
        \begin{array}{ll}
           \underline{J}^H \text{ preserves $\{s\} \times \{t\} \times T\hat{M}$ and } J|_{ T\hat{M}} = J_{(S, \pi), z} \text{ for }z\in S, \\
            \underline{J}^H_{(S, \pi), \, z}(\partial_s)=\partial_t+X_{H^{\mathrm{tot}}_{(S, \pi)}(z)}\text{ for $z \in S$}.
        \end{array}
    \right.
\end{align}

Observe that $\underline{J}^H_{(S, \pi), \, z}$ is \emph{adjusted} in the sense of \cite{BEHWZ2003sft} to the stable Hamiltonian structure on $S^1 \times \hat{M}$ given by 
\begin{equation}
(\omega^{H^{\mathrm{tot}}_{(S, \pi)}(z)}=\omega+dH^{\mathrm{tot}}_{(S, \pi)}(z) \wedge dt, \lambda^{H^{\mathrm{tot}}_{(s, \pi)}(z)}=dt),
\end{equation}
with the symplectic vector bundle and the Reeb vector field given by
\begin{equation}
    \xi^{H^{\mathrm{tot}}_{(S, \pi)}(z)}=T\hat{M} \, \, \, \text{ and } \, \, \, R^{H^{\mathrm{tot}}_{(S, \pi)}(z)}=\partial_t+X_{H^{\mathrm{tot}}_{(S, \pi)}(z)}.
\end{equation}

\subsection{Pseudo-holomorphic curves}\label{section-pseudo-holomophic-curves}
Let us pick two tuples of Hamiltonian orbits $\bm x, \bm x' \in \mathcal{P}_{\mathrm{psym}}$ as in Section~\ref{section-hamiltonian-partsym} and consistent Floer data on $\hurwitzsmbr$ as in Section~\ref{branched-manifolds} and set $\bm \mu = \bm \mu (\bm x)$ and $\bm \mu' = \bm \mu (\bm x')$. 

Consider the set of pairs 
\begin{equation}\label{pairs}
\Big( [(S, \pi, \bm p^{\pm}, \bm a^{\pm}, B)] \in \left(\hurwitzbr \right)_0, \, u=(\pi', v)\colon S \to \R \times S^1 \times \hat{M}\Big) 
\end{equation}
\noindent such that
\begin{equation}\label{pairs-CR}
\begin{dcases}
  [S, \pi']= [S, \pi];\\
(dv-Y_{(S, \pi)})^{0,1}=0, \text{ with respect to } J_{(S,\pi)};  \\
\lim_{s \to +\infty } v\circ \delta_i^+(s, \cdot) = x_i(\cdot) \text{ for }i=1, \ldots, l(\bm \mu); \\
\lim_{s \to -\infty } v\circ \delta_i^-(s, \cdot) = x_i'(\cdot) \text{ for }i=1, \ldots, l(\bm \mu').
\end{dcases}
\end{equation}
We say that the pairs
$$\Big( [(S_1, \pi_1, \bm p^{\pm}, \bm a^{\pm}, B)], \, u=(\pi'_1, v_1) \Big) \text{ and } \Big( [(S_2, \pi_2, \bm p^{\pm}, \bm a^{\pm}, B)], \, u=(\pi'_2, v_2) \Big)$$
\noindent are \emph{equivalent} if $[S_1, \pi_1]=[S_2, \pi_2]$ and $v_1=v_2$.

\begin{lemma}\label{sft-to-hamiltonian}
    A curve $u=(\pi, v) \colon S \to \R \times S^1 \times \hat{M}$ with $[S, \pi] \in \left( \hurwitzbr \right)_0$ is $\underline{J}^H_{(S, \pi)}$-holomorphic precisely when $v \colon S \to \hat{M}$ satisfies
    \begin{equation}\label{hamiltonian-cr}
        (dv-Y_{(S, \pi)})^{0,1}=0.
 \end{equation}
\end{lemma}
\begin{proof}
    This is a restatement of \cite[Proposition 2.2]{fabert2009}.
\end{proof}

\begin{remark}
    The consequence of Lemma~\ref{sft-to-hamiltonian} is that one may apply both the SFT technology and the Hamiltonian Floer approach depending on the circumstances.
\end{remark}
We define $\modulifloerbr$~\label{moduli-floer-branched}
as a weighted branched groupoid with $\modulifloerbr_0$ given by the equivalence classes of pairs as in~\eqref{pairs-CR} (i.e., objects of this groupoid are equivalence classes of pairs of a branched cover and a pseudo-holomorphic curve up to $s$-translation). We write $\modulifloerbraction$ for the space of pairs as above, not equivalence classes.

The manifold $\modulifloerbr_1$ of morphisms and the weight function come from the wnbc groupoid structure on $\hurwitzbr$ as explained in the discussion below and in Section~\ref{branched-manifolds}. In particular, we have that $\operatorname{Mor}(([S, \pi], v), ([S', \pi'], v'))$ is non empty if $\operatorname{Mor}([S, \pi], [S', \pi'])$ is non empty and $v=v'$.

For a small open $V \subset \left(\hurwitzbr \right)^{\vee}_H$ and the corresponding groupoid $\mathscr{V}=\hurwitzbr|_{\pi_H^{\vee-1}(V)}$ one may associate a trivial Banach fiber orbibundle $\mathscr{B}|_{\mathscr{V}} \to \mathscr{V}$ whose fiber over $$[(S, \pi, \bm p^{\pm}, \bm a^{\pm}, B)]~\in~\left(\hurwitzbr \right)_0$$ is the space $\mathscr{B}_{(S, \pi)}$ of maps in $W^{1,p}_\text{loc}(S, \hat{M})$ converging at infinity on cylindrical ends associated with $\bm p^\pm$ to some pair $(\bm x, \bm x')$.
The groupoid $\mathscr{B}|_{\mathscr{V}}$ inherits a wnb groupoid structure from ${\mathscr{V}}$. Further, there is a vector orbibundle $\mathscr{E}|_{\mathscr{V}} \to \mathscr{B}|_{\mathscr{V}}$ with fibers $\mathscr{E}_{[S, \pi]} \to \mathscr{B}_{[S, \pi]}$ given by $L^p(S, \Omega^{0,1}_S\otimes v^*T\hat{M})$, for $v\in \mathscr{B}_{[S, \pi]}$.
Then one may interpret $(dv-Y_{(S, \pi)})^{0,1}$ as a section $\sigma$ of this vector orbibundle and the zero-set of this section is the union of $\modulifloerbraction_\mathscr{V}$  for all pairs $(\bm x, \bm x')$. The derivative of this section at $((S, \pi), u)$ is a \emph{linearized operator}:
\begin{equation}\label{eq-linearized-operator}
    D_{[S, \pi], u} \colon (T\mathscr{B}|_{\mathscr{V}})_{[S, \pi], u}=T_{[S, \pi]}\hurwitzbr \times T_v\mathscr{B}_{[S, \pi]} \longrightarrow (\mathscr{E}_{[S, \pi]})_u.
\end{equation}

Then $\modulifloerbraction$ is \emph{transversely cut out} if for any such $U_0$ the section $\sigma$ is transverse in the sense of Section~\ref{branched-manifolds}. This is equivalent to the surjectivity of $D_{[S, \pi], u}$ for any point in the zero set of $\sigma$. We refer the reader to \cite[Remark 9.4]{seidel2008fukaya} for the reasoning behind considering the local orbibundle $\mathscr{B}|_{\mathscr{U}}$ instead of the global one $\mathscr{B}_{\hurwitzbr}$ (which would not be a Banach manifold).

\begin{remark}\label{rmk: regulkarity-of-u}
    We note that surjectivity of the linearized $\bar{\partial}$-operator with parameters given by~\eqref{eq-linearized-operator} is equivalent to surjectivety of linearized $\bar{\partial}$-operator associated with curve $u$, i.e. $u$ is regular whenever the operator~\eqref{eq-linearized-operator} is surjective.
\end{remark}

\begin{theorem}\label{transversality-theorem}
    For a generic Floer data on $\hurwitzbr$ the moduli space $\modulifloerbraction$ is a wnb groupoid.
\end{theorem}
\begin{proof}
    {\bf Case $\chi=0$.} First, consider the special case when $\chi=0$ (see Remark~\ref{rmk: chi=0 case}). The moduli space $\hurwitz$ is nonempty only if $\bm \mu = \bm \mu'$ and consists of a finite number of point,s each corresponding to a collection of $l(\bm \mu)$ cylinders $C_1, \ldots, C_{l(\bm \mu)}$. According to the classical transversality results in symplectic homology, there is a comeagre set of perturbations of $F_0$ which guarantee transversality for each of these cylinders alone (see, for instance, \cite[Section 8.6]{audin-damian2014}). Since, for a collection of disjoint cylinders, the linear differential operator assigned to~(\ref{hamiltonian-cr}) splits, there is a comeagre set of perturbations of $l(\bm \mu)$ copies of $F_0$ (considered on different circles $\R /\mu_i\Z \cong S^1$).
    Varying over all $\bm \mu \vdash \kappa$, one may pick a collection of Hamiltonians in ${\mathcal{ \bm H}}_{S^1}^\kappa(\hat{M})$ such that transversality is achieved for each cylinder alone for all Hurwitz spaces as above. Notice that different permutations may have shared desired Hamiltonian function assigned to them (e.g., $H^{\mathrm{tot}}_{1,1}$), but the transversality as above still can be achieved since the intersection of comeagre sets is comeagre.
    
    {\bf Case $\chi<0$.}
     We start by picking a family $\Omega_{[S, \pi]} \subset \text{Int}(S)$ of open neighborhoods smoothly depending on $[(S, \pi, \bm p^\pm, \bm a^\pm, B)] \in (\hurwitzsmbr)^{\vee}$, such that each such $\Omega_{[(S, \pi)]}$ is disjoint from cylindrical ends $\delta^i_\pm$ (in case of cylinder components that means the whole component) and finite cylinders $\delta^j_{i,r}$. We require that $\Omega_{[S, \pi]}$ intersects non-trivially with every non-cylindrical component of $S$. 

    We assume that we are already given a consistent Floer data $(H, F, J)$ such that its restriction to $\partial \left(\hurwitzsmbr\right)_0$ guarantees regularity. Let $\mathscr{F}^{\bm \mu, \bm \mu'}_{\kappa, \chi}$ denote the groupoid of Floer data on $\hurwitzsmbr$ with fixed $F$.  Following \cite[Section 9k]{seidel2008fukaya} we 
    now will pick a perturbation $(\delta H, \delta J) \in \mathscr{T} = T\mathscr{F}^{\bm \mu, \bm \mu'}_{\kappa, \chi}$ such that for perturbed Floer data the regularity holds for the interior of $\hurwitzsmbr$ as well (note that we are not going to perturb $F$). With $\delta H$ there is an associated perturbation $\delta K$ of the $1$-form $K$.  We require that $(\delta K, \delta J)$ vanishes away from $\Omega$.

    To each point $([S, \pi], u=(\pi', v)) \in \modulifloerbr_0$ there is an associated universal linearized operator $D_{\mathscr{T},[(S, \pi)], u}$ which differs from $D_{[(S, \pi)]), u}$ by additional terms responsible for varying Floer data (note that $D_{\mathscr{T},[(S, \pi)], u}$ only depends on $v$-part of $u$). In a local trivialization, it is given by
\begin{gather}
    D_{\mathscr{T}, [(S, \pi)], u} \colon \mathscr{T} \times (T\mathscr{B}|_{\mathscr{U}})_{[S, \pi], u} \to (\mathscr{E}_{[S, \pi]})_u,\\
(\delta K,\delta J,Z,X ) \mapsto (\delta Y)^{0,1}+\delta J \circ \frac{1}{2}(dv-Y)\circ j_S+D_{([S, \pi], u)}(Z, X). \nonumber
\end{gather}

Since $u$ is somewhere injective by assumptions ~\ref{hamiltonian-properties2} and~\ref{hamiltonian-properties3} on $H^{\mathrm{tot}}_{m,j}$'s, the Sard-Smale theorem then would imply regularity if one shows the surjectivity of the above map. 
The image of $D_{\mathscr{T}, [(S, \pi)], u}$ is closed since the last summand in the above is Fredholm. If $\eta$ is an element of the cokernel it must be smooth on $\text{int}(S)$ and satisfy $D^*_{([S, \pi], u)}\eta=0$. Additionally,

$$\int_S \langle (\delta Y)^{0,1}, \eta \rangle=0$$
for all $\delta Y$. This implies that $\eta$ has to vanish on all of $\Omega_{[S, \pi]}$ (for computation see \cite[Lemma 4.3]{fabert2009}).  
By the unique continuation theorem, it has to vanish on all non-cylindrical components of $S$, and cylindrical components are dealt with in the previous case. Notice that for this argument to work it suffices to consider only the perturbations for which $\delta Y (\partial_s)=0$ in an appropriate local chart.

The transversality of the section then implies that there is a wnb groupoid structure on the zero set as pointed out in Section~\ref{branched-manifolds}.
\end{proof}

\begin{remark}
    One could adopt the proof in \cite{cielebak-mohnke2007} to show the transversality via the SFT approach. This could be used to extend our invariants to more generic stable Hamiltonian structures where modeling the general problem on Hurwitz spaces instead of Deligne-Mumford moduli spaces is still applicable, i.e., in the presence of the additional $S^1$-direction (e.g., for mapping tori). We do not pursue this venture here.  
\end{remark}

\begin{lemma}\label{index-count}
    The dimension of $\modulifloerbr$ for generic $(H, F, J)$ is given by
    \begin{equation}
        \text{dim}_{\R}\modulifloerbr=\operatorname{ind}(u)=(n-2)\chi+|\bm x'|-|\bm x|-1.
    \end{equation}
\end{lemma}
\begin{proof}
    Let $([(S ,\pi)], u)$ be a representative of an element of $\modulifloerbr_0$.  For $v$ and fixed $S$ one may apply standard results \cite{bourgeois2002} to get
\begin{equation}
\text{ind}(v)=n\chi+|\bm x'|-|\bm x|.
\end{equation}
 Adding the dimension of the Hurwitz space $\text{dim}_{\R}\hurwitz=-2\chi-1$ to $\text{ind} (v)$ gives the result. We point out that this quantity is clearly equal to the Fredholm index of the curve $u$.
\end{proof}

\begin{lemma}\label{orientation-lemma}
    Assuming that $\modulifloerbr$ is a wnb groupoid, there is a canonical isomorphism 
    \begin{equation}\label{eq-orientation-floer}
       \big|  \modulifloerbr \big| \otimes |\R \partial_s| \otimes o_{\bm x}  \cong \big| \hurwitzbraction \big| \otimes o_{\bm x'} \otimes \zeta^{n\chi}.
    \end{equation}
    \begin{proof}
        We refer the reader to Section~\ref{section-orientation-lines} for the notations in~\eqref{eq-orientation-floer}, e.g., the definition of $\zeta$. The proof follows immediately from \cite[Lemma C.4]{abouzaid2010geometric}. 
    \end{proof}
\end{lemma}

Now, in the case where $|\bm x'|-|\bm x|+(n-2)\chi=1$, by permuting factors, we get an isomorphism
\begin{equation}\label{eq-hurwitz-orline}
    |\R\partial_s| \otimes o_{\bm x} \cong o_{\bm x'}\bigl[-n\chi\bigr] \otimes \big| \hurwitzbraction \big| \cong o_{\bm x'}\bigl[(2-n)\chi\bigr],
\end{equation}
where we note that there is no change of sign since $\hurwitzbraction$ is of even dimension. The last isomorphism in the row above comes from choosing the standard orientation on $\hurwitzaction$ coming from the complex manifold structure on it. Hence, choosing the standard orientation on $\R$, the above induces:
\begin{equation}\label{orientation-map-floer}
    \partial_ u \colon o_{\bm x} \to o_{\bm x'}\bigl[(2-n)\chi\bigr].
\end{equation}
We extend this map to any generator $[\bm x]\hbar^{j}$ of $SC^*_{\kappa, \mathrm{psym}}(\hat{M})$ by tensoring the above isomorphism with $\zeta^{(2-n)j}$ and multiplying with $(-1)^{(2-n)j}$:
\begin{equation}\label{orientation-map}
    \partial_u \colon \zeta^{(2-n)j }\otimes o_{\bm x} \to \zeta^{(2-n)j }\otimes o_{\bm x'}\bigl[(2-n)\chi\bigr] \cong o_{\bm x'}\bigl[(2-n)(\chi-j)\bigr] .
\end{equation}
The sign appears from the Koszul sign corresponding to the following chain of isomorphisms:
\begin{gather}
    |\R\partial_s| \otimes \zeta^{(2-n)j} \otimes o_{\bm x} \cong (-1)^{(2-n)j} \zeta^{(2-n)j} \otimes |\R\partial_s| \otimes o_{\bm x} \cong \zeta^{(2-n)j }\otimes o_{\bm x'}\bigl[(2-n)\chi\bigr] \cong \\
    \cong o_{\bm x'}\bigl[(2-n)(\chi-j)\bigr]. \nonumber
\end{gather}

\bigskip
\emph{Ghost bubbles and Kuranishi replacements.} For $n \ge 3$, there is a technical difficulty as in $1$-parameter families of curves ghost bubbles may occur as discussed in \cite{colin2020applications}. That is, a sequence of curves $\{u_i\}_{i=1}^{i=\infty}$, in $\modulifloerbr$, after passing to a subsequence, can limit to 
$$(u_\infty= u_{\infty,1}\cup u_{\infty,2}),$$ 
where:
\begin{itemize}
\item[(i)] the {\em main part} $u_{\infty,1}$ is the union of components that are not ghost bubbles; each component of the $u_{\infty,1}$ is somewhere injective; and 
\item[(ii)] $u_{\infty,2}$ is a union of \emph{ghost bubbles}, i.e., locally constant maps where the domain is a possibly disconnected compact Riemann surface and $u_{\infty,2}$ maps to points on $\operatorname{Im}(u_{\infty,1})$.
\end{itemize} 

To overcome this issue, we excise the portion $\mathcal{B}$ of $\modulifloerbr$ that is close to breaking into $u_\infty= u_{\infty,1}\cup u_{\infty,2}$ with $u_{\infty,2}\not=\varnothing$ and replace it with a Kuranishi model $\mathcal{B}'$ of a neighborhood of all possible $u_{\infty,1}\cup u_{\infty,2}$ satisfying (i) and (ii).  

This provides us with the \emph{Kuranishi replacements} as in \cite{colin2020applications} $\modulifloerbrKR$, $\modulifloerbrclosedKR$, and for all $\chi$, $\chi'$, $\bm x$, $\bm x'$, such that:
\begin{enumerate}[(KR1)]
\item \label{Kuranishi-replacements1} $\modulifloerbrKR$ are transversely cut out branched manifolds; 
\item \label{Kuranishi-replacements2} $\modulifloerbrclosedKR$ is a transversely cut out branched manifold that consists of perturbed holomorphic maps from connected closed Riemann surfaces which are homologous to a constant map (see e. g. \cite[Section 3.5]{Pardon2019contact}); and
\item \label{Kuranishi-replacements3} all the strata of $\partial \modulifloerbrKR$ are finite fiber products of moduli spaces of the form $\mathscr{M}^{\chi_i, \sharp}(\bm x_i, \bm x_{i+1}; H^{\mathrm{tot}}, J) \times \mathscr{M}^{\chi_j,\sharp}$.
\end{enumerate}

From now on, for $n\ge 3$ we pass to the necessary Kuranishi replacements for $\modulifloerbr$ and omit $\sharp$ from the notation. This, in particular, affects the discussion to follow in Section~\ref{section-differential}.

For a thorough discussion on Kuranishi replacements, we refer the reader to the works of Fukaya-Oh-Ohta-Ono \cite{FOOO,FOOO2,FOOO3} and McDuff-Wehrheim \cite{MW,McD, mcduffwehrheim2024}. Also see the approach of Hofer-Wysocki-Zehnder \cite{HWZ17} and Pardon~\cite{Pardon2016virtual, Pardon2019contact}.

\subsection{Compactness}
\label{section-compactness}

We first show that
\begin{lemma}
\label{lemma-cpt-finite}
    For a generic choice of Floer data $(H, F, J)$, given $\bm x \in \mathcal{P}_{\mathrm{psym}}$ the groupoid $\modulifloerbr$ is empty for all but finitely many tuples $\bm x' \in \mathcal{P}_{\mathrm{psym}}$.
\end{lemma}
\begin{proof}
    For a representative $([(S, \pi)], u=(\pi, v))$ of an element of $\modulifloerbr$, the curve $v$ is a solution of perturbed Cauchy-Riemann equation~\eqref{hamiltonian-cr}. Therefore, the general framework of the proof given in \cite[Lemma 5.2]{abouzaid2010geometric} can be applied. That is, according to \cite[Lemma A.1]{ganatra2012symplectic} for $R$ large enough the tuple $\bm x'$ with all its orbits lying in $[R, +\infty) \times \partial M$ has (very) negative action $\mathcal{A}(\bm x')$. Hence, one can apply \cite[Theorem A.1]{ganatra2012symplectic} for $v$ connecting $\bm x$ to any such $\bm x'$ to conclude that $\modulifloerbr$ is empty. There are only finitely many tuples of orbits $\bm x'$ inside of any compact subset $\hat{M}$ since all orbits are non-degenerate.
\end{proof}

\begin{theorem}
\label{theorem-cpt-moduli}
    For a generic consistent choice of $(H, F, J)$, given $(n-2)\chi+|\bm x'|-|\bm x|\leq2$, the wnb groupoid $\modulifloerbr$ admits a compactification $\modulifloercompbr$. When $(n-2)\chi+|\bm x'|-|\bm x|=2$, the boundary takes the form
    $$
       \bigsqcup_{\substack{\chi_1+\chi_2=\chi, \\ (n-2)\chi_1+|\bm x''|-|\bm x|= (n-2)\chi_2+|\bm x'|-|\bm x''|=1}} \mathscr{M}^{\chi_1}(\bm x, \bm x''; H^{\mathrm{tot}}, J) \times \mathscr{M}^{\chi_2}(\bm x'', \bm x'; H^{\mathrm{tot}}, J).
    $$
\end{theorem}
\begin{proof}
    For the entirety of the proof put $\bm \mu = \bm\mu (\bm x)$ and $\bm \mu' = \bm \mu( \bm x')$. Define $\modulifloercompbr$ to be a groupoid with $\modulifloercompbr_0$ consisting of pairs 
    $$\Big([(\bm C, S=S_1\cup \ldots \cup S_k, \pi, \{v^i_j\}, \{m^i_j\})], \bm u =\{ u_1 \colon S_1 \to \R \times S^1 \times \hat{M}, \ldots, u_k \colon S_k \to  \R \times S^1 \times \hat{M}\} \Big),$$
    where the former is a representative of an element of $\left(\hurwitzsmbr\right)_0$ with possibly some $S_i$'s being  disjoint unions of cylinders and each $u_i \in \mathscr{M}^{\chi(S_i)}(\bm x_{i-1}, \bm x_{i}; H^{\mathrm{tot}}, J)_0$ where $\bm x_0 =\bm x$, $\bm x_n = \bm x' $ and the rest of $\bm x_i$ are such that $$(n-2)\chi+|\bm x'|-|\bm x|=\op{ind}(u_1)+\dots+\op{ind}(u_k),$$ 
    $$\chi=\chi(S_1)+\dots+\chi(S_n)=\chi_1+\dots+\chi_k.$$ 
    The condition $(n-2)\chi+|\bm x'|-|\bm x|\leq2$ suggests that actually $k \le 2$ since for generic $(K, J)$ there are no index $0$ curves.
    We further denote such a pair as $\left([(\bm C, S, \pi)], \bm u \right)$. The morphisms $\modulifloercompbr_1$ and the weight function $\Lambda$ are induced from that on $\hurwitzsmbr$ (one can naturally extend it to those branched cover buildings where cylindrical levels are allowed) as in the case of $\modulifloerbr$; see Theorem~\ref{transversality-theorem}. We also refer to Figure~\ref{Step1}.
    
    We show that it is indeed a compact closure of $\modulifloerbr$. By this, we mean that \begin{equation}\label{branched-compactification}
        \overline{\modulifloerbr}=\modulifloercompbr.
    \end{equation}
    Consider a sequence $([(S_i, \pi_i)], u_i) \in \modulifloerbr$. First, we notice that by \cite[Theorem A.1]{ganatra2012symplectic} there exists a constant $C$ depending only on $\bm x, \bm x'$ and the choice of $F_0$ and $F_{m,j}$ as in~\ref{hamiltonian-properties2}, such that any curve $u$ in $\left(\modulifloerbr\right)_0$ has $\pi_{\hat{M}}(u)$ contained in $\hat{M} \setminus (C; +\infty] \times \partial M$. Hence, the elements of the sequence $u_i$ as above have their projections to $\hat{M}$ contained in a compact region of $\hat{M}$. Additionally, curves $u_i$ are $\underline{J}^H$-holomorphic as shown in Lemma~\ref{sft-to-hamiltonian}. Therefore, the SFT compactness theorem of \cite{BEHWZ2003sft} can be applied.
    For the compatibility of the SFT compactness with domain-dependent almost complex structures, see \cite[Section 5]{cielebak-mohnke2007} and \cite[Section 3]{fabert2009}.
    Let $u_\infty$ be an SFT limit of the above sequence. The domain of $u_\infty$ must be an element of $[(\bm C_\infty, S_\infty, \pi_\infty)] \in \hurwitzcomp$ (maybe with some additional inserted unbranched covers of cylinders or spheres). The curve $u_\infty$ is then pseudo-holomorphic with respect to the almost complex structure described in Remark~\ref{data-on-compactification}. 
    
    First, we show that $(\bm C_\infty, S_\infty, \pi_\infty)$ is a smooth branched cover building. We notice that $u_\infty$ restricted to any component of $S_\infty$, which is a branched cover of some stable tree of spheres in $\bm C_\infty$ must be constant since $\hat{M}$ is an exact symplectic manifold.
    We then consider the restriction $u_\infty^{\text{res}}$ of $u_\infty$ to non-ghost components. We note that the ghost components impose matching conditions at the preimages under $\pi_\infty$ of the nodes connecting the main cylinder to the mentioned tree. 
    
    If $n = 1$ the index of $u_\infty^{\text{res}}$ is at least $2$ less than the index of $u_i$ by Lemma~\ref{index-count} since there have to be at least two branching points in the target (on the $\R \times S^1$-side) of the ghost components, hence $u^{\mathrm{res}}_\infty$ has negative index and does not exist.
    
    If $n=2$, the Euler characteristic does not affect the index of $u_\infty^{\text{res}}$. We may apply \cite[Theorem 1.5]{ekholm-shende2022} (see also \cite[Theorem 1.1]{doan-walpuski2019}) to conclude that the presence of ghost components forces the expected dimension of $u_\infty^{\text{res}}$ to be at least $n=2$ less than the index of $u_i$. Since $u_\infty^{\text{res}}$ is transversely cut out, this implies that no such ghosts could exist.

    Alternatively, if one uses Kuranishi replacements in this case, given that components (of the normalization) of $u_\infty$ are regular, the total index of $v_\infty$ is equal to the index of $v_i$, but the boundary strata of $\hurwitzcompbr$ corresponding to sphere bubbling on the target cylinder are all of codimension at least $2$, hence the index of $u_\infty$ is again not greater than $\operatorname{ind}(u_i)-2$. This argument also goes through for $n \ge 3$. In this case, one may observe that the index of any closed component in $u_\infty$ is negative, therefore, such curves do not occur.
    
    Therefore, $(\bm C_\infty, S_\infty, \pi_\infty) \in \hurwitzsmbr$ with possibly some levels consisting only of cylinders (i.e., has Euler characteristic zero). By consistency of the Floer data $(H, F, J)$, the restriction of $u_\infty$ to each of the levels of $S_\infty$ is regular. If $\operatorname{ind}(u_i)=(n-2)\chi+|\bm x'|-|\bm x|=1$ then we have that the limit is a building of height $1$. This and regularity imply that for fixed $\bm x$ and $\bm x''$ the space $\modulifloerbr$ consists of finitely many points.  Moreover, the groupoid $$\modulifloercompbr=\modulifloerbr$$ is a manifold consisting of a finite number of points with $\Lambda$ taking the value of $1$ on all of these points. This can be obtained by choosing $R$ big enough in the construction of $\hurwitzsmbr$ (see Section~\ref{branched-manifolds}). 
    
    If $\operatorname{ind}(u_i)=2$, it follows that if there is more than $1$ level in the limit, the curve $(\bm C_\infty, S_\infty, \pi_\infty)$ is a height $k=2$ branched cover building $(u_\infty^1, u_\infty^2)$.
    
    Let $u^1_\infty$ connects $\bm x$ with $\bm x''$, and $u^2_\infty$ connects $\bm x''$ with $\bm x'$ and none of $u^1_{\infty}$, $u^2_{\infty}$ has domain with $\chi_i=0$. By Lemma~\ref{lemma-cpt-finite} there are only finitely many $\bm x'' \in \mathcal{P}_{\mathrm{psym}}$ for which such a pair $(u_\infty^1, u_\infty^2)$ may exist. And as we already saw, for fixed $\bm x''$ there are only finitely many such pairs. 
    
    Then we may conclude that for $(n-2)\chi+|\bm x'|-|\bm x|=2$ the boundary $$\partial \modulifloercompbr^{\vee}_H$$  is given by the union of images of
    \begin{equation}\label{strata}
        \mathscr{M}^{\chi_1}(\bm x, \bm x''; H^{\mathrm{tot}}, J) \times \mathscr{M}^{\chi_2}(\bm x'', \bm x'; H^{\mathrm{tot}}, J)
    \end{equation}
in $\modulifloercompbr^{\vee}_H$ where $$\chi_1+\chi_2=\chi,$$ 
    $$(n-2)\chi_1+|\bm x''|-|\bm x|=(n-2)\chi_2+|\bm x'|-|\bm x''|=1.$$
    The value of $\Lambda$ on a stratum~(\ref{strata}) equals $\frac{1}{N_{\bm \mu(\bm x'')}}$ if $\chi_1, \chi_2 <0$. Otherwise, it is equal to $1$ since such a case corresponds to one of the levels being a collection of cylinders.
    \vskip-.05in
\end{proof}
From now on, we will assume that $(H, F, J)$ are generic enough to satisfy all the results above and drop them from the notations for respective moduli groupoids.
\subsection{The differential}\label{section-differential}

We now define the differential on $SC_{\kappa, \mathrm{psym}}^*(\hat{M})$ by
\begin{gather}
\label{eq-diff}
    d\colon SC_{\kappa, \mathrm{psym}}^*(\hat{M})\to SC_{\kappa, \mathrm{psym}}^{*+1}(\hat{M}),\\
    d[\bm x]=\sum_{\substack{\bm x',\chi \colon |\bm x'|-|\bm x|=1, \\ u \in \mathscr{M}^{\chi=0}(\bm x; \bm x')}} \partial_u([\bm x])+
    \sum_{\substack{\bm x',\chi>0 \colon (n-2)\chi+|\bm x'|-|\bm x|=1, \\ u \in \modulifloerbrshort}}\frac{\partial_u([\bm x])}{N_{\bm \mu(\bm x')}}\hbar^{-\chi}  ,\nonumber
\end{gather}
where $[\bm x] \in o_{\bm x}^{\Q}$ is a generator.
One can write $d$ in the form 
$$d=d_0+d_1\hbar+d_2\hbar^2+\dots,$$
where $d_i$ is a count of curves of Euler characteristic $-i$. The relation $d^2=0$ is equivalent to an infinite system of equations on variables $d_i$'s starting with $d_0^2=0$. We summarize the discussion above via the following construction.

\begin{constr}\label{series-complex}
    Given a cochain complex $(C^*, d_0)$ over $\mathbb{Q}$ and a sequence of maps $d_i \colon C^* \to C^{*+1-(n-2)i}$ for $i \ge 1$ satisfying the relations 
    \begin{equation}\label{series-differential-relation}
        d_id_0+d_{i-1}d_1+\dots+d_1d_{i-1}+d_0d_i=0,
    \end{equation}
    we denote by $(\widetilde{C}^*, d)$ a cochain complex with
    $$\widetilde{C}^m=\sum_{i-(n-2)j=m}C^i\hbar^j, \text{ where } |\hbar|=2-n,$$
    which can be regarded as a subspace in $C^*\llbracket\hbar\rrbracket$. The differential $d$ is given by the formula
    $$d \coloneqq d_0+d_1\hbar+d_2\hbar^2+\dots.$$
    Clearly the relations~(\ref{series-differential-relation}) imply $d^2=0$. We also notice that one may obtain $(C^*, d_0)$ by setting $\hbar=0$ and we simply write this as $(\widetilde{C}^*|_{\hbar=0}, d|_{\hbar=0})=(C^*, d_0)$.
\end{constr}

\begin{remark}
    Analogously, given a chain complex $(C_*, \partial_0)$ and maps $\partial_1, \partial_2, \ldots$ one similarly constructs a chain complex $(\widetilde{C}_*, \partial)$.
\end{remark}
Clearly, $SC^*_{\kappa,\mathrm{psym}}(\hat{M}, d)$ is an instance of the above construction, and we are left to prove that relations~(\ref{series-differential-relation}) are satisfied.
\begin{theorem}
\label{theorem-d2}
    The differential $d$ makes $SC^*_{\kappa, \mathrm{psym}}(\hat{M})$ into a chain complex.
\end{theorem}
\begin{proof}
    The right-hand side of (\ref{eq-diff}) has non-zero terms for only finitely many $\bm x'$ due to Lemma \ref{lemma-cpt-finite}.
    Moreover, $\modulifloerbrshort$ is finite by Lemma~\ref{theorem-cpt-moduli}.
    
    By Theorem~\ref{theorem-cpt-moduli}, given $(n-2)\chi+|\bm x'|-|\bm x|=2$, the boundary of the compactification $\modulifloerbrshort$ consists of a finite number of broken curves as in~(\ref{strata}). 

Then for generators $[\bm x] \in o_{\bm x}^{\Q}$ and $\hbar^{-\chi}[\bm x'] \in o_{\bm x}^{\Q}$  in the count $\langle d^2 [\bm x],  \hbar^{-\chi} [\bm x'] \rangle$ each such height $2$ curve $([S_1, \pi_1], u_1), ([S_2, \pi_2], u_2)$ contributes 
\begin{gather}\label{boundary-differential}
\frac{1}{N_{\bm \mu(\bm x'')}}\cdot \frac{1}{N_{\bm \mu(\bm x')}} \cdot (\partial_{u_2} \circ \partial_{u_1}), \text{ if } \chi_1, \chi_2<0, \\
\frac{1}{N_{\bm \mu(\bm x')}} \cdot (\partial_{u_2} \circ \partial_{u_1}), \text{ otherwise.} \nonumber
\end{gather}

The Hausdorff quotient $\modulifloercompbrshort^{\vee}_H$ can be given a consistent choice of orientation on the smooth part and regarded as an oriented weighted graph satisfying the Kirchhoff junction rule at each interior vertex, where the interior vertices correspond to the branching locus and the exterior vertices correspond to points on the boundary $\partial \modulifloercompbrshort^{\vee}_H$ (see Figure~\ref{fig: branched-1dim} for an illustration).
Each exterior vertex then induces a map $o_{\bm x}^{\Q} \to o_{\bm x'}^{\Q}[(n-2)\chi)]$ by applying~\eqref{eq-orientation-floer} and using the trivialization $\modulifloercompbrshort^{\vee}_H$ induced by the outward pointing vector. Then the signed count of such maps is equal to $0$.

We then claim that for any pair $([S_1, \pi_1], u_1)$ and $([S_2, \pi_2], u_2)$ representing a point of $\partial \modulifloercompbrshort^{\vee}_H$ the composition map on orientation lines
\begin{equation}
    \partial_{u_2}  \circ \partial_{u_1} \colon o_{\bm x}^{\Q} \to o_{\bm x'}^{\Q}
\end{equation}
coincides with the one induced by the outward pointing vector at this pair to $\modulifloercompbrshort^{\vee}_H$. More explicitly, applying Lemma~\ref{orientation-lemma}, we see that we need to compare two isomorphisms
\begin{equation}
    o_{\bm x''}[(2-n)\chi]\cong \big|  \modulifloercompbrshort \big| \otimes |\R \partial_su| \otimes o_{\bm x},
\end{equation}
\begin{equation}
    o_{\bm x''}\bigl[(2-n)\chi\bigr]\cong |\R \partial_su_2| \otimes o_{\bm x'}\bigl[(2-n)\chi_1\bigr] \cong |\R \partial_su_2| \otimes |\R \partial_su_1| \otimes o_{\bm x},
\end{equation}
where $((S, \pi),u)$ is some curve near the boundary. The claim then follows since $\partial_su_2$ is an outward pointing vector.

\begin{figure}[h]
    \centering
    \includegraphics[width=12cm]{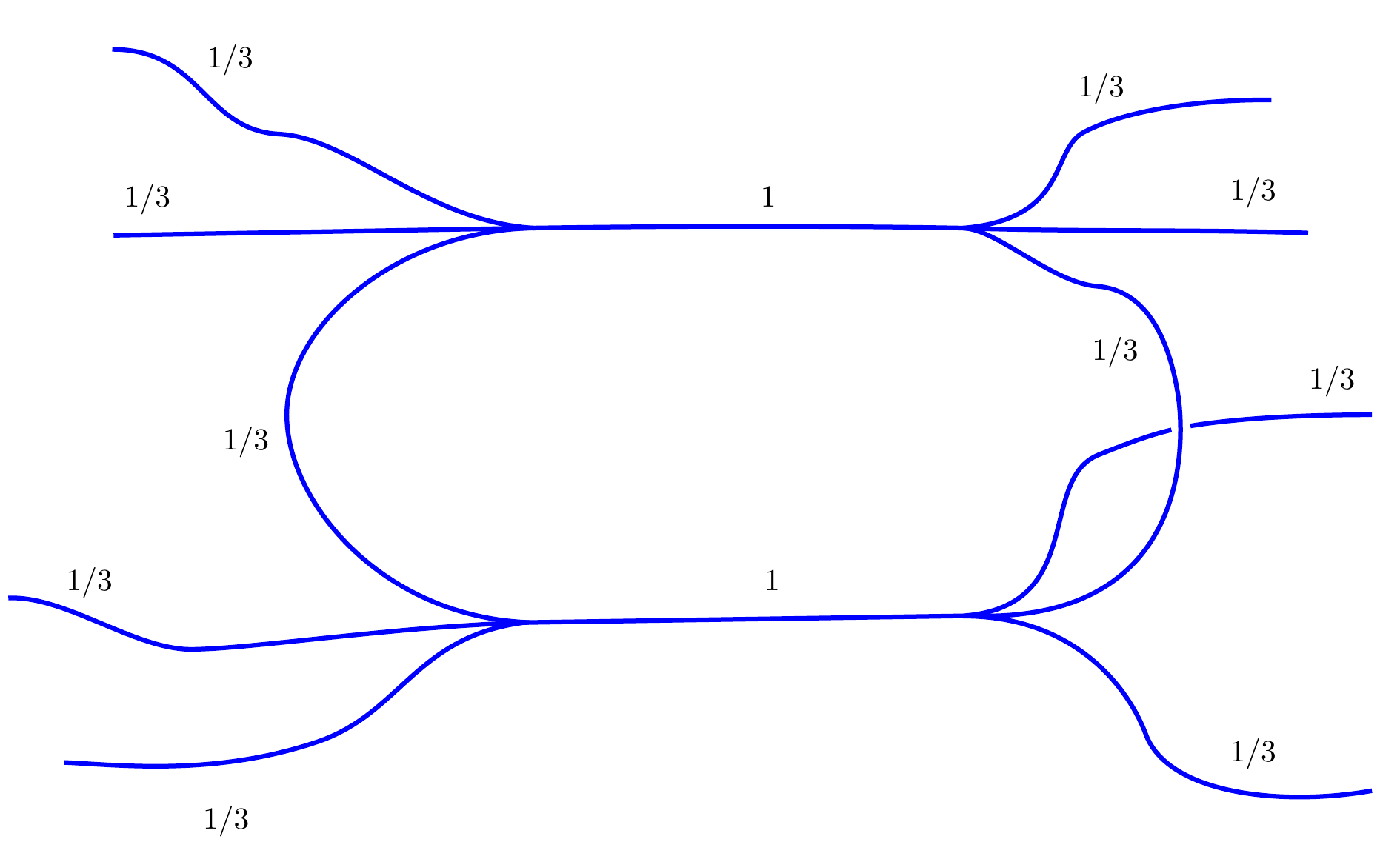}
    \caption{Here we show the maximal Hausdorff quotient with a weight function for a $1$-dimensional orientable branched manifold that one potentially can get as the moduli space $\modulifloercompbrshort$.}
    \label{fig: branched-1dim}
\end{figure}

Therefore the sum of contributions~\eqref{boundary-differential} corresponding to a given component of $\modulifloercompbrshort^{\vee}_H$ equals $0$, providing that $d^2|_{o_{\bm x}}=0$. We leave it to the reader to verify that $d^2|_{o_{\bm x}[(n-2)j]}=0$ for any $j>0$.
\end{proof}

We denote the homology of $(SC^*_{\kappa, \mathrm{psym}}(\hat{M}),d)$ by $SH^*_{\kappa, \mathrm{psym}}(M)$ or by $SH^*_{\kappa, \mathrm{psym}}(\hat{M})$.
We also write $SC^*_{\kappa, \mathrm{psym}}(\hat{M};\bm H,J_t)$  to indicate the specific choice of $\bm H=(H^{\mathrm{tot}}_{1,1}, \ldots, H^{\mathrm{tot}}_{\kappa, 1})$ as in~\ref{hamiltonian-properties1}-\ref{hamiltonian-properties3} and $J_t$ as in Section~\ref{section-almost-complex}.

\subsection{Linear Hamiltonians.}\label{section-linear-hamiltonians}
In this section, we give an alternative definition of the partially symmetrized version of Heegaard Floer symplectic cohomology via collections of linear Hamiltonians. We show that the two definitions coincide. The advantage of the linear Hamiltonian perspective is that the invariance under choices of Floer data is evident.

Since we assume everywhere in this section that we are working in the partially symmetrized context, the corresponding subscript would sometimes be dropped from the notation.

\begin{definition}
    We say that a collection $\bm H= \{H_j^m\}$ of Hamiltonian functions 
    \begin{equation*}
        H_j^m~\colon~\R/m \Z~\times~\hat{M} \to \R
    \end{equation*}
    for $j \in \Z, 1\le j \le \lfloor\frac{\kappa}{m} \rfloor$, is \emph{linear of slope $b$} if each of these functions is equal to $b \cdot r$ on the portion $[b, \infty) \times \partial M$ of the symplectization end $[1, \infty) \times \partial M$.
\end{definition}
\begin{claim}
    For almost any $b \in \R_{>0}$ all tuples of orbits in $\mathcal{P}_{\mathrm{psym}}(\bm H)$ for a tuple $\bm H$ of linear Hamiltonians of slope $b$ are contained in $M\cup[1,b)\times\partial M$.
\end{claim}
\begin{proof}
    It follows by considering $b \in \R_{>0}$ which are not equal to periods of any of the Reeb orbits in $\partial M$.
\end{proof}

\begin{claim}
    For a generic collection of linear Hamiltonians $\bm H=\{H_j^m\}$, all time-$m$ orbits of each $H_j^m$ are non-degenerate. Moreover, it can be assured that these orbits are disjoint.
\end{claim}
\begin{proof}
    See \cite[Lemma 1.2.13]{abouzaid2015symplectic}.
\end{proof}

Combining the two claims above, we conclude that for a generic tuple $\bm H$ properties \ref{hamiltonian-properties2} and \ref{hamiltonian-properties3} are satisfied, where instead of $H^{\operatorname{tot}}_{m,j}$ one considers $H^m_j$.

For a collection of linear Hamiltonians, one may also define a chain complex $$(SC_{\kappa, \mathrm{psym}}^*(M; \bm H, J_t), d)$$ by choosing a consistent Floer data on $\hurwitzsmbr$ adapted to $\bm H$. The proofs of the regularity and compactness of the associated moduli spaces are similar to those of Theorems~\ref{transversality-theorem} and~\ref{theorem-cpt-moduli} and are less technical.

We define a preorder on collections of linear Hamiltonians as follows:

\begin{definition}
     Given tuples $\bm H^+$ and $\bm H^-$ of linear Hamiltonians of slopes $b^+$ and $b^-$, respectively, we say that 
    \begin{equation}\label{preorder}
        \bm H^+ \preceq \bm H^-
    \end{equation}
    if $b^+ \le b^-$.
\end{definition}

Given two collections $\bm H^+ \preceq \bm H^-$ we construct the continuation map
\begin{equation}\label{continuation-map}
    \mathfrak{c} \colon SC_{\kappa, \mathrm{psym}}^*(M; \bm H^+, K^+, J^+) \to SC_{\kappa,\mathrm{psym}}^*(M; \bm H^-, K^-, J^-)
\end{equation}
where $(K^+, J^+)$ is some consistent Floer data compatible with $\bm H^+$ and $(K^-, J^-)$ is compatible with $\bm H^-$.

First of all, we set $H_j^m(s, t)$ to satisfy
    \begin{gather}
    H_j^m(s, t)=b_sr \text{, on the symplectization end for } r \ge b^-, \\
    H_j^m(s, t)=(H_j^m)^+ \text{ for } s \gg 0 \text{ and } H_j^m(s, t)=(H_j^m)^- \text{ for } s \ll 0,
\end{gather}
where $b_s$ is a value interpolating between $b^+$ and $b^-$ and can be regarded as a smooth function $\R  \to [b^+, b^-]$.
Similarly, we set $J_t^s$ to be $s$-dependent almost complex structure on $\hat{M}$ coinciding with $J_t^\pm$ near $\pm \infty$.

These choices essentially fix Floer data on $\hurwitzaction$ for $\chi=0$. Notice that this data is no longer $\R$-invariant, hence such data does not induce a choice of data on $\hurwitz$.

\begin{definition}
    
We say that a choice of Floer data $(\tilde{K}, \tilde{J})$ for all $\hurwitzaction$ is a \emph{continuation Floer data} for the pair $(\bm H^+, J^+)$ and $(\bm H^-, J^-)$ if it satisfies:
\begin{enumerate}\label{interpolation data}
    \item for $\chi=0$ it coincides with Floer data induced by $H^{\mathrm{tot}}_{m,j}(s, \cdot)$, $J_t^s$ as above;
    \item $H_{(S, \pi)}(z)=b_sr$ on the symplectization end, where $z \in S$ and $\pi_\R(z)=s$ for any $(S, \pi) \in \hurwitzaction$; 
    \item $\tilde{K}_{(S, \pi)}(z)=K_{(S, \pi)}^+(z)$ for $z \in S$ such that $\pi_\R(z) \ge 1$, similarly $\tilde{K}_{(S, \pi)}(z)=K_{(S, \pi)}^-(z)$ for $z \in S$ such that $\pi_\R(z) \le -1$, for any $(S, \pi) \in \hurwitzaction$;
    \item $\tilde{J}_{(S, \pi)}(z)=J^+(z)$ for $z \in S$ such that $\pi_\R(z) \ge 1$; similarly $\tilde{J}_{(S, \pi)}(z)=J_{(S, \pi)}^-(z)$ for $z \in S$ such that $\pi_\R(z) \le -1$, for any $(S, \pi) \in \hurwitzaction$.
\end{enumerate}
\end{definition}

We notice that since $\hurwitz$ is a quotient of $\hurwitzaction$, the wnb groupoid $\hurwitzbr$ may be adapted to introduce a branched structure $\hurwitzactionbr$ associated with $\hurwitzaction$.

Given tuples of orbits $\bm x^\pm \in \mathcal{P}_{\mathrm{psym}}(\bm H^\pm)$, and a choice of continuation Floer data $(\tilde{K}, \tilde{J})$ we define a \emph{continuation moduli space} 
$\mathscr{K}^\chi(\bm x^+, \bm x^-; \tilde{K}, \tilde{J})$ to be the groupoid with elements of 
$\Big(\mathscr{K}^\chi(\tilde{K}; \tilde{J}; \bm x^+; \bm x^-)\Big)_0$ consisting of pairs:

\begin{equation}\label{pairs-continuation}
\Big( (S, \pi, \bm p^{\pm}, \bm a^{\pm}, B) \in \left(\hurwitzactionbr\right)_0, \, u=(\pi, v)\colon S \to \R \times S^1 \times \hat{M}\Big) 
\end{equation}
\noindent such that
\begin{equation}
\begin{dcases}
(dv-\tilde{Y})^{0,1}=0 \text{ with respect to } \tilde{J}_{(S,\pi)};  \\
\lim_{s \to +\infty } v\circ \epsilon_i^+(s, \cdot) = x_i^+(\cdot); \\
\lim_{s \to -\infty } v\circ \epsilon_i^-(s, \cdot) = x_i^-(\cdot).
\end{dcases}
\end{equation}

The morphisms $\left( \mathscr{K}^\chi(\bm x^+; \bm x^-; \tilde{K}; \tilde{J}) \right)_1$ and the weight function are induced from $\hurwitzaction$; see the discussion in Section~\ref{section-pseudo-holomophic-curves}. 

\begin{proposition}
    For a generic choice of continuation Floer data $(\tilde{K}, \tilde{J})$ the moduli spaces $\mathscr{K}^\chi(\bm x^+, \bm x^-; \tilde{K}, \tilde{J})$ are wnb groupoids of dimension 
    \begin{equation}
        (n-2)\chi+|\bm x^-|-|\bm x^+|.
    \end{equation}
\end{proposition}
\begin{proof}
    We refer the reader to the proof of Theorem~\ref{transversality-theorem} for details.
\end{proof}

We now define the continuation map $\mathfrak{c}$ by counting curves in the moduli space above, i.e., given $\bm x^+ \in \mathcal{P}_{\mathrm{psym}}(\bm H^+)$, we set
\begin{equation}\label{eq-contmap}
    \mathfrak{c}([\bm x^+]) = \sum_{\substack{ \bm x^- \in \mathcal{P}_{\mathrm{psym}}( \bm H^-), \, u \in \mathscr{K}^\chi(\tilde{K}; \tilde{J}; \bm x^-; \bm x^+)_H \\ (n-2)\chi+|\bm x^-|-|\bm x^+|=0, \, \chi <0}}
      \frac{\mathfrak{c}_u([\bm x^+])}{N_{\bm \mu(\bm x^-)}},
\end{equation}
where $[\bm x^+] \in o_{\bm x^+}^{\Q}$ is some generator and $\mathfrak{c}_u$ is the morphism on orientation lines associated with $u$, as in~\eqref{eq-hurwitz-orline}:
\begin{equation}
    \mathfrak{c}_u \colon o_{\bm x^+} \rightarrow o_{\bm x^-}[(2-n)\chi].
\end{equation}

We point out that one should pass to Kuranishi replacements for $n\ge 3$ whenever necessary, as in Section~\ref{section-pseudo-holomophic-curves}.
\begin{lemma}
    The map $\mathfrak{c}$ is a chain map.
\end{lemma}
\begin{proof}
  We claim that $\mathscr{K}^\chi(\bm x^-, \bm x^+; \tilde{K}, \tilde{J})$ with $(n-2)\chi+|\bm x^-|-|\bm x^+|=1$ admits a compactification  (in the sense of~(\ref{branched-compactification})) with boundary $\partial |\mathscr{K}^\chi(\bm x^+; \bm x^-; \tilde{K}; \tilde{J})|_H$ covered by images of
      \begin{equation}\label{eq-continuation-boundary1}
       \bigsqcup_{\chi_1, \chi_2; \, \bm x^-_0 \in SC^*_\kappa(M, \bm H^-)} \mathscr{K}^{\chi_1}(\bm x^+,\,  \bm x^-_0; \tilde{K}, \tilde{J} ) \times \mathscr{M}^{\chi_2}(\bm x^-_0, \bm x^-; K^-, J^- ),
      \end{equation}
      \noindent where $$\chi_1+\chi_2=\chi,$$  
      $$(n-2)\chi_1+|\bm x^-_0|-|\bm x^+|=(n-2)\chi_2+|\bm x^-|-|\bm x^-_0|-1=0;$$
      and 
      \begin{equation}\label{eq-continuation-boundary2}
               \bigsqcup_{\chi_1, \chi_2; \, \bm x^+_0 \in SC^*_\kappa(M, \bm H^+) }\mathscr{M}^{\chi_1}(\bm x^+, \bm x^+_0; K^+, J^+ ) \times \mathscr{K}^{\chi_2}( \bm x^+_0; \bm x^-; \tilde{K}, \tilde{J}),
      \end{equation}
       \noindent where $$\chi_1+\chi_2=\chi,$$  
      $$(n-2)\chi_1+|\bm x^+_0|-|\bm x^+|-1=(n-2)\chi_2+|\bm x^-|-|\bm x^+_0|=0.$$
      The proof is similar to that of Theorem~\ref{theorem-cpt-moduli} but easier since for linear Hamiltonians all orbits are contained in the compact region. Each such building has the associated weight $N_{\bm \mu(\bm x^0)}$. We leave it to the reader to verify that compositions of morphisms on orientation lines associated with~\eqref{eq-continuation-boundary1} and~\eqref{eq-continuation-boundary2} for given $o_{\bm x^+}^{\Q}$ and $o_{\bm x^-}^{\Q}[(2-n)\chi]$ differ by $(-1)$ (also compare with the proof of Theorem~\ref{theorem-d2}).

Taking weights into account, we conclude that
\begin{equation}
    \mathfrak{c} \circ d - d \circ \mathfrak{c}=0.
\end{equation} \vskip-.2in\end{proof}

\begin{lemma}\label{lemma-cont-independent}
    The continuation map $\mathfrak{c}$ is independent of the choice of continuation Floer data $(\tilde{K}, \tilde{J})$.
\end{lemma}

\begin{proof}
    We omit the proof here and refer the reader to \cite[Lemma 1.6.13]{abouzaid2015symplectic} for the heuristic of the argument that can be adapted in the context of Hurwitz spaces.
\end{proof}

\begin{lemma}\label{composition-lemma}
    Given three tuples of Hamiltonians satisfying $\bm H^+ \preceq \bm H^{\text{\ding{73}}} \preceq \bm H^-$, there is a commutative diagram
\begin{center} 
\begin{tikzpicture}
  \node at (2,0) (left) {$SC_\kappa^*(\hat{M}; \bm H^+)$};
  
  \node at (6,0) (right) {$SC_\kappa^*(\hat{M}; \bm H^{\text{\ding{73}}})$};
  
  \node at (6,-3) (bottom) {$SC_\kappa^*(\hat{M}; \bm H^-)$};
  
  \draw[->] (left) -- (right) node[midway, above] {$\mathfrak{c}_{+ \to \text{\ding{73}}}$};
  \draw[->] (left) -- (bottom) node[midway, below left] {$\mathfrak{c}_{+ \to -}$};
  \draw[->] (right) -- (bottom) node[midway, right] {$\mathfrak{c}_{\text{\ding{73}} \to -}$};
\end{tikzpicture}
\end{center}    
\end{lemma}
\begin{proof}
    Let $(\tilde{K}^+, \tilde{J}^+)$ and $(\tilde{K}^-, \tilde{J}^-)$ be Floer data defining continuation maps $\mathfrak{c}_{+ \to \text{\ding{73}}}$ and  $\mathfrak{c}_{\text{\ding{73}} \to -}$ respectively.
    We pick the following collection of Floer data $(\tilde{K}^R, \tilde{J}^R)$ for $R>0$ on all $\hurwitzaction$, satisfying that for any $(S, \pi) \in \hurwitzaction$:
\begin{enumerate}
      \item $\tilde{K}^R_{(S, \pi)}(z)=\tilde{K}^+_{(S, \pi^{-R})}(z)$ for $z \in S$ such that $\pi_\R(z) \ge 0$ where $\pi^{-R}=\pi+(-R, 0)$; similarly $\tilde{K}^R_{(S, \pi)}(z)=\tilde{K}^-_{(S, \pi^R)}(z)$ for $z \in S$ such that $\pi_\R(z) \le 0$, where $\pi^R=\pi+(R,0)$;
    \item $\tilde{J}^R_{(S, \pi)}(z)=\tilde{J}^+_{(S, \pi^{-R})}(z)$ for $z \in S$ such that $\pi_\R(z) \ge 0$; similarly $\tilde{J}^R_{(S, \pi)}(z)=\tilde{J}^-_{(S, \pi^R)}(z)$ for $z \in S$ such that $\pi_\R(z) \le 0$. 
\end{enumerate}
We notice that for $z$ satisfying $-R \le \pi_{\R}(z) \le R$ the Floer data as above coincides with some $(K^{\text{\ding{73}}}, J^{\text{\ding{73}}})$ defining differential on $SC_\kappa^*(\hat{M}, \bm H^{\text{\ding{73}}})$.

The data $(\tilde{K}^R, \tilde{J}^R)$ also defines the continuation map $\mathfrak{c}_{+ \to -}$ on cohomology. Now we fix $\bm x^\pm \in SC_\kappa^*(\hat{M}, \bm H^\pm)$ and consider a pair $(u^+, u^-)$ with 
$$u^+ \in \mathscr{K}^{\chi_1}(\bm x^+,\,  \bm x^{\text{\ding{73}}}_0;\tilde{K}^+, \tilde{J}^+) \text{ and } u^- \in \mathscr{K}^{\chi_2}( \bm x^{\text{\ding{73}}}_0,\,  \bm x^-; \tilde{K}^-, \tilde{J}^-), \text{ where}$$ 
$$(n-2)\chi_1+|\bm x'_0|-|\bm x^+|=(n-2)\chi_2+|\bm x^-|-|\bm x^{\text{\ding{73}}}_0|=0.$$
We claim then that for $R \gg 0$ there is a glued curve
$$u^+ \#_R u^- \in \mathscr{K}^{\chi}(\bm x^+,  \bm x^-; \tilde{K}^R, \tilde{J}^R),$$ where $\chi=\chi_1+\chi_2$. Moreover, we claim that for $R \gg 0$ all elements of $|\mathscr{K}^{\chi}(\bm x^+,\,  \bm x^-;\tilde{K}^R, \tilde{J}^R)|$ are obtained in this fashion. The only nuance is that this map is not injective, and the curve $u^+ \#_R u^-$ as above is obtained in $N_{\bm \mu (\bm x'_0)}$ many ways in the same way as in the description of boundary degenerations in Theorem~\ref{theorem-cpt-moduli}. Therefore, there is a bijection between elements of the finite set $|\mathscr{K}^{\chi}(\bm x^+,  \bm x^-; \tilde{K}^R, \tilde{J}^R)|$ taken with weights as above and

$$ \bigsqcup_{\chi_1+\chi_2=\chi; \, \bm x'_0 \in SC^*_\kappa(M, \bm H')}  |\mathscr{K}^{\chi_1}(\bm x^+,\,  \bm x^{\text{\ding{73}}}_0;\tilde{K}^+, \tilde{J}^+)| \times |\mathscr{K}^{\chi_2}(\bm x^{\text{\ding{73}}}_0, \bm x^-; \tilde{K}^-, \tilde{J}^-)|.$$

This implies
\begin{equation}\label{eq-composition}
    \mathfrak{c}_{+ \to -} = \mathfrak{c}_{\text{\ding{73}} \to -} \circ \mathfrak{c}_{+ \to \text{\ding{73}}}.
\end{equation}
\vskip-.2in
\end{proof}

\begin{corollary}
    For $\bm H$ and $\bm H^{\text{\ding{73}}}$ of the same slope, their symplectic cohomologies are isomorphic via the continuation map:
$$\mathfrak{c} \colon SH_\kappa^*(\hat{M}, \bm H) \xrightarrow{\cong} SH_\kappa^*(\hat{M}, \bm H^{\text{\ding{73}}}).$$
\end{corollary}

Lemmas~\ref{lemma-cont-independent} and~\ref{composition-lemma} allow one to define Heegaard Floer symplectic homology as a direct limit:
\begin{equation}\label{linear-ham-defn}
    SH_{\kappa, \mathrm{psym}}^*(\hat{M}) = \lim_{\mathfrak{c}} SH_\kappa^*(\hat{M}, \bm H).
\end{equation}
At last, one may restrict to a sequence $\{\bm H^m\}_{m \in \mathbb{N}}$ with unbounded slopes, i.e.,

\begin{lemma}\label{lemma-limit-sequence}
    For any sequence of tuples of linear Hamiltonians $\bm H^m$ with slopes $b_k$ satisfying $\lim_{k \to \infty} b_k=+\infty$, there is an isomorphism 
    \begin{equation}
        \lim_{i \to \infty} SH_\kappa^*(\hat{M}, \bm H^i)=SH_\kappa^*(\hat{M}).
    \end{equation}
\end{lemma}
\begin{proof}
    We omit the proof, which follows the lines of \cite[Lemma 1.6.17]{abouzaid2015symplectic}.
\end{proof}
\begin{theorem}\label{theorem-linear-to-quadratic}
    Given a tuple $ \bm H^{\mathrm{tot}} \in {\mathcal{ \bm H}}_{S^1}^\kappa(\hat{M})$ and $J_t \in \mathcal{J}(\hat{M})$ , there is an isomorphism
    \begin{equation}
        SH_{\kappa, \mathrm{psym}}^*(\hat{M}; H^{\mathrm{tot}}, J) \cong SH_{\kappa, \mathrm{psym}}^*(\hat{M}).
     \end{equation}
\end{theorem}
\begin{proof}[Sketch of the proof.]
Here, we only sketch the proof and refer the reader to a more detailed treatment in the case of $\kappa=1$ to \cite[Appendix 3]{ritter2013}.
    Set $\bm H^k$ equal to the tuple of linear Hamiltonians of slope $k$ and equal to $\bm H^{\mathrm{tot}}$ on the complement of $(k, +\infty) \times \partial M \subset \hat{M}$. Let us denote by $SC_\kappa^*(\bm H^{\mathrm{tot}}; \mathcal{A}>c)$ the subcomplex of $SC_\kappa^*(\hat{M}; \bm H^{\mathrm{tot}})$ generated by tuples of orbits $\bm x$ with action $\mathcal{A}(\bm x)>c$. Note that for $c \ll -k$ we may naturally define a continuation map by interpolating between $\bm H^{\mathrm{tot}}$ and $\bm H^k$ which on the level of cohomology takes the form
$$SH_\kappa^*(\bm H^k)\to SH_\kappa^*(\bm H^{\mathrm{tot}}; \mathcal{A}>c).$$
    Additionally, for $k' \gg -c$, there is a natural cochain map induced by inclusion at the level of tuples of orbits
$$SC_\kappa^*(\bm H^{\mathrm{tot}}; \mathcal{A}>c) \to SC_\kappa^*(\bm H^{k'}).$$
Composing the above for $k' \gg -c \gg k$ we get
\begin{equation}
    SH_\kappa^*(\bm H^k) \rightarrow SH_\kappa^*(\bm H^{\mathrm{tot}} ; \mathcal{A}>c) \rightarrow SH_\kappa^*(\bm H^{k^{\prime}})
\end{equation} 
and clearly, this composition is the continuation map.

    On the other hand, for $-c' \gg k' \gg -c$ the composition of maps
\begin{equation}
    SC_\kappa^*(\bm H^{\mathrm{tot}}; \mathcal{A}>c) \rightarrow SC_\kappa^*(H^{k^{\prime}}) \rightarrow SC_\kappa^*(\bm H^{\mathrm{tot}} ; \mathcal{A}>c^{\prime})
\end{equation}
is the natural inclusion. The claim then follows by Lemma~\ref{lemma-limit-sequence} and since
$$SH_\kappa^*(\bm H^{\mathrm{tot}})= \lim_{c \to -\infty} SH_\kappa^*(\bm H^{\mathrm{tot}}; \mathcal{A}>c).$$
\vskip -.2in
\end{proof}

   As a corollary, we get that $SH^*_{\kappa, \mathrm{psym}}(\hat{M}; \bm H, J_t; (H, F, J))$ is independent of the choice of $\bm H^{\mathrm{tot}}\in{\mathcal{ \bm H}}_{S^1}^\kappa(\hat{M})$ and $J_t \in \mathcal{J}(\hat{M})$ and the choice of Floer data $(H, F, J)$ that we made in our construction. It is also well known that the standard symplectic cohomology $SH^*_{\kappa=1}(\hat{M})$ defined as a direct limit is an invariant under \emph{Liouville isomorphisms}; see \cite[Section 3]{seidel2006biased}, \cite[Theorem 8]{ritter2010}. The same argument works for any $\kappa$ and is similar to the one we provide in the proof of Theorem~\ref{theorem-linear-to-quadratic}. Hence we conclude that Theorem~\ref{thm-main} holds true for $SH^*_{\kappa, \mathrm{psym}}(M)$.

\subsection{Heegaard Floer symplectic cohomology with unsymmetrized orbit tuples}
\label{section-hamiltonian-unsym}

In this section, we give another version of Heegaard Floer symplectic cohomology $SC_{\kappa, \mathrm{unsym}}(\hat{M})$ which we expect to produce closely related chain complex to $SC_{\kappa, \mathrm{psym}}(\hat{M})$ even at the level of cohomology. On the chain level $SC_{\kappa, \mathrm{unsym}}(\hat{M})$ will have more generators in every degree than $SC_{\kappa, \mathrm{psym}}(\hat{M})$.

First, we fix some notation. For a given permutation $\sigma \in \mathfrak{S}_\kappa$ we assign its cycle decomposition $\sigma=c_1(\sigma) \circ \dots \circ c_{ l(\sigma)}(\sigma)$, with $\bm \mu (\sigma) = (\mu_1(\sigma), \ldots, \mu_{ l(\sigma)}(\sigma))$ being a cycle type (i.e., $\mu_i(\sigma)$ is the length of cycle $c_i(\sigma)$) of $\sigma$ and $ l(\sigma)$ denotes the length of the partition $\bm \mu (\sigma)$. We apply here the following order on cycles: we say that the cycle $c > c'$ for two cycles $c, c' \in \mathfrak{S}_\kappa$ if $c$ is longer than $c'$ or if the least element permuted by $c$ is smaller than the least element permuted by $c'$. For example, $c=(531)>c'=(246)$ since $1<2$.

Note that for a given permutation $\sigma \in \mathfrak{S}_\kappa$ the above relation provides a total order on the elements of the cyclic decomposition of $\sigma$ and $c_i(\sigma)$ in the above is assumed to be the $i$-the largest cycle in the decomposition of $\sigma$ under this relation. 

Let $H_0 \in \mathcal{H}(\hat{M})$ and $F_0 \in C^{\infty}(S^1 \times \hat{M})$ satisfy~\ref{function-perturbation-properties1} and~\ref{function-perturbation-properties2}.

We then choose $\kappa$ different functions $F_i\colon S^1\times \hat{M}\to\mathbb{R}$ for $i=1,\ldots,\kappa$ satisfying
\begin{enumerate}[(H1')]
    \item \label{hamiltonian-properties-unsym-1} each $F_i \colon \R / \Z \times \hat{M} \to \R$ is a non-negative function such that $|F_i-F_0|$ is absolutely bounded by a small constant and vanishes away from the locus of Hamiltonian orbits of $H_0^{\mathrm{tot}}=H_0+F_0$ of all orbits of periods not greater than $\kappa$, and for $t\in[0,\delta]\cup[1-\delta,1]\subset S^1$ it holds that $F_0(t)=F_i(t)$. We then denote $H_i^{\mathrm{tot}}=H_0+F_i$ for $i=1,\dots,\kappa$.
    \item \label{hamiltonian-properties-unsym-2} for any cyclic permutation $c=(j_1\cdots j_k) \in \mathfrak{S}_\kappa$ we define $F_c$ to be the function $F_c \colon \R/ k\Z \times \hat{M} \to \R$ such that $F_c(i-1+t)=F_{j_i}(t)$ for $t \in [0,1)$ and $i\in \Z$ such that $1 \le i \le k$ (note that $F_c$ is smooth by~\ref{hamiltonian-properties-unsym-1}); moreover, it is required that any time-$k$ orbit of $H_c^{\mathrm{tot}}=H_0+F_c$ is non-degenerate and lies in a small neighborhood of the corresponding time-$k$ orbit of $H_0^{\mathrm{tot}}$;
    \item \label{hamiltonian-properties-unsym-3} Hamiltonian time-$k$ orbits for different $H^{\mathrm{tot}}_c$'s are disjoint and embedded.
\end{enumerate}
We leave it to the reader to verify that such collections $\{F_i\}_{i=1}^\kappa$ exist and are sufficiently generic (see also Section~\ref{section-hamiltonian-partsym}).

Consider the set $\mathcal{P}_{\mathrm{unsym}}$ of ordered $\kappa$-tuples of Hamiltonian chords defined as
\begin{align}
    \label{eq-paths-unsym}
    \mathcal{P}_{\mathrm{unsym}}=\left\{\bm x=(x_1,\ldots,x_\kappa)\left|
        \begin{array}{ll}
            \text{$x_i\colon[0,1]\to\hat{M}$ is the integral flow of $\phi^t_{H^{\mathrm{tot}}_i}$,}\\
            \text{$x_i(1)=x_{\sigma(i)}(0)$ for some permutation $\sigma\in \mathfrak{S}_\kappa$.}
        \end{array}
    \right.
    \right\}
\end{align}

 For each $\bm x \in \mathcal{P}_{\mathrm{unsym}}$ we would denote the associated permutation by $\sigma(\bm x) \in \mathfrak{S}_\kappa$. For a given $\bm x$ with $\sigma=\sigma(\bm x)$ to each cycle $c_i(\sigma)$ with $i=1, \ldots,  l(\sigma)$  we associate a loop $x^{c_i(\sigma)}$ in $\hat{M}$ formed of chords $x_j$ with $j$ such that $c_i(\sigma)(j) \neq j$. Clearly, $x^{c_i(\sigma)}$ is a Hamiltonian orbit with respect to $H^{\mathrm{tot}}_{c_i(\sigma)}$. See Figure \ref{fig-generator} for an example.
We would also use the notation $\ell(\bm x)$ for $ l(\sigma)$.
 
\begin{figure}
    \centering
    \includegraphics[width=10cm]{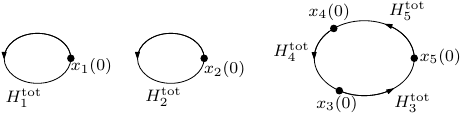}
    \caption{A $\kappa$-tuple ($\kappa=5$) of Hamiltonian chords $\bm$ as an element of $\mathcal{P}_{\mathrm{unsym}}$, where $H^{\mathrm{tot}}_i$'s stand for perturbed Hamiltonian functions. The permutation associated with this generator is given by $(1)(2)(354)$ in the cyclic form and $\bm x$ is comprised of loops $x_{(1)}, x_{(2)}, x_{(354)}$. 
  }
    \label{fig-generator}
\end{figure}

\begin{remark}
    Notice that the presentation above allows one to view $\bm x$ as a loop in the configuration space $\mathrm{Conf}_\kappa(\hat{M})$.
\end{remark}

As in the partially symmetrized case, to an orbit $x^{c_i(\sigma)}$ we associate a differential operator
\begin{equation}\label{loop-operator-unsym}
    D_{x^{c_i(\sigma)}} \colon W^{1,2}(\C, \C^{n}) \to L^2 (\C, \C^n)
\end{equation}
 following \cite[Section C.6]{abouzaid2010geometric}, the orientation line of $x^{c_i(\sigma)}$ 
 \begin{equation}
     o_{x^{c_i(\sigma)}} = \text{det}(D_{x^{c_i(\sigma)}})
 \end{equation}
and the \emph{orientation line} $o_{\bm x}$ of $\bm x \in \mathcal{P}_\mathrm{unsym}$ is given as follows
     \begin{equation}\label{orientation-eq-unsym}
         o_{\bm x} \coloneqq o_{x^{c_1(\sigma)}} \otimes \dots \otimes o_{x^{c_{ l(\sigma)}(\sigma)}}.
     \end{equation}
Regarding $o_{\bm x}$ as a $\Z$-module we set
    \begin{equation}\label{orientation-rational-eq-unsym}
        o_{\bm x}^{\Q} = \Q \otimes_{\Z}o_{\bm x}.
   \end{equation}
Note that each orbit $x^{c_i(\sigma)}$ has a well-defined Conley-Zehnder index $CZ(x^{c_i(\sigma)})$.
We then define the grading of $\bm x$ by
\begin{equation}\label{grading-unsym}
    |\bm x|=\kappa n-CZ(x^{c_1(\sigma)})-\dots-CZ(x^{c_{ l(\sigma)}(\sigma)}).
\end{equation}

We set 
\begin{equation}  
V_{\kappa, \mathrm{unsym}} = \bigoplus_{\bm x \in \mathcal{P}_{\mathrm{psym}}}o_{\bm x}^{\Q}.
\end{equation}
It is a graded vector space with $$V^i_{\kappa, \mathrm{unsym}}=\bigoplus_{|\bm x|=i}o_{\bm x}^{\Q}.$$
Consider a formal variable $\hbar$ with grading $|\hbar|=2-n$. 

The Heegaard Floer symplectic cochain complex in the unsymmetrized case then has as its underlying graded abelian group
\begin{equation}\label{HFSH-graded-abelian}
    SC_{\kappa, \mathrm{unsym}}^*(\hat{M}) = \bigoplus_{m \in \Z}SC_{\kappa, \mathrm{unsym}}^m(\hat{M})
\end{equation}
with
\begin{equation}
SC_{\kappa, \mathrm{unsym}}^m(\hat{M}) = \sum_{i+(2-n)j=m}V^i_{\kappa, \mathrm{psym}}\hbar^j=\sum_{i+(2-n)j=m}V^i_{\kappa, \mathrm{psym}}[(n-2)j].
\end{equation}

The group $SC_{\kappa, \mathrm{unsym}}^m(\hat{M})$ may be regarded as a graded vector subspace of $V_{\kappa, \mathrm{unsym}}\llbracket\hbar\rrbracket$.
We notice that $SC^*_{\kappa, \mathrm{unsym}}$ is a graded $\mathbb{Q}[h]$-module (but it is not a $\mathbb{Q}\llbracket\hbar\rrbracket$-module). In particular, multiplication by $\hbar$ is a $(2-n)$-degree map on $SC^*_{\kappa, \mathrm{psym}}(\hat{M})$.

For each $\kappa$-tuple of chords $\bm x=(x_1,\ldots,x_\kappa)\in \mathcal{P}_{\mathrm{unsym}}$, we define the action of $\bm x$ via
\begin{equation}\label{action}
    \mathcal{A}(\bm x)=\sum^\kappa_{i=1}\mathcal{A}(x_i),
\end{equation}
where the action of chord $x_i$ is given by the expression
\begin{equation}
    \mathcal{A}(x_i)=\int^1_0 x_i^*\lambda-\int^1_0 H^{\mathrm{tot}}_i(t,x_i(t))dt.
\end{equation}

Now we outline how to define a differential on $SC^*_{\kappa, \mathrm{unsym}}(\hat{M})$ following the strategy we applied in previous sections to $SC^*_{\kappa, \mathrm{psym}}(\hat{M})$.

Pick two permutations $\sigma, \sigma' \in \mathfrak{S}_\kappa$.
\begin{definition}\label{def-hurwitz-space-unsym}
    The \emph{Hurwitz space} $\hurwitzactionunsym$ is the space of ismorphism classes \emph{branched covers}, where a branched cover
    \begin{equation*}
        (S, \pi \colon S~\longrightarrow~\R ~\times~S^1, \bm p^+, \bm a^+, \bm p^-, \bm a^-, B)
    \end{equation*}
    is the following collection of data: 
    \begin{enumerate}
        \item A Riemann surface $S$ without boundary of Euler characteristic $\chi$ punctured at positive punctures $\bm p^+=\{p^+_1, \ldots, p^+_{l(\bm \mu)}\}$ and negative punctures $\bm p^-=\{p^-_1, \ldots, p^-_{l(\bm \mu ')}\}$.
        \item For each positive puncture $p_i^+$  we pick $\mu_i(\sigma)$ asymptotic markers $a_{i1}^+, \ldots, a_{i\mu_i(\sigma)} \in \mathcal{S}^1_{p_i^+}S$. We assume the total order on the positive asymptotic markers $a_1^+, \ldots, a_\kappa^+$ such that the asymptotic markers belonging to $\mathcal{S}^1_{p_i^+}S$ appear on this circle in the order prescribed by the cycle $c_i(\sigma)$ (note that $\mathcal{S}^1_{p_i^+}S$ posses the induced orientation from $S)$. There is an analogous choice of negative asymptotic markers $a_1^-, \ldots, a_\kappa^-$ satisfying similar conditions with respect to $\sigma'$.
        \item A holomorphic branched cover $\pi \colon S \longrightarrow \R \times S^1$ with simple branched points away from $\pm \infty$. At $+\infty$ (resp. at $- \infty$) the branching partial is given by $\bm \mu(\sigma)$ (resp. by $\bm \mu(\sigma')$), i.e., a point $p_i^+$ (resp. a point $p_i^-$) is a ramification point of multiplicity $\mu_i(\sigma)$ (resp. multiplicity $\mu_i(\sigma')$) of an extended map $\overline{\pi} \colon \overline{S} \longrightarrow S^2$. We additionally require that the map induced by $\pi$ from $\mathcal{S}^1_{p_i^+}\overline{S}$ to $\mathcal{S}^1_{+ \infty} S^2$ sends all associated with $p_i^+$ asymptotic markers to the marker $1 \in \mathcal{S}^1_{+ \infty} S^2$,  together with similar requirements for negative punctures of $S$.
        \item An ordered tuple $B$ of points $q_1, \ldots, q_b$ in $\R \times S^1$ of simple branching, where $b= - \chi$.
    \end{enumerate}

    Two branched covers $(S, \pi, \bm p^\pm, \bm a^\pm, B)$ and $(S', \pi', \bm p'^\pm, \bm a'^\pm, B')$ are said to be \emph{isomorphic} if there exists a biholomorphism $\alpha \colon \overline{S} \rightarrow \overline{S'}$ such that $\alpha(p^\pm_i)=p_i'^\pm$, $\alpha(a^\pm_i)=a_i'^\pm$ and $\pi' \circ \alpha = \pi$. 
\end{definition}

Without going into much detail, we claim that there is a compactification $\hurwitzcompunsym$ containing a subset $\hurwitzsmunsym \subset \hurwitzcompunsym$ of branched buildings with smooth levels, see Section~\ref{hurwitz-spaces}. Furthermore, one may introduce a branched structure $(\hurwitzsmbrunsym, \Lambda)$ on $\hurwitzsmunsym$ following the strategy of Section~\ref{branched-manifolds}. We highlight here one of the main differences.

Recall that a partition $\bm \mu \vdash \kappa$, which we regard as a cycle type, corresponds to the unique conjugacy class in $\mathfrak{S}_\kappa$. Then the analogue of codimension $1$ boundary degeneration described in Definition~\eqref{def-hurwitz-2-level} is denoted by
\begin{equation}\label{eq-hurwitz-unsym-2-level}
    {\mathcal{H}^{\sigma, \bm \mu}_{\kappa, \chi_1} \times_{\bm m} \mathcal{H}^{\bm \mu, \sigma'}_{\kappa, \chi_2} }.
\end{equation}
Essentially, it is a space of height $2$ buildings with levels glued along $l(\bm \mu)$ nodes with matching markers $\bm m$ and branching profile over the corresponding node given by $\bm \mu$. Then the analogue of the forgetful covering~\eqref{hurwitz-boundary-cover}, which is essential for the construction of $\hurwitzsmbr$, is given by
\begin{equation}\label{hurwitz-unsym-boundary-cover}
    bf_{\mathrm{unsym}}^1 \colon \bigsqcup_{\sigma'' \colon \bm \mu(\sigma'')=\bm \mu} {\mathcal{H}^{\sigma, \sigma''}_{\kappa, \chi_1} \times \mathcal{H}^{\sigma'', \sigma'}_{\kappa, \chi_2} } \to {\mathcal{H}^{\sigma, \bm \mu}_{\kappa, \chi_1} \times_{\bm m} \mathcal{H}^{\bm \mu, \sigma'}_{\kappa, \chi_2} }.
\end{equation}
We note that $bf_{\mathrm{unsym}}^1$ is a $\kappa!$-sheeted cover as for given matching $\bm m$ there are exactly $\kappa!$ ways to assign asymptotic markers with such matching (in the sense of~\eqref{eq-asymptotic-to-matching}). Then the weight function $\Lambda$ takes value $\frac{1}{\kappa!}$ on each element of ${\mathcal{H}^{\sigma, \sigma''}_{\kappa, \chi_1} \times \mathcal{H}^{\sigma'', \sigma'}_{\kappa, \chi_2} } \times (R, +\infty] \subset (\hurwitzsmbrunsym)^{\vee}_H.$

Following the lines of Section~\ref{section-floer-data} one may introduce consistent cylindrical ends and finite cylinders data on $\hurwitzbrunsym$. This may be further utilized to introduce consistent Floer data $(H,F, J)$ on $\hurwitzbrunsym$. We note that for any $[S, \pi] \in \hurwitzbrunsym$ the function $H^{\mathrm{tot}}_{(S, \pi)}=H_{(S,\pi)}+F_{(S, \pi)}$ restricts to $H^{\mathrm{tot}}_c$ on a given cylinder (away from the collar) with assigned cycle $c$.

Given such consistent Floer data and two tuples $\bm x, \bm x' \in \mathcal{P}_{\mathrm{unsym}}$ with $\sigma=\sigma(\bm x)$ and $\sigma'=\sigma(\bm x')$, we define $\modulifloerbractionunsym$ to be the set of pairs
\begin{equation}\label{pairs-unsym}
\Big( [(S, \pi, \bm p^{\pm}, \bm a^{\pm}, B)] \in \left(\hurwitzbrunsym \right)_0, \, u=(\pi', v): S \to \R \times S^1 \times \hat{M}\Big) 
\end{equation}
\noindent such that
\begin{equation}
\begin{dcases}
  [S, \pi']= [S, \pi];\\
(dv-Y_{(S, \pi)})^{0,1}=0, \text{ with respect to } J_{(S,\pi)};  \\
\lim_{s \to +\infty } v\circ \delta_i^+(s, \cdot) = x^{c_i(\sigma)}(\cdot) \text{ for }i=1, \ldots,  l(\sigma); \\
\lim_{s \to -\infty } v\circ \delta_i^-(s, \cdot) = x^{c_i(\sigma')}(\cdot) \text{ for }i=1, \ldots, \ell(\sigma').
\end{dcases}
\end{equation}

We set $\modulifloerbrunsym$ to be the quotient of $\modulifloerbractionunsym$ under the following equivalence relation.
We say that two pairs $$\Big( [(S_1, \pi_1, \bm p^{\pm}, \bm a^{\pm}, B)], \, u=(\pi'_1, v_1) \Big) \text{ and } \Big( [(S_2, \pi_2, \bm p^{\pm}, \bm a^{\pm}, B)], \, u=(\pi'_2, v_2) \Big)$$
\noindent are equivalent if $[S_1, \pi_1]=[S_2, \pi_2]$ and $v_1=v_2$.

\begin{proposition}\label{transversality-theorem-unsym}
    For a generic Floer data on $\hurwitzbrunsym$ the moduli space $\modulifloerbractionunsym$ and $\modulifloerbrunsym
    $ are wnb groupoids. Moreover, the dimension of $\modulifloerbrunsym$ is given by
    $$ \op{dim}_{\R}\modulifloerbrunsym=(n-2)\chi+|\bm x'|-|\bm x|-1.$$
\end{proposition}
The proof of this proposition is similar to that of Theorem~\ref{transversality-theorem}.

In case when $\op{dim}_{\R}\modulifloerbrunsym=1$ there is an induced morphism
\begin{equation}\label{orientation-map-unsym}
    \partial_ u \colon o_{\bm x} \to o_{\bm x'}\bigl[(2-n)\chi\bigr],
\end{equation}
similar to~\eqref{orientation-map-floer}.

The differential on $SC_{\kappa, \mathrm{unsym}}(\hat{M})$ is then defined by the formula
\begin{equation}
    d[\bm x]=\sum_{\substack{\bm x',\chi \colon |\bm x'|-|\bm x|=1, \\ u \in \mathscr{M}_{\mathrm{unsym}}^{\chi=0}(\bm x; \bm x')}} \partial_u([\bm x])+
    \sum_{\substack{\bm x',\chi \colon (n-2)\chi+|\bm x'|-|\bm x|=1, \\ u \in \modulifloerbrunsym}} \frac{\partial_u([\bm x])}{\kappa!}\hbar^{-\chi},
\end{equation}
where $[\bm x]$ is a generator in $ o_{\bm x}^{\Q}$ for $\bm x \in \mathcal{P}_{\mathrm{unsym}}$.

The proof of the fact that the above forms a chain complex is based on the following
\begin{proposition}
       For a generic consistent choice of $(H, F, J)$, given $(n-2)\chi+|\bm x'|-|\bm x|\leq2$, the wnb groupoid $\mathscr{M}_{\mathrm{unsym}}^{\chi}(\bm x, \bm x'; H^{\mathrm{tot}}, J)$ admits a compactification $\overline{\mathscr{M}}_{\mathrm{unsym}}^{\chi}(\bm x, \bm x'; H^{\mathrm{tot}}, J)$. In case $(n-2)\chi+|\bm x'|-|\bm x|=2$ the boundary takes the form
    $$
       \bigsqcup_{\substack{\chi_1+\chi_2=\chi, \\ (n-2)\chi_1+|\bm x''|-|\bm x|= (n-2)\chi_2+|\bm x'|-|\bm x''|=1}} \mathscr{M}_{\mathrm{unsym}}^{\chi_1}(\bm x, \bm x''; H^{\mathrm{tot}}, J) \times \mathscr{M}_{\mathrm{unsym}}^{\chi_2}(\bm x'', \bm x'; H^{\mathrm{tot}}, J).
    $$
\end{proposition}

The proof is similar to that of Theorem~\ref{theorem-cpt-moduli} and left to the reader.
We then state the following 
\begin{theorem}
\label{theorem-d2-unsym}
    The differential $d$ makes $SC^*_{\kappa, \mathrm{unsym}}(\hat{M})$ into a chain complex.
\end{theorem}
We conclude this section by pointing out that the unsymmetrized version can also be defined via tuples of linear Hamiltonians as in Section~\ref{section-linear-hamiltonians} and, therefore, $SH_{\kappa, \mathrm{unsym}}(\hat{M})$ is also an invariant of the Liouville domain $M$.

\subsection{Orientation lines conventions}\label{section-orientation-lines}

 Here we introduce some basics on orientation lines that we use to have all chain complexes and chain maps appearing in this text well-defined over $\Q$. We use the conventions from \cite{abouzaid2015symplectic} with slight adjustements. 
 
 Given a finite dimensional real vector space $V$, its determinant line is the $1$-dimensional $\Z$-graded  real vector space $\operatorname{det} V=\Lambda^{\operatorname{dim}(V) } V$ supported in degree $\operatorname{dim}(V)$. 

To any $1$-dimensional graded real vector space $\delta$ we associate an orientation line $|\delta|$, which is the rank $1$ graded free abelian group generated by the two possible orientations of $\delta$, modulo the relation that their sum vanishes. The orientation line $|\delta|$ is supported in the same degree as $\delta$. When $\delta=\operatorname{det} V$ we denote its orientation line by $|V|$. 
Given an orientable manifold $M$ we will denote by $|M|$ the orientation line generated by the orientations of $M$ and supported in degree $\operatorname{dim}_{\R}M$.

Given two vector spaces $U$ and $V$, let $W= U \oplus V$. There is then an isomorphism
$$|U| \otimes |V| \cong |W|$$
given by first choosing a basis corresponding to a given orientation of $U$ and then extending it to the basis of $W$ by adding (on the right) a basis of $V$ corresponding to a given orientation of $V$. Notice that this procedure gives rise to the \emph{Koszul sign}:
\begin{equation}\label{eq-koszul-sign}
    |W| \cong |U| \otimes |V| \cong (-1)^{\operatorname{dim}(U)\cdot\operatorname{dim}(V)} |V| \otimes |U| \cong (-1)^{\operatorname{dim}(U)\cdot\operatorname{dim}(V)} |W|.
\end{equation}

More generally, for an exact sequence of vector spaces 
$$0 \to V \to W \to U \to 0$$
which splits, we assign an isomorphism
\begin{equation}
    |U| \otimes |V| \cong |W|.
\end{equation}

We note that this convention is different from the one used in \cite{abouzaid2015symplectic}. We are motivated by the following example. Let $N$ be an orientable submanifold of an orientable manifold $M$. Let us denote by $\eta$ the normal bundle $TM/TN$ over $N$. There is then a short exact sequence
$$0 \to TN \to TM|_{N} \to \eta \to 0.$$
The standard convention for inducing an orientation on $N$ from an orientation on $M$ is to first set a positive orientation on $\eta$ (which in some cases is straightforward as $\eta$ could be a trivial bundle) and then declare an orientation on $N$ to be positive if a basis with such orientation extends a basis with positive orientation on $\eta$ to a basis with positive orientation on $M$. In other words,
\begin{equation}\label{eq: conormal-orientation-identity}
    |\eta| \otimes |N| \cong |M|.
\end{equation}
Given a $\mathbb{Z}$-graded line $\ell$ (rank $1$ free abelian group), its dual line $\ell^{-1}=$ $\operatorname{Hom}_{\mathbb{Z}}(\ell, \mathbb{Z})$ is, by definition, supported in the opposite degree as $\ell$. Here, we use the convention that $\Z$ is supported in degree $0$. There is a canonical isomorphism $\ell^{-1} \otimes \ell \cong \mathbb{Z}$ induced by evaluation. According to~\eqref{eq-koszul-sign} there is an isomorphism $(-1)^{\operatorname{deg}(\ell)}\ell \otimes \ell^{-1} \cong \Z$.

We will denote by $\zeta$ the trivial rank $1$ free abelian group (i.e., with prescribed isomorphism with $\Z$) supported in degree $1$.

Given a $\mathbb{Z}$-graded rank $1$ free abelian group $\ell$, we denote by $\ell[k]$ the rank $1$ group obtained by shifting the degree down by $k \in \mathbb{Z}$, i.e., $\ell[k]$ is supported in degree $\operatorname{deg}(\ell)-k$. We will use the agreement

\begin{gather}\label{eq-orientation-lines-shift}
    \ell[-k]\cong \ell \otimes \zeta^{k}, \text{ for }k>0, \\
    \ell[-k] \cong \zeta^{-k} \otimes \ell, \text{ for }k<0. \nonumber
\end{gather}

For any such rank $1$ free abelian group $\ell$ we associate 
\begin{equation}
    \ell^{\Q} = \ell \otimes_{\Z} \Q,
\end{equation} 
which is a $\Z$-graded $\Q$-vector space supported in $\operatorname{deg}(\ell)$.
\section{The free multiloop complex}
\label{section-morse}
In this section, we introduce a certain Morse chain complex which is intended to be a counterpart of the unsymmetrized version of HFSH. We note that this complex is the closed string analogue of the Morse complex introduced in \cite{honda2025morse}.
Throughout this section, $Q$ is a closed orientable $n$-manifold.

\subsection{Free multiloops}
The analogue of $SH_\kappa(T^*Q)$ on the Morse side should, on the chain level, be based on time-$1$ free loops in the configuration space $\operatorname{UConf}_\kappa(Q)$. 
To model this, we consider $\kappa$-strand multiloops on $Q$. We start by recalling the setup of \cite{abbondandolo2010floer} for the Morse complex of paths alone.

Consider the space of continuous time-1 paths on $Q$ with free endpoints:
\begin{equation*}
    \Omega(Q)=\{\gamma\in C^0([0,1],Q)\}.
\end{equation*}
We denote the subset of $\Omega(Q)$ in the class $W^{1,2}$ by $\Omega^1(Q)$. Further, we denote by $\Lambda^1(Q)$ the subset of $\Omega^1(Q)$ consisting of closed loops. It is naturally a Hilbert manifold with a standard metric on it, see e.g. \cite[Section 2.2]{abbondandolo2006floer}. We also consider the spaces of Sobolev  $W^{m,2}$ paths and we denote via $\Omega^{m,2}(Q)$ and $\Lambda^{m,2}(Q)$ the corresponding path and loop spaces (note that for the free loop space above one additionally requires that first $m-1$ derivatives coincide at the endpoints). For the spaces we introduce below, we assume similar $W^{m,2}$-analogs and do not specify this explicitly.

Fix a generic Riemannian metric $g$ on $Q$. 
Let $$L=L(t,q,v)\colon [0,1]\times TQ\to\mathbb{R}$$ be a \emph{Lagrangian function} which is the \emph{Fenchel dual} of a Hamiltonian $H$ that we introduced in Section~\ref{section-hamiltonian-unsym}.

We then define the Lagrangian action functional $\mathcal{A}_{L}\colon\Omega^1(Q)\to\mathbb{R}$ by
\begin{equation}
    \mathcal{A}_{L}(\gamma)=\int_0^1 L(t,\gamma(t),\dot{\gamma}(t))\,dt.
\end{equation}

Let $\varphi^t_{L}$ be the integral curve on $Q$ satisfying the Euler-Lagrangian equation:
\begin{equation}\label{Euler-Lagrange}
    \frac{d}{dt}\partial_v L(t,\gamma(t),\dot{\gamma}(t))=\partial_q L(t,\gamma(t),\dot{\gamma}(t)).
\end{equation}

The critical points of $\mathcal{A}_L$ are exactly the solutions of~\eqref{Euler-Lagrange}.
For $L$ as one of the functions above, the Lagrangian action functional $\mathcal{A}_L$ is not of class $C^2$, unless $L$ is a polynomial of degree at most two on each fiber of $TQ$, hence in general one may not construct a Morse complex on $\Omega(Q)$ by means of the gradient vector field of $\mathcal{A}_L$. To overcome this obstacle, in \cite{abbondandolo2009pseudo-gradient} the authors introduce the notion of pseudo-gradient vector field $X$ for such functions on Hilbert manifolds and prove the existence of pseudo-gradients for Lagrangian action functional.
The following summarizes the key properties of $\mathcal{A}_L$ and its pseudo-gradient vector fields (see \cite[Theorem 4.1]{abbondandolo2009pseudo-gradient}, \cite[Appendix A.1]{abbondandolo2010floer}):
\begin{theorem}\label{theorem-action+gradient}
    \begin{enumerate}
        \item For a generic choice of $L$ all critical points of $\mathcal{A}_L$ are non-degenerate.
        \item  The action functional $\mathcal{A}_L$ is bounded below, and there exists a smooth pseudo-gradient vector field $X$ on $\Lambda^1(Q)$, i.e.,
           
        \begin{enumerate}[(PG1)]
            \item \label{PG1} $\mathcal{A}_L$ is a Lyapunov function for $X$;
            \item \label{PG2} $X$ is a Morse vector field whose singular points (= critical points of $\mathcal{A}_L$) have finite Morse index;
            \item \label{PG3} the pair $(X, \mathcal{A}_L)$ satisfies the Palais-Smale condition;
            \item \label{PG4} 
            $X$ is forward-complete, i.e., the flow $\Phi(t, x)$ is defined for all $t>0$ and $x \in \Lambda^1(Q)$;
            \item \label{PG5} X satisfies the Morse-Smale condition, i.e., for each pair of critical points $\gamma, \gamma' \in \Lambda^1(Q)$, the unstable manifold of $\gamma$ is transverse to the stable manifold of $\gamma'$.
        \end{enumerate}
          
        \item  Any generic small perturbation of $X$ which coincides with $X$ on a neighborhood of the critical points of $\mathcal{A}_L$, is still a pseudo-gradient (i.e., satisfies (PG1)-(PG5)).
    \end{enumerate}
\end{theorem}

For the generalization of the above to the set of what we call free multiloops, choose Lagrangians $L_i$ to be \emph{Fenchel duals} (see~\cite{abbondandolo2010floer} for the precise definition) of $H^{\mathrm{tot}}_i$'s that we picked in Section~\ref{section-hamiltonian-unsym} for $H_0=\frac{1}{2}\lvert p\rvert ^2$. 

By the assumption~\ref{hamiltonian-properties-unsym-1} on the $H_i^{\mathrm{tot}}$ it follows that for $\delta$ such that $0<\delta<1/2$ as in~\ref{hamiltonian-properties-unsym-1}
 \begin{equation}\label{lagrangian-gluing}
 L_i(t,\cdot, \cdot)=L_0(t,\cdot, \cdot)\, \text{ for }\, t\in[0,\delta]\cup[1-\delta,1].
\end{equation}

Moreover, it is straightforward (see, e.g., \cite{abbondandolo2006floer}) to show that each $L=L_i$ for $i=1,\ldots, \kappa$ satisfy the following two properties
\begin{enumerate}[(L1)]
    \item \label{L1-condition} There is a continuous function $\ell_1$ on $Q$ such that for every $(t, x, v) \in [0,1] \times T Q$ with $x\in Q$ and $v\in T_xQ$,
    $$
    \begin{aligned}
        \|\nabla_{v v} L(t, x, v)\| & \leq \ell_1(x), \\
        \|\nabla_{v x} L(t, x, v)\| & \leq \ell_1(x)(1+|v|), \\
        \|\nabla_{x x} L(t, x, v)\| & \leq \ell_1(x)(1+|v|^2),
    \end{aligned}
    $$
    where $|v|=g(v,v)^{1/2}$.
    \item \label{L2-condition} There is a continuous positive function $\ell_2$ on $Q$ such that $\nabla_{v v} L(t, x, v) \geq$ $\ell_2(x) I$, for every $(t, x, v) \in[0,1] \times TQ$.
\end{enumerate}

To each permutation $\sigma \in \mathfrak{S}_\kappa$ we associate the set $\Omega^1_\sigma(Q)$ of $\kappa$-tuples of multipaths on $Q$ which satisfy
\begin{align}
    \label{eq-lag-paths}
    \Omega^1_\sigma(Q)=\left\{\bm \gamma=(\gamma_1,\ldots,\gamma_\kappa)\in\Omega^1(Q)^\kappa\,|\gamma_i(1)=\gamma_{\sigma(i)}(0)\right\}.
\end{align}
We call elements of such a set \emph{free multiloops} since each such multipath consists of $ l(\sigma)$ loops.
Specifically, each free multiloop $\bm \gamma$ consists of closed loops $\gamma^{c_1(\sigma)}, \ldots, \gamma^{c_{l(\sigma)}(\sigma)}$ with $$\gamma^{c_j(\sigma)} \in \Lambda^1_{\mu_j(\sigma)}=\{\gamma \in W^{1,2}([0, \mu_j(\sigma)], Q) \mid \gamma(0)=\gamma(\mu_j(\sigma))\},$$
and therefore, we have an isomorphism
$$\Omega^1_{\sigma}(Q)\cong \Lambda^1_{\mu_1(\sigma)}(Q) \times \dots \times \Lambda^1_{\mu_{ l(\sigma)}}(Q).$$
In light of the above isomorphism, we set 
\begin{equation}
    \Omega^m_{\sigma}(Q)\coloneqq \Lambda^{m,2}_{\mu_1(\sigma)}(Q) \times \dots \times \Lambda^{m,2}_{\mu_{ l(\sigma)}}(Q).
\end{equation}

We denote by $\bm L$ the collection $(L_1,\ldots,L_\kappa)$ and define the Lagrangian action functional on $\Omega^1_\sigma(Q)$ by
\begin{equation}
    \mathcal{A}_{\bm L}(\bm \gamma)=\sum_{i=1}^\kappa \mathcal{A}_{L_i}(\gamma_i),
\end{equation}
and similarly we define action functionals $\mathcal{A}_{L_{c_i(\sigma)}}$ on $\Lambda^1_{c_i(\sigma)}(Q)$; their sum is equal to $\mathcal{A}_{\bm L}$.
One can also view $\Omega^1_\sigma(Q)$ as set of paths $\bm \gamma \in \Omega^1(Q^\kappa)$ with boundary conditions
\begin{equation}
    (\bm \gamma(0), \bm \gamma(1)) \in R_\sigma \subset Q^\kappa \times Q^\kappa,
\end{equation}
where $R_\sigma=\left\{(\bm x, \bm y) \in  Q^\kappa \times Q^\kappa \, | \, x_{\sigma(i)}=y_i \, \right\}$.

The critical points of $\mathcal{A}_{\bm L}$ form a subset $\mathcal{P}_\sigma^L\subset\Omega^1_\sigma(Q)$ defined as
\begin{align}
    \label{eq-lag-paths-crit}
    \mathcal{P}_\sigma^{\bm L}=\left\{\bm \gamma \in\Omega^1_\sigma(Q)\,|\,\text{$\gamma_i\colon[0,1]\to Q$ is the integral flow of $\varphi^t_{L_i}$}\right\}.
\end{align}
When there is no ambiguity, we also denote $(\bm \gamma,\sigma)$ simply by $\bm \gamma$.
We denote the permutation element $\sigma$ associated to $\bm \gamma$ by $\sigma({\bm \gamma})$.
Note that each cycle in $\sigma({\bm \gamma})$ corresponds to a loop in $Q$. 
We further require that for all $\bm \gamma\in\mathcal{P}_\sigma^{\bm L}$,
\begin{enumerate}
    \item $\gamma_i(t)\neq \gamma_j(t)$ for all $t$ and $i\neq j$;
    \item each $\gamma^{c_i(\sigma)}$ is a non-degenerate critical point of $\mathcal{A}_{L_{c_i(\sigma)}}$.
\end{enumerate}

The following lemma follows immediately from Theorem~\ref{theorem-action+gradient}.
\begin{lemma}
    For a generic choice of metric and $\{H_1^{\mathrm{tot}},\ldots,H_\kappa^{\mathrm{tot}}\}$, the above conditions can be achieved. Moreover, $\mathcal{A}_{\bm L}$ is a Morse function on $\Omega^1_\sigma(Q)$ such that the set of critical points of $\mathcal{A}_{\bm L}$ is $\mathcal{P}_\sigma^{\bm L}$, all of which are non-degenerate and have finite Morse indices.
\end{lemma}

Clearly for $\bm \gamma =(\gamma_1, \ldots, \gamma_\kappa) \in \mathcal{P}^{\bm L}_\sigma$ the Morse index $|\bm \gamma|$ can be computed as a sum of Morse indices of $\gamma_i$:
\begin{equation}
    |\bm \gamma|=I_{L_1}(\gamma_1)+\ldots+I_{L_\kappa}(\gamma_\kappa).
\end{equation}

We denote the set of all multiloops by
\begin{equation}
    \Lambda^1_{[\kappa]}(Q)= \bigsqcup_{\sigma \in \mathfrak{S}_\kappa} \Omega^1_\sigma(Q)
\end{equation}
and the set of all critical points of $\bm L$ by
\begin{equation}
    \mathcal{P}^{\bm L}_\kappa= \bigsqcup_{\sigma \in \mathfrak{S}_\kappa}  \mathcal{P}_\sigma^{\bm L}.
\end{equation}
There is a conjugation action on $\Lambda^1_{[\kappa]}(Q)$ by any $\rho \in \mathfrak{S}_\kappa$ given by the following diffeomorphism
    \begin{gather}
        {c}_{\rho} \colon \Lambda^1_{[\kappa]}(Q) \to \Lambda^1_{[\kappa]}(Q), \\
        \bm \gamma \mapsto (\gamma_{\rho(1)}, \ldots, \gamma_{\rho(\kappa)}). \nonumber
    \end{gather}
    Notice that $$c_\rho \left( \Omega^1_\sigma(Q)\right)=\left( \Omega^1_{\rho \circ \sigma \circ \rho^{-1}}(Q)\right).$$

\subsection{Switching map.}
\subsubsection{Motivation.} The main new ingredient that we need for defining the deformed Morse chain complex for free multiloops is the so-called \emph{switching map}. For some pair of distinct elements of $[\kappa]=\{1, \ldots, \kappa\}$, i.e. an element $\{i, j\} \in \binom{[\kappa]}{2}$, consider the subspace $\Delta_{I}^{ij}$ of $[0,1] \times \Lambda^1_{[\kappa]}(Q)$ given by the preimage of the diagonal $\Delta_Q \subset Q \times Q$ under 
\begin{gather*}
ev^{ij}_I: [0,1]\times \Lambda^1_{[\kappa]}(Q) \to Q\times Q,\\
ev^{ij}_I(\theta, \bm \gamma)=(\gamma_i(\theta),\gamma_j(\theta)).
\end{gather*}
Then \emph{the naive switching map} $\Delta^{ij}_I \to \Delta^{ij}_I$ is given by 
$$(\theta, \gamma_1, \ldots, \gamma_\kappa) \mapsto (\theta, \gamma_1, \ldots, \gamma_i|_{[0, \theta]}\# \gamma_j|_{[\theta, 1]}, \ldots, \gamma_j|_{[0, \theta]}\# \gamma_i|_{[\theta, 1]}, \ldots, \gamma_\kappa).$$
To make this naive map work well with the moduli of Morse flow trajectories, we need to introduce a tedious upgrade of the above map, which we do in the rest of this section. Notice that one source of problems comes from the fact that $\Delta^{ij}_I$ is not even a $C^1$-submanifold in $\Lambda^1_{[\kappa]}(Q)$. 

\begin{figure}[ht]
    \centering
    \includegraphics[width=0.8\linewidth]{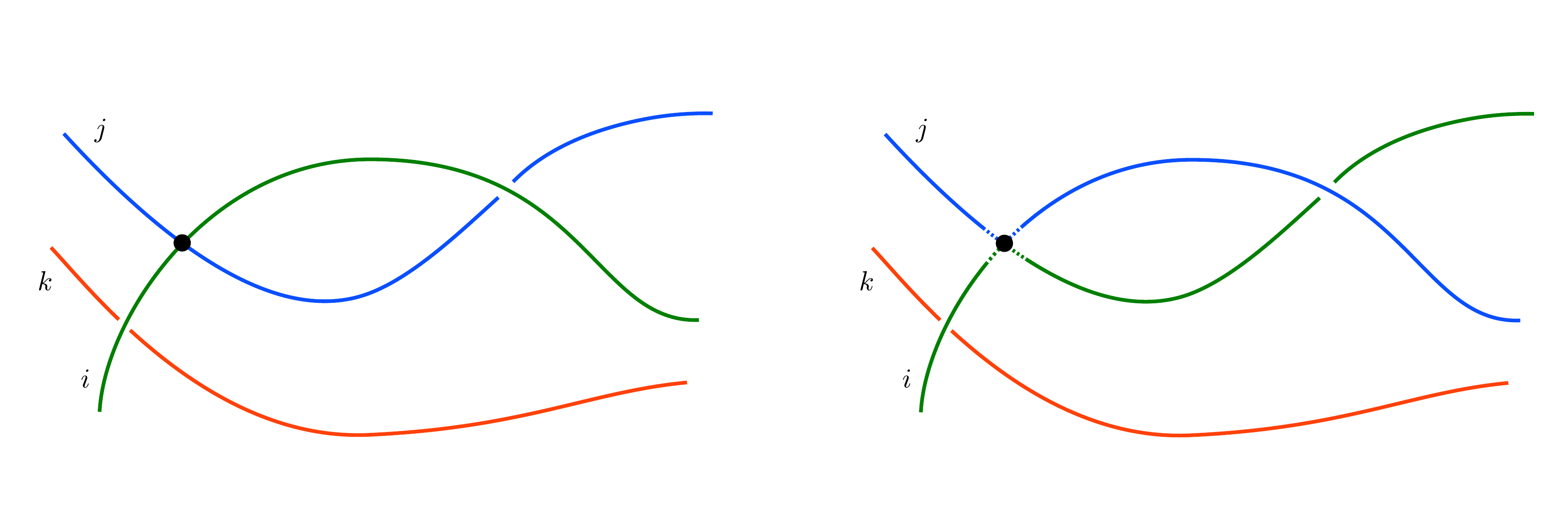}
    \caption{Here we schematically illustrate the effect of the switching map $sw^{ij}_{I_1}$ on the left tuple of paths, where we have locally drawn $3$ out of $\kappa$ paths. All paths except $\gamma_i$ and $\gamma_j$ are unaffected, and the portion of $\gamma_i$ after the intersection point at $\theta$ gets replaced with the reparametrized portion of $\gamma_j$ and vice versa. Dashed lines depict the portions of $\gamma_i$ and $\gamma_j$ affected by the smoothing reparametrization.}
    \label{fig:switching-map}
\end{figure}
\subsubsection{Construction of the switching map.}\label{section: switching map} Here we closely follow the exposition given in \cite[Section 2.2]{honda2025morse} and only slightly change the definition given there, adapting it to the closed string setting.

Fix an integer $\ell \ge 1$. Consider a subset $A=\{i_1, \ldots, i_{\tilde{\ell}}\} \subset [\ell]$ and an $\tilde{\ell}$-tuple of elements of $\binom{[\kappa]}{2}$ given by $\bm \tau=(\tau_1=\{j_1^1,j_1^2\}, \ldots, \tau_{\tilde{\ell}}=\{j^1_{\tilde{\ell}},j^2_{\tilde{\ell}}\})$. For any space $X$ we denote by $\operatorname{Conf}_A(X)$ the subspace of $X^\ell$ consisting of tuples $(x_1, \ldots, x_\ell)$ such that $x_i \neq x_j$ for any two $i, j \in A$, and it is isomorphic to $\operatorname{Conf}_{|A|}(X) \times X^{\ell-|A|}$. We associate with $A$ and $\bm \tau$ a \emph{multidiagonal} subspace $\Delta^{\bm \tau}_A$ of $\operatorname{Conf}_{A}([0,1])\times \Lambda_{[\kappa]}^1(Q)$ given by
\begin{equation}
    \Delta^{\bm \tau}_A\coloneqq (ev_{I_A}^{\bm \tau})^{-1}(\Delta_Q^{\tilde{\ell}}), \text{ where}
\end{equation}
\begin{gather}\label{eq: evaluation-map}
    ev_{I_A}^{\bm \tau} = ev^{\tau_1, \dots, \tau_{\tilde{\ell}}}_{I_{i_1}, \dots, I_{i_{\tilde{\ell}}}}\colon \operatorname{Conf}_A([0,1]) \times \Lambda^{1}_{[\kappa]}(Q) \to Q^{2\tilde{\ell}}, \\
    (\theta_1, \ldots, \theta_\ell, \bm \gamma) \mapsto ((\gamma_{j^1_1}(\theta_{i_1}),\gamma_{j^2_1}(\theta_{i_1})),\ldots,(\gamma_{j^1_{\tilde{\ell}}}(\theta_{i_{\tilde{\ell}}}),\gamma_{j^2_{\tilde{\ell}}}(\theta_{i_{\tilde{\ell}}}))) \nonumber.
\end{gather}

We point out that restriction of $\Delta^{\bm \tau}_A$ to $\operatorname{Conf}_A([0,1]) \times\Lambda^{m,2}_{[\kappa]}(Q)$ is a $C^{m-1}$-submanifold with boundary as shown in \cite[Section 2.2.1]{honda2025morse}. By taking unions of such subspaces and their closures over various choices of $\bm \tau$ we define subsets
\begin{equation}
    \Delta_{I_A}=\Delta_{I_{i_1}, \dots, I_{i_{\tilde{\ell}}}}, \quad\overline{\Delta}_{I_A}=\overline{\Delta}_{I_{i_1}, \dots, I_{i_{\tilde{\ell}}}} \subset [0,1]^{\ell} \times \Lambda^1_{[\kappa]}(Q).
\end{equation}

We also introduce
\begin{gather}
    \overline{\Delta}_{\R^\ell}^A=\operatorname{ev}_A^{-1}(\Delta^{\operatorname{thin}}_{\R^{\tilde{\ell}}}), \nonumber \\
    \operatorname{ev}_A \colon \R^\ell \to \R^{\tilde{\ell}}, \quad(s_1, \ldots, s_\ell) \mapsto (s_{i_1}, \ldots, s_{i_{\tilde{\ell}}}), \nonumber
\end{gather}
where $\Delta^{\operatorname{thin}}_{\R^\ell}$ is the thin diagonal in $\R^\ell=\R\times\dots\times\R$. We denote by $\Delta_{\R^\ell}^A$ the open submanifold of $\overline{\Delta}_{\R^\ell}^A$ consisting of points such that $s_j \neq s_i$ for $j \in [\ell] \setminus A$ and $i \in A$. Let $\overline{\Delta}_{\R^\ell}^A$ be its closure in $\R^\ell$. Note that $\overline{\Delta}^A_{\R^\ell}=\R^\ell$ if $|A| \le 1$. We define
\begin{equation}
    \Delta^{|}_{I_A}\coloneqq\Delta_{\R^{\ell}}^A \times \Delta_{I_A}, \quad \overline{\Delta}^{|}_{I_A}\coloneqq \overline{\Delta}^A_{\R^\ell} \times \overline{\Delta}_{I_A} \subset \R^\ell\times[0,1]^\ell\times \Lambda^1_{[\kappa]}(Q).
\end{equation}

Notice that each of these manifolds is invariant under the $\R$ action on $\R^\ell$ by simultaneous translations of each coordinate. 
We set
$$\operatorname{Conf}_S([0,1])=\{(\theta_1, \ldots, \theta_p) \in [0,1]^p \mid \theta_i=\theta_j \text{ iff }i \sim_Sj\},$$
where $\sim_S$ is an equivalence relation associated with a partition $S$. 

Given a $1$-periodic diffeomorphism $f \colon \R \to \R$, it induces a diffeomorphism $$f^{\circ} \colon \Lambda^1_p(Q) \to \Lambda^1_p(Q)$$ by pre-composition $\gamma \mapsto \gamma\circ f$ for any integer $p \ge 1$ since $f$ induces a diffeomorphism of $\R/p\Z$. Therefore, any such $f$ also induces a diffeomorphism $f^\circ$  of $\Lambda^1_{[\kappa]}(Q)$.

For each $p \ge 1$ we pick a continuous family of such diffeomorphisms $f_{\theta_1, \ldots, \theta_p}$ parametrized by $(\theta_1, \ldots, \theta_p) \in [0,1]^p$ satisfying 
\begin{gather}
f_{\theta_1, \dots, \theta_p}(\theta_i)=\theta_i \quad \text{ for all $i$ such that  $1\le i \le p$;} \label{eq: family-property1 for r}   \\
    f_{\theta_1,\dots, \theta_p}^{(k)}(\theta_i)=0 \quad \text{ for all $k\geq 1$ and $i$ such that $1\leq i\leq p$};\\
    |f_{\theta_1, \ldots, \theta_p}'| \le 1+c \text{ for some small  } c> 0 \label{eq: family-property3 for r} ;\\
    f_{\theta_1, \ldots, \theta_p} \text{ restricts to a smooth family on }\operatorname{Conf}_S([0,1]) \text{ for any partition } S \text{ of }[p]; \\
    f_{\theta_1, \ldots, \theta_p} \text{ descends to a well-defined family on }T^p=([0,1]/(0 \sim 1))^p \label{eq: family-property5}.
\end{gather}
We omit the explicit construction and refer to \cite[Section 2.2.2]{honda2025morse}. We point out that we do not require $f_{\theta_1, \ldots, \theta_p}(0) \equiv 0$ or $f_{\theta_1, \ldots, \theta_p}(1) \equiv 1$.

\smallskip
We now construct a continuous family of smooth $1$-periodic diffeomorphisms $\nu_{[\ell]}=:\R\to \R$ parametrized by $\R^\ell \times [0,1]^\ell \times \Lambda^1_{[\kappa]}(Q)$. First choose a cover $\{U_A\}_{A \subset [\ell]}$ of $\R^\ell \times [0,1]^\ell \times \Lambda^1_{[\kappa]}(Q)$ 
such that:
\begin{enumerate}
\item[(C1)] $\overline{\Delta}^{|}_{I_A} \subset \bigcup_{A' \subseteq A} U_{A'}$  for all $A \subset [\ell]$;
\item[(C2)] $U_{A} \cap \overline{\Delta}^{|}_{I_{A'}}=\emptyset$  for all pairs $A \subsetneq A' \subset [\ell]$;
\item[(C3)] $U_A$ is contained in an $\varepsilon$-neighborhood of $\Delta_A^{|}$ for each $A \subset [\ell]$ for a fixed $\varepsilon>0$; each $U_A$ is $\R$-invariant.
\end{enumerate}
Fix a partition of unity $\{\psi_A\}_{A \subset [\ell]}$ subordinate to $\{U_A\}_{A \subset [\ell]}$ with $\psi_A$'s $\R$-invariant. Given $A=\{j_1, \dots, j_{r'}\} \subset [\ell]$ we define $$f_A:U_A\times\R\to \R, \quad ((s_1, \ldots, s_\ell,\theta_1,\dots,\theta_\ell,\bm\gamma),x)\mapsto f_{\theta_{j_1},\dots,\theta_{j_{r'}}}(x).$$ 
We then set
\begin{equation}
    \nu_{[\ell]}=\sum_{S \subset [\ell]}\psi_Sf_S.
\end{equation}

By (C1)--(C3), the restriction of $\nu_{[\ell]}$ to $\overline{\Delta}_A^{|}$ is equal to $f_A$ on the complement 
of the union $\bigcup_{A' \subsetneq A} U_{A'}$, which by (C3) is contained within a small neighborhood of $\partial \overline{\Delta}_A^{|}$. In particular, its restriction to the most of $\Delta^{|}_{I_j}=\R^\ell\times \Delta_{I_j}$ is equal to $f_{\theta_j}$. We observe that $\nu_\ell$ descends to a function parametrized by $\R^\ell \times T^\ell \times \Lambda^1_{[\kappa]}(Q)$ by~\eqref{eq: family-property5}.

\begin{definition}\label{defn: switching-map}
    For $\overline{\Delta}^{ij, \,|}_{I_k} \subset \R^\ell \times [0,1]^{\ell} \times \Lambda^1_{[\kappa]}(Q)$ with $\{i, j\} \in \binom{[\kappa]}{2}$ and $k \in [\ell]$ the \emph{switching map} $sw_{I_k}^{ij} \colon \overline{\Delta}^{ij, \, |}_{I_k} \to \overline{\Delta}^{ij,\, |}_{I_k}$ is defined by the formula
    \begin{gather}
        sw_{I_k}^{ij}(\bm s, \bm \theta, \bm \gamma)=sw^{ij, \operatorname{ naive}}_{I_k}(\bm s, \bm \theta, \nu_{[\ell]}(\bm s, \bm \theta, \bm \gamma)^{\circ}(\bm \gamma)), \text{ where}
        \\
        sw^{ij, \operatorname{naive}}_{I_k}(\bm s, \bm \theta, \bm \gamma')=(\bm s, \bm \theta, \gamma_1', \ldots, \gamma_i'|_{[0, \theta_k]}\# \gamma_j'|_{[\theta_k, 1]}, \ldots, \gamma_j'|_{[0, \theta_k]}\# \gamma_i'|_{[\theta_k, 1]}, \ldots, \gamma_\kappa'), \nonumber
    \end{gather}
    and $\#$ stands for the concatenation. See Figure~\ref{fig:switching-map} for an illustration.
\end{definition}

In the following lemma, we state the properties satisfied by the switching map and refer the reader to \cite[Proposition 2.6]{honda2025morse} for proofs.
\begin{figure}[h]
    \centering
    \includegraphics[width=6cm]{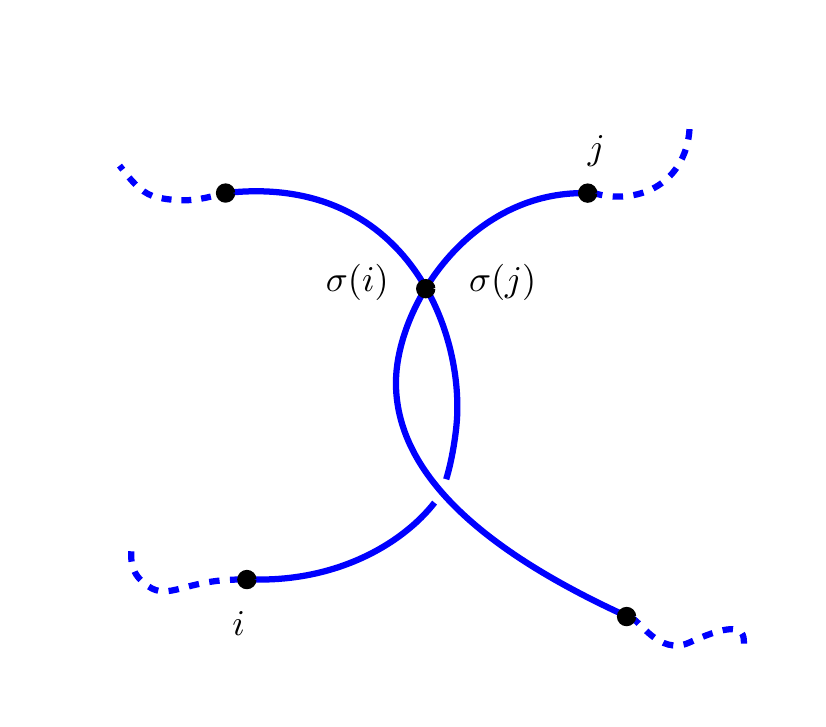}
    \caption{Here we depict a portion of the tuple $\bm \gamma \in \Omega^1_\sigma(Q)$ such that $\gamma_i(1)=\gamma_{\sigma(i)}(0)=\gamma_{\sigma(j)}(0)=\gamma_j(1)$; the strands are oriented counterclockwise. It is straightforward to verify that ${}^{(\sigma(i)\sigma(j))}(\sigma \cdot (ij))=\sigma \cdot (\sigma(i)\sigma(j)),$ which confirms~\eqref{eq: loop-conjugation-relation}.}
    \label{fig: circular-switch}
\end{figure}

\begin{lemma}
    For arbitrary indices $p_1, p_2 \in [\ell]$:
\begin{enumerate}
\item the map $sw^{ij}_{I_{p_1}}$ is a $C^{m-1}$-immersion when restricted to $\operatorname{Conf}_{S}(\R \times [0,1]) \times \Lambda^{m,2}_{[\kappa]}(Q)$ for an arbitrary partition $S$ of $[\ell]$;
\item the following commutativity relations hold:
\begin{gather}
    sw^{ij}_{I_{p_2}}(sw^{ij}_{I_{p_1}}(\bm s,\bm \theta, \bm \gamma))=sw^{ij}_{I_{p_1}}(sw^{ij}_{I_{p_2}}(\bm s,\bm \theta, \bm \gamma) ), \text{ if }(\bm s,\bm \theta, \bm \gamma) \in \Delta^{ij, ij,\, |}_{I_{p_1},I_{p_2}}, \label{eq : commutativity-switching}\\
    sw^{jk}_{I_{p_2}}(sw^{ij}_{I_{p_1}}(\bm s, \bm \theta, \bm \gamma))=sw^{ij}_{I_{p_1}}(sw^{ik}_{I_{p_2}}(\bm s, \bm \theta, \bm \gamma)), \text{ if }(\bm s, \bm \theta, \bm \gamma) \in \Delta^{ij, ik, \, |}_{I_{p_1},I_{p_2}} \text{ and }\theta_{p_1}\le\theta_{p_2}, \label{eq-switching-commuting-1}\\
    sw^{ik}_{I_{p_2}}(sw^{ij}_{I_{p_1}}(\bm s,\bm \theta, \bm \gamma))=sw^{kj}_{I_{p_1}}(sw^{ik}_{I_{p_2}}(\bm s,\bm \theta, \bm \gamma)), \text{ if }(\bm s,\bm \theta, \bm \gamma) \in \Delta^{ij, ik, \,|}_{I_{p_1},I_{p_2}} \text{ and }\theta_{p_1} \ge \theta_{p_2}, \label{eq-switching-commuting-2} \\
    sw_{I_{p_1}}^{ij}(sw^{kl}_{I_{p_2}}(\bm s,\bm \theta, \bm \gamma))=sw_{I_{p_2}}^{ij}(sw_{I_{p_1}}^{kl}(\bm s,\bm \theta, \bm \gamma)), \text{ if } (\bm s,\bm \theta, \bm \gamma) \in \Delta^{ij, kl, \, |}_{I_{p_1},I_{p_2}}, \\
    c^{|}_{\tau}(sw^{ij}_{I_{p_1}}(\bm s, \bm \theta, \bm \gamma))=sw_{I_{p_1}}^{\sigma(i)\sigma(j)}(s, \bm \theta', \bm \gamma), \text{ if }\theta_{p_1}=1, \label{eq: loop-conjugation-relation}
\end{gather}
where $i, j, k, l$ are assumed to be distinct, $\bm s=(s_1, \ldots, s_{\ell})$, $\bm \theta=(\theta_1, \dots, \theta_\ell)$; $\bm \theta'$ in~\eqref{eq: loop-conjugation-relation} is obtained from $\bm \theta$ by replacing $\theta_{p_1}$ with $\theta_{p_1}'=0$, $\sigma=\sigma(\gamma)$; $\tau$ is the transposition $(\sigma(i)\sigma(j))$, and 
\begin{equation}\label{eq: conjugation-with-parameters}
    c_\tau^|(\bm s, \bm \theta, \bm \gamma)\coloneqq(\bm s, \bm \theta, c_\tau(\bm \gamma)).
\end{equation}
\item for $(\bm s, \bm \theta, \bm \gamma')=sw^{ij}_{I_{p_1}}(\bm s, \bm \theta, \bm \gamma)$ it holds
\begin{equation}
    \mathcal{A}_{\bm L}(\bm \gamma') \le (1+c)\mathcal{A}_{\bm L}(\bm \gamma)
\end{equation}
for a fixed small $c>0$ as in~\eqref{eq: family-property3 for r}.
\end{enumerate}
\end{lemma}

\begin{remark}
    The property~\eqref{eq: loop-conjugation-relation} is the only new interesting relation in the closed multiloop case when compared to based multiloops in \cite[Proposition 2.6]{honda2025morse}. It is verified by direct computation, and we illustrate it in Figure~\ref{fig: circular-switch}.
\end{remark}

\begin{remark}
    The last property above follows immediately from \ref{L1-condition}-\ref{L2-condition} and~\eqref{eq: family-property3 for r}.
\end{remark}

\subsection{The free multiloop complex}

\subsubsection{Recollection of the standard Morse complex for $\Lambda^1(Q)$} We first recall the definition of the Morse chain complex $CM_*\left(\Lambda^1(Q)\right)$ in the standard $\kappa=1$ setting following \cite{abbondandolo2006floer, abbondandolo2010floer}. The chains are given by the graded vector space over $\mathbb{Q}$ generated by the critical points of $\mathcal{A}_L$. By Theorem~\ref{theorem-action+gradient} there exists a pseudo-gradient vector field $X$ for $\mathcal{A}_L$. The Morse differential is then given by counting index $1$ \emph{pseudo-gradient trajectories} of $\mathcal{A}_L$ associated with the flow of $X$
\begin{equation}
\frac{\partial \Gamma}{\partial s}=X(\Gamma),    
\end{equation}
where we are parametrizing paths in $\Lambda^1(Q)$ by the variable $s$.

Given a critical point $\gamma \in \mathcal{P}^L$ one may define the \emph{stable} and \emph{unstable} manifolds of $\gamma$ with respect to $X$:
\begin{equation}
\begin{aligned}
    W^s(\gamma ; X) & \coloneqq\{\Gamma \in C^1([0, +\infty)_s, \Omega^{1,2}(M, \bm q, \bm q')) \mid   \, \lim_{s \to +\infty} \Gamma= \bm \gamma,  \, \lim_{s \to +\infty} \partial_s \Gamma=0, \, \partial_s\Gamma+X=0\}, \\
    W^u(\gamma ; X) & \coloneqq\{\Gamma \in C^1((-\infty, 0]_s, \Omega^{1,2}(M, \bm q, \bm q')) \mid   \, \lim_{s \to -\infty} \Gamma= \bm \gamma,  \, \lim_{s \to -\infty} \partial_s \Gamma=0, \, \partial_s\Gamma+X=0\} .
\end{aligned}
\end{equation}
Note that by condition~\ref{PG3} the unstable manifold is always finite-dimensional.

To each critical $\gamma$ we associate a graded rank $1$ abelian group $o_\gamma$ generated by orientations of the unstable manifold $W^u(\gamma; X)$ with grading given by
\begin{equation}
    |\gamma|=-\operatorname{ind}_{\operatorname{Morse}}(\gamma).
\end{equation}
We stress that we are using negative (cohomological) grading.

We denote by $\widetilde{\mathcal{M}}\left(\gamma, \gamma' ; X\right)$ the moduli space of pseudo-gradient trajectories from $\gamma$ to $\gamma'$ with respect to $X$ and $\mathcal{M}\left(\gamma, \gamma' ; X\right)=\widetilde{\mathcal{M}}\left(\gamma, \gamma^{\prime} ; X\right) / \R$ be the quotient by the $\R$-translation. Notice that $\widetilde{\mathcal{M}}\left(\gamma, \gamma' ; X\right)$ can be defined as a fiber product of $W^u(\gamma; X)$ and $W^s(\gamma'; X)$. 
We set
\begin{equation}
    C M_{-*}\left(\Lambda^1(Q)\right)= \bigoplus_{\gamma \in \mathcal{P}^L}o_\gamma.
\end{equation}
Then we define

\begin{gather}
    \partial\colon C M_{-*}\left(\Lambda^1(Q)\right) \rightarrow CM_{-*+1}\left(\Lambda^1(Q)\right) \\
    \partial [\gamma]=\sum_{\substack{\Gamma \in \mathcal{M}\left(\gamma, \gamma' ; X\right) \\ |\gamma'|=|\gamma|+1}} \# \partial_\Gamma([\gamma]), \nonumber
\end{gather}
where $[\gamma]$ is a generator of $o_\gamma$, and 
$$\partial_\Gamma \colon o_\gamma \to o_{\gamma'}$$
is a natural morphism associated with the trajectory $\Gamma$.
The cohomology of this complex is independent of the choice of $X$ according to \cite{abbondandolo2010floer}.

\subsubsection{Perturbation spaces}\label{section: perturbation spaces}
Here we shed more light on statement (3) of Theorem~\ref{theorem-action+gradient} and introduce particular perturbation spaces which we use to achieve transversality for various moduli spaces. First, recall that given a smooth loop $\gamma \in \Lambda^1_T(Q)$ for $T \in \Z_{>0}$, and a small open ball $U \subset \R^n$ there is a local chart $\mathcal{U}_{\gamma} \subset \Lambda^1_T(Q)$ given by
\begin{gather}
    \operatorname{exp}_* \colon \Lambda^1_T(U) \to \Lambda^1_T(Q), \\
    \xi \mapsto \operatorname{exp}_{\gamma(\cdot)}(\xi(\cdot)),
\end{gather}
where we use trivialization $[0, T] \times U \subset [0, T] \times \R^n \cong \gamma^*TQ$, and we refer to \cite[Section 3]{abbondandolo2009pseudo-gradient} and \cite[Appendix A]{honda2025morse} for more details.

We claim that there exists a countable locally finite cover $\{\mathcal{U}_{\bm \gamma_j}\}_{j \in \mathcal{J}}$ of $\Lambda^{1}_{[\kappa]}(Q)$, where $\bm \gamma_j=(\gamma_{j}^{c_1(\sigma)},\dots, \gamma_{j}^{c_{l(\sigma)}(\sigma)})$ is a \emph{smooth} multiloop, here $\sigma=\sigma(\bm \gamma_j)$, and $\mathcal{U}_{\bm \gamma_j}$ is the product of neighborhoods $\mathcal{U}_{\gamma_{j}^{c_i(\sigma)}}$ as above.  We also assume the following: 
\begin{enumerate}
\item[(C1)] if some $\gamma\in \mathcal{U}_{\gamma_{ji}}$ is a critical point, then $\gamma$ lies in only one chart $\mathcal{U}_{\gamma_{ji}}$ (called a {\em critical point chart}) and in that case $\gamma=\gamma_{ji}$;
\item [(C2)] critical point charts will be arbitrarily small and will be shrunk as necessary in response to given closed subsets in $\R^\ell \times [0,1]^\ell\times \Lambda_{[\kappa]}^{1}(Q).$
\end{enumerate}
We fix a partition of unity $\{\psi_j\}_{j \in \mathcal{J}}$ subordinate to this cover.

For such a cover we define $\mathfrak{X}$ to consist of smooth vector fields $Y$ on $\Lambda^1_{[\kappa]}(Q)$ which admit a presentation $Y=\sum_{j \in \mathcal{J}} \psi_j Y_j$, where $Y_j=(Y_j^{c_1(\sigma)}, \dots,Y_j^{c_l(\sigma)})$ and each $Y_j^{c_i(\sigma)}$ is linear in $\mathcal{U}_{\gamma_{j}^{c_i(\sigma)}}$ with restriction to $\gamma_{j}^{c_i(\sigma)}$ of finite $C_{\bm \varepsilon}$-norm for an appropriate choice of $\bm \varepsilon=(\varepsilon_k)_{k \ge 1}$.
The sequence $\{Y_j\}_{j \in \mathcal{J}}$ is required to be bounded, which makes $\mathfrak{X}$ into a separable Banach space, see \cite[Appendix B]{honda2025morse} for more details. Further we define $\mathfrak{X}_-$ and $\mathfrak{X}_+$ to be subspaces of $C_{\bm\varepsilon}((-\infty, 0], \mathfrak{X})$ and $C_{\bm\varepsilon}([0,+\infty), \mathfrak{X})$ consisting of functions vanishing away fram $[-1,0]$ and $[0,1]$, respectively. 

Next, fix a nondecreasing smooth function $g \colon [0, +\infty) \to [0, +\infty)$ such that $g(l)=\frac{l}{3}$ for $l \le 1$ and $g(l)=1$ for $l\ge 3$, and define
\begin{gather}\label{eq: finite-length-parameter-space}
    \mathfrak{X}_0=
    \{X\in C_{\bm\varepsilon}([0,+\infty)^2, \mathfrak{X})~ \mid  ~ \mbox{$X(l,s)=0$ for $s \in  (g(l), \operatorname{max}(l-g(l),l-1))\cup [l, +\infty)$}; \\
    \lim_{l \to +\infty}\operatorname{split}X(l,\cdot) \text{ exists in } \mathfrak{X}_+ \times \mathfrak{X}_-; \text{ all partial derivatives} \nonumber\\
    \text{ of } X 
    \text{ vanish at } (0,s) \text{ for all } s \in [0,+\infty)\}, \nonumber
\end{gather}
where, writing $X_\lambda(s)\coloneqq X(\lambda,s)$ and $X_\lambda^T(s)\coloneqq  X_\lambda(s+T)$,
$$\operatorname{split}X_\lambda=(X_\lambda|_{[0,1]}\#0|_{[1, +\infty)}, 0|_{(-\infty,-1]}\#X_\lambda^{\lambda+1}|_{[-1,0]}).$$
We refer to \cite[Appendix B]{honda2025morse} and \cite[Section 1]{mescher2018} for more details.

Given an integer $\ell \ge 0$, and a permutation $\bm h=(h_1, \ldots, h_\ell)$ in $\mathfrak{S}_\ell$ encoding the component $\operatorname{Conf}_\ell^{\bm h}(\R)$ of $\operatorname{Conf}_\ell(\R)$ we define perturbation spaces
$\mathfrak{X}(\ell, \bm h, \bm \tau,  \bm \rho)$ that are necessary to work with to achieve \emph{consistency} of the data in the sense that we explain below. For the rest of the discussion, we regard the tuple $\bm h$ also as a strict linear order on the given $\ell$ elements.

The basic building blocks are the spaces $\mathfrak{X}_{\pm}(k)$ and $\mathfrak{X}_0^i(k)$, $k=0,1,2,\dots$, $i=1,\dots, k$ introduced in \cite[Appendix C, Equations (C.1)-(C.3)]{honda2025morse} and we refer the reader there for the detailed definition. We recall that element of each of theses spaces is either an element of $C^\infty(\mathscr{D}_k \times [0,+\infty), \mathfrak{X})$ or $ C^\infty(\mathscr{D}_k \times (-\infty, 0], \mathfrak{X})$, where $\mathscr{D}_k=[0, +\infty)^k$. We denote by $(\lambda_1, \ldots, \lambda_k) \in \mathscr{D}_k$ a point of this space.

For each $k\ge 0$ we define the following map from the closure of $\operatorname{Conf}^{\bm h}_k(\R)$ in $\R^k$
\begin{gather}
    \mathcal{D}_{\bm h}\colon \overline{\operatorname{Conf}}^{\bm h}_k(\R) \to \mathscr{D}_{k-1}, \label{eq: d-map-conf} \\
    (s_1, \ldots, s_k) \mapsto (s_{h_1}-s_{h_2}, \dots, s_{h_{k-1}}-s_{h_k}). \nonumber
\end{gather}
Set $\mathfrak{X}_{\pm}(k, \bm h)$ and $\mathfrak{X}_{0}^i(k, \bm h)$ to be the pullbacks of the spaces introduced above (corresponding to $(k-1)$) under $\mathcal{D}_{\bm h}$. Notice that $\overline{\operatorname{Conf}}^{\bm h}_k(\R) \cong \R \times \mathscr{D}_{k-1}$. For $Y \in \mathfrak{X}_{0}^i(k, \bm h)$ and given $\bm s \in \operatorname{Conf}^{\bm h}_k(\R)$ we denote by $Y_{\bm s, i} \in \mathfrak{X}_0$ the restriction of $Y$ to $$\R \times \{s_{h_1}-s_{h_2}\} \times \dots \times \{s_{h_{i-1}}-s_{h_i}\} \times [0, +\infty) \times \{s_{h_{i+1}}-s_{h_{i+2}}\} \times \dots \times \{s_{h_{\ell-1}}-s_{h_{\ell}}\},$$
i.e. its ``length" parameter $\lambda$ is the $i$-th distance.

We then have that an element $Y$ can be regarded as an element of $C^\infty(\overline{\operatorname{Conf}}^{\bm h}_k(\R) \times [0, +\infty), \mathfrak{X})$ or $C^\infty(\overline{\operatorname{Conf}}^{\bm h}_k(\R) \times (-\infty, 0], \mathfrak{X})$, and is invariant under simultaneous parallel transport of the $k$ points in $\R$. One may interpret the defining conditions for these spaces as allowing to extend $Y$ as above to the codimension $1$ boundary of the natural compactification of $\operatorname{Conf}^{\bm h}_k(\R)$ obtained by ``breaking" the real line into two parts whenever $(s_{h_i}-s_{h_{i+1}}) \to +\infty$ for some $i=1,\dots, k$. The codimension $1$ boundary of such compactification is given by
\begin{equation}\label{eq: conf-boundary-decomposition}
\bigsqcup_{i=1}^k\overline{\operatorname{Conf}}^{\bm h_i^1}_{i}(\R) \times \overline{\operatorname{Conf}}^{\bm h^2_{i}}_{k-i}(\R),
\end{equation}
Where $(\bm h_i^1, \bm h_i^2) \in \mathfrak{S}_i \times \mathfrak{S}_{k-i}$ are permutations on $i$ and $(k-i)$ elements respectively, given by restricting the strict order $\bm h$ to the $i$ greatest (resp., $k-i$ smallest) elements.

To each $k\in \Z_{\ge 1}$ we also associate the \emph{background perturbation spaces}
\begin{equation} \label{back pm}
    \mathfrak{X}_{ \pm}^{\operatorname{back}}(k, \bm h)\coloneqq  \left\{\bm Z=(Z_1, \dots, Z_k) \in \mathfrak{X}_{\pm}(k-2)^{{k-1}} \mid   \lim _{\lambda_i' \rightarrow \infty} Z_j|_{\lambda_i=\lambda_i'} =\lim _{\lambda_{j-1}' \rightarrow \infty}Z_i|_{\lambda_{j-1}=\lambda_{j-1}'} ~\forall j>i\right\}.
\end{equation}
One should think of the $i$-th component $\mathfrak{X}_{\pm}(k-2)$ in the product above as of a subspace of $C^\infty(\overline{\operatorname{Conf}}^{\bm h_i^1}_{i}(\R) \times \overline{\operatorname{Conf}}^{\bm h^2_{i}}_{k-i}(\R) \times [0, +\infty), \mathfrak{X})$ (or of the analogous space for $(-\infty, 0]$) pulled back under
$$(\mathcal{D}_{\bm h_i^1}, \mathcal{D}_{\bm h_i^2}) \colon \overline{\operatorname{Conf}}^{\bm h_i^1}_{i}(\R) \times \overline{\operatorname{Conf}}^{\bm h^2_{i}}_{k-i}(\R) \to \mathscr{D}_{k-1}.$$

In the light of boundary decomposition~\eqref{eq: conf-boundary-decomposition} we introduce a ``face map'' which assigns to $Y \in \mathfrak{X}_{\pm}(k, \bm h)$ its {\em background perturbation data}
\begin{gather}\label{eq-constructing-background}
F_{\pm,k}:\mathfrak{X}_{\pm}(k, \bm h) \to \mathfrak{X}_{ \pm}^{\operatorname{back} }(k, \bm h), \quad
Y\mapsto  (\lim_{\lambda_1'\rightarrow \infty} Y|_{s_{h_1}-s_{h_2}=\lambda_1'}, \dots, \lim_{\lambda_{k-1}' \rightarrow \infty} Y|_{s_{h_{k-1}-s_{h_k}}=\lambda_{k-1}'}).   
\end{gather}

Similarly, one defines the spaces $\mathfrak{X}_0^{\operatorname{back},i}(k, \bm h)$ for $i=1,\dots, k-1$, see~\cite[Equation (C.6)]{honda2025morse}, and there are face maps
\begin{gather} \label{eq-constructing-background-part2}
    F_{0,k,i}:  \mathfrak{X}_0^i(k, \bm h) \to \mathfrak{X}_0^{\operatorname{back},i}(k, \bm h),\\
    Y \mapsto (\lim_{\lambda_1' \rightarrow \infty}  Y|_{s_{h_1}-s_{h_2}=\lambda_1'}, \dots,\widehat{\lim_{\lambda_i' \rightarrow \infty} Y|_{s_{h_i}-s_{h_{i+1}}=\lambda_i'}},\dots, \lim_{\lambda_{k-1}' \rightarrow \infty} Y|_{s_{h_{k-1}}-s_{h_k}=\lambda_{k-1}'}, \lim_{\lambda_i' \to \infty} \operatorname{split}Y|_{s_{h_i}-s_{h_{i+1}}=\lambda_i'}).\nonumber
\end{gather}

Given $\bm Z \in \mathfrak{X}^{\operatorname{back}}_{\pm}(k, \bm h)$ and $\bm Z' \in \mathfrak{X}^{\operatorname{back},i}_{0}(k, \bm h)$ we set $\mathfrak{X}_{ \pm}(k, \bm h, \bm Z):=F_{\pm,k}^{-1}(\bm Z)$ and $\mathfrak{X}_{0}(k, \bm h, \bm Z'):=F_{0,k,i}^{-1} (\bm Z')$.

Using the above building blocks we now construct the spaces $\mathfrak{X}_{\operatorname{sw}}(\ell)$, $\ell\geq 1$, of \emph{$\ell$-switching perturbation data} and $\mathfrak{X}^{\operatorname{back}}_{\operatorname{sw}}(\ell)$ of \emph{background $\ell$-switching perturbation data}. Given an $\ell$-tuple of transpositions $\bm \tau=(\tau_1=(i_1j_1), \ldots, \tau_\ell=(i_\ell j_\ell))$, an $\ell$-tuple $\rho=(\rho_1, \ldots, \rho_\ell)$ of permutations in $\mathfrak{S}_\kappa$, and a permutation $\bm h \in \mathfrak{S}_\ell$, we set:
\begin{gather}\label{eqn: sw}
\mathfrak{X}_{\operatorname{sw}}(\ell, \bm h, \bm \tau, \bm \rho)\coloneqq  \mathfrak{X}_-(\ell, \bm h) \times \mathfrak{X}_0^1(\ell, \bm h)\times\dots\times  \mathfrak{X}_0^{\ell-1}(\ell, \bm h) \times \mathfrak{X}_+(\ell, \bm h),\\
\mathfrak{X}_{\operatorname{sw}}(\ell)\coloneqq\prod_{\bm h \in \mathfrak{S}_\ell, \bm \tau \in \binom{[\kappa]}{2}^\ell, \bm \rho \in \mathfrak{S}_\kappa^\ell}  \mathfrak{X}_{\operatorname{sw}}(\ell, \bm h, \bm \tau, \bm \rho), \label{eqn: sw1}\\
\label{eqn: sw2}
    \mathfrak{X}^{\operatorname{back}}_{\operatorname{sw}}(\ell, \bm h, \bm \tau, \bm \rho)\coloneqq \mathfrak{X}^{\operatorname{back}}_-(\ell, \bm h) \times \mathfrak{X}_0^{\operatorname{back},1}(\ell, \bm h)\times\dots\times  \mathfrak{X}_0^{\operatorname{back},\ell-1}(\ell, \bm h) \times \mathfrak{X}^{\operatorname{back}}_+(\ell, \bm h),\\
    \mathfrak{X}^{\operatorname{back}}_{\operatorname{sw}}(\ell) \coloneqq\prod_{\bm h \in \mathfrak{S}_\ell, \bm \tau \in \binom{[\kappa]}{2}^\ell, \bm \rho \in \mathfrak{S}_\kappa^\ell} \mathfrak{X}^{\operatorname{back}}_{\operatorname{sw}}(\ell, \bm h, \bm \tau, \bm \rho).
\end{gather}
Given $\bm Z_{\bm h, \bm \tau, \bm \rho}=(\bm Z^-, \bm Z_1^0, \dots, \bm Z_{\ell-1}^0, \bm Z^+) \in \mathfrak{X}^{\operatorname{back}}_{\operatorname{sw}}(\ell, \bm h, \bm \tau, \bm \rho)$ we define
the space of {\em $\ell$-switching data consistent with background data $\bm Z_{\bm h, \bm \tau, \bm \rho}$}:
\begin{equation}
    \mathfrak{X}_{\operatorname{sw}}(\ell, \bm h, \bm \tau, \bm \rho, \bm Z_{\bm h, \bm \tau, \bm \rho})\coloneqq\mathfrak{X}_-(\ell, \bm h, \bm Z^-) \times \mathfrak{X}_0^1(\ell, \bm h, \bm Z^0_1) \times \dots \times \mathfrak{X}_0^{\ell-1}(\ell, \bm h, \bm Z^0_{\ell-1}) \times \mathfrak{X}_+(\ell, \bm h, \bm Z^+).
\end{equation}

And, similarly, for a given $\bm Z=(\bm Z_{\bm h, \bm \tau, \bm \rho})_{\bm h, \bm \tau, \bm \rho} \in \mathfrak{X}^{\operatorname{back}}_{\operatorname{sw}}(\ell) $ we have
\begin{equation}
    \mathfrak{X}_{\operatorname{sw}}(\ell, \bm Z)\coloneqq\prod_{\bm h \in \mathfrak{S}_\ell, \bm \tau \in \binom{[\kappa]}{2}^\ell, \bm \rho \in \mathfrak{S}_\kappa^\ell}\mathfrak{X}_{\operatorname{sw}}(\ell, \bm h, \bm \tau, \bm \rho, \bm Z_{\bm h, \bm \tau, \bm \rho}).
\end{equation}

We say that $\bm Y ^\ell =(\bm Y^\ell_{\bm h, \bm \tau, \bm \rho})_{\bm h, \bm \tau, \bm \rho} \in \mathfrak{X}_{\operatorname{sw}}(\ell)$ is \emph{commutative} if   
    \begin{equation}\label{eq-perturbation-switching}
    \bm Y^\ell_{\bm h, \bm \tau, \bm \rho}|_{s_{h_i}-s_{h_{i+1}}=0}=\bm Y^\ell_{\bm h^i, \bm \tau^{i,1}, \bm \rho}|_{s_{h^i_i}-s_{h^i_{i+1}}=0} =\bm Y^\ell_{\bm h^i, \bm \tau^{i,2}, \bm \rho}|_{s_{h^i_i}-s_{h^i_{i+1}}=0}\text{ for all }i=1, \dots, \ell-1,
 \end{equation}
where $\bm h^i$ is obtained from $\bm h$ by transposing $s_{h_i}$ and $s_{h_{i+1}}$, and $\bm \tau^{i,1}$ and $\bm \tau^{i,2}$ are obtained from $\bm \tau$ by modifying $\tau_i$ and $\tau_{i+1}$ using the following rules:
\begin{gather}
   (\tau_i^{i,1}, \tau_{i+1}^{i,1})\coloneqq ({}^{\tau_i\rho_i^{-1}} {\tau_{i+1}}, {}^{\rho_i} {\tau_{i}}), \label{eq: data-commutativity-rule-1}  \\ 
    (\tau_i^{i,2}, \tau_{i+1}^{i,2})\coloneqq ({}^{\rho_i^{-1}} {\tau_{i+1}}, {}^{\tau_{i+1}\rho_i} {\tau_{i}}). \label{eq: data-commutativity-rule-2} 
\end{gather}
where we use the standard notation for conjugation in $\mathfrak{S}_\kappa$ above, i.e., ${}^{\sigma'}\sigma=\sigma' \circ \sigma \circ \sigma'^{-1}$. These rules arise naturally from the commutativity properties~\eqref{eq-switching-commuting-1} and~\eqref{eq-switching-commuting-2} of the switching map.

Given a permutation $\sigma \in \mathfrak{S}_\kappa$ and tuple $\bm \tau$ and $\bm \rho$ as above we set $\bm \sigma(\sigma, \bm \tau, \bm \rho)\coloneqq(\sigma_0, \sigma_1, \dots, \sigma_\ell)$ with 
\begin{equation}\label{eq: data-cyclic}
    \sigma_0=\sigma, \quad \sigma_{i+1}={}^{\rho_i}(\sigma_i\tau_i), \, \text{for }i=1, \ldots, \ell .
\end{equation}
Set the restriction $\bm Y^\ell_{\bm h, \bm \tau, \bm \rho}|_\sigma$ of $\bm Y^\ell_{\bm h, \bm \tau, \bm \rho}=(Y_0, Y_1, \dots, Y_\ell)$ to be the restriction of each $Y_i$ to $\Omega^1_{\sigma_i}(Q)$. 
We say that $\bm Y ^\ell$ as above is \emph{cyclic} if for $i=1, \ldots, \ell$ and for any permutation $\sigma \in \mathfrak{S}_\kappa$
\begin{equation}\label{eq:cyclic-data}
    \bm Y^\ell_{\bm h, \bm \tau, \bm \rho}|_\sigma=\bm Y^\ell_{\bm h, \bm \tau_i(\sigma), \bm \rho_i(\sigma)}|_\sigma,
\end{equation}
where $\bm \tau_i(\sigma)$ is obtained from $\bm \tau$ by substituting $\tau_i$ with ${}^{\sigma_i}\tau_i$ and $\bm \rho_i(\sigma)$ is obtained by replacing $\rho_i$ with $\rho_i{}^{\sigma_i}\tau_i$. The condition~\eqref{eq:cyclic-data} is motivated by the conjugation property~\eqref{eq: loop-conjugation-relation} of the switching map.

\begin{definition}\label{defn-universal-data}
    A \emph{universal switching data} is a collection $\bm Y=(\bm Y^\ell)_{\ell \ge 1}$ of $\ell$-switching data $\bm Y^\ell\in \mathfrak{X}_{\operatorname{sw}}(\ell)$ for all integers $\ell \ge 1$ such that the following hold:
    \begin{enumerate}
        \item For each $\ell$ the data $\bm Y^\ell$ is commutative and cyclic.
        \item The data is \emph{consistent}, i.e., $\bm Y^\ell_{\bm h, \bm \tau, \bm \rho} \in \mathfrak{X}_{\operatorname{sw}}(\ell, \bm Z_{\bm h, \bm \tau, \bm \rho}^{< \ell})$, where $\bm Z_{\bm h, \bm \tau, \bm \rho}^{<\ell}$ is obtained from the collection $(\bm Y^1, \dots, \bm Y^{\ell-1}$) as follows:
        
        Let us write $\bm Z_{\bm \tau, \bm s}^{< \ell}=(\bm Z^-, \bm Z_1^0, \dots, \bm Z_{\ell-1}^0, \bm Z^+)$, where 
        $$\bm Z^{\pm}=(Z^{\pm}_1, \dots, Z^{\pm}_{\ell-1}) \quad \mbox{ and } \quad\bm Z^0_i=(Z^0_{i1}, \dots, Z^{0}_{i(i-1)}, Z_{ii}^+, Z_{ii}^-, Z^{0}_{i(i+1)}, \dots,  Z^0_{i(\ell-1)}).$$ Then the vector 
        $$(Z^-_{\ell'}, Z^0_{1\ell'}, \dots, Z^0_{(\ell'-1)\ell'}, Z^+_{\ell'\ell'}, Z^-_{\ell'\ell'}, Z^0_{(\ell'+1)\ell'}, \dots, Z^0_{(\ell-1)\ell'}, Z^+_{\ell'})$$
        is equal to $(\bm Y^{\ell'}, \bm Y^{\ell-\ell'})$ for each $\ell'$ such that $1 \le \ell' \le \ell-1$.
    \end{enumerate}
\end{definition}

\subsubsection{Multiloop complex}
For $\kappa \geq 1$ and $\bm \gamma \in \Omega_\sigma^1(Q)$, there is a straightforward definition of a pseudo-gradient trajectory on $\Omega_\sigma^1(Q)$ given by
$$
\frac{\partial \bm \Gamma}{\partial s}=X(\gamma),
$$
where $X$ is a pseudo-gradient vector field for $\mathcal{A}_{\bm L}$ on $\Lambda^1_{[\kappa]}(Q)$ which exists as follows from Theorem~\ref{theorem-action+gradient}.

 To the critical multiloop $\bm \gamma$ as above we associate rank-$1$ abelian group $o_{\bm \gamma}$ generated by orientations of $W^u(\bm \gamma; \bm X)$, i.e.,
 \begin{equation}
     o_{\bm \gamma}=o_{\gamma^{c_1(\sigma)}}\otimes\dots\otimes o_{\gamma^{c_{ l(\sigma)}}},
 \end{equation}
 and as before, we set
\begin{equation}
    o_{\bm \gamma}^{\Q}=\Q \otimes_{\Z} o_{\bm \gamma}.
\end{equation}

The Morse cohomology, whose differential is defined by counting pseudo-gradient multi-trajectories between critical points as above, gives nothing but $ l(\sigma)$ copies of the Morse homology $HM_*(\Lambda^{1}(Q))$ of the loop space of $Q$. In order to mimic higher genus curves appearing in the Floer differential~\eqref{eq-diff} we deform the definition of the Morse differential so that, 
given a gradient trajectory $\Gamma(s)$ starting at $\bm \gamma$, whenever there exists $s$ such that $\Gamma(s)$ crosses the big diagonal of $\text{Sym}^\kappa(Q)$, additional \emph{Morse flow lines with switches} that bifurcate from the original one are added. This we do by means of the switching map that we previously introduced.

Specifically, given two multiloops $\bm \gamma \in \Omega^1_\sigma(Q)$ and $\bm \gamma' \in \Omega^1_{\sigma'}(Q)$ in $\mathcal{P}_\kappa^{\bm L}$, a perturbation data
$$\bm Y=(Y_-, Y_1, \dots, Y_{\ell-1}, Y_+) \in \mathfrak{X}_- \times (\mathfrak{X}_0)^{\ell-1} \times \mathfrak{X}_+,$$ a tuple $\bm h=(h_1, \ldots, h_\ell)$, a permutation of elements of $[\ell]$, which also encodes a component of $\operatorname{Conf}_\ell(\R)$, a tuple of transpositions $\bm \tau=(\tau_1=(i_1j_1), \ldots, \tau_\ell=(i_\ell j_\ell))$, and a collection $\bm \rho=(\rho_1, \ldots , \rho_\ell)$  of permutations in $\mathfrak{S}_\kappa$,  we define $\widetilde{\mathcal{M}}_\ell(\bm \gamma,\bm \gamma'; \bm Y, \bm h, \bm \tau, \bm \rho)$ as follows:

\begin{definition}
    \label{def-morse-diff}
    A \emph{Morse flow line with switchings} $\bm \Gamma\in\widetilde{\mathcal{M}}_\ell(\bm \gamma,\bm \gamma'; \bm Y, \bm h, \bm \tau,  \bm \rho)$ is a tuple
    \begin{equation*}
        \bm \Gamma=\left(\bm s, \bm \theta,(\Gamma_0,\ldots,\Gamma_\ell)\right), 
    \end{equation*}
    where $(\bm s, \bm \theta) \in \operatorname{Conf}_\ell(\R \times [0,1]) \subset \R^\ell \times [0,1]^\ell$ with $\bm s=(s_1, \ldots, s_{\ell})$ in the closure $\overline{\bm h}$ of $\bm h$, i.e. $s_{h_1}\ge\dots \ge s_{h_\ell}$. Set $\bm \lambda= (\lambda_0=0, \lambda_1=s_{h_1}-s_{h_2}, \ldots, \lambda_{\ell-1}=s_{h_{\ell-1}}-s_{h_\ell})$.
    We require that:
    \begin{enumerate}
        \item The maps $$\left\{\begin{array}{l} \Gamma_0 \colon (-\infty, 0]\to \Lambda^{1}_{[\kappa]}(Q);\\
        \Gamma_i \colon [0,\lambda_i]\to \Lambda^{1}_{[\kappa]}(Q) \quad \text{ for }\quad  i=1,\dots,\ell-1;\\
        \Gamma_\ell \colon [0, +\infty) \to \Lambda^{1}_{[\kappa]}(Q)
        \end{array}\right.$$
        are continuously differentiable.
        
        \item $\partial_s\Gamma_i(s)=(X+Y_i(\lambda_i, s))_{\Gamma_i(s)}$ for $\, i=0,\dots,\ell$ (by definition, for $i=0, \ell$, $Y_i(\lambda_i, s)$ is simply $Y_i(s)$).
        
        \item $\Gamma_0(+\infty)=\bm \gamma$, $\Gamma_\ell(-\infty)=\bm \gamma'$.

        \item \label{item: switch-condition} For $k=1, \dots, \ell$, $(\bm s, \bm \theta, \Gamma_{k-1}(\lambda_{k-1})) \in \Delta^{\tau_k, \, |}_{I_{h_k}}$  and
        $$c^|_{\rho_{k}}\circ sw^{i_kj_k}_{I_{h_k}}(\bm s, \bm \theta, \Gamma_{k-1}(\lambda_{k-1}))=(\bm s, \bm \theta, \Gamma_{k}(0)).$$
    \end{enumerate}
\end{definition}
We refer the reader to Figure~\ref{fig-example-of-MFLS} for an example of a portion of a Morse flow line with switchings.

\begin{figure}[h]
    \centering
    \includegraphics[width=11cm]{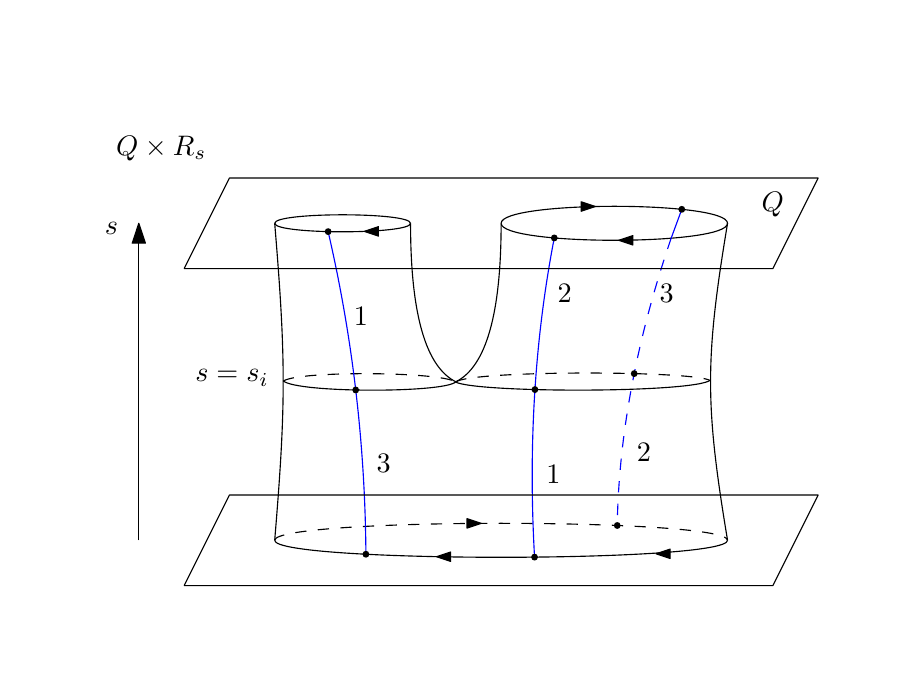}
    \caption{The image of $\bm \Gamma$ viewed as the graph in $Q\times\mathbb{R}_s$, where $\kappa=3$ and we have $\Gamma_{i-1}(\lambda_{i-1}) \in \Omega^1_{\sigma_{i-1}}(Q), \Gamma_{i}(0) \in \Omega^1_{\sigma_{i}}(Q)$  with $\sigma_{i-1}=\begin{pmatrix}
    1 & 2 & 3 \\
    1 & 3 & 2 
  \end{pmatrix}$, $\sigma_{i}=\begin{pmatrix}
    1 & 2 & 3 \\
    3 & 1 & 2 
  \end{pmatrix}$, $\rho_{i}=\begin{pmatrix}
    1 & 2 & 3 \\
    1 & 2 & 3 
  \end{pmatrix}$ and $\tau_i=(12)$.
    The blue arcs labeled by $j=1, \ldots, \kappa$ denote the image of $\Gamma(s)_j(0)$.}
    \label{fig-example-of-MFLS}
\end{figure}
More generally, given a background perturbation data $\bm Z_{\bm h, \bm \tau, \bm \rho} \in \mathfrak{X}^{\operatorname{back}}_{\operatorname{sw}}(\ell, \bm h, \bm \tau, \bm \rho)$ and a switching data $\bm Y^\ell_{\bm h, \bm \tau, \bm \rho} \in \mathfrak{X}_{\operatorname{sw}}(\ell, \bm h, \bm \tau, \bm \rho, \bm Z_{\bm h, \bm \tau, \bm \rho})$ we can similarly define the moduli space $\mathcal{M}_\ell(\bm \gamma,\bm \gamma'; \bm Y^\ell_{\bm h, \bm \tau, \bm \rho}, \bm \tau, \bm h, \bm \rho)$ by appropriately substituting $\bm Y$ with  $Y^\ell_{\bm h, \bm \tau, \bm \rho}$ in the definition above. We say that a universal perturbation data $\bm Y=(\bm Y^\ell)_{\ell\ge 1}$ is \emph{regular} if all the moduli spaces $\mathcal{M}_\ell(\bm \gamma,\bm \gamma'; \bm Y^\ell_{\bm h, \bm \tau, \bm \rho}, \bm \tau, \bm h, \bm \rho)$ are smooth.
\begin{proposition}
    \label{lemma-morse-transversality}
    For given $\bm \gamma, \bm \gamma' \in \mathcal{P}^{\bm L}_\kappa$, $\ell$, $\bm\tau$, $\bm \rho$ and $\bm h$:
    \begin{enumerate}
        \item  for generic $\bm Y \in \mathfrak{X}_- \times (\mathfrak{X}_0)^{\ell-1} \times \mathfrak{X}_+$ the space $\widetilde{\mathcal{M}}_\ell(\bm \gamma,\bm \gamma'; \bm Y, \bm h, \bm \tau,  \bm \rho)$ is a smooth manifold of dimension
$$\operatorname{dim}\widetilde{\mathcal{M}}_\ell(\bm \gamma,\bm \gamma'; \bm Y, \bm \tau, \bm h, \bm \rho)=\operatorname{ind}(\bm \gamma,\bm \gamma',\ell)=-(n-2)\ell+|\bm\gamma'|-|\bm\gamma|;$$
in particular, if $\operatorname{dim}\widetilde{\mathcal{M}}_\ell(\bm \gamma,\bm \gamma'; \bm Y, \bm h, \bm \tau,  \bm \rho)$ is negative, then $\widetilde{\mathcal{M}}_\ell(\bm \gamma,\bm \gamma'; \bm Y, \bm h, \bm \tau,  \bm \rho)=\varnothing$.
\item There exists a universal regular perturbation data $\bm Y=(\bm Y^\ell)_{\ell\ge 1}$.
\item  Fixing a choice of universal regular perturbation data and of orientations on unstable manifolds of $\bm \gamma$ and $\bm \gamma'$ provides a natural orientation of $\widetilde{\mathcal{M}}_\ell(\bm \gamma, \bm \gamma'; \bm Y^\ell_{\bm h, \bm \tau, \bm \rho}, \bm h, \bm \tau,  \bm \rho)$,  which depends on the orientation of $Q$.
    \end{enumerate}
\end{proposition}

This proposition is a consequence of the discussion that follows.

First, given two immersed $C^{m-1}$-submanifolds $N_1$ and $N_2$ of $\R^\ell \times [0,1]^ \ell \times \Lambda^{m,2}_{[\kappa]}(Q)$ we define the space of finite time trajectories connecting $N_1$ and $N_2$:
\begin{gather}
     \tilde W^{ft}(N_1, N_2; \mathfrak{X}_0)=\{(\bm s, \bm\theta, \lambda, \Gamma, Y,  x_1, x_2) \mid   (\bm s,\bm\theta)\in \R^\ell \times [0,1]^\ell,  \, \lambda>0,\\
    \Gamma \in C^1([0, \lambda], \Lambda^1_{[\kappa]}(Q)), Y\in \mathfrak{X}_0, (x_1,x_2) \in N_1 \times N_2,\nonumber \\ 
    \partial_s \Gamma-(X+Y_\lambda)=0, (\bm s, \bm\theta, \Gamma(0))=x_1, (\bm s, \bm\theta, \Gamma(\lambda))=x_2\}. \nonumber
\end{gather}

This space comes equipped with two evaluation maps, which we record as follows
\begin{gather} \label{eqn: tilde E zero}
    \tilde E_0=(\tilde E_0^s, \tilde E_0^t) \colon \tilde W^{ft}(N_1, N_2; \mathfrak{X}_0) \to (\R^\ell \times [0,1]^{\ell} \times \Lambda^1_{[\kappa]}(Q))^{2},\\
    (\bm s,\bm\theta, \lambda, \Gamma, Y) \mapsto ((\bm s, \bm\theta, \Gamma(0)), (\bm s, \bm\theta, \Gamma(\lambda)).\nonumber
\end{gather}
We write $\tilde W^{ft}(-, N_2; \mathfrak{X}_0)$ to denote the space of finite-time trajectories where the start point $(\bm s, \bm \theta, \Gamma(0))$ does not have to belong to $N_1$, and this is clearly a $C^{m-1}$ Banach manifold. Similarly, there is a $C^{m-1}$ Banach manifold $\tilde W^{ft}(N_1, -; \mathfrak{X}_0)$ of trajectories starting at $N_1$, and it is diffeomorphic to $[0, +\infty) \times N_1 \times \mathfrak{X}_0$ by uniqueness of solutions (assuming $N_1$ does not contain critical points). We notice that $$\tilde W^{ft}(N_1, N_2; \mathfrak{X}_0)=\tilde W^{ft}(N_1, -; \mathfrak{X}_0) \tensor*[_{\tilde{E}_0^t}]{\times}{} N_2,$$
and since $\tilde E_0^t$ is a submersion (away from small neighborhoods of critical points), we conclude that $\tilde W^{ft}(N_1, N_2; \mathfrak{X}_0)$ is a $C^{m-1}$-smooth Banach manifold.
For more details, we refer to \cite[Lemma B.5]{honda2025morse}.

We introduce the space $\tilde W^{ft}_{ij}(N_1, N_2; \mathfrak{X}_0)$ to be the subspace of $\tilde W^{ft}(N_1, N_2; \mathfrak{X}_0)$ consisting of tuples $(\bm s, \bm\theta, \lambda, \Gamma, Y,  x_1, x_2)$ with the additional condition
\begin{equation}
    s_i-s_j=\lambda.
\end{equation}
This space is a fiber product 
\begin{equation}
    \tilde W^{ft}(N_1, -; \mathfrak{X}_0)\tensor*[_{(\lambda, \tilde{E}_0^t)}]{\times}{_{(s_i-s_j, id)}} N_2
\end{equation} for maps 
\begin{gather}
    (\lambda, \tilde{E}_0^t) \colon \tilde W^{ft}(N_1, -; \mathfrak{X}_0) \to \R \times (\R^\ell \times [0,1]^\ell \times \Lambda^1_{[\kappa]}(Q)), \\(\bm s, \bm \theta, l, \Gamma, x_1) \mapsto (\lambda, \tilde{E}_0^t); \nonumber \\
(s_i-s_j, \tilde{E}_0^t) \colon N_2 \to \R \times (\R^\ell \times [0,1]^\ell \times \Lambda^1_{[\kappa]}(Q)), \\
\quad x_2 \mapsto (s_i(x_2)-s_j(x_2), x_2). \nonumber
\end{gather}
These two maps are also transversal, since $(\lambda, \tilde E_0^t)$ is a submersion.

For the data parametrized by $\operatorname{Conf}_\ell(\R)$ we also introduce the space
\begin{gather}\label{eq: finite-trajectory-with-parameters}
     \tilde W^{ft}(N_1, N_2; \mathfrak{X}^i_0(\ell, \bm h))=\{(\bm s, \bm\theta, \Gamma, Y,  x_1, x_2) \mid   (\bm s,\bm\theta)\in \operatorname{Conf}_\ell^{\bm h}(\R) \times [0,1]^\ell,  \\
    \Gamma \in C^1([0, s_{h_i}-s_{h_{i+1}}], \Lambda^1_{[\kappa]}(Q)), Y\in \mathfrak{X}^i_0(\ell, \bm h), (x_1,x_2) \in N_1 \times N_2,\nonumber \\ 
    \partial_s \Gamma-(X+Y_{\bm s, i}(s_{h_i}-s_{h_{i+1}}, \cdot))=0, (\bm s, \bm\theta, \Gamma(0))=x_1, \, (\bm s, \bm\theta, \Gamma(s_{h_i}-s_{h_{i+!}}))=x_2\}, \nonumber
\end{gather}
A similar argument to the one above shows that it is a $C^{m-1}$ smooth Banach manifold.

For $\bm \gamma \in \mathcal{P}^{\bm L}_\kappa$ we introduce unstable and stable manifolds of trajectories connecting $\bm \gamma$ to a given immersed $C^{m-1}$ submanifold $N$ of $\Lambda^{m,2}_{[\kappa]}(Q)$:
\begin{gather}
\tilde W^u(\bm \gamma, N; \mathfrak{X}_-) \coloneqq  \{(\bm s,\bm\theta, \Gamma, Y) \ \mid (\bm s, \bm \theta) \in \R^\ell \times [0,1]^{\ell}, \Gamma \in C^1((-\infty, 0], \Lambda^1_{[\kappa]}(Q)), \\ Y \in \mathfrak{X}_-, \, 
\partial_s\Gamma-(X+Y)=0, \, (\bm s,\bm \theta, \Gamma(0)) \in N\},    \nonumber \\
\tilde W^s(N, \bm \gamma; \mathfrak{X}_+) \coloneqq  \{(\bm s,\bm\theta, \Gamma, Y) \ \mid (\bm s, \bm \theta) \in \R^\ell \times [0,1]^{\ell}, \Gamma \in C^1([0,+\infty), \Lambda^1_{[\kappa]}(Q)), \\ Y \in \mathfrak{X}_+, \, 
\partial_s\Gamma-(X+Y)=0, \, (\bm s,\bm \theta, \Gamma(0)) \in N\}.    \nonumber
\end{gather}

These spaces are also $C^{m-1}$ Banach manifolds. We also point out that trajectories $\Gamma$ in all of the above spaces stay within $\Lambda^{m,2}_{[\kappa]}(Q)$, which follows from \cite[Theorem A.10]{honda2025morse} (the proof in the closed string case is slightly easier). Similarly, there are evaluation maps 
\begin{gather} \label{eqn: tilde E minus plus}
    \tilde E_- \colon \tilde W^{u}(\bm \gamma, N; \mathfrak{X}_-) \to \R^\ell \times [0,1]^{\ell} \times \Lambda^1_{[\kappa]}(Q),\\
    (\bm s,\bm\theta, \Gamma, Y) \mapsto (\bm s, \bm\theta, \Gamma(0));\nonumber \\
     \tilde E_+ \colon \tilde W^{s}(N, \bm \gamma; \mathfrak{X}_+) \to \R^\ell \times [0,1]^{\ell} \times \Lambda^1_{[\kappa]}(Q),\\
    (\bm s,\bm\theta, \Gamma, Y) \mapsto (\bm s, \bm\theta, \Gamma(0)).\nonumber
\end{gather}

As in~\eqref{eq: finite-trajectory-with-parameters}, we also have Banach manifolds $\tilde W^u(\bm \gamma, N; \mathfrak{X}_-(\ell, \bm h))$ and $\tilde W^s(N, \bm \gamma; \mathfrak{X}_+(\ell, \bm h))$ defined similarly.

We can now introduce the space of main interest as a fibered product of the spaces that we just introduced. We set
\begin{gather}\label{eq: defn-MFLS-parameters}
    \tilde W_\ell(\bm \gamma, \bm \gamma'; \mathfrak{X}_- \times  (\mathfrak{X}_0)^{\ell-1} \times \mathfrak{X}_+, \bm \tau, \bm h)\coloneqq\tilde W^u(\bm \gamma, \Delta_{I_{h_1}}^{i_1j_1, \, |};\mathfrak{X}_-)\tensor*[_{c^|_{\rho_1}\circ sw^{i_1j_1}_{I_{h_1}} \circ \tilde E_-}]{\times}{_{\tilde E_0^s}}\tilde W^{ft}_{h_1h_2}(\Delta_{I_{h_1}}^{i_1j_1, \, |},\Delta^{i_2j_2, \, |}_{I_{h_2}}; \mathfrak{X}_0 )\\
    \qquad\qquad 
    \tensor*[_{c^|_{\rho_2}\circ sw^{i_2j_2}_{I_{h_2}} \circ \tilde E_0^t}]{\times}{_{\tilde E_0^s}} 
    \cdots
    \tensor*[_{c^|_{\rho_{\ell-1}}\circ sw^{i_{\ell-1}j_{\ell-1}}_{I_{h_{\ell-1}}} \circ \tilde E_0^t}]{\times}{_{\tilde E_0^s}}\tilde W^{ft}_{h_{\ell-1}h_{\ell}}(\Delta_{I_{h_{\ell-1}}}^{i_{\ell-1}j_{\ell-1}, \, |},\Delta_{I_{h_\ell}}^{i_\ell j_\ell, \, |};\mathfrak{X}_0 ) \nonumber \\
    \tensor*[_{c^|_{\rho_\ell}\circ sw^{i_\ell j_\ell}_{I_{h_\ell}} \circ \tilde E_0^t}]{\times}{_{\tilde E_+}} \tilde W^s(\Delta_{I_{h_\ell}}^{i_\ell j_\ell, \, |},\bm \gamma'; \mathfrak{X}_+). \nonumber
\end{gather}

We note that the fiber above $\bm Y \in \mathfrak{X}_- \times  (\mathfrak{X}_0)^{\ell-1} \times \mathfrak{X}_+$ under the natural projection is $\mathcal{M}_\ell(\bm \gamma,\bm \gamma'; \bm Y, \bm h, \bm \tau,  \bm \rho)$. 
\begin{proof}[Proof of Proposition~\ref{lemma-morse-transversality}]
From the discussion above, it follows that each term in the above fiber product, when restricted to $\R^\ell \times [0,1]^\ell \times \Lambda^{m,2}_{[\kappa]}(Q)$, is a $C^{m-1}$ Banach manifold. Using \cite[Theorem A.10]{honda2025morse}, we deduce that all elements of the moduli space~\eqref{eq: defn-MFLS-parameters} are Morse flow lines consisting of trajectories contained in the product of $\R^\ell \times [0,1]^\ell$ with the space of smooth $\kappa$-multiloops. Now to show that~\eqref{eq: defn-MFLS-parameters} is a smooth Banach manifold, one can apply an inductive argument, since this space is obtained by taking consecutive fiber products. Explicitly, by setting
\begin{gather}\label{eq: defn-MFLS-parameters-k}
    \tilde W_k(\bm \gamma, \bm \gamma'; \bm h, \bm \tau,  \bm \rho)=\tilde W^u(\bm \gamma, \Delta_{I_{h_1}}^{i_1j_1, \, |};\mathfrak{X}_-)\tensor*[_{c^|_{\rho_1}\circ sw^{i_1j_1}_{I_{h_1}} \circ \tilde E_-}]{\times}{_{\tilde E_0^s}} \cdots \nonumber\\
    \tensor*[_{c^|_{\rho_{k}}\circ sw^{i_{k}j_{k}}_{I_{h_{k}}} \circ \tilde E_0^t}]{\times}{_{\tilde E_0^s}}\tilde W^{ft}_{h_{k}h_{k+1}}(\Delta_{I_{h_{k}}}^{i_{k}j_{k}, \, |},\Delta_{I_{h_{k+1}}}^{i_{k+1} j_{k+1}, \, |};\mathfrak{X}_0 ), \nonumber
\end{gather}
for $0 \le k< \ell$ we see that 
$$\tilde W_{k}(\bm \gamma, \bm \gamma'; \bm h, \bm \tau,  \bm \rho)=\tilde W_{k-1}(\bm \gamma, \bm \gamma'; \bm \tau, \bm h) \tensor*[_{c^|_{\rho_{k}}\circ sw^{i_{k}j_{k}}_{I_{h_k}} \circ \tilde E_0^t}]{\times}{_{\tilde E_0^s}} \tilde W^{ft}(-, -; \mathfrak{X}_0) \tensor*[_{\tilde E_0^t}]{\times}{} \Delta^{i_{k+1}j_{k+1}, \, |}_{I_{h_{k+1}}}.$$
The transversality of the $\tilde E_0^{t}$ and $\Delta^{i_{k+1}j_{k+1}, \, |}_{I_{h_{k+1}}}$, and of $c^|_{\rho_{k}}\circ sw^{i_{k}j_{k}}_{I_{s_{k}}} \circ \tilde E_0^t$ and $\tilde E_0^s$ is straightforward according to the choice of perturbation space $\mathfrak{X}_0$. We refer to the proof of \cite[Theorem B.6]{honda2025morse} for more details. Notice that $\tilde W_\ell(\bm \gamma, \bm \gamma'; \mathfrak{X}_- \times  (\mathfrak{X}_0)^{\ell-1} \times \mathfrak{X}_+, \bm \tau, \bm h)$ is obtained from $\tilde W_{\ell-1}(\bm \gamma, \bm \gamma'; \bm \tau, \bm h)$ via slightly different fiber product construction, but transversality can be shown via the analogous argument. The first statement of the Proposition~\ref{lemma-morse-transversality} now follows from the Sard-Smale theorem.

The proof of the second statement is inductive and analogous to that of the first statement. We omit the details here and refer the reader to \cite[Theorem C.4]{honda2025morse} and \cite[Lemma 5.1]{mescher2018} for similar arguments. 

Let us now focus on the third statement. In order to describe $\mathcal{M}_\ell(\bm \gamma,\bm \gamma'; \bm Y^\ell_{\bm h, \bm \tau, \bm \rho}, \bm h, \bm \tau,  \bm \rho)$ as above in terms of a fiber above $\bm Y^\ell_{\bm h, \bm \rho, \bm \tau} \in \mathfrak{X}_{\operatorname{sw}}(\ell, \bm h, \bm \tau, \bm \rho)$, one simply replaces $\tilde W_{k}(\bm \gamma, \bm \gamma'; \bm h, \bm \tau,  \bm \rho)$ with $\tilde W_{k}'(\bm \gamma, \bm \gamma'; \bm h, \bm \tau,  \bm \rho)$ where instead of $\mathfrak{X}_{0}$ and $\mathfrak{X}_{\pm}$ one uses spaces $\mathfrak{X}^i_0(\bm \ell, \bm h)$ and $\mathfrak{X}_{\pm}(\bm \ell, \bm h)$ instead.

Now, for a fixed choice of perturbation data $\bm Y^\ell_{\bm h, \bm \rho, \bm \tau}=(Y_-, Y_1, \ldots, Y_{\ell-1}, Y_+)$ as above we denote by $\mathcal{M}_k(\bm \gamma, \bm \gamma'; \bm h, \bm \tau,  \bm \rho)$ the fiber of $\tilde W_{k}'(\bm \gamma, \bm \gamma'; \bm h, \bm \tau,  \bm \rho)$ above (the appropriate truncation of) $\bm Y^\ell_{\bm h. \bm \rho, \bm \tau}$, which is a smooth manifold. Observe that each $\Delta^{i_kj_k. \, |}_{h_k}$ is a submanifold of $(\R \times [0,1])^\ell \times \Lambda^{m,2}_{[\kappa]}(Q)$ whose coorientation at each point is given by $|Q| \cong \zeta^n$, since it is a preimage of the diagonal $\Delta_Q \subset Q \times Q$ under the evaluation map as in~\eqref{eq: evaluation-map}. Then, by~\eqref{eq: conormal-orientation-identity} we have
\begin{equation}\label{eq: orinetation-identity-prop1}
    \zeta^n \otimes |\mathcal{M}_0(\bm \gamma, \bm \gamma'; \bm h, \bm \tau,  \bm \rho)| \cong |\R \times [0,1]|^\ell \otimes o_{\bm \gamma}^{-1} \cong \zeta^{2\ell} \otimes o_{\bm \gamma}^{-1},
\end{equation}
since $\mathcal{M}_1(\bm \gamma, \bm \gamma'; \bm \tau, \bm h)$ is a fiber product of $(\R \times [0,1])^\ell \times  W^u(\bm \gamma; Y_-)$ and $\Delta^{i_1j_1, \, |}_{I_{h_1}}$. Similarly, since $$\mathcal{M}_{k}(\bm \gamma, \bm \gamma'; \bm \tau, \bm h)=\tilde{W}^{ft}(\mathcal{M}_{k-1}(\bm \gamma, \bm \gamma'; \bm \tau, \bm h), \Delta^{i_{k+1}j_{k+1}, \, |}_{I_{h_{k+1}}}; Y_{k}),$$
where $\mathcal{M}_{k-1}(\bm \gamma, \bm \gamma'; \bm \tau, \bm h)$ is immersed via $c^|_{\rho_{k}}\circ sw^{i_{k}j_{k}}_{I_{h_k}}\circ \tilde{E}_0^t$. Hence, $\mathcal{M}_{k}(\bm \gamma, \bm \gamma'; \bm \tau, \bm h)$ is a fiber product of $\R \times\mathcal{M}_{k-1}(\bm \gamma, \bm \gamma'; \bm \tau, \bm h)$ and $\Delta^{i_{k+1}j_{k+1}, \, |}_{I_{h_{k+1}}}$ over $\R \times (\R \times [0,1])^\ell \times \Lambda^{m,2}_{[\kappa]}(Q)$, implying
\begin{equation}\label{eq: orinetation-identity-prop2}
    |\R| \otimes \zeta^n \otimes |\mathcal{M}_{k}(\bm \gamma, \bm \gamma'; \bm h, \bm \tau, \bm \rho)| \cong |\R| \otimes |\mathcal{M}_{k-1}(\bm \gamma, \bm \gamma'; \bm h, \bm \tau, \bm \rho)|.
\end{equation}

Finally, one shows via a similar argument
\begin{equation}\label{eq: orinetation-identity-prop3}
    o_{\bm \gamma'}^{-1} \otimes |\widetilde{\mathcal{M}}_{\ell}(\bm \gamma, \bm \gamma'; \bm Y^\ell_{\bm h, \bm \tau, \bm \rho},\bm h, \bm \tau,  \bm \rho)| \cong |\mathcal{M}_{\ell-1}(\bm \gamma, \bm \gamma'; \bm h, \bm \tau,  \bm \rho )|.
\end{equation}

Combining~\eqref{eq: orinetation-identity-prop1}-\eqref{eq: orinetation-identity-prop3} we obtain
\begin{equation}\label{eq: orientation-identity-MFLS}
    |\widetilde{\mathcal{M}}_{\ell}(\bm \gamma, \bm \gamma'; \bm Y, \bm h, \bm \tau,  \bm \rho)| \otimes o_{\bm \gamma} \cong o_{\bm \gamma'}[(n-2) \ell].
\end{equation}
\end{proof}

From now on we fix a choice of universal regular data $\bm Y=(\bm Y^\ell)_{\ell \ge 1}$ and denote the moduli space $\widetilde{\mathcal{M}}_\ell(\bm \gamma, \bm \gamma'; \bm Y^\ell_{\bm h, \bm \tau, \bm \rho}, \bm h, \bm \tau,  \bm \rho)$ simply as $\widetilde{\mathcal{M}}_\ell(\bm \gamma, \bm \gamma'; \bm Y, \bm h, \bm \tau,  \bm \rho)$ slightly abusing the notation.
Note that there is vertical translation $\mathbb{R}$-action on $\widetilde{\mathcal{M}}_{\ell}(\bm \gamma, \bm \gamma'; \bm Y, \bm h, \bm \tau,  \bm \rho)$ by the simultaneous translation of elements of $\bm s$ as in Subsection~\ref{section: switching map}. 
We denote the quotient by $\mathcal{M}_{\ell}(\bm \gamma, \bm \gamma'; \bm Y, \bm h, \bm \tau,  \bm \rho)$. 
Given $\Gamma\in \widetilde{\mathcal{M}}_{\ell}(\bm \gamma, \bm \gamma'; \bm Y, \bm h, \bm \tau,  \bm \rho)$, we denote its class in $\mathcal{M}_{\ell}(\bm \gamma, \bm \gamma'; \bm Y, \bm h, \bm \tau,  \bm \rho)$ by $[\Gamma]$. Now observe that if $\operatorname{ind}(\bm \gamma, \bm \gamma', \ell)=1$ the isomorphism~\eqref{eq: orientation-identity-MFLS} provides a natural morphism for each $[\bm \Gamma] \in \mathcal{M}_{\ell}(\bm \gamma, \bm \gamma'; \bm Y, \bm h, \bm \tau,  \bm \rho)$
\begin{equation}\label{eq: orientations-map-morse}
    \partial_{[\bm \Gamma]} \colon o_{\bm \gamma}^{\Q} \to o_{\bm \gamma'}^{\Q}[(n-2)\ell],
\end{equation}
since $|\widetilde{\mathcal{M}}_{\ell}(\bm \gamma, \bm \gamma'; \bm Y, \bm h, \bm \tau,  \bm \rho)|=|\R_s|$, i.e., the orientation trivializes and is given by the vertical $s$-coordinate, and we choose trivialization given by $\partial_s$, the \emph{positive pseudo-gradient direction}.

\begin{definition} \label{defn: free multiloop complex}
Let $CM_{0,-*}(\Lambda^1_{[\kappa]}(Q))$ be the $\Q$-module generated by the orientation lines associated with elements of $\mathcal{P}^{\bm L}$. Then the {\em free multiloop complex $CM_{-*}(\Lambda^1_{[\kappa]}(Q))$ of $Q$} is a cochain complex whose underlying module is $CM_{0,-*}(\Lambda^1_{[\kappa]}(Q))\llbracket \hbar\rrbracket$ when $n=2$ and $CM_{0,-*}(\Lambda^1_{[\kappa]}(Q)) \otimes_{\Q}\Q[\hbar]$ when $n>2$, where grading of $\hbar$ is $|\hbar|=2-n$.  The differential is given by
\begin{gather}
\label{eq-morse-diff}
    \partial\colon CM_{-*}(\Lambda^1_{[\kappa]} (Q))\to CM_{-*+1}(\Lambda^1_{[\kappa]} (Q)),\\
    \partial([\bm \gamma]\hbar^j)=\sum_{\substack{\ell \geq0, |\bm \gamma'|=|\bm \gamma|-(n-2)\ell-1, \\ [\bm \Gamma] \in \mathcal{P}(\bm \gamma,\bm \gamma';\ell)}}
    \frac{\partial_{[\bm \Gamma]}([\bm \gamma]\hbar^j)}{\ell!\cdot (\kappa!)^\ell}=\sum \partial_{\ell}([\gamma]\hbar^j),\nonumber
\end{gather}
where $[\bm \gamma]\hbar^j$ is some generator of $o_{\bm \gamma}^{\Q}[(2-n)j]$ for $j \ge 0$. 
\end{definition}
Notice that above we extend the map~\eqref{eq: orientations-map-morse} to any generator $[\bm \gamma]\hbar^{j}$ of $CM_{-*}\bigl(\Lambda^1_{[\kappa]}(Q)\bigr)$ by tensoring with $\zeta^{(2-n)j}$ on the left:
\begin{equation}\label{orientation-map-Morse}
    \partial_{[\bm \Gamma]} \colon   \zeta^{(2-n)j} \otimes o_{\bm \gamma}  \to   \zeta^{(n-2)j} \otimes o_{\bm \gamma'}\bigl[(n-2)\ell\bigr]  \cong o_{\bm \gamma'}\bigl[(n-2)(\ell+j)\bigr] .
\end{equation}

We point out that the sum in~\eqref{eq-morse-diff} is indeed finite for $n>2$ by Proposition~\ref{lemma-morse-transversality}. We devote the following section to showing that this is indeed a cochain complex.

 \subsection{Compactness}
\label{section-morse-compactness}
We introduce the notation $\widetilde{\mathcal{M}}_{\ell}(\bm \gamma, \bm \gamma', \Delta^{\tau_i, \tau_{i+1}, \, |}_{I_{s_i}, I_{s_{i+1}}}; \bm Y, \bm h, \bm \tau, \bm \rho)$ for the space given by the fiber product construction similar to the one in~\eqref{eq: defn-MFLS-parameters} with the only difference that we replace the portion
\begin{equation*}
    \tilde W^{ft}(\Delta_{I_{h_{i-1}}}^{\tau_{i-1}, \, |},\Delta^{\tau_i, \, |}_{I_{h_i}}; Y_{i-1}) \tensor*[_{c^|_{\rho_i}\circ sw^{\tau_i}_{I_{h_i}} \circ \tilde E_-}]{\times}{_{\tilde E_0^s}}\tilde W^{ft}(\Delta^{\tau_i, \, |}_{I_{h_i}}, \Delta_{I_{h_{i+1}}}^{\tau_{i+1}, \, |}; Y_{i})\tensor*[_{c^|_{\rho_{i+1}}\circ sw^{\tau_{i+1}}_{I_{h_{i+1}}} \circ \tilde E_-}]{\times}{_{\tilde E_0^s}}\tilde W^{ft}(\Delta^{\tau_{i+1}, \, |}_{I_{h_{i+1}}}, \Delta_{I_{h_{i+2}}}^{\tau_{i+2},\, |}; Y_{i+1})
\end{equation*}
with
\begin{equation*}
    \tilde W^{ft}(\Delta_{I_{h_{i-1}}}^{\tau_{i-1}, \, |},\Delta^{\tau_i, \tau_{i+1}, \, |}_{I_{s_i}, I_{s_{i+1}}}; Y_{i-1})\tensor*[_{c^|_{\rho_{i+1}}\circ sw^{\tau_{i+1}}_{I_{h_{i+1}}}\circ c^|_{\rho_i}\circ sw^{\tau_i}_{I_{h_i}} \circ \tilde E_-}]{\times}{_{\tilde E_0^s}}\tilde W^{ft}(\Delta^{\tau_i, \tau_{i+1}, \, |}_{I_{s_i}, I_{s_{i+1}}}, \Delta_{I_{h_{i+2}}}^{\tau_{i+2},\, |}; Y_{i+1}).
\end{equation*}

We claim that we may assume that for our choice of data $\bm Y$ any such space $$\widetilde{\mathcal{M}}_{\ell}(\bm \gamma, \bm \gamma', \Delta^{\tau_i, \tau_{i+1}, \, |}_{I_{s_i}, I_{s_{i+1}}}; \bm Y, \bm h, \bm \tau, \bm \rho)$$ is also a smooth manifold by an argument analogous to the one presented in the proof of Proposition~\ref{lemma-morse-transversality}, see also \cite[Lemma C.3]{honda2025morse}. Moreover, we can assume that the same holds for any moduli space obtained via similar replacement of a fiber product as above for a consecutive tuple $\tau_i, \ldots, \tau_{i+j}$ by any element of the stratification of $\overline{\Delta}^{\tau_i, \ldots, \tau_{i+j}, \, |}_A$ discussed in Section~\ref{section: switching map} with $A=\{h_i, \ldots, h_{i+j}\}$. 

\begin{lemma}\label{thm: compactness-for-MFLS} 
Let $\bm Y=(\bm Y^\ell)_{\ell \ge 1}$ be a choice of universal switching data.
\begin{enumerate}
    \item If $\op{ind}(\bm \gamma')-\op{ind}(\bm \gamma)-(n-2)\ell-1=0$, then $\widetilde{\mathcal{M}}_{\ell}(\bm \gamma, \bm \gamma'; \bm Y, \bm h, \bm \tau, \bm \rho)$ consists of a finite number of points.
    \item If $\op{ind}(\bm \gamma')-\op{ind}(\bm \gamma)-(n-2)\ell-1=1$, then $\widetilde{\mathcal{M}}_{\ell}(\bm \gamma, \bm \gamma'; \bm Y, \bm h, \bm \tau, \bm \rho)$ admits a compactification $\overline{\mathcal{M}}_{\ell}(\bm \gamma, \bm \gamma'; \bm Y, \bm h, \bm \tau, \bm \rho)$ whose boundary is covered by the following strata:
    \begin{gather}
        \widetilde{\mathcal{M}}_{\ell'}(\bm \gamma, \bm \gamma''; \bm Y, \bm h_{\ell'}^1, \bm \tau^1_{\ell'}, \bm \rho^1_{\ell'}) \times \widetilde{\mathcal{M}}_{\ell-\ell'}(\bm \gamma'', \bm \gamma'; \bm Y, \bm h_{\ell'}^2, \bm \tau^2_{\ell'}, \bm \rho^2_{\ell'}), \quad 0 \le \ell' \le \ell,\label{eq-compactness-breaking} \\
        \widetilde{\mathcal{M}}_{\ell}(\bm \gamma, \bm \gamma', \Delta^{\tau_i, \tau_{i+1}}_{I_{s_i}, I_{s_{i+1}}}; \bm Y,  \bm h, \bm \tau, \bm \rho), \label{eq-compactness-two-switches} \\
        \widetilde{\mathcal{M}}_{\ell}(\bm \gamma, \bm \gamma'; \bm Y, \bm h, \bm \tau, \bm \rho)|_{\theta_{\ell'}=0}, \quad \widetilde{\mathcal{M}}_{\ell}(\bm \gamma, \bm \gamma'; \bm Y, \bm h, \bm \tau, \bm \rho)|_{\theta_{\ell'}=1},  \, 1 \le \ell' \le \ell \label{eq-compactness-circle}
    \end{gather}
    where $(\bm \tau_{\ell'}^1, \bm \tau^2_{\ell'})=\bm \tau, (\bm \rho_{\ell'}^1, \bm \rho^2_{\ell'})=\bm \rho$ with $|\bm \tau_1|=|\bm \rho^1_{\ell'}|=\ell'$; $(\bm h^1_{\ell'}, \bm h^2_{\ell'})$ is obtained from $\bm h$ as in~\eqref{eq: conf-boundary-decomposition}; and $\bm \gamma''$ is a critical point such that 
    $$\op{ind}(\bm \gamma)-\op{ind}(\bm \gamma'') -(n-2)\ell'-1=\op{ind}(\bm \gamma'')-\op{ind}(\bm \gamma') -(n-2)(\ell-\ell')-1=0.$$
    \end{enumerate}
\end{lemma}
\begin{proof}[Sketch]
    The proof of this theorem is analogous to that of \cite[Theorem C.5]{honda2025morse}. The only new different case showing up in the free multiloop setting is given by~\eqref{eq-compactness-circle}, which corresponds to one of the switching coordinates $\theta_\ell'$ approaching the boundary of $[0,1]$, which is allowed in this case.
\end{proof}

\begin{corollary} \label{lemma-morse-cpt-moduli}
   Given $\bm \gamma, \bm \gamma' \in \mathcal{P}^{\bm L}_\kappa$ such that $|\bm \gamma|-|\bm \gamma'|-(n-2)\ell =2$, the space $$\bigcup_{\bm h \in \mathfrak{S}_\ell, \bm \tau \in \binom{[\kappa]}{2}^\ell, \bm \rho \in \mathfrak{S}_\kappa^\ell} \widetilde{\mathcal{M}}_{\ell}(\bm \gamma, \bm \gamma'; \bm Y, \bm h, \bm \tau, \bm \rho)$$ admits a compactification $\overline{\mathcal{M}}_\ell(\bm \gamma, \bm \gamma'; \bm Y)$ with boundary covered only by components of the form~\eqref{eq-compactness-breaking}.
\end{corollary}
\begin{proof}[Sketch of the proof]
 By Lemma~\ref{thm: compactness-for-MFLS} (2), we need to show that components~\eqref{eq-compactness-two-switches} and~\eqref{eq-compactness-circle} can be identified for different choices of tuples $(\bm h, \bm \tau, \bm \rho)$, and therefore such points may be regarded as interior points of 
 $$\bigcup_{\bm h \in \mathfrak{S}_\ell, \bm \tau \in \binom{[\kappa]}{2}^\ell, \bm \rho \in \mathfrak{S}_\kappa^\ell} \overline{\mathcal{M}}_{\ell}(\bm \gamma, \bm \gamma'; \bm Y, \bm h, \bm \tau, \bm \rho).$$
 First, the above holds for~\eqref{eq-compactness-two-switches} since any point of $\widetilde{\mathcal{M}}_{\ell}(\bm \gamma, \bm \gamma', \Delta^{\tau_i, \tau_{i+1}}_{I_{s_{h_i}}, I_{s_{h_{i+1}}}}; \bm Y,  \bm h, \bm \tau, \bm \rho)$ is also a point of $\widetilde{\mathcal{M}}_{\ell}(\bm \gamma, \bm \gamma', \Delta^{\tau^{i,j}_i, \tau^{i,j}_{i+1}}_{I_{s_{h_i}}, I_{s_{h_{i+1}}}}; \bm Y,  \bm h^i, \bm \tau^{i,j}, \bm \rho)$ for $j=1$ or $j=2$ depending on whether $\theta_{h_i}$ is greater or smaller than $\theta_{h_{i+1}}$. Here, we use the notation of~\eqref{eq-perturbation-switching} and the commutativity of the perturbation data $\bm Y$. We illustrate this in Figure~\ref{fig-degenerate-s}.

 Second, any point of  $\widetilde{\mathcal{M}}_{\ell}(\bm \gamma, \bm \gamma'; \bm Y, \bm h, \bm \tau, \bm \rho)|_{\theta_{\ell'}=1}$ is also a point of $\widetilde{\mathcal{M}}_{\ell}(\bm \gamma, \bm \gamma'; \bm Y, \bm h, \bm \tau_{\ell'}(\sigma), \bm \rho_{\ell'}(\sigma))|_{\theta_{\ell'}=0}$ by cyclicity of $\bm Y$, where $\sigma=\sigma(\bm \gamma)$ and we use the notation of~\eqref{eq:cyclic-data}. Hence, the claim of the corollary follows.
 \end{proof}
    \begin{figure}
       \centering
       \includegraphics[width=17cm]{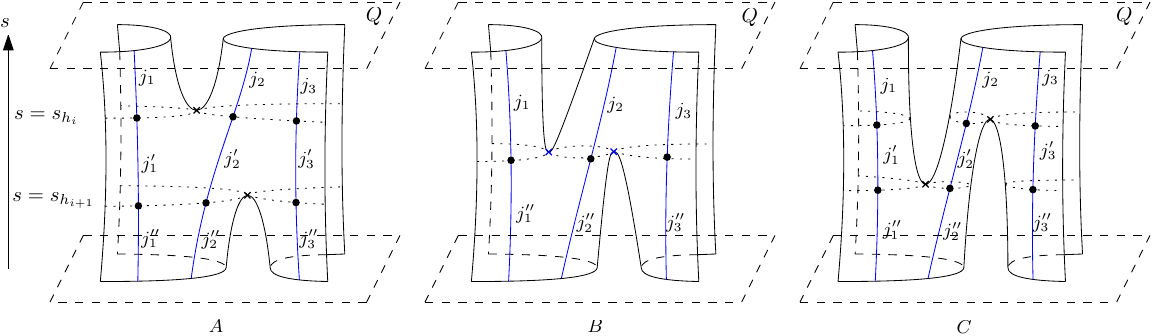}
       \caption{The image of $\bm \Gamma$ viewed as the graph in $Q\times\mathbb{R}_s$. The blue arc labelled by $j$ denotes the image of $\Gamma_i(s)_j(0)$. The left side (A) illustrates an MFLS corresponding to $(s_{h_i}-s_{h_{i+1}})\to 0^+$; the limit is shown by (B), where two switchings occur simultaneously; the curves of family (A) have the same limiting curves in family (C) corresponding to $(s_{h^i_i}-s_{h^i_{i+1}}) \to 0^+$. Here $\tau_i=(j_1j_2)$ and $\tau_{i+1}=(j_2'j_3')$.
       }
       \label{fig-degenerate-s}
    \end{figure}

 As an immediate corollary of Lemma~\ref{lemma-morse-cpt-moduli} we obtain
\begin{proposition}
\label{proposition-d2}
 The graded module $CM_{-*}(\Lambda^1_{[\kappa]}(Q))$ equipped with differential $\partial$ is indeed a cochain complex.
\end{proposition}
\begin{proof}
Note that conditions \ref{PG1}-\ref{PG5} imply that $\mathcal{P}^{\bm L}_\kappa \cap \{\mathcal{A}_{\bm L} \le a \}$ is finite for any $a \in \R$. Hence by Lemma~\ref{thm: compactness-for-MFLS}(1) the homomorphism $\partial$ given by \eqref{eq-morse-diff} is well-defined. To show that $CM_{-*}(\Lambda^1_{[\kappa]}(Q))$ is indeed a cochain complex it only remains to show that $\partial^2=0$.

    Given $|\bm \gamma|-|\bm \gamma'|-(n-2)\ell=2$, by Corollary~\ref{lemma-morse-cpt-moduli} the boundary of the compactification $\overline{\mathcal{M}}_{\ell}(\bm \gamma,\bm \gamma';\bm Y)$ may be regarded as a finite number of broken Morse flow lines with switches, i.e.,
    \begin{equation}\label{eq: boundary-strata-MFLS}
        \bigsqcup_{\substack{0 \le \ell' \le \ell, \\ |\bm \gamma|-|\bm \gamma''|=|\bm \gamma''|-|\bm \gamma'|=(n-2)\ell+1}}\widetilde{\mathcal{M}}_{\ell'}(\bm \gamma, \bm \gamma''; \bm Y, \bm h_{\ell'}^1, \bm \tau^1_{\ell'}, \bm \rho^1_{\ell'}) \times \widetilde{\mathcal{M}}_{\ell-\ell'}(\bm \gamma'', \bm \gamma'; \bm Y, \bm h_{\ell'}^2, \bm \tau^2_{\ell'}, \bm \rho^2_{\ell'}).
    \end{equation}
    Notice that given pair $(\bm h^1, \bm h^2) \in \mathfrak{S}_{\ell'} \times \mathfrak{S}_{\ell-\ell'
    }$ it corresponds to $\binom{\ell}{\ell'}$ elements $\bm h \in \mathfrak{S}_\ell$ under the restriction map (see description after~\eqref{eq: conf-boundary-decomposition}). Therefore, each broken trajectory in~\eqref{eq: boundary-strata-MFLS} contributes to the boundary of $\overline{\mathcal{M}}_{\ell}(\bm \gamma,\bm \gamma';\bm Y)$ the $\ell!$ times more than it contributes to $\partial_{\ell-\ell'} \circ \partial_{\ell'}$.
    
    To finish the proof, it is enough to compare the composition map $\partial_{[\bm \Gamma_2]} \circ \partial_{[\bm \Gamma_1]}$ and the map induced by boundary orientation for 
    \begin{equation*}
        (\bm \Gamma_1, \bm \Gamma_2) \in \widetilde{\mathcal{M}}_{\ell'}(\bm \gamma, \bm \gamma''; \bm Y, \bm h_{\ell'}^1, \bm \tau^1_{\ell'}, \bm \rho^1_{\ell'}) \times \widetilde{\mathcal{M}}_{\ell-\ell'}(\bm \gamma'', \bm \gamma'; \bm Y, \bm h_{\ell'}^2, \bm \tau^2_{\ell'}, \bm \rho^2_{\ell'}).
    \end{equation*}
    That is,
    \begin{equation}
        |\R \partial_s|^{-1} \otimes\big|\widetilde{\mathcal{M}}_{\ell}(\bm \gamma,\bm \gamma';\bm Y, \bm h, \bm \tau, \bm \rho)\big|^{-1} \otimes o_{\bm \gamma} \cong o_{\bm \gamma'}\bigl[(2-n)\ell\bigr],
    \end{equation}
    \begin{equation}
        |\R \partial_s \bm \Gamma_2|^{-1} \otimes |\R \partial_s \bm \Gamma_1|^{-1} \otimes o_{\bm \gamma} \cong |\R \partial_s \bm \Gamma_2|^{-1} \otimes o_{\bm \gamma''} \bigl[(2-n)\ell' \bigr] \cong o_{\bm \gamma'} \bigl[(2-n)\ell \bigr].
\end{equation} 
Now it remains to notice that $\partial_s \bm \Gamma_1$ with the preferred orientation as in~\eqref{eq: orientations-map-morse} is an outward pointing vector. Therefore, $\partial|_{o_{\bm \gamma}}^2=0$. We leave it to the reader to verify that $\partial|_{o_{\bm \gamma}[(2-n)j]}^2=0$ for any $j>0$.
\end{proof}

We write $CM_*(\Lambda^1_{[\kappa]} (Q);\bm L, \bm X, \bm Y)$ and $HM_*(\Lambda^1_{[\kappa]} (Q);\bm L,\bm X, \bm Y)$ to indicate the specific choice of $\bm L=\{L_1,\ldots,L_\kappa\}$ and pseudo-gradient vector fields $\bm X$ and perturbation data $\bm Y$.
\begin{proposition}
    The homology groups $HM_*(\Lambda^1_{[\kappa]} (Q); \bm L, \bm X, \bm Y)$ do not depend on the choice of $(\bm L, X, \bm Y)$.
\end{proposition}
\begin{proof}[Sketch]
According to \cite{abbondandolo2010floer}, changing $\bm L$ or pseudo-gradient fields $\bm X$ induces chain homotopy at the level of standard Morse chain complexes. Hence, if we were considering only pseudo-gradient trajectories with $\ell=0$, the statement would immediately follow. We then claim that showing that changing one of $\bm L$, $\bm X$, or $\bm Y$ alone induces a chain homotopy in the presence of switchings is similar to the above-mentioned line of the argument.
\end{proof}

\section{The isomorphism between the Heegaard Floer symplectic cohomology and the homology of the free multiloop complex}

In this section, we define a chain map from the Floer complex to the Morse complex
\begin{equation*}
    \mathcal{F}\colon SC^*_{\kappa, \mathrm{unsym}}(T^*Q)\to CM_{\kappa n-*}(\Lambda^1_{[\kappa]} Q),
\end{equation*}
by counting rigid elements in a mixed moduli space that combines holomorphic curves and piecewise gradient trajectories. We establish an analogue of the Viterbo isomorphism theorem in the context of Heegaard Floer symplectic cohomology by showing that $\mathcal{F}$ is a quasi-isomorphism. 

In this section, we assume that $Q$ is an orientable closed manifold  with vanishing second \emph{Stiefel-Whitney class} $$w_2(TQ)~\in~H^2(T^*Q; \mathbb{Z}/2).$$
This allows one to fix a spin structure on $Q$ regarded as a Lagrangian submanifold in $T^*Q$. We also note that the approach in this section is parallel to \cite[Section 3]{honda2025morse}.

We point out that we work here with the unsymmetrized version of HFSH introduced in Section~\ref{section-hamiltonian-unsym}, and sometimes we abuse notations and do not specify that we work in this context, as it is understood throughout this section.
\subsection{The half-cylinder Hurwitz space}
We start by introducing the space of branched covers of a half-cylinder in the context of Section~\ref{section-hamiltonian-unsym}.
\begin{definition}
    Let $\sigma, \sigma_0 \in \mathfrak{S}_\kappa$. For given $ \chi>0$ define $\mathcal{H}^{\sigma, \sigma_0, \ge 0}_{\kappa, \chi}$ to be the set of tuples
    $$\left( \pi \colon S \to \R_{\ge 0}\times S^1, \bm p^+, \bm a^+, \bm P^0, \bm a^0, B\right) \text{, where}$$
    \begin{enumerate}
        \item $S$ is a Riemann surface of Euler characteristic $\chi$ with positive punctures $\bm p^+=(p^+_1, \ldots, p^+_{l(\sigma)})$ and boundary components $\bm P^0=(P^0_1, \ldots, P^0_{\ell(\sigma_0)})$. Denote by $\overline{S}$ the compactification of $S$ such that $S= \overline{S} \setminus \bm p^+$.
        
        \item  $\bm a^+= (a_1^+, \ldots, a_\kappa^+)$ is a collection of positive asymptotic markers at positive punctures $p_i^+$  such that for each $p_i^+$ there are exactly $\mu_i(\sigma)$ asymptotic markers 
        $a_{i_1}^+, \ldots, a_{i_{\mu_i}}^+ \in \mathcal{S}^1_{p_i^+}S = (T_{p_i^+}S \setminus \{0\})/\R_{>0}$; these are exactly the indices permuted by $c_i(\sigma)$, and they appear on $\mathcal{S}^1_{p_i^+}$ in the order prescribed by the cycle $c_i(\sigma)$.
        
        Analogously, we choose marked points, which we simply call \emph{markers}, on $\bm P^0$ such that for each boundary component $P^0_i$ there are exactly $\mu_i(\sigma_0)$ markers.

         The permutation on the set of markers induced by walking along $P_i^0$ in the boundary orientation direction is required to be equal to $c_i(\sigma_0)$.
        
        \item $\pi \colon S \longrightarrow \R_{\ge 0} \times S^1$ is a holomorphic branched cover with simple branched points away from $+\infty$ and boundary $\partial S=\bm P^0$. At $+\infty$ the branching profile is given by $\bm \mu(\sigma)$ and for each boundary component $P^0_i$ the restriction $\pi|_{P^0_i}$ is the $\mu_i^0$-cover of $\{0\} \times S^1$. 
       
        \item The map induced by $\pi$ from $\mathcal{S}^1_{p_i^+}S$ to $\mathcal{S}^1_{+ \infty} (\R_{\ge 0} \times S^1)$ sends all asymptotic markers assigned to $p_i^+$ to the asymptotic marker $1 \in \mathcal{S}^1_{+ \infty} (\R_{\ge 0} \times S^1)$. The markers $a_{i}^0$'s are all sent to $(0, 0) \in \R_{\ge 0} \times S^1$ under $\pi$.
        \item $B=(q_1,\ldots,q_b)$, is an ordered collection  of points in $\R_{> 0} \times S^1$ of simple branching, where $b= - \chi$.
    \end{enumerate}
\end{definition}
\begin{figure}
    \centering
    \includegraphics[width=8cm]{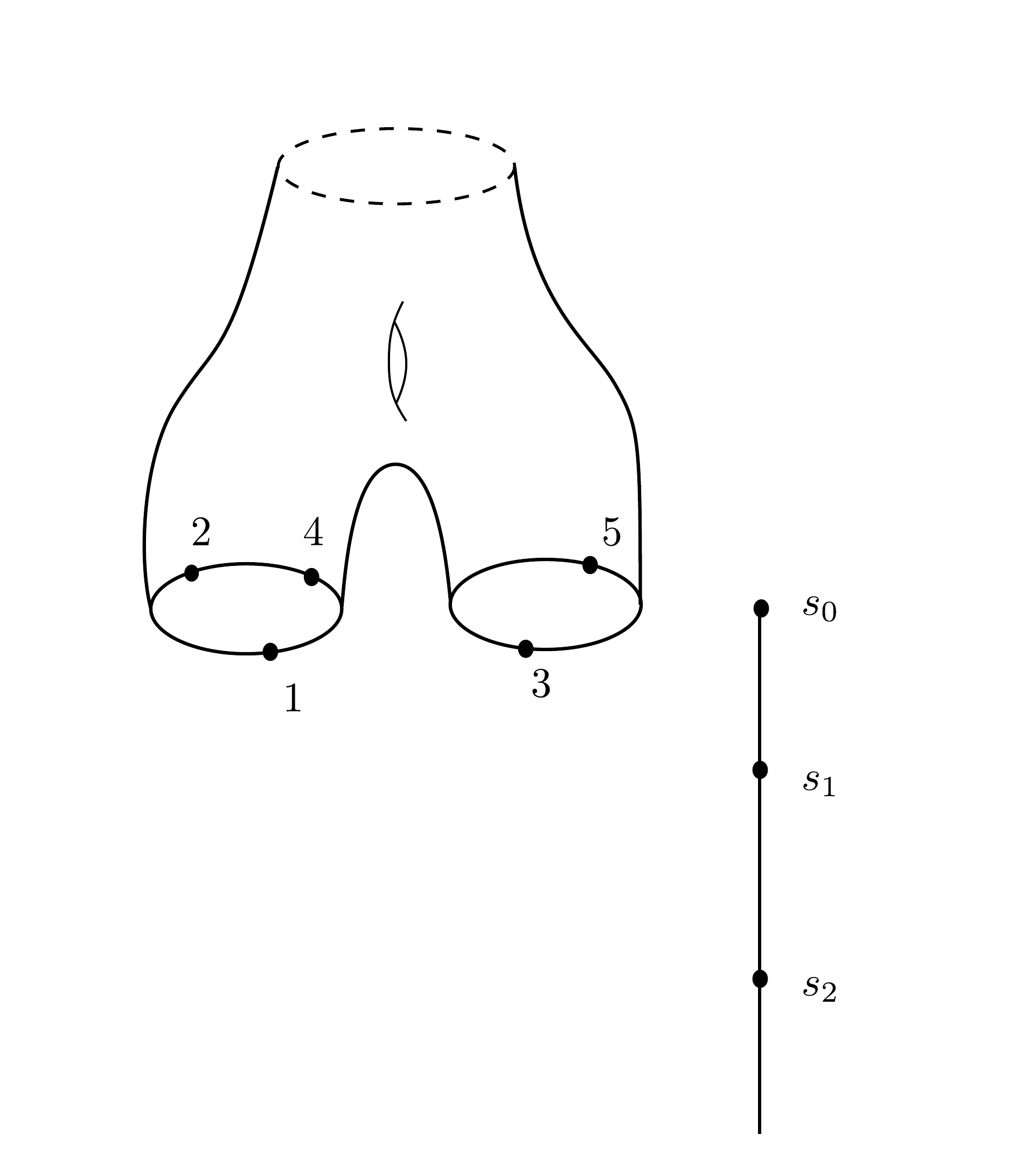}
    \caption{The tuple $((S, \pi), \bm \Gamma)$ above is an element of $\mathcal{T}^{\sigma, \sigma_0, \sigma'}_{-3, \ell}$ for $\kappa=5$ with $\sigma_0=(142)(35)$.}
    \label{fig-mixed-moduli}
\end{figure}

Following the approach of Section~\ref{hurwitz-spaces} we define a portion $\hurwitzhalfsm$ of the compactification $\hurwitzhalfcomp$ (which we will not specify here). 
Namely, an element of $\hurwitzhalfsm$ may have several levels of covers of the cylinder $\R \times S^1$ above the cover of the half cylinder. That is 

\begin{definition}
    The space of \emph{smooth half-cylindrical buildings} $\hurwitzhalfsm$ consists of
    $$ [(\overrightarrow{C}=(C_n,\ldots,C_0), S=S_n\cup \ldots \cup S_0, (\bm v^{n-1}, \ldots, \bm v^0), (\bm m^{n-1}, \ldots, \bm m^0), \pi]$$
    where
    \begin{enumerate}
    \item $\pi|_{S_0} \colon S_0 \to \R_{\ge 0} \times S^1$ is a branched cover of a half cylinder with data of markers on the boundary, but without asymptotic markers. In other words, it comes from an element of ${\mathcal{H}^{\sigma', \sigma_0, \ge 0}_{\kappa, \chi_2}}$ for some $\sigma' \in \mathfrak{S}_\kappa$ by forgetting asymptotic markers;
    \item $\bm m^0$ is a tuple of matching markers corresponding to the nodes $\bm v^0$ connecting $S_0$ and $S_1$ (see (4) of Definition~\ref{hurwitz-spaces} for more details);
        \item $[ (\overrightarrow{C}_{>0}=(C_n, \ldots, C_1),  S_{>0}=S_n\cup \ldots \cup S_1, (\bm v^{n-1}, \ldots, \bm v^1), (\bm m^{n-1}, \ldots, \bm m^1), \pi|_{S_{>0}}]$ comes from an element of $\widehat{\mathcal{H}}^{\sigma, \sigma'}_{\kappa, \chi_1}$ with forgotten asymptotic markers at negative punctures of $C_1$  (see Definition~\ref{hurwitz-spaces}).
        \item It holds that $\chi_1+\chi_2=\chi$.
        \end{enumerate}
\end{definition}

Similarly to $\hurwitzsm$ (see Definition~\ref{def-hurwitz-2-level}) a codimension $1$ boundary of $\hurwitzhalfsm$ can be given as the images under natural embeddings of spaces
\begin{equation}\label{half-cylinder-boundary}
   \bigsqcup_{\substack{\chi_1+\chi_2=\chi, \\ \sigma' \in \mathfrak{S}_\kappa}}\mathcal{H}^{\sigma, \sigma'}_{\kappa, \chi_1} \times_{\bm m} \mathcal{H}^{\sigma', \sigma_0, \ge 0}_{\kappa, \chi_2} .
\end{equation}
There is a natural projection
\begin{equation}
    \overline{\mathcal{H}}^{\sigma, \sigma_0, \ge 0}_{\kappa, \chi} \to \overline{\mathcal{M}}_{0, 1, -\chi+1}
\end{equation}
where the later space is the moduli space of genus $0$ curves with $1$ boundary component and $-\chi+1$ marked points. Notice that on smooth locus (i.e. away from nodal curves in the target) this map is a finite cover.

\subsection{Floer data for the half-cylinder Hurwitz space}
We set up the perturbation data on the half-cylinder Hurwitz spaces following Section~\ref{section-floer-data}. First, we note that there is a wnbc groupoid $\hurwitzhalfbrsm$ built via construction similar to the one given in Section~\ref{branched-manifolds}. 

As in the Section~\ref{section-floer-data} there is a smooth choice of positive cylindrical ends $\{\delta^i_+\}_{i=1}^{i=l(\sigma)}$ and finite cylinders $\delta^j_{i,r}$ over the whole $\hurwitzhalfbrsm$.

    A Floer datum on $(S, \pi) \in (\hurwitzhalfbrsm)^{\vee}_H$ (here $(S, \pi)$ is assumed to be away from the boundary) is a collection $(H_{(S,\pi)}, F_{(S, \pi)}, J_{(S, \pi)})$, where 
  \begin{enumerate}[(FDH1)]
    \item\label{half-floer-data-1} $H_{(S, \pi)} \colon S \to \mathcal{H}(\hat{M})$ is a domain-dependent Hamiltonian function, required to coincide with $H_0$ on cylindrical ends and finite cylinders, and vanishes on $\pi$-preimages of small neighborhoods of branched points and also on $\pi^{-1}([0,1] \times S^1)$;
    \item\label{half-floer-data-2} a domain-dependent function $F_{(S,\pi)} \in C^\infty (\hat{M})$ on $S$ such that conditions given in~\ref{def-floer-datum2} are satisfied (we note that on a cylinder with assigned cycle $c \in \mathfrak{S}_\kappa$ the function $F$ has to coincide with $F_c$ away from the collars) and additionally $F$ vanishes on $\pi$-preimages of $\varepsilon$-neighborhoods of branched points and also on $\pi^{-1}([0,1] \times S^1)$;
    \item \label{half-floer-data-3} a domain-dependent almost complex structure $J_{(S, \pi)}\colon S \to\mathcal{J}_{\hat{M}}$ coinciding with $J_t$ as in~(\ref{fixed-almost-complex}) at each cylindrical end and at each finite cylinder, and coincides with $J^0$ on $\pi$-preimages of small neighborhoods of branched points and also on $\pi^{-1}([0,1] \times S^1)$.
\end{enumerate}

As in the case of Hurwitz spaces, there exists a choice of \emph{consistent Floer data} $(H,F,J)$ on $\hurwitzhalfbrsm$ in the sense of Section~\ref{section-floer-data}. As before we will use the notation $H^{\mathrm{tot}}=H+F$.

\subsection{The mixed moduli space}
By complete analogy with Section~\ref{section: perturbation spaces}, for a strict order $\bm h \in \mathfrak{S}_\ell$ we can define spaces  $\mathfrak{X}^{\le 0}_{+}(k, \bm h)$ and $\mathfrak{X}_{0}^{\le 0, i}(k, \bm h)$ for $i=0, \dots, k$ simply by replacing $\operatorname{Conf}_\ell(\R)$ with $\operatorname{Conf}_{\ell}(\R_{\le 0})$, and the map~\eqref{eq: d-map-conf} with
\begin{gather}
    \mathcal{D}^{\le 0}_{\bm h}\colon \overline{\operatorname{Conf}}^{\bm h}_k(\R_{\le 0}) \to \mathscr{D}_{k}, \label{eq: new d-map-conf} \\
    (s_1, \ldots, s_k) \mapsto (-s_{h_1}, s_{h_1}-s_{h_2}, \dots, s_{h_{k-1}}-s_{h_k}). \nonumber
\end{gather}
Using these building blocks one defines $\mathfrak{X}^{\le 0}_{\operatorname{sw}}(\ell)$ analogous to~\eqref{eqn: sw}-\eqref{eqn: sw1}, with the only difference that one replaces $\mathfrak{X}^{\le 0}_{-}(\ell, \bm h)$ with $\mathfrak{X}^{\le 0, 0}_{0}(\ell, \bm h)$ in~\eqref{eqn: sw}. We leave the details to the reader.

Let us pick a Floer data $(H, F,J)$ on $\hurwitzhalfbrsm$ and an element $\bm Y=(Y_0, \ldots, Y_{\ell-1}, Y_\ell) \in \mathfrak{X}^{\le 0}_{\operatorname{sw}}(\ell)$.

Consider $\bm x\in \mathcal{P}_{\mathrm{unsym}}(T^*Q)$, $\bm \gamma \in \mathcal{P}_\kappa^{\bm L}$ and a permutation $\sigma_0 \in \mathfrak{S}_\kappa$. Let $\bm h$ be a strict order on $[b+\ell]$ regarded as an element of $\mathfrak{S}_{b+\ell}$, where $b=\kappa-\chi$; $\bm \tau$ be an $\ell$-tuple of elements of $\binom{[\kappa]}{2}$, and $\bm \rho$ be an $\ell$-tuple of permutations in $\mathfrak{S}_\kappa$.
We set $\sigma= \sigma(\bm x)$ and $\sigma'=\sigma(\bm \gamma)$.

We define the mixed moduli space $\mathcal{T}^{\sigma_0}_{\chi, \ell}(\bm x,\bm \gamma; H^{\mathrm{tot}}, J, \bm Y, \bm \tau, \bm h, \bm \rho)$ to be the space of tuples
\begin{equation}\label{mixed-moduli-tuples}
    \left((S, \pi \colon S \to \R_{\ge 0} \times S^1, B), u=(\pi, v), \bm \Gamma\right),
\end{equation}
where $u$ is a map
\begin{equation*}
    u\colon S\to \R_{\ge 0} \times S^1 \times T^*Q,
\end{equation*}
and \begin{equation*}
        \bm \Gamma=\left(\bm s, \bm \theta,(\Gamma_0,\ldots,\Gamma_\ell)\right), 
    \end{equation*}
    where 
     such that $(\bm s, \bm \theta) \in \operatorname{Conf}_{\ell}(\R_{\le0} \times [0,1])$, with $\bm s=(s_1, \ldots, s_{\ell})$, such that $(\bm s, \bm \theta)$ in the closure $\overline{\bm h}$ of $\bm h$ regarded as a component of $\operatorname{Conf}_{\ell}(\R \times [0,1])$.
     Set $\bm \lambda= (\lambda_0=-s_{h_1}, \lambda_1=s_{h_1}-s_{h_2}, \ldots, \lambda_{\ell-1}=s_{h_{\ell-1}}-s_{h_\ell})$. The above collection is required to satisfy several conditions, which we list below.
     
(A) First, $(S, \pi) \in\hurwitzhalfbrsm$ the map $u=(\pi, v)$ is required to satisfy
\begin{equation}\label{eq-half-cylinder-moduli}
    \begin{dcases}
        (dv-Y_{(S, \pi)})^{0,1}=0, \text{ with respect to } J_{(S,\pi)};  \\
        \lim_{s \to +\infty } v\circ \delta_i^+(s, \cdot) = x^{c_i(\sigma)}(\cdot) \text{ for }i=1, \ldots,  l(\sigma); \\
        \operatorname{Im}(v|_{\partial S}) \subset Q.
    \end{dcases}
\end{equation}

Here $Y_{(S, \pi)}$ is the $1$-form associated $H^{\mathrm{tot}}_{(S, \pi)}$ as in~\eqref{vector-form}. We point out that as in Lemma~\ref{sft-to-hamiltonian} the map $u$ can be regarded as a pseudo-holomorphic map with respect to the almost complex structure $\underline{J}^{H^{\mathrm{tot}}}_{(S, \pi)}$ defined as in Section~\ref{section-floer-data}. We denote the moduli space of such curves $u$ via $\hurwitzhalfbrsm(\bm x;  H^{\mathrm{tot}},  J)$.

\smallskip
(B) Second, we introduce the \emph{evaluation map}
\begin{gather}
    \mathscr{E}: \hurwitzhalfbrsm(\bm x; H^{\mathrm{tot}}, J) \times (\R_{\le 0} \times [0,1])^\ell\to \Lambda^{1}_{[\kappa]}(Q),\\
    (u, \bm s, \bm \theta)\mapsto \mathscr{E}(u, \bm s, \bm \theta), \nonumber
\end{gather} 
as follows: If $u$ does not have any branch points along $\{0\} \times S^1$, then we let $\mathscr{E}(u, \bm s, \bm \theta)$ be the restriction of $u$ to $\bm \gamma'=\pi^{-1}(\{0\} \times S^1)$, precomposed with $\nu^{\circ b}_{[b+\ell]}$ evaluated at $((B, (\bm s, \overline{\bm \theta})), \bm \gamma')$, regarded as a point of $(\R \times S^1)^{b+\ell} \times \Lambda^1_{[\kappa]}(Q)$ (here $\overline{\bm \theta}$ is the class of $\bm \theta$ in $T^\ell$). Notice that we use markers $a_i^0$ to order strands in $\bm \gamma'$.

If $u$ has branch points along $\{0\} \times S^1$, then let 
$$\bm \gamma' = \lim_{\delta\to 0^+} v|_{\pi^{-1}(\{\delta\} \times [0,1])},$$ 
and we set 
\begin{equation}\label{eq: evaluation map}
\mathscr{E}(u)=(\gamma'_1\circ\nu_{[b+\ell]}, \dots, \gamma'_\kappa \circ \nu_{[b+\ell]}).
\end{equation}

Observe that $\mathscr{E}(u) \in \Omega^{m,2}(M, \bm q)$ for arbitrary $m \ge 1$.

\smallskip
(C) Finally, we require that
    \begin{enumerate}
        \item The maps $$\left\{\begin{array}{l} 
        \Gamma_i \colon [0,\lambda_i]\to \Lambda^{1}_{[\kappa]}(Q) \quad \text{ for }\quad  i=0,\dots,\ell-1;\\
        \Gamma_\ell \colon [0, +\infty) \to \Lambda^{1}_{[\kappa]}(Q)
        \end{array}\right.$$
        are continuously differentiable.
        
        \item $\partial_s\Gamma_i(s)=(X+\bm Y_{\bm s, i}(\lambda_i, s))_{\Gamma_i(s)}$ for $\, i=0,\dots,\ell$.
        
        \item $\Gamma_0(0)=\mathscr{E}(u)$, $\Gamma_\ell(-\infty)=\bm \gamma'$.

        \item \label{item: mix-switch-condition} For $k=1, \dots, \ell$, $(\bm s, \bm \theta, \Gamma_{k-1}(\lambda_{k-1})) \in \Delta^{\tau_k, \, |}_{I_{h_k}}$  and
        $$c^|_{\rho_{k}}\circ sw^{i_kj_k}_{I_{h_k}}(\bm s, \bm \theta, \Gamma_{k-1}(\lambda_{k-1}))=(\bm s, \bm \theta, \Gamma_{k}(0)).$$
 
    \end{enumerate}

 There is a notion of \emph{universal perturbation data for mixed moduli spaces} $(H,  F, J, \bm Y)$ for all tuples $(\kappa, \chi, \ell)$ which can be given by combining the notion of consistency on the Floer side and the notion of universality on the Morse side as in Definition~\ref{defn-universal-data}. It is important to note that such data should be compatible with already chosen data for HFSH and for Morse flow lines with switchings, corresponding to vertical ``breakings''. We are not giving the explicit definition here and leave the details to the reader. Such data is \emph{regular} if all the mixed moduli spaces are wnb groupoids.

\begin{lemma}
\label{lemma-F-transversality}
    There exists a choice of regular universal perturbation data.
    For this data, a tuple of Hamiltonian orbits $\bm x\in \mathcal{P}_{unsym}(T^*Q)$, a multiloop $\bm \gamma\in \mathcal{P}_\kappa^{\bm L}$ the space 
    \begin{equation*}
        \mathcal{T}^{\sigma_0}_{\chi, \ell}(\bm x,\bm \gamma; H^{\mathrm{tot}}, J, \bm Y, \bm \tau, \bm h, \bm \rho )
    \end{equation*}
    is a wnb groupoid of dimension 
    $$\kappa n-|\bm x|-|\bm \gamma|-(n-2)(\ell-\chi).$$
    Moreover, given an orientation of $Q$ and a spin structure on the pull-back of tangent bundle $TQ$ to $T^*Q$ there is an induced orientation of $\mathcal{T}^{\sigma_0}_{\chi, \ell}(\bm x,\bm \gamma; H^{\mathrm{tot}}, J, \bm Y, \bm \tau, \bm h, \bm \rho )$.
\end{lemma}
\begin{proof}
First, one can show that $\hurwitzhalfbrsm(\bm x;  H^{\mathrm{tot}}, J)$ is a smooth manifold along the lines of the argument in Theorem~\ref{transversality-theorem} (compare with \cite[Section 3]{abouzaid2011cotangent} for the $\kappa=1$ case). 
For index calculation, we glue the operator $D_u$ associated with a given element of $\hurwitzhalfbrsm(\bm x;  H^{\mathrm{tot}}, J)$ to the operator $D_{\bm x}$ associated
with a tuple of Hamiltonian orbits $\bm x$, obtained by the direct product of operators $D_{x^{c_i(\sigma)}}$ (see~\eqref{loop-operator-unsym}) to obtain an operator $D_u \# D_{\bm x}$ from a surface with no punctures. Then
$$\operatorname{ind}(D_u \# D_{\bm x})=n (\chi+\ell(x))$$
by \cite[Proposition 11.13]{seidel2008fukaya} and 
$$\operatorname{ind}(D_{\bm x})=n(\ell(\bm x)-\kappa)+|\bm x|.$$
Hence $$\operatorname{ind}(D_u)=\kappa n-|\bm x|+n \chi$$
and then one concludes as in Lemma~\ref{index-count} that the dimension of $\hurwitzhalfbrsm(\bm x;  H^{\mathrm{tot}},  J)$ is equal to
\begin{equation}\label{dim-half-moduli}
\kappa n -|\bm x|+(n-2)\chi.    
\end{equation}

Given a spin structure on $TQ$ regarded as a subbundle of $TT^*Q$, following \cite[Proposition 11.13]{seidel2008fukaya} and \cite[Appendix A]{abouzaid2011cotangent}, there are isomorphisms
\begin{equation}\label{eq-halfcyl-orientation}
    \big| \hurwitzhalfbrsm(\bm x;  H^{\mathrm{tot}},  J) \big| \cong o_{\bm x}^{-1} \otimes \big| \hurwitzhalfbrsm \big| \otimes |Q|^{\ell(\sigma_0)} \otimes \zeta^{\kappa n+n(\chi-\ell(\sigma_0)} \cong \zeta^{(n-2)\chi} \otimes o_{\bm x}^{-1} \otimes \zeta^{\kappa n}.
\end{equation}
For the later isomorphism, we use the fixed orientation on $Q$ to get trivialization $|Q| \cong \zeta^{n}$ and the standard orientation on $\hurwitzhalfbrsm$ to get $|\hurwitzhalfbrsm| \cong \zeta^{-2\chi}$; we note that in the second isomorphism, there is a Koszul sign involved that we do not mention explicitly.

 Next, we claim that for generic Floer data the evaluation map $\mathscr{E}$ is an immersion of the space $\hurwitzhalfbrsm(\bm x; H^{\mathrm{tot}}, J)$  when restricted to the fiber above $(\bm s, \bm \theta) \in (\R_{\le 0} \times [0,1])^\ell$ for arbitrary $(\bm s, \bm \theta)$. It is implied by \cite[Theorem 3.8]{abbondandolo2006floer}, and more generally by the discussion in \cite[Section 5]{abbondandolo2010floer}. Finally, one may conclude the proof by applying a similar argument to the proof of Proposition~\ref{lemma-morse-transversality} to show that for a generic choice of $\bm Y$ the Morse flow lines with switchings connecting the image of $\hurwitzhalfbrsm(\bm x; H^{\mathrm{tot}}, J)$ under $\mathscr{E}$ with $\bm \gamma$ constitute a smooth manifold. The main step of such an argument consists of showing that the following fiber product is a Banach manifold
 \begin{equation}
 \hurwitzhalfbrsm(\bm x;  H^{\mathrm{tot}}, J)\tensor*[_{\mathscr{E}}]{\times}{_{\tilde E_0^s}} W_\ell(-, \bm \gamma; \mathfrak{X}_{\operatorname{sw}}^{\le 0}(\ell)),
 \end{equation}
where we denote by $W_\ell(-, \bm \gamma; \mathfrak{X}_{\operatorname{sw}}^{\le 0}(\ell))$ the moduli space of MFLS with $\ell$ switches connecting an arbitrary point of $\Lambda^1_{[\kappa]}(Q)$ to $\bm \gamma$, defined by analogy with~\eqref{eq: defn-MFLS-parameters}. 
 As for the orientation on $\mathcal{T}^{\sigma_0}_{\chi, \ell}(\bm x,\bm \gamma;  H^{\mathrm{tot}},  J, \bm Y, \bm \tau, \bm h, \bm \rho)$, one may argue as in the proof of Proposition~\ref{lemma-morse-transversality} that
\begin{equation}\label{eq-mixed-orientation}
o_{\bm \gamma} \otimes \big| \mathcal{T}^{\sigma_0}_{\chi, \ell}(\bm x,\bm \gamma; H^{\mathrm{tot}}, J, \bm Y, \bm \tau, \bm h, \bm \rho) \big| \cong \zeta^{(2-n)\ell} \otimes \big|  \hurwitzhalf(\bm x; H^{\mathrm{tot}}, J)\big|.
\end{equation}

By combining~\eqref{eq-halfcyl-orientation} and~\eqref{eq-mixed-orientation} one gets an orientation on $\mathcal{T}^{\sigma_0}_{\chi, \ell}(\bm x,\bm \gamma; H^{\mathrm{tot}}, J, \bm Y, \bm \tau, \bm h, \bm \rho)$ induced by a choice of a generator of $o_{\bm x}$.

As in the proof of Proposition~\ref{lemma-morse-transversality} we assert that one can achieve universality of the data $( H, F, J, \bm Y)$ via an inductive argument.
\end{proof}

From now on we pick regular and universal data $( H, F,  J, \bm Y)$ and omit it form notation of mixed moduli spaces.

\begin{lemma}
\label{lemma-F-cpt-finite}
    For regular and universal data, a positive integer $m$, the mixed moduli space
    $$\bigsqcup_{\substack{\ell-\chi=m, \, \sigma_0 \in \mathfrak{S}_\kappa, \\\bm h \in \mathfrak{S}_\ell, \bm \tau \in \binom{[\kappa]}{2}^\ell, \bm \rho \in \mathfrak{S}_\kappa^\ell}}\mathcal{T}^{\sigma_0}_{\chi, \ell}( \bm x,\bm \gamma; \bm \tau, \bm h, \bm \rho)$$
    is empty for all but finitely many $\bm \gamma\in \mathcal{P}_\kappa^{\bm L}$.
\end{lemma}
\begin{proof}
    We omit the proof as it follows from the Palais-Smale condition~\ref{PG4}. See~\cite{abbondandolo2006floer} for details.
\end{proof}

\smallskip
\noindent
\emph{Aside on Kuranishi replacements.}
As in Section~\ref{section-pseudo-holomophic-curves}, for $n\ge 3$, a sequence of elements in the mixed moduli space may converge to a pair with Floer side admitting ghost bubbles. More precisely, a sequence of curves $(u_i,\bm \Gamma_{(i)})$, $i=1,2,\dots$, in $\mathcal{T}^{\sigma_0}_{\chi, \ell}(\bm x,\bm \gamma; \bm \tau, \bm h, \bm \rho)$, after passing to a subsequence, can limit to 
$$(u_\infty= u_{\infty,1}\cup u_{\infty,2}, \bm \Gamma_{(\infty)}),$$ 
where:
\begin{itemize}
\item[(i)] the main part $u_{\infty,1}$ is the union of components that are not ghost bubbles; each component of the $u_{\infty,1}$ is somewhere injective; and 
\item[(ii)] $u_{\infty,2}$ is a union of ghost bubbles, i.e., locally constant maps where the domain is a possibly disconnected compact Riemann surface (possibly with boundary) and $u_{\infty,2}$ maps to points on $\operatorname{Im}(u_{\infty,1})$.
\end{itemize}  
To remedy this we introduce Kuranishi replacements $\mathcal{T}^{\sigma_0, \sharp}_{\chi, \ell}(\bm x,\bm \gamma; \bm \tau, \bm h, \bm \rho)$ satisfying similar requirements to \ref{Kuranishi-replacements1}-\ref{Kuranishi-replacements3} (see \cite[Section 3.2.2]{honda2025morse} for more specifics). As before, we will pass to Kuranishi replacements whenever necessary and drop $\sharp$ from the notation.

\bigskip
Before diving into the discussion of the compactification of mixed moduli spaces of dimension $1$, we introduce two new types of boundary behavior. Let us fix tuples $\bm x\in \mathcal{P}_{\mathrm{unsym}}(T^*Q)$ and $\bm \gamma\in \mathcal{P}^{\bm L}_\kappa$ satisfying for some $\chi \le 0, \ell \ge 0$:
$$\kappa n-|\bm x|-1=|\bm \gamma|+(n-2)(\ell-\chi).$$
\begin{figure}[h]
    \centering
    \subfloat{{\includegraphics[width=15cm]{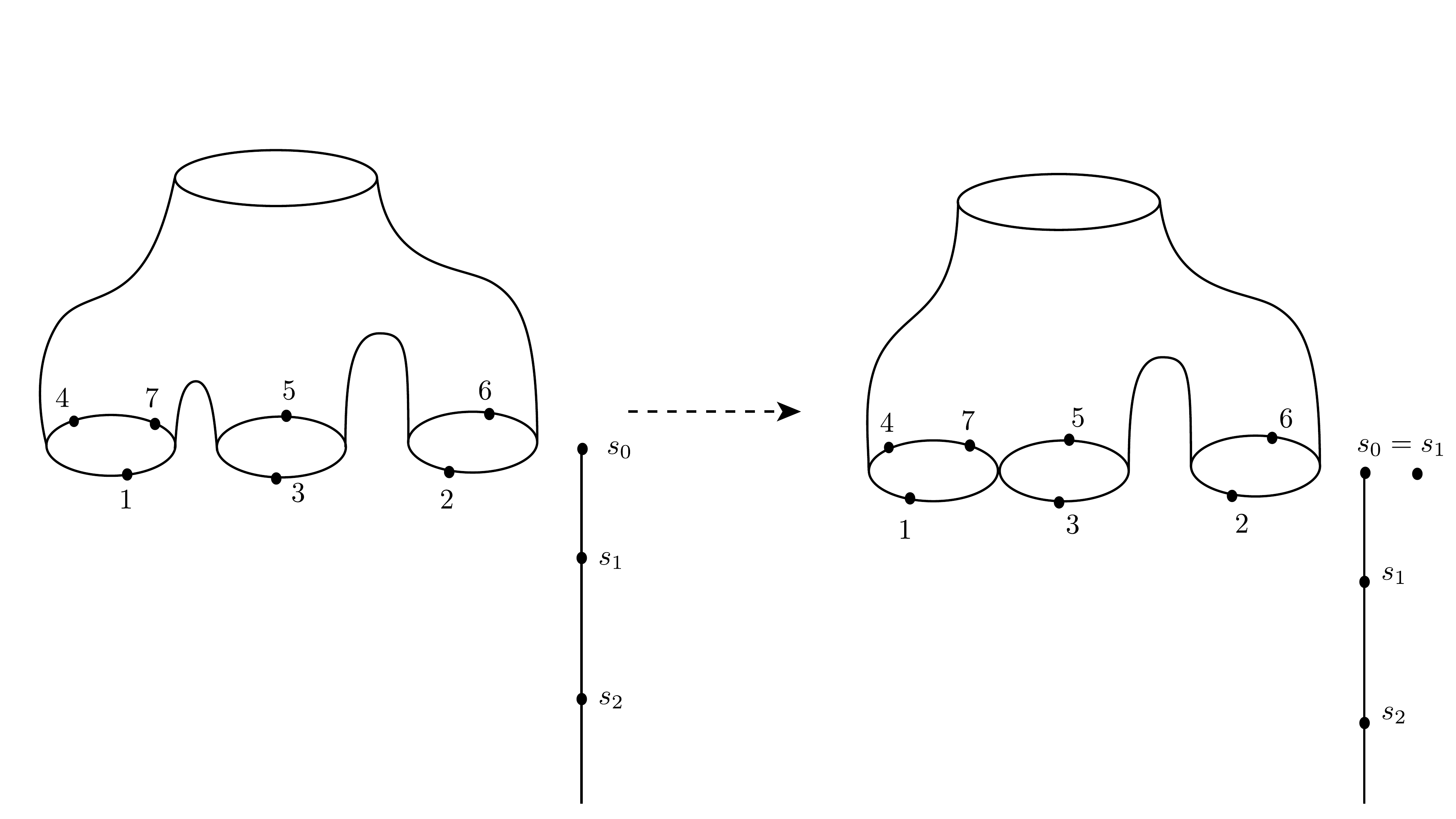} }}
    \caption{Approaching a Type I singularity.}
    \label{fig-typeI}
\end{figure}

\begin{definition}\label{def-typeI}
   \emph{The Type I singular mixed moduli space } $\mathcal{T}^{\sigma_0, \mathrm{I}}_{\chi, \ell}(\bm x,\bm \gamma; \bm \tau, \bm h, \bm \rho)$ is a space of tuples 
    \begin{equation*}
        \left((S, \pi), u, \bm \Gamma\right)
    \end{equation*}
    as in~\eqref{mixed-moduli-tuples} with the following differences: 
    \begin{enumerate}
        \item $(S, \pi)$ represents a point in $\hurwitzhalfcomp$, such that (only) one of the branch points in $B$ has $s$-coordinate equal to $0$, let $\tau'=(ij)$ be such that $\mathscr{E}(u)_i$ and $\mathscr{E}(u)_j$ intersect;
        \item $\bm \Gamma^{\mathrm{I}}=\left((\Gamma_0^{\mathrm{I}}, \ldots, \Gamma_{\ell}^{\mathrm{I}}), \bm s^{\mathrm{I}}=(s_1^{\mathrm{I}}, \ldots, s_{\ell+1}^{\mathrm{I}}), \bm \tau^{\mathrm{I}}=(\tau_1^{\mathrm{I}}, \ldots,\tau_{\ell+1}^{\mathrm{I}}), \bm h^{\mathrm{I}}=(h_1^{\mathrm{I}}, \ldots, h_{\ell+1}^{\mathrm{I}}), \bm \rho^{\mathrm{I}}=(\rho_1^{\mathrm{I}}, \ldots, \rho_{\ell+1}^{\mathrm{I}})\right)$ 
        is a tuple satisfying 
        \begin{itemize}
            \item the order $\bm h^{\mathrm{I}}$ restricted to the least $\ell$ elements coincides with $\bm h$;
            \item $\bm \rho=(\rho_2^{\mathrm{I}}, \ldots, \rho_{\ell+1}^{\mathrm{I}})$ and $\rho_1^{\mathrm{I}} \in \mathfrak{S}_\kappa$ is arbitrary;
            \item $\bm \tau=(\tau_{2}^{\mathrm{I}},\ldots,\tau_{\ell+1}^{\mathrm{I}}))$ and $\tau'={}^{\rho_1}\tau_1^{\mathrm{I}}$;
            \item $\Gamma_0^{\mathrm{I}}$ is a stationary trajectory of length $0$.
        \end{itemize}
    \end{enumerate}
\end{definition}

\begin{remark}\label{rmk-typeI}
    Heuristically, a Type I boundary singularity is a result of a ramification point approaching the boundary and adding a new switching point on the Morse side at $s=0$ (see Figure~\ref{fig-typeI}). The number of switching points of the curves in $\mathcal{T}^{\sigma_0, \mathrm{I}}_{\chi, \ell}(\bm x,\bm \gamma)$ is equal to $\ell+1$. We point out that choices in (2) above allow us to regard this singularity as an element of the boundary of $\mathcal{T}^{\sigma_0, \mathrm{I}}_{\chi+1, \ell+1}(\bm x,\bm \gamma)$ (after forgetting the marker of the boundary branch point), and for given $(\bm \tau, \bm h, \bm \rho)$ there are $(\ell+1)\cdot \kappa!$ ways of doing so in accordance to different choices of data as in (2) (not altering the geometry).
\end{remark}
\begin{definition}\label{def-typeII}
   \emph{The Type II singular mixed moduli space} $\mathcal{T}^{\sigma_0, \mathrm{II}}_{\chi, \ell}(\bm x,\bm \gamma; \bm \tau, \bm h, \bm \rho)$ is the space of tuples 
    \begin{equation*}
        \left((S, \pi), u, \bm \Gamma\right)
    \end{equation*}
    as in~\eqref{mixed-moduli-tuples} with the following differences: 
    \begin{enumerate}
    \item $(S^{\mathrm{II}}, \pi^{\mathrm{II}} \colon S \to \R_{\ge 0} \times S^1)$ is a branched cover with a set of branched points consisting of $|B|=-\chi+1$ points with (only) one of these points having $s$-coordinate equal to $0$;
        
        \item $\left(\bm \Gamma, \bm s=(s_1,  \ldots, s_{\ell}), \bm \tau, \bm \rho \right)$
      is a tuple with conditions that $s_{h_1}=0$.
    \end{enumerate}
\end{definition}
\begin{figure}[h]
    \centering
    \subfloat{{\includegraphics[width=16cm]{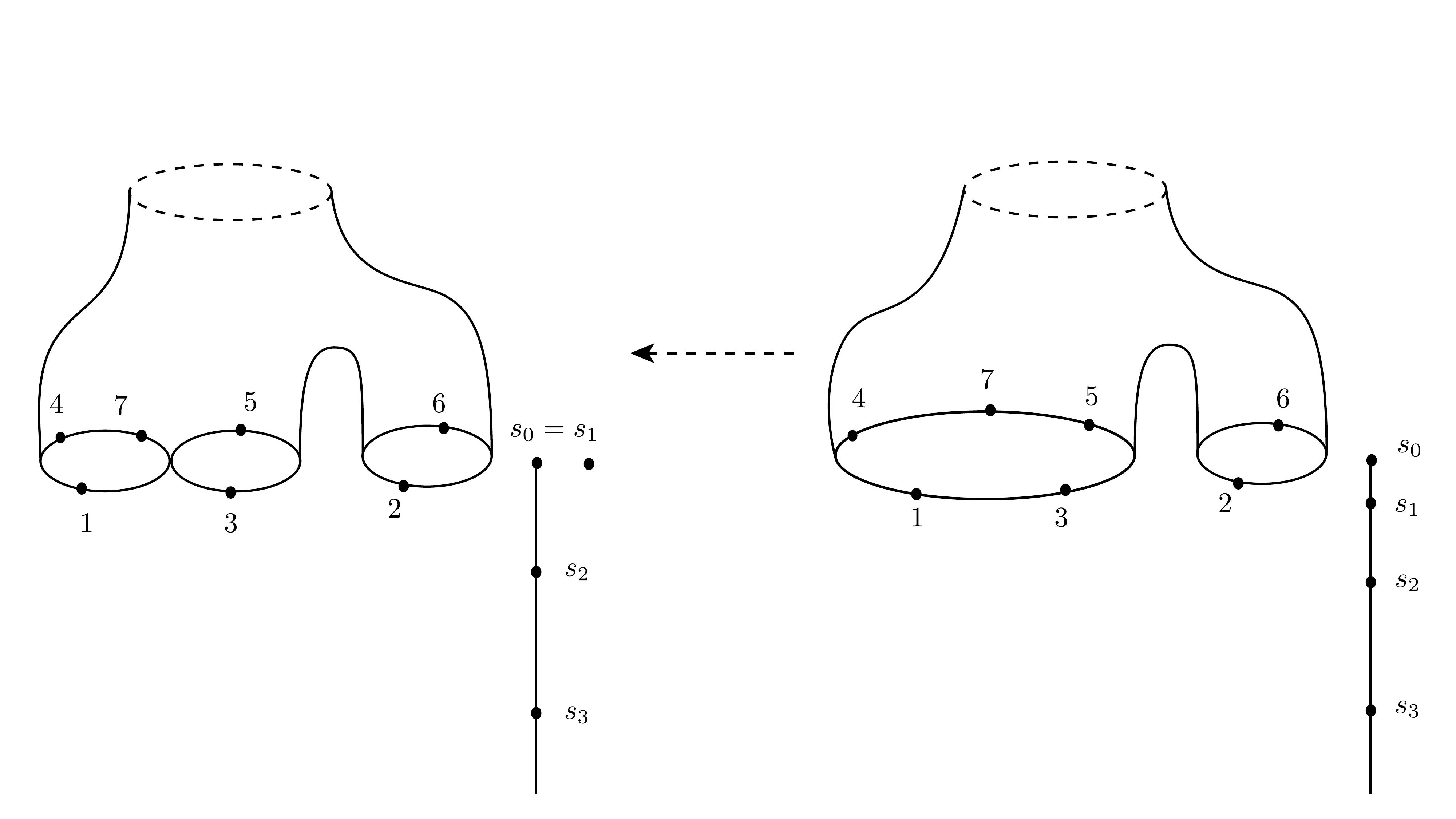} }}
    \caption{Approaching a Type II singularity.}
    \label{fig-typeII}
\end{figure}
\begin{remark}\label{rmk-typeII}
       A Type II boundary singularity is a result of $s_{h_1}$ approaching $0 \in \R_{\le 0}$.  This implies that for the limit curve $u=(\pi, v)$, the component $v$ has a boundary hyperbolic node. We replace the domain of $u$ by adding an additional branched point at the image of the node under $\pi$, and replace $\pi$ with $\pi'$ ramified at this node, see Figure~\ref{fig-typeII} for a schematic picture. The Euler characteristic of curves in $\mathcal{T}^{\sigma_0, \mathrm{II}}_{\chi, \ell}(\bm x,\bm \gamma)$ is equal to $\chi-1$. Note that there are $(-\chi+1)$ ways to choose such a domain $\pi'$ since one has to order the branched points (while preserving the order induced on the ``old" branched points).
\end{remark}

\begin{lemma}
\label{lemma-F-compactness}
    For a generic choice of universal and regular data $(H, F, J, \bm Y)$, a tuple of Hamiltonian orbits $\bm x\in \mathcal{P}_{\mathrm{unsym}}(T^*Q)$ and a multiloop $\bm \gamma\in \mathcal{P}_\kappa^{\bm L}$ satisfying  
    \begin{equation}\label{mixed-moduli-dim-1}
    \kappa n-|\bm x|-|\bm \gamma|-(n-2)(\ell-\chi)= 1,
    \end{equation}
    the wnb groupoid $$\bigcup_{\substack{\bm h \in \mathfrak{S}_\ell, \bm \tau \in \binom{[\kappa]}{2}^\ell, \bm \rho \in \mathfrak{S}_\kappa^\ell}} \ \mathcal{T}^{\sigma_0}_{\chi, \ell}(\bm x,\bm \gamma; \bm h, \bm \tau, \bm \rho)$$ admits a compactification $\overline{\mathcal{T}}^{\sigma_0}_{\chi, \ell}(\bm x,\bm \gamma)$ with boundary points contained in the images of the following spaces under natural inclusions
    \begin{gather}
        \bigsqcup_{\substack{\bm x' \in SC^*_\kappa(T^*Q), \, (n-2)\chi_1+|\bm x'|-|\bm x|=1, 
        \\ \chi_1+\chi_2=\chi, \, \bm h \in \mathfrak{S}_\ell,  \bm \tau \in \binom{[\kappa]}{2}^\ell, \bm \rho \in \mathfrak{S}_\kappa^\ell}} \mathscr{M}^{\chi_1}(\bm x, \bm x') \times \mathcal{T}^{\sigma_0, \op{ind}=0}_{\chi_2, \ell}(\bm x',\bm \gamma; \bm h, \bm \tau, \bm \rho) ; \label{eq-mixed-moduli-boundary1}\\
        \bigsqcup_{\substack{\bm \gamma' \in CM_{\kappa n-*}(\Lambda^1_{[\kappa]}(Q)), 
        \\ \ell_1+\ell_2=\ell, (\bm h_1, \bm h_2) \in \mathfrak{S}_{\ell_1} \times\mathfrak{S}_{\ell_2}, (\bm \tau_1, \bm \tau_2) \in \binom{[\kappa]}{2}^\ell, (\bm \rho_1, \bm \rho_2) \in \mathfrak{S}_\kappa^{\ell_1} \times \mathfrak{S}_\kappa^{\ell_2} }}  \mathcal{T}^{\sigma_0}_{\chi, \ell_1}(\bm x,\bm \gamma'; \bm h_1, \bm \tau_1, \bm \rho_1) \times \widetilde{\mathcal{M}}_{\ell_2}(\bm \gamma', \bm \gamma; \bm h_2, \bm \tau_2, \bm \rho_2) \label{eq-mixed-moduli-boundary2};\\
        \bigsqcup_{\bm h \in \mathfrak{S}_\ell, \bm \tau \in \binom{[\kappa]}{2}^\ell, \bm \rho \in \mathfrak{S}_\kappa^\ell} \mathcal{T}^{\sigma_0, \mathrm{I}}_{\chi, \ell}(\bm x,\bm \gamma; \bm \tau, \bm h, \bm \rho); \\
        \bigsqcup_{\bm h \in \mathfrak{S}_\ell, \bm \tau \in \binom{[\kappa]}{2}^\ell, \bm \rho \in \mathfrak{S}_\kappa^\ell} \mathcal{T}^{\sigma_0, \mathrm{II}}_{\chi, \ell}(\bm x,\bm \gamma; \bm \tau, \bm h, \bm \rho).
    \end{gather}
\end{lemma}
\begin{proof}[Sketch of the proof]
    As in the proof of Theorem~\ref{theorem-cpt-moduli}, one shows that there is a portion of the boundary corresponding to the vertical stretching on the side of the curve $u$, and we leave it to the reader to fill in the details. Points of this type are described by~\eqref{eq-mixed-moduli-boundary1} and the weight of such a point is equal to $\frac{1}{\kappa!}$.

    The second type of boundary degeneration corresponds to stretching on the Morse side and is treated in Corollary~\ref{lemma-morse-cpt-moduli}. Notice that, as in the proof of this corollary, one can show that points corresponding to two switches ``occurring at the same time" (i.e., the ones described in~\eqref{eq-compactness-two-switches}) can be regarded as interior points.
    
One last case that we have to analyze is the boundary nodal degeneration with a possible occurrence of almost ghost bubbles attached along the boundary. Let us first treat this scenario. Namely, assume that $(u_i, \Psi_i)$ limits to $(u_\infty=u_{\infty, 1} \cup u_{\infty, 2}, \Psi_\infty)$, where $u_{\infty, 2}$ is null-homologous (see \ref{Kuranishi-replacements2}) has domain attached to the domain of $u_{\infty,1}$ along $k$ boundary nodes. We then have
$$\operatorname{ind} u_\infty=\operatorname{ind}u_i+k(n-2)$$
by~\eqref{dim-half-moduli}. On the other hand, the appearance of a boundary node along the Lagrangian $\{0\} \times S^1 \times Q$ is a codimension $(n+1)-2\cdot 1=n-1$ condition since the Lagrangian is $(n+1)$-dimensional and the boundary of the domain is mapped to a real $1$-dimensional curve. In total, we get a dimension drop of
$$k(n-1)-k(n-2)=k.$$
Hence, $u_{\infty,2}$ can be attached only along a boundary node, and it must have only one boundary component. This component can not be a disk, since there have to be at least two ramification points on it. Therefore, $\operatorname{ind}(u_{\infty,2}) <0$ and it violates~\ref{Kuranishi-replacements3} and~\ref{Kuranishi-replacements2}.

    Finally, from the above analysis, it follows that there could be no more than $1$ boundary node. If no ghost bubbles appear, this leaves us with the two singular types of boundary points $\mathcal{T}^{\sigma_0, \mathrm{I}}_{\chi, \ell}(\bm x,\bm \gamma)$ and $\mathcal{T}^{\sigma_0, \mathrm{II}}_{\chi, \ell}(\bm x,\bm \gamma)$.
\end{proof}

\subsection{Gluing analysis.}
In this subsection, we state the appropriate gluing result that we use for the local analysis near Type II boundary singularities. It is an adaptation of the gluing result of \cite{HH25openDM}, whose analysis in turn follows that in \cite{swaminathan2021Rel}, where the interior nodal gluing is shown. Both works are based on the polyfold theory introduced in \cite{HWZ17}, and here we assume familiarity with it. We note that this result was also asserted in \cite{ekholm-shende2022, ES2024bare}.

In what follows, we work with the category $C^\infty/\cdot$ of rel-$C^\infty$ manifolds, and we refer the reader to \cite[Section 2]{swaminathan2021Rel} for a detailed discussion. An object in this category is a pair of topological spaces $(Y, T)$ together with a structure morphism $p_Y \colon Y \to T$. Roughly speaking, one may think of $Y$ as having smooth fibers of fixed dimension.

Consider a pair $(X,L)$ of a symplectic manifold $X$, and a Lagrangian submanifold $L$ in $X$. We fix a tuple $\bm x=(x^1, \ldots, x^l)$ of $l$ simple marked Reeb orbits, and a proper analytic family $\pi_{\mathcal{S}} \colon \mathcal{S} \to \mathcal{W}$ of nodal curves with $l$ punctures and $l'$ ordered boundary components of Euler characteristic $\chi$, where $\mathcal{W}$ is a complex manifold possibly with boundary. 

For each $w \in \mathcal{W}$, we denote by $(S_w, p_1^+(w), \ldots, p_{l(\sigma)}^+(w))$ the fiber curve over $w$. Additionally, we assume a choice of cylindrical ends $\delta^+_{w,i}$ near punctures and of finite cylinders, and
we fix a choice of domain-dependent data of almost complex structures $J_w$ on $X$ on each $S_w$, such that 
 there is an open neighborhood $U \subset \mathcal{S}$ of the locus of the nodes where 
 \begin{equation}\label{eq-data-trivial-near-nodes}
 J_w\equiv J_{\operatorname{fix}},
\end{equation}
for some fixed almost complex structure $J_{\operatorname{fix}}$ on $X$. We denote this data via $J_{\mathcal{W}}$.

To all the above data, we associate a functor \begin{equation}
    \mathfrak{M}\coloneqq \mathfrak{M}(\pi_S,D_{\mathcal{W}}) \colon (C^\infty/\cdot)^{\operatorname{op}} \to \operatorname{Set}
\end{equation}
 by setting $\mathfrak{M}(Y/S)$ to be the set of diagrams

 \[
\begin{tikzcd}
\mathcal{S} & \mathcal{S}_T \arrow[l] \arrow[d] & \mathcal{S}_Y \arrow[l] \arrow[r, "F"] \arrow[d] & X \\
\mathcal{W} & T \arrow[l, "\varphi"] & Y \arrow[l, "p_Y"]  & 
\arrow[from=1-1, to=2-1, "\pi_{\mathcal{S}}"']
\end{tikzcd}
\]
satisfying 
\begin{enumerate}[($\mathfrak{M}1$)]
    \item $\varphi$ is continuous;
    \item both squares are cartesian;
    \item $F \colon \mathcal{S}_Y/ \mathcal{S}_T \to X/*$ is a rel-$C^\infty$ map with $F(\partial S_y) \subset L$ for all $y \in Y$;
    \item \label{eq: functor-dbar-regularity} each $u_y \coloneqq F|_{S_y}$ is a smooth stable map satisfying
    \begin{equation}\label{eq-parametrized-perturbed-dbar}
        (du_y)^{0,1}=0 \text{ with respect to }J_y,
    \end{equation}
    and $u_y$ is required to be \emph{regular}.
    \item  $u_y$ converges to $x^i$ near  the puncture $p_i^+$ with respect to the cylindrical end $\delta^+_{y,i}$.
\end{enumerate}

We denote by $\overline{\mathcal{M}}^{\operatorname{reg}}_{\chi, \bm x, \sigma_0}(\pi_{\mathcal{S}}, D_{\mathcal{W}})$ the space of pairs $(w, u)$ with $w \in \mathcal{W}$, and $u \colon (S_w, \partial S_w) \to (X, L)$ regular and satisfying~\eqref{eq-parametrized-perturbed-dbar}. Considering graphs of such maps in $\mathcal{S} \times X$ we endow it with Gromov-Hausdorff metric as in \cite{swaminathan2021Rel}. It is then naturally a $\mathcal{W}$-space with the obvious projection.
\begin{proposition}\label{thm:boundary-nodal-gluing}
    The functor $\mathfrak{M}$ is representable by rel-$C^\infty$ manifold $\overline{\mathcal{M}}^{\operatorname{reg}}_{\chi, \bm x, \sigma_0}(\pi_{\mathcal{S}}, D_{\mathcal{W}}) / \mathcal{W}$.
\end{proposition}

\begin{proof}[Sketch of proof]
    This theorem is an upgrade of \cite[Theorem A.4]{HH25openDM} to the case of perturbed Cauchy-Riemann operator with domain-dependent data $J_\mathcal{W}$ varying over $\mathcal{W}$. Due to the assumption that~\eqref{eq-data-trivial-near-nodes} in the neighborhood of the nodes in the versal setting (see \cite[Lemma 4.1]{swaminathan2021Rel}) all the local analysis of \cite{HH25openDM}, which is based on \cite[Appendix B]{Pardon2016virtual}, is the same. The only possible issue is the verification of sc-smoothness and sc-Fredholmness of the associated $\overline{\partial}_{J_\mathcal{W}}$-section (which clearly depends on all the data); in the unperturbed case, locally it is given by \cite[Equation (4.5)]{HWZ17}. Notice that sc-smoothness can be verified locally and by~\eqref{eq-data-trivial-near-nodes} in the neighborhood of the nodes, the analysis of \cite[Proposition 4.21]{HWZ17} and \cite[Theorem A.12]{HH25openDM} is applicable. Away from the nodes, this section is smooth since the data is assumed to vary smoothly over $\mathcal{W}$, see the discussion following \cite[Equation (4.16)]{HWZ17} for more details.
\end{proof}

\begin{remark}
    Note that the space $\overline{\mathcal{M}}^{\operatorname{reg}}_{\chi, \bm x, \sigma_0}(\pi_{\mathcal{S}}, D_{\mathcal{W}})$ does not contain SFT-type buildings and only involves stable nodal curves. Nevertheless, this is enough for our applications. We also note that only analysis of hyperbolic boundary nodes is required for our applications.
\end{remark}

\subsection{Chain map}

To define a morphism $\mathcal{F}$
\begin{equation}
    \mathcal{F}\colon SC^*_{\kappa, \mathrm{unsym}}(T^*Q)\to CM_{\kappa n-*}\bigl(\Lambda^1_{[\kappa]} (Q)\bigr),\\
\end{equation}
we first consider the canonical isomorphism
\begin{equation}\label{eq-mixedmoduli-orline}
      o_{\bm \gamma}[(2-n)(\ell-\chi)] \otimes o_{\bm x} \cong \big|\mathcal{T}^{\sigma_0}_{\chi, \ell}(\bm x,\bm \gamma)\big|^{-1} \otimes \zeta^{\kappa n}.
\end{equation}

When the dimension of the mixed moduli space is $0$ the above provides the isomorphism
\begin{equation}
    o_{\bm x} \cong \bigl(o_{\bm \gamma}[(2-n)(\ell-\chi)] \bigr)^{-1} \otimes \zeta^{\kappa n},
\end{equation}
which induces a morphism
\begin{equation}
    f_{[u, \bm \Gamma]} \colon o_{\bm x} \rightarrow \bigl(o_{\bm \gamma}[(2-n)(\ell-\chi)] \bigr)^{-1}[-\kappa n].
\end{equation}

We extend it to $o_{\bm x}[(n-2)j]$ for $j>0$ by simply multiplying the above with $\zeta^{(2-n)j}$ on the left.
Then $\mathcal{F}$ is given by
\begin{equation}\label{eq-F-}
     \mathcal{F}^-([\bm x]\hbar^j)=\sum_{\substack{|\bm \gamma|=\kappa n-|\bm x|-(n-2)(\ell-\chi),\\ \ell\geq0,\chi\leq0, \, \sigma_0 \in \mathfrak{S}_\kappa, \\
    (u, \bm \Gamma) \in \mathcal{T}^{\sigma_0}_{\chi, \ell}(\bm x,\bm \gamma) }} (-1)^{|\bm x|+(n-2)j} \cdot \frac{f_{(u, \bm \Gamma)}([\bm x]\hbar^j)}{(-\chi!)\cdot \ell! \cdot(\kappa!)^\ell}  .
\end{equation}

By Lemma \ref{lemma-F-compactness}, $\mathcal{T}^{\sigma_0}_{\chi, \ell}( \bm x,\bm \gamma)$ is finite, hence the map $\mathcal{F}$ is well-defined.
\begin{theorem}\label{thm-F-chain-map}
    The map $\mathcal{F}$ is a chain map.
\end{theorem}
\begin{proof}
 We analyze the codimension 1 boundary of $\mathcal{T}^{\sigma_0}_{\chi, \ell}( \bm x,\bm \gamma)$. A special codimension 1 phenomenon is the nodal degeneration along 
 the zero section $Q$ which, after adding a branched point at the preimage of the node as explained in Remark~\ref{rmk-typeII}, corresponds to a point $((S, \pi_*), u_*, \bm \Gamma_*)$ of the Type II singular space $\mathcal{T}^{\sigma_0, \mathrm{II}}_{\chi+1, \ell+1}( \bm x,\bm \gamma)$,  i.e. it corresponds to a $1$-parameter family with a switching coordinate $s_{h_1}$ approaching $s=0$ level. This point can also be regarded as a point of the Type I singular mixed moduli space $\mathcal{T}^{\sigma_1, \mathrm{I}}_{\chi, \ell}(\bm x,\bm \gamma)$ by adding an additional branch point on the domain (there are $(-\chi+1)$ ways to do this). Let us denote thus obtained tuple by  $((S^{\mathrm{I}}, \pi^{\mathrm{I}}_*), u^{\mathrm{I}}_*, \bm \Gamma^{\mathrm{I}}_*)$.
 
 We give a sketch of the proof that there is a $1$-dimensional family of points in $\mathcal{T}^{\sigma_1}_{\chi, \ell}(\bm x,\bm \gamma)$ converging to such a point. We note that 
 $$\mathscr{E}(u^{\mathrm{I}}_*, \bm s^{\mathrm{I}}, \bm \theta^{\mathrm{I}})= \operatorname{pr}_{\Lambda^1_{[\kappa]}(Q)}(c_{\rho_1}^| \circ sw_{I_{h_1}}^{\tau_1}(\bm s, \bm \theta, \mathscr{E}(u_*, \bm s, \bm \theta))).$$

 Let $V(\pi^{\mathrm{I}}_*)$ denote a small neighborhood of $(S^{\mathrm{I}}, \pi_*^{\mathrm{I}}) \in \overline{\mathscr{H}}^{\sigma, \sigma_0, \ge 0}_{\kappa, \chi+1}$, which we may assume to be diffeomorphic to $D^{-2\chi}_{\ge 0}$, i.e. the intersection of the unit ball in $\R^{-2\chi}$ with the upper half-space.
 Let  $U(u_*)$ be a neighborhood of $u^{\mathrm{I}}_*$ in $\overline{\mathscr{H}}^{\sigma, \sigma_0, \ge 0}_{\kappa, \chi}(\bm x)$ which projects to $V(\pi^{\mathrm{I}}_*)$ under the natural forgetful map.  In the view of Lemma~\ref{sft-to-hamiltonian}, we have that $U(u_*)/V(\pi^{\mathrm{I}}_*)$ is a rel-$C^\infty$ manifold by Proposition~\ref{thm:boundary-nodal-gluing}. The proof of Lemma~\ref{lemma-F-transversality} implies that 
\begin{equation}\label{eq: mixed-space-with-data}
    (U(u_*^{\mathrm{I}}) \times \operatorname{Conf}_{\bm h}(\R_{\le 0}) \times [0,1]^\ell)\tensor*[_{\mathscr{E}}]{\times}{_{\tilde E_0^s}} W_\ell(-, \bm \gamma; \mathfrak{X}_{\operatorname{sw}}^{\le 0}(\ell))
\end{equation}
is a rel-$C^\infty$ Banach manifold over $V(\pi^{\mathrm{I}}_*)$. It also holds that the data $\bm Y$ is a regular value of the natural projection from~\eqref{eq: mixed-space-with-data} to $\mathfrak{X}_{\operatorname{sw}}^{\le 0}(\ell)$. Therefore, by \cite[Corollary C.14]{HH25openDM} the compactification $\overline{\mathcal{T}}^{\sigma_0}_{\chi, \ell}( \bm x,\bm \gamma)$ is locally a rel-$C^\infty$ submanifold of~\eqref{eq: mixed-space-with-data}. Notice that it is a rel-$C^\infty$ manifold over a submanifold (with boundary) of $V(\pi_*^{\mathrm{I}})$.

Conversely, one may also show that for a $1$-dimensional family converging to a point of $\mathcal{T}^{\sigma_0, \mathrm{I}}_{\chi, \ell}(\bm x, \bm \gamma)$ and any $\rho_1^{\mathrm{I}} \in \mathfrak{S}_\kappa$ as in Definition~\ref{def-typeI} there is a continuation via a family of curves in $\mathcal{T}^{\sigma_0'}_{\chi+1, \ell+1}(\bm x, \bm \gamma).$

Together, the above two statements imply that counting points of singular type I moduli spaces $\mathcal{T}^{\sigma_0, \mathrm{I}}_{\chi, \ell}(\bm x, \bm \gamma)$ with weights $(\ell+1)\cdot\kappa!$ and the signs induced by boundary orientation of $\mathcal{T}^{\sigma_0}_{\chi, \ell}(\bm x, \bm \gamma)$ gives the opposite of the sum of points of $\mathcal{T}^{\sigma_0, \mathrm{II}}_{\chi+1, \ell+1}(\bm x, \bm \gamma)$ with signs and weights $(-\chi)$. That is

\begin{equation}\label{typeI-typeII-vanishing}
    (\ell+1)\cdot\kappa! \cdot \sum_{\substack{\sigma_0 \in \mathfrak{S}_\kappa, \\ (u_{\mathrm{I}}, \bm \Gamma_{\mathrm{I}}) \in \mathcal{T}^{\sigma_0, \mathrm{I}}_{\chi, \ell}(\bm x, \bm \gamma)}} \varepsilon^{\mathrm{I}}_{(u_{\mathrm{I}}, \bm \Gamma_{\mathrm{I}})}+(-\chi)\cdot\sum_{\substack{\sigma_0 \in \mathfrak{S}_\kappa, \\ (u_{\mathrm{II}}, \bm \Gamma_{\mathrm{II}}) \in \mathcal{T}^{\sigma_0, \mathrm{II}}_{\chi+1, \ell+1}(\bm x, \bm \gamma)}} \varepsilon^{\mathrm{II}}_{(u_{\mathrm{II}}, \bm \Gamma_{\mathrm{II}})}=0.
\end{equation}

According to Lemma~\ref{lemma-F-compactness}, the remaining codimension-$1$ degenerations correspond to stretching along the $s$-direction, and correspond to composing $\mathcal{F}^-$ with differentials of the Heegaard Floer symplectic cochain complex and of the Morse multiloop complex, see~\eqref{eq-mixed-moduli-boundary1} and~\eqref{eq-mixed-moduli-boundary2}. 
   
   Now, by counting elements of the boundaries of the spaces $\mathcal{T}^{\sigma_0}_{\chi, \ell}(\bm x, \bm \gamma)$ for all $\ell, -\chi \ge 0$ and $\sigma_0 \in \mathfrak{S}_\kappa$ with signs induced by boundary orientations and respective weights $\frac{1}{(-\chi!)\cdot \ell!\cdot(\kappa!)^\ell}$ we obtain
   \begin{equation}\label{eq-chainmorph-reln}
       (-1)^{\kappa n}\partial \circ \mathcal{F}^{-}-(-1)^{\kappa n}\mathcal{F}^- \circ d=0,
   \end{equation}
where vanishing of terms corresponding to points of singular moduli spaces follows from~\eqref{typeI-typeII-vanishing}.
It only remains to show that boundary orientation of $\overline{\mathcal{T}}^{\sigma_0}_{\chi, \ell}(\bm x, \bm \gamma)$ indeed provides the sign as stated above.

Let us consider a pair $$(u, (v, \bm \Gamma)) \in \mathscr{M}^{\chi_1}(\bm x, \bm x') \times \mathcal{T}^{\sigma_0}_{\chi_2, \ell}(\bm x', \bm \gamma) \subset \partial \overline{\mathcal{T}}^{\sigma_0}_{\chi, \ell}(\bm x, \bm \gamma) \text{ with } $$
\begin{equation}\label{eq-boundary-dim1}
|\bm x'|-|\bm x|+(n-2)\chi_1=1, \, \kappa n-|\bm x'|-|\bm \gamma|-(n-2)(\ell-\chi_2)=0 \text{ and }\chi_1+\chi_2=\chi,
\end{equation}
and compare the map $f_{(v, \bm \Gamma)}\circ d_{u}$ with the map $o_{\bm x} \to (o_{\bm \gamma}[(2-n)(\ell-\chi)])^{-1}[-\kappa n]$ induced by boundary orientation at this pair. We notice that under the condition $\kappa n -|\bm x|-|\bm \gamma|-(n-2)(\ell-\chi)=1$ the isomorphism~\eqref{eq-mixedmoduli-orline} is equivalent to 
\begin{equation}\label{eq-mixedmoduli-boundary-orline}
    (-1)^{|\bm \gamma|+(n-2)(\ell-\chi)+1}\big|\mathcal{T}^{\sigma_0}_{\chi, \ell}(\bm x, \bm \gamma) \big| \otimes o_{\bm x} \cong \bigl(o_{\bm \gamma}[(2-n)(\ell-\chi)]\bigr)^{-1} \otimes \zeta^{\kappa n}.
\end{equation}
Recalling~\eqref{eq-hurwitz-orline} we note that composition $f_{(v, \bm \Gamma)}\circ d_{u}$ is given by isomorphisms
\begin{equation}
    |\R \partial_su| \otimes o_{\bm x} \cong o_{\bm x'}[(2-n)\chi_1] \cong \bigl(o_{\bm \gamma}[(2-n)(\ell-\chi)] \bigr) \otimes \zeta^{\kappa n}.
\end{equation}
Clearly, $\partial_su$ is an outward pointing vector on $\partial \overline{\mathcal{T}}^{\sigma_0}_{\chi, \ell}(\bm x, \bm \gamma)$, which implies that $|\R \partial_su| \cong \big|\mathcal{T}^{\sigma_0}_{\chi, \ell}(\bm x, \bm \gamma) \big|$. We also point out that 
\begin{equation}
    |\bm \gamma|+(n-2)(\ell-\chi)+1=\kappa n+1-|\bm x'|-(n-2)\chi_1 ,
\end{equation}
which follows from~\eqref{eq-boundary-dim1}. Now recalling the sign in front of $f_{(v, \bm \Gamma)}$ in~\eqref{eq-F-} we conclude that the map induced by~\eqref{eq-mixedmoduli-boundary-orline} is equal to the restriction of $-(-1)^{\kappa n} \mathcal{F}^- \circ d$.

As for the other type of boundary contributions, we consider
$$((u, \bm \Gamma), [\bm \Gamma]) \in \mathcal{T}^{\sigma_0}_{\chi, \ell_1}(\bm x, \bm \gamma') \times \mathcal{P}(\bm \gamma', \bm \gamma; \ell_2) \text{ with}$$
\begin{equation}\label{eq-boundary-dim2}
    \kappa n-|\bm x|-|\bm \gamma'|-(n-2)(\ell_1-\chi_2)=0, \, |\bm \gamma'|-|\bm \gamma|-(n-2)\ell_2=1, \text{ and }\ell_1+\ell_2=\ell.
\end{equation}

The composition $\partial_{[\bm \Gamma]} \circ f_{(u, \Psi)}$ is then given by the composition
\begin{equation}
    (-|\R \partial_s \bm \Gamma|)\otimes o_{\bm x} \cong (-|\partial_s \bm \Gamma|)\otimes \big(o_{\bm \gamma'}[(2-n)(\ell_1-\chi)] \big)^{-1} [-\kappa n] \cong \big(o_{\bm \gamma'}[(2-n)(\ell-\chi)] \big)^{-1}\otimes [-\kappa n].
\end{equation}

The vector $-\partial_s \bm \Gamma$ is outward pointing, hence $-|\R \partial_s \bm \Gamma| \cong \big|\mathcal{T}^{\sigma_0}_{\chi, \ell}(\bm x, \bm \gamma) \big|$, which means that $\partial_{[\bm \Gamma]} \circ f_{(u, \bm \Gamma)}$ differs from the map induced by boundary orientation only by a term $(-1)^{|\bm \gamma|+(n-2)(\ell-\chi)+1}$. Now, noting that $|\bm \gamma|+(n-2)(\ell-\chi)+1=\kappa n-|\bm x|$, we conclude that boundary induced map coincides with $(-1)^{\kappa n} \partial \circ \mathcal{F}^-$.
\end{proof}
\begin{remark}
    We point out that our proof of Theorem~\ref{thm-F-chain-map} is based on the proof of the analogous statement for HDHF that is worked out in \cite{honda2025morse}.
\end{remark}
\subsection{The isomorphism}

For $\kappa=1$, the complex $SC_{\kappa, \mathrm{unsym}}^*(T^*Q)$ coincides with the symplectic chain complex $SC^*(T^*Q)$ computing the symplectic cohomology of $T^*Q$, whereas the complex $CM^*_\kappa(\Lambda^1_{[\kappa]} Q)$ coincides with the Morse chain complex $CM^*(\Lambda)$. Originally due to Viterbo \cite{viterbo1995generating, viterbo1999functors}, the homology groups of these two chain complexes are known to be isomorphic. In \cite{abouzaid2011cotangent}, Abouzaid gives a refined proof of this fact, accounting for the twisting of the symplectic homology depending on the second Stiefel-Whitney class of $Q$. We claim that the map constructed there (see \cite[Section 3.2]{abouzaid2011cotangent}) coincides with $\mathcal{F}$ on the level of homology. First, we notice that
the morphism $\mathcal{V}$ in \cite{abouzaid2011cotangent} on the chain level has as its target \emph{generic} (in Morse-theoretic sense) singular chains $C_*(\Lambda Q)$ of the loop space. We claim that $\mathcal{F}$ is chain homotopic to a composition of $\mathcal{V}$ with a map $\mathcal{A}_L$ from generic singular chains to Morse complex given by counting rigid flows from singular chains to critical points of the appropriate Morse function (see, e.g. \cite[Section 2]{hutchings-lee1999circle-valued}; we also refer the reader to a similar discussion in the case of based loop space in \cite[Section 5]{abouzaid2012wrapped}). Therefore, we conclude:

\begin{theorem}{\cite[Proposition 1.4]{abouzaid2011cotangent}}
\label{thm-viterbo}
    For $\kappa=1$, $\mathcal{F}$ induces a homology isomorphism.
\end{theorem}

\begin{remark}
    In the context of linear Hamiltonians the map constructed in \cite[Section 4]{abouzaid2015symplectic} is heuristically the same map that we construct for quadratic at infinity Hamiltonians. One may argue that the limit of the maps given there is equal to $\mathcal{F}$. 

    We also point out that the morphism $\mathcal{F}$ is proved to be an isomorphism in \cite{abouzaid2011cotangent} by showing that it is an inverse of the map constructed in \cite{abbondandolo2006floer}, and is known to be an isomorphism. It would be of its own interest to construct an analogue of this morphism for HFSH.
\end{remark}
\begin{lemma}\label{qi-hbar=0}
    The restriction of $\mathcal{F}$ to $\hbar=0$ is an isomorphism.
\end{lemma}
\begin{proof}
    Recall that with each tuple $\bm x\in \mathcal{P}_{\mathrm{unsym}}(T^*Q)$ (resp. $\bm \gamma\in CM_{-*}(\Lambda^1_{[\kappa]} Q)$) there is an associated permutation $\sigma(\bm x)\in \mathfrak{S}_\kappa$ (resp. $\sigma(\bm \gamma)\in \mathfrak{S}_\kappa$). 
    By definition the restriction of $SC^*_{\kappa, \mathrm{unsym}}(T^*Q)$ to $\hbar=0$ is split according to permutation type:
    \begin{equation*}
        (SC^*_\kappa(T^*Q)|_{\hbar=0}, d_0)=\bigoplus_{\sigma\in S_\kappa} (SC^*_\kappa(T^*Q)|_{\hbar=0,\sigma}, d_0),
    \end{equation*}
    where $SC^*_\kappa(T^*Q)|_{\hbar=0,\sigma}$ is generated by tuples of Hamiltonian orbits with associated permutation equal to $\sigma$.
    Similarly,
    \begin{equation*}
        (CM_{-*}(\Lambda^1_{[\kappa]} Q)|_{\hbar=0}, \partial_0)=\bigoplus_{\sigma\in S_\kappa} (CM_{-*}(\Lambda^1_{[\kappa]} Q)|_{\hbar=0,\sigma}, \partial_0).
    \end{equation*}

    Note that $\mathcal{F}|_{\hbar=0}$ also respects the splitting according to permutation type. 
    It suffices to show
    \begin{equation}
    \label{eq-F-0}
        \mathcal{F}|_{\hbar=0}\colon SC^*_\kappa(T^*Q)|_{\hbar=0,\sigma}\to CM_{\kappa n-*}(\Lambda^1_{[\kappa]} Q)|_{\hbar=0,\sigma}
    \end{equation}
    is an isomorphism. Note that for any generator $[\bm x]\in SC^*_\kappa(T^*Q)|_{\hbar=0,\sigma}$
    with $[\bm x] \in o_{\bm x}^{\Q}$ for some $\bm x \in \mathcal{P}_{\mathrm{unsym}}(T^*Q)$ the count $\mathcal{F}|_{\hbar=0}(\bm x)$ has non-zero entries only for elements in $o_{\bm \gamma}^{\Q}$ with $\bm \gamma$ such that $\sigma(\bm \gamma)=\sigma(\bm x)$. 
    We denote by $\gamma^{c_1(\sigma)}, \ldots, \gamma^{c_{ l(\sigma)}(\sigma)}$ the loops constituting $\bm \gamma$. Then we have the following.
    \begin{equation*}
        \mathcal{T}^{\text{Id}}_{0,0}(\bm x,\bm \gamma)\simeq\prod_{i=1}^{i= l(\sigma)}\mathcal{T}_{0, 0}(x^{c_i(\sigma)},\gamma^{c_i(\sigma)}).
    \end{equation*}
    It then suffices to show (\ref{eq-F-0}) holds when $\sigma$ is a cycle, which is given by Theorem \ref{thm-viterbo}.
\end{proof}

\begin{lemma}\label{qi-nonsense}
    Given cochain complexes $(C^*, d_0)$ and $(D_*, \partial_0)$. 
    Let $F$ be a degree $m$ chain map $$F \colon (\widetilde{C}^*, d) \to (\widetilde{D}_*, \partial)$$
    between associated complexes as in Construction~\ref{series-complex} given by 
    $$F=F_0+F_1\hbar+F_2\hbar^2+\dots.$$
    If $F_0 \colon (C^*,d_0) \to (D_*, \partial_0)$ is a quasi-isomorphism, then $F$ is a quasi-isomorphism as well.
\end{lemma}
\begin{proof}
For simplicity, we replace $(D_*, \partial_0)$ with the associated cochain complex $(D_s^*, \partial^0)$ obtained by shifting the degree by $m$. Similarly, we obtain $(\widetilde{D}_s^*, \partial)$. Then $F_s \colon (\widetilde{C}^*, d) \to (\widetilde{D}_s^*, \partial)$ is a chain map of degree $0$. 

    Now, we notice that 
    $$Cone(F_s)|_{\hbar=0}=Cone(F_0).$$
    We recall that a chain map is a quasi-isomorphism if and only if its cone is acyclic. According to the conditions of the lemma, $Cone(F_0)$ is acyclic, and hence $Cone(F_s)|_{\hbar=0}$ is acyclic. It remains to show that it implies that $Cone(F_s)$ is acyclic. From now on, denote this cochain complex via $(E^*, d_E)$.
    Let $$v=v_0+v_1\hbar+v_2\hbar^2+\dots$$ be a cocycle of degree $k$ in $E^k$. Here $|v_i|=k+i(n-2)$. Let us show that there is $w \in E^{k-1}$ such that $d_Ew=v$. First, since $(E^*|_{\hbar=0}, d_E^0)$ is a quasi-isomorphism, there is $w_0 \in E^{k-1}|_{\hbar=0}$ such that 
    \begin{equation}\label{acyclic-series}
        d_E^0 w_0=v_0.
    \end{equation}
    Then we may write
    $$v-d_Ew_0=\tilde{v}_1\hbar+\tilde{v}_2\hbar+\dots.$$
    Clearly, $v-d_E w_0$ is a cocycle, and since multiplication by $\hbar$ is a chain injective map, hence 
    $\tilde{v}=\tilde{v}_1+\tilde{v}_2\hbar+\ldots$
    is a cocycle. Arguing as in~(\ref{acyclic-series}) there is a $w_1 \in E^{k+(n-2)-1}$ satisfying
    $$d_E^0 w_1=\tilde{v}_1.$$
    Clearly, one has 
    $$v-d_E(w_0+w_1\hbar)=\tilde{\tilde{v}}_2\hbar^2+\dots.$$
    Repeating the procedure above one obtains a sequence $w_0, w_1, \ldots,$ such that $$d_E(w_0+w_1\hbar+w_2\hbar^2+\dots)=v.$$ \vskip-.2in\end{proof}
\begin{theorem}\label{thm-f-isomorphism}
    $\mathcal{F}$ induces an isomorphism on homology.
\end{theorem}
\begin{proof}
    This statement follows from Lemmas~\ref{qi-hbar=0} and~\ref{qi-nonsense}.
\end{proof}
\clearpage
\printnomenclature

\printbibliography

\end{document}